\documentclass[a4paper]{article}

\usepackage[top=1.2in, bottom=1.5in, left=1.1in, right=1.8in]{geometry}
\usepackage{amsfonts}
\usepackage{amsmath}
\usepackage{amsthm,amssymb}
\usepackage{CJK,graphicx}
\usepackage{amscd}
\usepackage{amssymb}
\usepackage{mathrsfs}
\usepackage[all,cmtip]{xy}
\usepackage{lmodern}
\usepackage[Symbol]{upgreek}
\usepackage{bm}
\usepackage{eucal}
\usepackage[nopar]{lipsum}
\usepackage{rotating}
\usepackage{tikz}
\usepackage{calc}
\usepackage{tikz-cd}
\usepackage{mathabx}
\usepackage{tikz-3dplot}
\usepackage{hyperref}
\usetikzlibrary{calc}
\usetikzlibrary{decorations.markings,arrows}
\usetikzlibrary{shapes.geometric}
\usepackage{lipsum}
\usepackage{makecell}
\usepackage{times}

\newsavebox\CBox
\newcommand\hcancel[2][0.5pt]{%
	\ifmmode\sbox\CBox{$#2$}\else\sbox\CBox{#2}\fi%
	\makebox[0pt][l]{\usebox\CBox}%
	\rule[0.5\ht\CBox-#1/2]{\wd\CBox}{#1}}

\newtheorem{theorem}{Theorem}
\newtheorem{corollary}[theorem]{Corollary}
\newtheorem{definition}[theorem]{Definition}
\newtheorem{lemma}[theorem]{Lemma}
\newtheorem{proposition}[theorem]{Proposition}

\newtheorem{conjecture}[theorem]{Conjecture}

\newtheorem{remark}[theorem]{Remark}

\numberwithin{equation}{section}

\begin{document}
	
\title{\textbf{Aspherical Lagrangian submanifolds, Audin's conjecture and cyclic dilations}}\author{Yin Li}
\newcommand{\Addresses}{{
		\bigskip
		\footnotesize
		\textsc{Department of Mathematics, Uppsala University, 753 10 Uppsala, Sweden}\par\nopagebreak
		\textit{E-mail address}: \texttt{yin.li@math.uu.se}
	}}
\date{}\maketitle

\begin{abstract}
Given a closed, oriented Lagrangian submanifold $L$ in a Liouville domain $\overline{M}$, one can define a Maurer-Cartan element with respect to a certain $L_\infty$-structure on the string homology $\widehat{H}_\ast^{S^1}(\mathcal{L}L;\mathbb{R})$, completed with respect to the action filtration. When the first Gutt-Hutchings capacity \cite{gh} of $\overline{M}$ is finite, and $L$ is a $K(\pi,1)$ space, we show that $L$ bounds a pseudoholomorphic disc of Maslov index 2. This confirms a general form of Audin's conjecture \cite{ma} and generalizes the works of Fukaya \cite{kf} and Irie \cite{ki2} in the case of $\mathbb{C}^n$ to a wide class of Liouville manifolds, which includes low degree smooth affine hypersurfaces in $\mathbb{C}^{n+1}$. In particular, when $\dim_\mathbb{R}(\overline{M})=6$,  every closed, orientable, prime Lagrangian 3-manifold $L\subset\overline{M}$ is diffeomorphic either to a spherical space form, or $S^1\times\Sigma_g$, where $\Sigma_g$ is a closed oriented surface.
\end{abstract}

\section{Introduction}

\subsection{Closed Lagrangian submanifolds in $\mathbb{C}^n$}\label{section:Lag}

The study of Lagrangian submanifolds plays a central role in symplectic geometry, whose importance is reflected by Weinstein's famous creed ``everything is a Lagrangian submanifold". As one of the first non-trivial consequences of Gromov's ground breaking work \cite{mg} on pseudoholomorphic curves, we know that $H^1(L;\mathbb{Q})\neq0$ for any closed Lagrangian submanifold $L\subset\mathbb{C}^n$. 

Gromov's argument for the non-vanishing of $H^1(L;\mathbb{Q})$ can be summarized as follows. Take a compactly supported Hamiltonian function $H_t:[0,1]\times\mathbb{C}^n\rightarrow\mathbb{R}$ which displaces $L$ from itself, denote by $X_{H_t}$ its associated vector field. Choose a cut-off function $\chi:\mathbb{R}\rightarrow[0,1]$ such that $\chi\equiv0$ on $\mathbb{R}_{\leq0}$ and $\chi\equiv1$ on $\mathbb{R}_{\geq1}$. For each $r\geq0$, consider the function $\chi_r(s):=\chi(r+s)\chi(r-s)$. Let $D\subset\mathbb{C}$ be the closed unit disc, we identify $D\setminus\{-1,1\}$ with the strip $\mathbb{R}\times[0,1]$, and denote by $s$ and $t$ the coordinates on $\mathbb{R}$ and $[0,1]$, respectively. Consider the moduli space $\mathcal{N}^r(L,\beta)$ of maps $u:(D,\partial D)\rightarrow(\mathbb{C}^n,L)$ in the relative homotopy class $\beta\in\pi_2(\mathbb{C}^n,L)$, which satisfy the perturbed Cauchy-Riemann equation
\begin{equation}\label{eq:p-CR}
\left(du-X_{\chi_r(s)H_t}(u)\otimes dt\right)^{0,1}=0,
\end{equation}
where the $(0,1)$-part is taken with respect to the standard complex structure on $\mathbb{C}^n$. Note that when $r=0$, $\chi_0\equiv0$, so (\ref{eq:p-CR}) reduces to the ordinary Cauchy-Riemann equation. When $\beta=0$, one can show that $\mathcal{N}^1(L,0)=\emptyset$, so $\mathcal{N}^{\geq0}(L,0):=\bigcup_{r\in[0,1]}\mathcal{N}^r(L,0)$ has a single boundary component $\mathcal{N}^0(L,0)$, which consists of constant solutions to (\ref{eq:p-CR}). Since $\left[\mathcal{N}^0(L,0)\right]\neq0$ in $H_n(\mathcal{L}L;\mathbb{R})$, where $\mathcal{L}L$ is the free loop space of $L$, the moduli space $\mathcal{N}^{\geq0}(L,0)$ cannot be compact, and there is a divergent sequence of maps $(u_k)_{k\geq1}$ in $\mathcal{N}^{\geq0}(L,0)$. It follows from elliptic regularity and the removable singularity theorem that $L$ bounds a non-constant holomorphic disc, which implies that $H^1(L;\mathbb{Q})\neq0$.

Fukaya's new observation, which was sketched in \cite{kf} and carried out in detail by Irie in \cite{ki2,ki3}, is that the new boundary component arising from compactifying the moduli space $\mathcal{N}(L,\beta)$ can be expressed in terms of string topology operations after pushing forward its virtual fundamental chain to the free loop space $\mathcal{L}L$. More precisely, for a suitable $L_\infty$-structure $(\ell_k)_{k\geq1}$ on the homology group $H_\ast(\mathcal{L}L;\mathbb{R})$, there holds
\begin{equation}\label{eq:def}
\sum_{k=2}^\infty\frac{1}{(k-1)!}\ell_k(y,x,\cdots,x)=(-1)^{n+1}[L],
\end{equation}
where $[L]\in H_n(\mathcal{L}L;\mathbb{R})$ is the cycle of constant loops, and $x\in\widehat{H}_\ast(\mathcal{L}L;\mathbb{R})$ is a Maurer-Cartan element. Both of $x$ and $y$ lie in the \textit{completed} free loop space homology. Here, the completion is taken with respect to the energy filtration
\begin{equation}
F^\Xi H_{n-\ast}(\mathcal{L}L;\mathbb{R}):=\bigoplus_{\omega_\mathit{std}(\bar{a})>\Xi}H_{n-\ast}(\mathcal{L}(a)L;\mathbb{R}),
\end{equation}
where $\omega_\mathit{std}$ denotes the standard symplectic form on $\mathbb{C}^n$, $\bar{a}\in H_2(\mathbb{C}^n,L)$ satisfies $\partial\bar{a}=a\in H_1(L;\mathbb{Z})$, and $\mathcal{L}(a)L\subset\mathcal{L}L$ consists of loops in the class $a$. When $L$ is a $K(\pi,1)$ space (for technical reasons, one also assumes that $L$ is orientable and relatively \textit{Spin}), the identity (\ref{eq:def}) implies that $L$ bounds a non-constant holomorphic disc of Maslov index 2. This provides a refinement of Gromov's result and in particular confirms Audin's conjecture \cite{ma}. More interestingly, when $n=3$, it leads to the following classification of prime Lagrangian submanifolds.

\begin{theorem}[Fukaya \cite{kf}, Irie \cite{ki2}]\label{theorem:3}
A closed, connected, orientable and prime 3-manifold $L$ admits a Lagrangian embedding into $\mathbb{C}^3$ if and only if it is diffeomorphic to $S^1\times\Sigma_g$, where $\Sigma_g$ is a closed, orientable surface.
\end{theorem}

Under the additional assumption that $L$ is monotone, one can prove the same result without assuming that $L$ is prime, see \cite{ek}.
\bigskip

In this paper, we extend the ideas of Fukaya and Irie to study closed Lagrangian submanifolds in a more general class of Liouville domains, namely those with finite first Gutt-Hutchings capacity (see Section \ref{section:GH} for details). We prove a general version of Audin's conjecture (cf. Corollary \ref{corollary:n}) and a generalization of Theorem \ref{theorem:3} (cf. Corollary \ref{corollary:3}), which in particular enables us to classify the prime Lagrangian 3-folds in many Milnor fibers. See Remark \ref{remark:example}.

\subsection{Cyclic dilations and Gutt-Hutchings capacities}\label{section:GH}

In this paper, by a \textit{Liouville domain} we mean a compact symplectic manifold with boundary $(\overline{M},d\theta_M)$, such that the \textit{Liouville vector field} $Z_M$ defined by $i_{Z_M}d\theta_M=\theta_M$ points strictly outwards along the boundary $\partial\overline{M}$. A \textit{Liouville manifold} is the completion of $\overline{M}$ by attaching the cylindrical end $[1,\infty)\times\partial\overline{M}$, i.e.
\begin{equation}
M=\overline{M}\cup_{\partial\overline{M}}[1,\infty)\times\partial\overline{M}.
\end{equation}
It follows that $(M,d\theta_M)$ is an exact symplectic manifold without boundary such that the Liouville vector field $Z$ is complete.

In order to generalize the work of Fukaya and Irie to more interesting Liouville manifolds, we will use a notion introduced by the author \cite{yl2} (also independently by Zhou \cite{zz2}). From now on, $M$ will be a $2n$-dimensional Liouville manifold with $c_1(M)=0$, and we fix a trivialization of its canonical bundle $K_M$, so that its symplectic cohomology $\mathit{SH}^\ast(M)$ and (positive) $S^1$-equivariant symplectic cohomology $\mathit{SH}_{S^1}^\ast(M)$ (see $\cite{bo}$ or $\cite{ps2}$ for its definition) are well-defined $\mathbb{Z}$-graded vector spaces over some field $\mathbb{K}$ (for the purpose of this paper, $\mathbb{K}\supset\mathbb{Q}$ will be a field containing rational numbers, or simply $\mathbb{R}$). As is the case of singular homologies, the $S^1$-equivariant symplectic cohomology is related to the ordinary symplectic cohomology by the (Gysin) long exact sequence
\begin{equation}\label{eq:Gysin}
\cdots\rightarrow\mathit{SH}^{\ast-1}(M)\xrightarrow{I}\mathit{SH}_{S^1}^{\ast-1}(M)\xrightarrow{S}\mathit{SH}_{S^1}^{\ast+1}(M)\xrightarrow{B}\mathit{SH}^\ast(M)\rightarrow\cdots,
\end{equation}
where the erasing map $I$ is induced by the obvious inclusion of cochain complexes, and the marking map $B$ is defined in terms of the higher BV structures on the cochain complex $\mathit{SC}^\ast(M)$ underlying $\mathit{SH}^\ast(M)$.

\begin{definition}
A cohomology class $\tilde{b}\in\mathit{SH}_{S^1}^1(M)$ is called a cyclic dilation if its image under the marking map $B:\mathit{SH}_{S^1}^\ast(M)\rightarrow\mathit{SH}^{\ast-1}(M)$ gives the identity, i.e. $B(\tilde{b})=1$.
\end{definition}

\begin{remark}
Note that the terminology used in this article differs from that of $\cite{yl2}$, where $\tilde{b}\in\mathit{SH}_{S^1}^1(M)$ is called a cyclic dilation as long as its image under the marking map $B$ hits an invertible element $h\in\mathit{SH}^0(M)^\times$. Here we reserve the terminology for the more restrictive situation when $h=1$ is the identity, and refer to the general case when $h$ is allowed to be an arbitrary invertible element as a cyclic quasi-dilation, see Definition \ref{definition:cyclic-quasi}. Requiring $h=1$ is equivalent to the existence of a $k$-dilation for some $k\in\mathbb{N}$ under the terminology of Zhou $\cite{zz2}$.
\end{remark}

The notion of a cyclic dilation generalizes that of a \textit{dilation} introduced earlier by Seidel-Solomon $\cite{ss}$. Recall that a dilation is a class $b\in\mathit{SH}^1(M)$ such that its image under the Batalin-Vilkovisky (BV) operator $\Delta:\mathit{SH}^\ast(M)\rightarrow\mathit{SH}^{\ast-1}(M)$ hits the identity, i.e. $\Delta(b)=1$. Every dilation gives rise to a cyclic dilation via the erasing map $I:\mathit{SH}^\ast(M)\rightarrow\mathit{SH}_{S^1}^\ast(M)$. See $\cite{yl2}$, Section 4.2 for details. However, the converse is not true. A typical example is the Milnor fiber of a 3-fold triple point
\begin{equation}
\{z_1^3+z_2^3+z_3^3+z_4^3=0\}\subset\mathbb{C}^4,
\end{equation}
in which case there is no dilation in $\mathit{SH}^1(M)$ (cf. \cite{ps4}, Example 2.7), but there is a cyclic dilation in $\mathit{SH}_{S^1}^1(M)$ (cf. \cite{yl2}, Section 6.1 and \cite{zz2}, Section 5.2). The study of cyclic dilations enables us to extend some of Seidel-Solomon's results to a wider class Liouville manifolds, among which is the following obstruction to exact Lagrangian embedding.

\begin{proposition}[$\cite{yl2}$, Corollary 40]\label{proposition:aspherical}
Let $M$ be a Liouville manifold which admits a cyclic dilation $\tilde{b}\in\mathit{SH}_{S^1}^1(M)$. Then $M$ does not contain a closed exact Lagrangian submanifold which is a $K(\pi,1)$ space.
\end{proposition}

The existence of a cyclic dilation $\tilde{b}\in\mathit{SH}_{S^1}^1(M)$ can also be interpreted in terms of certain symplectic capacities of the Liouville domain $\overline{M}\subset M$. Recall that the cochain complex computing $\mathit{SH}_{S^1}^\ast(M)$ is defined as
\begin{equation}\label{eq:SC}
\left(\mathit{SC}^\ast_{S^1}(M):=\mathit{SC}^\ast(M)\otimes_\mathbb{K}\mathbb{K}(\!(u)\!)/u\mathbb{K}[\![u]\!],\partial^{S^1}:=\partial+u\delta_1+u^2\delta_2+\cdots\right),
\end{equation}
where $\partial$ is the ordinary Floer differential, $\delta_1$ is the cochain level BV operator, and $u$ is a formal variable of degree 2. The action filtration on $\mathit{SC}^\ast(M)$ induces a filtration $F^\bullet$ on $\mathit{SC}^\ast_{S^1}(M)$, and the \textit{$k$-th Gutt-Hutchings capacity} of $\overline{M}$, introduced in $\cite{gh}$, is defined to be
\begin{equation}
c_k^\mathit{GH}(M):=\inf\left\{a\left\vert\partial^{S^1}(x)=u^{-k+1}e_M\textrm{ for some }x\in F^{\leq a}\mathit{SC}_{S^1}^\ast(M)\right.\right\},
\end{equation}
where $e_M\in\mathit{SC}^0(M)$ is the cochain level representative of the identity $1\in\mathit{SH}^0(M)$. It follows from the definition that (cf. \cite{yl2}, Section 4.2)

\begin{proposition}
A Liouville manifold $M$ admits a cyclic dilation if and only if the Liouville domain $\overline{M}$ has finite first Gutt-Hutchings capacity, i.e. $c_1^\mathit{GH}(M)<\infty$.
\end{proposition}

In other words, $M$ admits a cyclic dilation if and only if the identity $1\in\mathit{SH}^0(M)$ lies in the kernel of the erasing map $I:\mathit{SH}^0(M)\rightarrow\mathit{SH}_{S^1}^0(M)$.
\bigskip

In this paper, we will generalize the results of Fukaya and Irie described in Section \ref{section:Lag} to Liouville manifolds with cyclic dilations. To describe the motivation of considering this particular class of Liouville manifolds, we appeal to a more conceptual and largely speculative interpretation of Fukaya and Irie's results, with a focus on the key identity (\ref{eq:def}), which the author learned from \cite{jlf}, Section 6. Let $\overline{N}$ be another Liouville domain, and $N$ its completion as a Liouville manifold. For a Liouville embedding $\iota_L:\overline{N}\hookrightarrow M$, Viterbo functoriality $\cite{cv}$ gives rise to a map $\mathit{SH}^\ast(M)\rightarrow\mathit{SH}^\ast(N)$, which is, among other things, a morphism between $\mathbb{Z}$-graded $\mathbb{K}$-algebras. Gromov's theorem on Lagrangian embeddings in $\mathbb{C}^n$ then follows from the existence of such a map and the fact that $\mathit{SH}^\ast(\mathbb{C}^n)=0$. In general, when we have a (possibly non-exact) symplectic embedding $\iota_S:\overline{N}\hookrightarrow M$, the original version of Viterbo functoriality no longer holds. However, from the symplectic embedding $\iota_S$ one can construct a Maurer-Cartan element $s\in\mathit{SC}^\ast(N)$ by counting, roughly speaking, holomorphic thimbles in the symplectic cap $\overline{M\setminus\overline{N}}$ (the closure of $M\setminus\overline{N}$) which are asymptotic to the Reeb orbits in $\partial\overline{N}$, see $\cite{bsv}$ and $\cite{gs}$, Section 4.4. One can then deform the Floer differential $\partial$ using the Maurer-Cartan element $s$ to obtain a twisted cochain complex
\begin{equation}
\left(\mathit{SC}^\ast(N),\partial+\sum_{k=2}^\infty\frac{1}{(k-1)!}\ell_k(\cdot,s,\cdots,s)\right),
\end{equation}
whose cohomology we denote by $\mathit{SH}^\ast(N,s)$. Here,
\begin{equation}
\ell_k:\mathit{SC}^\ast(N)^{\otimes k}\rightarrow\mathit{SC}^{\ast+3-2k}(N)
\end{equation}
is the $k$-th $L_\infty$-structure on the (undeformed) symplectic cochain complex $\mathit{SC}^\ast(N)$ (cf. \cite{dt}, Section 4.2). This enables us to define a $\mathbb{K}$-algebra map
\begin{equation}\label{eq:tw-Vit}
\mathit{SH}^\ast(M)\rightarrow\mathit{SH}^\ast(N,s),
\end{equation}
which should be viewed as the Viterbo functoriality for the \textit{non-exact} symplectic embedding $\iota_S$. Note that when $\overline{N}$ is the Weinstein neighborhood of a closed Lagrangian submanifold $L\subset M$, the Maurer-Cartan cochain $s$ can be equivalently interpreted as a chain $s_L\in C_{n-\ast}(\mathcal{L}L;\mathbb{R})$ on the free loop space, defined by evaluating the fundamental chain of the moduli space of holomorphic discs in $M$ with boundary on $L$, via a domain-stretching argument similar to that of \cite{dt}. In the special case of a Lagrangian embedding $L\hookrightarrow\mathbb{C}^n$, it follows from (\ref{eq:tw-Vit}) that
\begin{equation}\label{eq:van}
\mathit{SH}^\ast(T^\ast L,s)\cong H_{n-\ast}(\mathcal{L}L,s_L)=0,
\end{equation}
where $H_{n-\ast}(\mathcal{L}L,s_L)$ denotes the free loop space homology deformed by $s_L$, using the chain level loop bracket. (\ref{eq:def}) then follows from (\ref{eq:van}) since the inclusion of constant loops $L\hookrightarrow\mathcal{L}L$ defines a coboundary in the deformed chain complex $C_{n-\ast}(\mathcal{L}L,s_L)$.

It is reasonable to expect that similar considerations as above would work in the $S^1$-equivariant setting. In terms of symplectic field theory, this is known as the \textit{Cieliebak-Latschev formalism}. Namely one should be able to construct a Maurer-Cartan element $\tilde{s}\in\mathit{SC}_{S^1}^\ast(N)$ associated to the symplectic embedding $\iota_S:\overline{N}\hookrightarrow M$, so that the $S^1$-equivariant analogue of Viterbo functoriality (also known as the \textit{Cieliebak-Latschev map}, cf. \cite{cl,cg}) exists after deformation, which gives rise to a morphism
\begin{equation}\label{eq:CL}
\mathit{SH}^\ast_{S^1}(M)\rightarrow\mathit{SH}_{S^1}^\ast(N,\tilde{s}),
\end{equation}
where the right-hand side is defined using the $L_\infty$-structure
\begin{equation}
\tilde{\ell}_k:\mathit{SC}_{S^1}^\ast(N)^{\otimes k}\rightarrow\mathit{SC}_{S^1}^{\ast+4-3k}(N).
\end{equation}
The analogue of (\ref{eq:def}) in the $S^1$-equivariant case translates to the fact that the lifting $\tilde{e}_N:=e_N\otimes1\in\mathit{SC}_{S^1}^0(N)$, where $e_N\in\mathit{SC}^0(N)$ is the identity, defines a coboundary in the deformed cochain complex
\begin{equation}
\left(\mathit{SC}_{S^1}^\ast(N),\partial_{S^1}+\sum_{k=2}^\infty\frac{1}{(k-1)!}\tilde{\ell}_k(\cdot,\tilde{s},\cdots,\tilde{s})\right).
\end{equation}
In order to obtain an analogue of (\ref{eq:def}) for a more general class of Liouville manifolds, it is therefore natural to impose the condition that the image of the identity $e_M\in\mathit{SC}^\ast(M)$ defines a coboundary under the obvious inclusion $\mathit{SC}^\ast(M)\hookrightarrow\mathit{SC}_{S^1}^\ast(M)$. As we have seen in the discussions above, this means that $M$ admits a cyclic dilation.

\subsection{New results}\label{section:results}

We turn to the actual substance of this paper. Let $M$ be a Liouville manifold with $c_1(M)=0$. As part of the general set up, we fix a background class $[\alpha]\in H^2(M;\mathbb{Z}_2)$ and a $\mathbb{Z}_2$-gerbe $\alpha$ representing this class. From now on, whenever we write $\mathit{SH}^\ast(M)$ and $\mathit{SH}_{S^1}^\ast(M)$, they should be understood as symplectic cohomologies twisted by $\alpha$.

Let $L\subset M$ be a closed Lagrangian submanifold. Recall that for $a\in H_1(L;\mathbb{Z})$, $\mathcal{L}(a)L\subset\mathcal{L}L$ is the subspace consisting of loops in $a$. We introduce the notation
\begin{equation}\label{eq:StrH}
\mathbb{H}^{S^1}_\ast(a):=H_{\ast+n+\mu(a)-1}^{S^1}(\mathcal{L}(a)L;\mathbb{R}),
\end{equation}
where $\mu:H_1(L;\mathbb{Z})\rightarrow\mathbb{Z}$ is the Maslov index, and the right-hand side is the $S^1$-equivariant homology of $\mathcal{L}(a)L$, with the $S^1$-action given by reparametrization of loops. Consider the direct sum
\begin{equation}\label{eq:H-eq}
\mathbb{H}_\ast^{S^1}:=\bigoplus_{a\in H_1(L;\mathbb{Z})}\mathbb{H}^{S^1}_\ast(a),
\end{equation}
which carries the aforementioned action filtration
\begin{equation}
F^\Xi\mathbb{H}_\ast^{S^1}:=\bigoplus_{\theta_M(a)>\Xi}\mathbb{H}^{S^1}_\ast(a),
\end{equation}
where $\theta_M$ is the Liouville 1-form on $M$. This enables us to define the completion
\begin{equation}\label{eq:H-c}
\widehat{\mathbb{H}}_\ast^{S^1}:=\varprojlim_{\Xi\rightarrow\infty}\mathbb{H}_\ast^{S^1}/F^\Xi\mathbb{H}_\ast^{S^1}.
\end{equation}

\begin{theorem}\label{theorem:main}
Let $M$ be a $2n$-dimensional Liouville manifold with $c_1(M)=0$, and assume that $M$ admits a cyclic dilation. Let $L\subset M$ be a closed Lagrangian submanifold which is oriented and relatively $\mathit{Spin}$ with respect to $\alpha$. Then there exists an $L_\infty$-structure $(\tilde{\ell}_k)_{k\geq1}$ on $\mathbb{H}_\ast^{S^1}$, together with homology classes $\tilde{x}\in\widehat{\mathbb{H}}_{-2}^{S^1}$, $\tilde{y}\in\widehat{\mathbb{H}}_2^{S^1}$, such that
\begin{itemize}
	\item[(i)] $\tilde{\ell}_1=0$.
	\item[(ii)] The $L_\infty$-structure $(\tilde{\ell}_k)_{k\geq1}$ respects the decomposition of $\mathbb{H}_\ast^{S^1}$ according to $a\in H_1(L;\mathbb{Z})$. In particular, it extends to the completion $\widehat{\mathbb{H}}_\ast^{S^1}$, and we will use the same notation $(\tilde{\ell}_k)_{k\geq1}$ to denote its extension.
	\item[(iii)] There exists a constant $c>0$ such that $\tilde{x}\in F^c\widehat{\mathbb{H}}_{-2}^{S^1}$.
	\item[(iv)] $\tilde{x}$ and $\tilde{y}$ satisfy the following relations:
	\begin{equation}
	\sum_{k=2}^\infty\frac{1}{k!}\tilde{\ell}_k(\tilde{x},\cdots,\tilde{x})=0,
	\end{equation}
    \begin{equation}\label{eq:def1}
    \left(\sum_{k=2}^\infty\frac{1}{(k-1)!}\tilde{\ell}_k(\tilde{y},\tilde{x},\cdots,\tilde{x})\right)_{a=0}=(-1)^{n+1}[\![L]\!],
    \end{equation}
    where the infinite sums on the left-hand side make sense because of (iii), the subscript $a=0$ means throwing away the high energy part, and $[\![L]\!]$ denotes the image of the fundamental class of $L$ under the composition
    \begin{equation}
    H_\ast(L;\mathbb{R})\rightarrow H_\ast(\mathcal{L}(0)L;\mathbb{R})\rightarrow H_\ast^{S^1}(\mathcal{L}(0)L;\mathbb{R}),
    \end{equation}
    where the first map is induced by the inclusion of constant loops, and the second map is the erasing map.
\end{itemize}
\end{theorem}

Note that when $L\subset M$ is $\mathit{Spin}$, we can pick $\alpha=0$ in the above theorem. This shall be our default choice when dealing with Lagrangian submanifolds with $\mathit{Spin}$ structures.

\begin{remark}\label{remark:example}
It follows from \cite{zz2}, Section 5.2 that Theorem \ref{theorem:main} applies to any Milnor fiber of the ``low degree" Brieskorn singularities in $\mathbb{C}^{n+1}$ defined by the polynomials
\begin{equation}
z_1^k+\cdots+z_{n+1}^k=0,\textrm{ }2\leq k\leq n.
\end{equation}
On the other hand, a more elementary argument by Seidel-Solomon (cf. \cite{ss}, Section 7) implies that the assumptions of Theorem \ref{theorem:main} are satisfied for any Milnor fiber of an isolated singularity which is triply stabilized.

Other examples of Liouville manifolds with cyclic dilations include cotangent bundles of rationally inessential manifolds. See \cite{yl2}, Section 4.2.
\end{remark}

In order to prove Theorem \ref{theorem:main}, we will follow the general strategies of Fukaya \cite{kf} and Irie \cite{ki2}, and extend them to the $S^1$-equivariant case. This will involve some new constructions in chain level string topology (cf. Section \ref{section:BV}) and some new (parametrized) moduli spaces defined using holomorphic maps with Hamiltonian asymptotic and Lagrangian boundary conditions, which we will introduce in Sections \ref{section:disc} and \ref{section:CG}. Analyzing these moduli spaces constitutes the main technical part of this paper. The heuristic argument presented in Section \ref{section:GH} has many conceptual advantages, but can be hard to carry out for technical reasons. To name a few, the geometric $L_\infty$-structure on $\mathit{SC}^\ast(M)$ hasn't been rigorously constructed\footnote{After the first draft of this paper was written, there are now several constructions of the $L_\infty$-structure on $\mathit{SC}^\ast(M)$, see for example \cite{beas} and \cite{pos}.}, and the general form of the non-exact Viterbo functoriality (\ref{eq:tw-Vit}) hasn't been established. In the $S^1$-equivariant case, significant progress has been made by Cieliebak-Latschev \cite{jlt} to construct the morphism (\ref{eq:CL}) from the perspective of symplectic field theory, which doesn't fully fit with our purposes as the linearized contact homology doesn't take into account of the unit $e_M\otimes1\in\mathit{SC}^\ast_{S^1}(M)$.

As a Corollary to Theorem \ref{theorem:main}, we prove the following general version of Audin's conjecture for Lagrangian $K(\pi,1)$ spaces in Liouville manifolds with finite first Gutt-Hutchings capacity $c_1^\mathit{GH}(M)$.

\begin{corollary}\label{corollary:n}
Let $M$ be a Liouville manifold with $c_1^\mathit{GH}(M)<\infty$, and $L\subset M$ a closed, orientable Lagrangian submanifold which is a $K(\pi,1)$ space, and $\mathit{Spin}$ relative to $\alpha$. Then there exists a finite covering $\widetilde{L}$ of $L$ which is homotopy equivalent to a product $S^1\times K$ for some $(n-1)$-dimensional CW complex $K$. Moreover, $\pi_1(\widetilde{L})\subset\pi_1(L)$ is the centralizer of some element $\gamma\in\pi_1(L)$ which has Maslov class equal to 2 and positive symplectic area.
\end{corollary}

\begin{remark}\label{remark:Audin}
The original conjecture of Audin \cite{ma} states that any Lagrangian torus $L\subset\mathbb{C}^n$ has Maslov number $2$. It is first rigorously proved by Cieliebak-Mohnke \cite{cm} using neck stretching for holomorphic curves with local tangency constraints. An alternative approach was outlined earlier by Fukaya \cite{kf} and realized later by Irie \cite{ki2,ki3}, which shows that the same conclusion holds for Lagrangian submanifolds $L\subset\mathbb{C}^n$ which are $K(\pi,1)$ spaces and $\mathit{Spin}$. When $M$ is a subcritical Weinstein manifold, and $L\subset M$ is a monotone Lagrangian $K(\pi,1)$ space (without the $\mathit{Spin}$ assumption), this is proved by Damian \cite{md}. The holomorphic curves used in our argument are actually more related to those appeared in Cieliebak-Mohnke \cite{cm}, although the general idea of the proof is to more extent inspired by Fukaya \cite{kf}. 

Our generalization above has the feature that $L\subset M$ doesn't need to be displaceable. For example, it is shown by Lekili-Maydayskiy \cite{lm} that non-displaceable monotone Lagrangian tori exist in 4-dimensional $A_k$ Milnor fibers for any $k\geq1$. In \cite{ad}, a $1$-parameter family of monotone Lagrangian tori in $T^\ast S^3$ is shown to split-generate the monotone Fukaya category. In all these examples, the Lagrangian tori have minimal Maslov number $2$.
\end{remark}

\begin{remark}
In another direction, Keating constructed infinitely many monotone Lagrangian $K(\pi,1)$ spaces (not necessarily orientable) in certain affine hypersurfaces of complex dimensions 2 and 3 with any possible minimal Maslov index. See \cite{ak3}, Theorems 1.1 and 1.2. The affine hypersurfaces appeared in her construction are of the form
\begin{equation}
\left\{z_1^2+z_2^4+z_3^k=1\right\}\subset\mathbb{C}^3,\textrm{ where }k\gg1
\end{equation}
in complex dimension 2, and
\begin{equation}
\left\{z_1^2+z_2^4+z_3^{k_3}+z_4^{k_4}=1\right\}\subset\mathbb{C}^4,\textrm{ where }k_3\gg1,k_4\gg1
\end{equation}
in complex dimension 3. Note that for $k\geq4$ (resp. $\frac{1}{k_3}+\frac{1}{k_4}\leq\frac{1}{4}$), these affine hypersurfaces do not admit cyclic dilations. In fact, it follows from \cite{zz2}, Theorem 3.27 and \cite{mm}, Theorem 2.5 that if a smooth affine variety admits a cyclic dilation, then it must be $\mathbb{A}^1$-uniruled. By \cite{mm}, Lemma 7.1, this forces its log Kodaira dimension to be $-\infty$. However, the condition $k\geq4$ (resp. $\frac{1}{k_3}+\frac{1}{k_4}\leq\frac{1}{4}$) would imply that the affine hypersurfaces under consideration have non-negative log Kodaira dimensions. Compare with the examples given in Remark \ref{remark:example}.
\end{remark}

In fact, we will prove something slightly stronger than the statement of Corollary \ref{corollary:n}, namely the class $\bar{\gamma}\in\pi_2(M,L)$ with $\partial\bar{\gamma}=\gamma$ and positive symplectic area is represented by a $J_M$-holomorphic disc for any convex almost complex $J_M$ on $M$. Note that aside from technical assumptions, this unveils the deep reason behind the validity of Proposition \ref{proposition:aspherical}: 

\textit{Liouville manifolds with cyclic dilations cannot contain exact Lagrangian $K(\pi,1)$ spaces because any aspherical Lagrangian submanifold $L\subset M$ necessarily bounds a non-constant pseudoholomorphic disc.}

When $n=3$, we have the following generalization to Theorem \ref{theorem:3}. Recall that a 3-dimensional \textit{spherical space form} is a quotient $S^3/\Gamma$, where $\Gamma\subset\mathit{SO}(4)$ is a finite subgroup acting freely on $S^3\cong\mathit{SO}(4)/\mathit{SO}(3)$.

\begin{corollary}\label{corollary:3}
Let $M$ be a 6-dimensional Liouville manifold which admits a cyclic dilation, and let $L\subset M$ be a closed, orientable, prime Lagrangian submanifold, then $L$ is diffeomorphic to a spherical space form, or a product $S^1\times\Sigma_g$, where $\Sigma_g$ is an orientable closed surface.
\end{corollary}

Our result is sharp in the sense that there exist Liouville 6-manifolds with cyclic dilations containing Lagrangian submanifolds of each of the topological types allowed by Corollary \ref{corollary:3}. For $S^1\times\Sigma_g$, the fact that they can be embedded as closed Lagrangian submanifolds in $\mathbb{C}^3$ follows from $\cite{kf}$, Theorem 2.6. For a spherical space form $L=S^3/\Gamma$, it follows from $\cite{ss}$, Example 6.4 that $T^\ast L$ admits a dilation, therefore also a cyclic dilation.

Corollaries \ref{corollary:n} and \ref{corollary:3} will be proved in Section \ref{section:corollary}.
\bigskip

The paper is organized as follows. In Section \ref{section:corollary}, we explain how to deduce the applications, Corollaries \ref{corollary:n} and \ref{corollary:3}, from our main result, Theorem \ref{theorem:main}. In Section \ref{section:dR-BV}, we recall a de Rham chain model of the free loop space homology introduced by Irie \cite{ki1} and Wang \cite{ywt}, and modify their constructions to produce two de Rham models $C_\ast^{S^1}$ and $C_\ast^\lambda$ of $S^1$-equivariant chains on the free loop space. These two de Rham models are related by a quasi-isomorphism $C_\ast^{S^1}\twoheadrightarrow C_\ast^\lambda$ (cf. Section \ref{section:BV}). We then introduce a chain level refinement of Chas-Sullivan's string bracket \cite{cs} on $C_\ast^\lambda$, which enables us to provide a chain level statement of Theorem \ref{theorem:main} (cf. Theorem \ref{theorem:chain}). Section \ref{section:moduli} contains the main technical input of this paper. Inspired by the work of Cohen-Ganatra \cite{cg}, we introduce the relevant moduli spaces whose virtual fundamental chains define the finite energy de Rham chains approximating the chains satisfying Theorem \ref{theorem:chain}, and analyze the boundary strata of their compactifications. The proof of Theorem \ref{theorem:main} is completed in Section \ref{section:proof}. One interesting aspect of our argument is that the geometric de Rham chains produced using the moduli spaces of holomorphic curves are first defined on the chain complex $C_\ast^{S^1}$, on which the chain level string bracket $\{\cdot,\cdot\}$ \textit{isn't} well-defined as a Lie bracket, but the $S^1$-equivariant differential takes a more adorable form. They will then be mapped to $C_\ast^\lambda$ via the quasi-isomorphism to give the chains required by Theorem \ref{theorem:chain}. The final section, Section \ref{section:mix}, is devoted to the discussions of some potential implications and variations of our results. In Appendix \ref{section:Kuranishi}, we record some key facts in virtual perturbation theory that are used in the main contents of this paper. The orientation conventions of the moduli spaces and the sign computations will be dealt with in Appendix \ref{section:orientation}.

\section*{Acknowledgements}
During a discussion in November 2019, Paul Seidel asked me whether the notion of a cyclic dilation introduced in \cite{yl2} has any cool applications. Hopefully this paper serves as an interesting application.  I would like to thank Yi Wang for helpful discussions and many useful suggestions, especially those related to various versions of cyclic (co)homologies. I also thank Kai Cieliebak for useful comments and Kei Irie for sending me an erratum to his paper \cite{ki2}, which enables me to make corresponding corrections for various minor issues in earlier versions of this paper. I'm indebted to the anonymous referee, Shah Faisal and Shuhao Li for carefully reading the paper and pointing out numerous issues.

The work is partially funded by Simons grant \#385571.

\section{Proof of Corollaries}\label{section:corollary}

In this section, we prove Corollaries \ref{corollary:n} and \ref{corollary:3} by assuming Theorem \ref{theorem:main}. The argument here largely follows the exposition of \cite{jlt} in the special case when $M=\mathbb{C}^n$, with only slight modifications.

For the proof of Corollary \ref{corollary:n}, we will be using the following two lemmas. Recall that $L$ is a closed manifold of dimension $n$. For a loop $\gamma:S^1\rightarrow L$, denote by $Z_\gamma\subset\pi_1(L)$ its centralizer in the fundamental group. 

\begin{lemma}[\cite{jlt}, Lemma 5.1]\label{lemma:homeo}
Let $\pi:\widetilde{L}\rightarrow L$ be a connected covering of $L$ associated to the subgroup $Z_\gamma$, and let $\tilde{\gamma}\in\pi_1(\widetilde{L})$ be a lift of $\gamma$. Then $\pi$ induces a homeomorphism $\mathcal{L}(\tilde{\gamma})\widetilde{L}\rightarrow\mathcal{L}(\gamma)L$ between the corresponding components of the free loop spaces.
\end{lemma}

\begin{lemma}[\cite{jlt}, Lemma 5.2]\label{lemma:h-equiv}
In the situation of the previous lemma, assume in addition that $L$ is a $K(\pi,1)$ space. Then the evaluation $\mathcal{L}(\tilde{\gamma})\widetilde{L}\rightarrow\widetilde{L}$ at the base point is a homotopy equivalence.
\end{lemma}

The following is a simple consequence of Lemmas \ref{lemma:homeo} and \ref{lemma:h-equiv}.

\begin{corollary}\label{corollary:h-type}
If $L$ is a $K(\pi,1)$ space, then every component of $\mathcal{L}L$ has the homotopy type of a CW complex of dimension at most $n$.
\end{corollary}

We can now prove Corollary \ref{corollary:n}, which in particular shows that if $M$ admits a cyclic dilation, and the closed Lagrangian submanifold $L\subset M$ is a $K(\pi,1)$ space and relatively $\mathit{Spin}$, then $L$ has minimal Maslov index $2$.

\begin{proof}[Proof of Corollary \ref{corollary:n}]
We can rewrite the identity (\ref{eq:def1}) as
\begin{equation}
\sum_{k=1}^\infty\frac{1}{k!}\sum_{a=a_1+\cdots+a_{k-1}}\tilde{\ell}_{k+1}\left(\tilde{y}(-a),\tilde{x}(a_1),\cdots,\tilde{x}(a_k)\right)=(-1)^{n+1}[\![L]\!],
\end{equation}
where $a,a_1,\cdots,a_k\in H_1(L;\mathbb{Z})$, and $\tilde{x}(a_i)$ denotes the component of $\tilde{x}$ in the summand $H_{\ast+n+\mu(a_i)-1}^{S^1}(\mathcal{L}(a_i)L;\mathbb{R})$ with respect to the decomposition (\ref{eq:H-eq}). We now use the assumption that $L$ is a $K(\pi,1)$ space, because of the topological splitting on the free loop space, we have $[\![L]\!]\neq0$.\footnote{Note that this is not true for general $L$. For example, $[\![L]\!]=0$ in $H_\ast^{S^1}(\mathcal{L}L;\mathbb{R})$ if $L$ is simply-connected, cf. \cite{jz}, Corollary 1.1.6.} It follows that there must be some integer $k\geq1$ and homology classes $a,a_1,\cdots,a_k$ such that
\begin{equation}\label{eq:non-van} 
\tilde{\ell}_{k+1}\left(\tilde{y}(-a),\tilde{x}(a_1),\cdots,\tilde{x}(a_k)\right)\neq0.
\end{equation}
The gradings of these inputs are given by
\begin{equation}
|\tilde{y}(-a)|=n+1-\mu(a)\textrm{ and }|\tilde{x}(a_i)|=n-3+\mu(a_i).
\end{equation}
Since $L$ is a $K(\pi,1)$, it follows from Corollary \ref{corollary:h-type} that the vector spaces $H_\ast^{S^1}(\mathcal{L}(b)L;\mathbb{R})$ are supported in degrees $0\leq\ast\leq n-1$ for every $b\neq0\in H_1(L;\mathbb{Z})$. In order for (\ref{eq:non-van}) to be true, none of the inputs can be a multiple of $[\![L]\!]$, so we must have
\begin{equation}
2\leq\mu(a)\leq n+1\textrm{ and }3-n\leq\mu(a_i)\leq2.
\end{equation}
By our assumption, $\sum_{i=1}^k\mu(a_i)=\mu(a)\geq2$, so there must be some $i$ for which $\mu(a_i)>0$. Recall that for an orientable Lagrangian submanifold $L$, $\mu(a_i)$ is even, so the constraint $0<\mu(a_i)\leq2$ actually implies that $\mu(a_i)=2$. Since $\tilde{x}(a_i)\neq0$ and it is defined in terms of counting pseudoholomorphic discs in the relative homotopy class $\bar{a}_i\in\pi_2(M,L)$ with $\partial\bar{a}_i=a_i$ (see Section \ref{section:disc} for the moduli spaces defining the chain underlying $\tilde{x}(a_i)$, and refer to (\ref{eq:MC}) for the precise relation between this chain and the homology class $\tilde{x}(a_i)$), we see that $L$ bounds a non-constant pseudoholomorphic disc of Maslov index 2, which in particular has positive symplectic area.

Set $\gamma:=a_i$, and let $Z_\gamma\subset\pi_1(L)$ be its centralizer. We have a short exact sequence
\begin{equation}
0\rightarrow\ker(\mu|_{Z_\gamma})\rightarrow Z_\gamma\xrightarrow{\frac{1}{2}\mu}\mathbb{Z}\rightarrow0.
\end{equation}
It follows that the map $\rho:\mathbb{Z}\times\ker(\mu|_{Z_\gamma})\rightarrow Z_\gamma$ defined by $\rho(k,g)=\gamma^k\cdot g$ is an isomorphism. Since $L$ is a $K(\pi,1)$ space, the covering space $\widetilde{L}$ of $L$ associated to $Z_\gamma$ is also an Eilenberg-MacLane space $K\left(\mathbb{Z}\times\ker(\mu|_{Z_\gamma}),1\right)$. In particular, it is homotopy equivalent to $S^1\times K$ for some $(n-1)$-dimensional CW complex $K$ which is a $K\left(\ker(\mu|_{Z_\gamma}),1\right)$ space. 

It remains to show that $\widetilde{L}$ is a finite cover of $L$. Note that since $\mu(\gamma)=2$ and $|\tilde{x}(\gamma)|=n-1$, we have $H_n(\mathcal{L}(\gamma)L;\mathbb{R})\neq0$. By Lemmas \ref{lemma:homeo} and \ref{lemma:h-equiv}, $\mathcal{L}(\gamma)L$ is homotopy equivalent to the $n$-manifold $\widetilde{L}$. Since $H_n(\widetilde{L};\mathbb{R})\neq0$, it follows that $\widetilde{L}$ is compact, which forces $\widetilde{L}\rightarrow L$ to be a finite covering.
\end{proof}

\begin{remark}\label{remark:manifold}
In \cite{jlf}, Section 5, it is claimed that the $K\left(\ker(\mu|_{Z_\gamma}),1\right)$ space $K$ can be realized as a manifold. However, it is unclear to the author how to prove this claim in general. When $n=3$, the fact that $\widetilde{L}$ is homotopy equivalent to $S^1\times\Sigma_g$ for some closed oriented surface $\Sigma_g$ follows from Stallings' fibration theorem (\cite{jh3}, Theorem 11.6). When $n\geq6$, the fact that $K$ can be taken to be a closed manifold can likely be proved using Farrell's fibration theorem \cite{tft}.\footnote{The author thanks the anonymous referee for pointing out this issue.}
\end{remark}

\begin{proof}[Proof of Corollary \ref{corollary:3}]
Let $L$ be a compact, orientable, prime $3$-manifold. It is known that $L$ is either diffeomorphic to $S^1\times S^2$, or $L$ is irreducible, meaning that every $S^2\subset L$ bounds a ball.

Let $L\subset M$ be an irreducible Lagrangian $3$-manifold. There there are two possibilities, either $\pi_1(L)$ is infinite, or it is finite. When $\pi_1(L)$ is infinite, the universal cover of $L$ is non-compact. In this case the argument is identical to the proof in Section 5 of \cite{jlf}, Corollary 1.2, which implies that $L$ is diffeomorphic to $S^1\times\Sigma_g$ for some $g\geq1$. When $\pi_1(L)$ is finite, applying Perelman's proof of the Geometrization Conjecture we see that $L$ is a spherical space form. 
\end{proof}

\section{de Rham complex and string bracket}\label{section:dR-BV}

To prove Theorem \ref{theorem:main}, we first introduce in this section a chain model of the free loop space homology due to Irie $\cite{ki1}$ and Wang \cite{ywt}, and then modify their constructions to produce a chain model for the string homology $\mathbb{H}_\ast^{S^1}$ (cf. (\ref{eq:H-eq})), on which an odd Lie bracket (which is supposed to play the role of the chain level string bracket) can be defined. This enables us to reformulate in Theorem \ref{theorem:chain} the statement of Theorem \ref{theorem:main} at chain level in Section \ref{section:chain}. In order to produce the chains in the completed de Rham complex as required by Theorem \ref{theorem:chain}, we will follow the strategy of Irie \cite{ki2} and approximate them using chains up to finite energy level. This is done in Section \ref{section:approximation}.

\subsection{A de Rham chain model}\label{section:de Rham}

Let $L$ be a closed, orientable manifold of dimension $n$. We first recall a convenient model for the space of Moore loops with marked points in $L$, which is due to Wang \cite{ywt}. Denote by $\Pi_1L$ the fundamental groupoid of $L$, which assigns to each ordered pair of points $(p,q)\in L^2$ the collection of equivalence classes of smooth paths from $p$ to $q$. Denote by
\begin{equation}
s:\Pi_1L\rightarrow L\textrm{ and }t:\Pi_1L\rightarrow L
\end{equation}
the source and the target maps, which assigns the points $p$ and $q$ to each equivalence class of ordered pair $(p,q)$, respectively. There is an obvious map
\begin{equation}
\left\{(c_0,c_1)\in(\Pi_1L)^2|t(c_0)=s(c_1)\right\}\rightarrow\Pi_1L,\textrm{ }(c_0,c_1)\mapsto c_0\ast c_1
\end{equation}
induced by the concatenation of two paths. For every $k\in\mathbb{Z}_{\geq0}$, denote by $\mathcal{P}_{k+1}L\subset(\Pi_1L)^{k+1}$ the subspace consisting of $(c_0,\cdots,c_k)$ such that $t(c_i)=s(c_{i+1})$ for $0\leq i\leq k-1$. Define\footnote{The spaces $\mathcal{P}_{k+1}L$ and $\mathcal{L}_{k+1}L$ are denoted by $\mathcal{P}_kL$ and $\mathcal{L}_kL$, respectively in \cite{ywt}. Our notations here are compatible with \cite{ki1} and \cite{ki2}, where the roles of $\mathcal{P}_{k+1}L$ and $\mathcal{L}_{k+1}L$ are played by the space of Moore loops with $k+1$ marked points.}
\begin{equation}\label{eq:flsk}
\mathcal{L}_{k+1}L:=\left\{(c_0,\cdots,c_k)\in\mathcal{P}_{k+1}L|t(c_k)=s(c_0)\right\}.
\end{equation}
Note that compared to the space of Moore loops with $k+1$ marked points in $L$ (cf. \cite{ki1}, Section 7 and \cite{ki2}, Section 4.1), the $\mathcal{L}_{k+1}L$ defined above has the advantage that it is a smooth oriented manifold of dimension $(k+1)n$.
 
From now on, we will omit $L$ from the notations, and simply write $\mathcal{L}_{k+1}$ for (\ref{eq:flsk}). These spaces are equipped with smooth evaluation maps
\begin{equation}
\mathit{ev}_j^\mathcal{L}:\mathcal{L}_{k+1}\rightarrow L,\textrm{ }(c_0,\cdots,c_k)\mapsto s(c_j),\textrm{ }0\leq j\leq k
\end{equation}
and concatenation maps
\begin{equation}\label{eq:con}
\begin{split}
\mathit{con}_j:\mathcal{L}_{k+1}\textrm{ }{{}_{\mathit{ev}_j^\mathcal{L}}\times_{\mathit{ev}_0^\mathcal{L}}}\textrm{ }\mathcal{L}_{k'+1}&\rightarrow\mathcal{L}_{k+k'}, \\
\left((c_0,\cdots,c_k),(c_0',\cdots,c_k')\right)&\mapsto\left\{\begin{array}{ll}
(c_0,\cdots,c_{j-2},c_{j-1}\ast c_0',c_1',\cdots,c_{k'-1}',c_{k'}'\ast c_j,\cdots,c_k) & k'\geq1, \\
(c_0,\cdots,c_{j-2},c_{j-1}\ast c_0'\ast c_j,c_{j+1},\cdots,c_k) & k'=0.
\end{array}\right.
\end{split}
\end{equation}
It is easy to see that the concatenation maps are compatible with the decomposition $\mathcal{L}_{k+1}=\bigsqcup_{a\in H_1(L;\mathbb{Z})}\mathcal{L}_{k+1}(a)$ of $\mathcal{L}_{k+1}$ into different homotopy classes, where $\mathcal{L}_{k+1}(a)\subset\mathcal{L}_{k+1}$ is the subspace consisting of loops $\gamma=c_0\ast\cdots\ast c_k\in\Pi_1L$ with $[\gamma]=a$. It follows that we have a map
\begin{equation}\label{eq:conca}
\mathit{con}_j:\mathcal{L}_{k+1}(a)\textrm{ }{{}_{\mathit{ev}_j^\mathcal{L}}\times_{\mathit{ev}_0^\mathcal{L}}}\textrm{ }\mathcal{L}_{k'+1}(a')\rightarrow\mathcal{L}_{k+k'}(a+a')
\end{equation}
for $a,a'\in H_1(L;\mathbb{Z})$. The following definition is inspired by \cite{ki2}, Definition 4.1.

\begin{definition}\label{definition:smooth}
Let $U$ be a smooth manifold and consider the map $\varphi:U\rightarrow\mathcal{L}_{k+1}$, which can be written as $\varphi(u)=\left(c_0(u),\cdots,c_k(u)\right)$. We say that $\varphi$ is a smooth map if the map $\varphi$ is $C^\infty$ and the composition $\mathit{ev}_0^\mathcal{L}\circ\varphi:U\rightarrow L$ is a submersion.
\end{definition}

For $N\in\mathbb{N}$, let $\mathfrak{U}_N$ be the collection of oriented submanifolds in $\mathbb{R}^N$, and define $\mathfrak{U}:=\bigsqcup_{N\geq1}\mathfrak{U}_N$. Let $\mathcal{P}(\mathcal{L}_{k+1}(a))$ denote the set of pairs $(U,\varphi)$, where $U\in\mathfrak{U}$ and $\varphi:U\rightarrow\mathcal{L}_{k+1}(a)$ is a smooth map in the sense of Definition \ref{definition:smooth}. In the terminology of \cite{ki1}, the pair $(U,\varphi)$ is called a \textit{plot} of $\mathcal{L}_{k+1}$.

For each $N$, consider the vector space
\begin{equation}\label{eq:sp}
\bigoplus_{(U,\varphi)\in\mathcal{P}(\mathcal{L}_{k+1}(a))}A_c^{\dim(U)-N}(U),
\end{equation}
where $A_c^\ast(U)$ denotes the space of compactly supported differential forms on $U$. Denote by $Z_N$ the subspace of (\ref{eq:sp}) defined by
\begin{equation}
	\begin{split}
	&\left\{(U,\varphi,\pi_!\omega)-(U',\varphi\circ\pi,\omega)|(U,\varphi)\in\mathcal{P}(\mathcal{L}_{k+1}(a)),U'\in\mathfrak{U}, \right. \\
	&\left.\omega\in A_c^{\dim(U')-N}(U'),\pi:U'\rightarrow U\textrm{ is a submersion}\right\}.
	\end{split}
\end{equation}
As a graded vector space, the $N$th degree \textit{de Rham chain complex} of $\mathcal{L}_{k+1}(a)$ is the quotient
\begin{equation}
C_N^\mathit{dR}(\mathcal{L}_{k+1}(a)):=\left(\bigoplus_{(U,\varphi)\in\mathcal{P}(\mathcal{L}_{k+1}(a))}A_c^{\dim(U)-N}(U)\right)/Z_N.
\end{equation}
By abuse of notations, we will write the chains in $C_N^\mathit{dR}(\mathcal{L}_{k+1}(a))$ as $(U,\varphi,\omega)$ instead of their equivalence classes. The boundary operator $\partial:C_\ast^\mathit{dR}(\mathcal{L}_{k+1}(a))\rightarrow C_{\ast-1}^\mathit{dR}(\mathcal{L}_{k+1}(a))$ is defined by taking the de Rham differential
\begin{equation}
\partial(U,\varphi,\omega):=(-1)^{|\omega|+1}(U,\varphi,d\omega).
\end{equation}
One can check that $\partial$ is well-defined, and $\partial^2=0$. The homology of $\left(C_\ast^\mathit{dR}(\mathcal{L}_{k+1}(a)),\partial\right)$ will be denoted by $H_\ast^\mathit{dR}(\mathcal{L}_{k+1}(a))$.

By \cite{ywt}, Theorem 2.2.1, after forming the total complex $C_\ast^\mathit{dR}(a):=\prod_{k=0}^\infty C_\ast^\mathit{dR}(\mathcal{L}_{k+1}(a))$ (cf. (\ref{eq:tot})), the homology group $H_\ast\left(C_\ast^\mathit{dR}(a)\right)$ is in fact isomorphic to the singular homology $H_\ast^\mathit{sing}(\mathcal{L}(a)L;\mathbb{R})$ defined using the $C^\infty$-topology on the free loop space $\mathcal{L}(a)L$.\footnote{Note that unlike in the case of Moore loops, we no longer have an isomorphism $H_\ast^\mathit{dR}(\mathcal{L}_{k+1}(a))\cong H_\ast^\mathit{sing}(\mathcal{L}(a)L;\mathbb{R})$ for each $k\in\mathbb{Z}_{\geq0}$. This will affect a few things later on, see in particular the proof of Lemma \ref{lemma:def-0}.}

One nice property of de Rham chains is that one can take their fiber products. For $k\in\mathbb{N}$, $k'\in\mathbb{Z}_{\geq0}$, and $1\leq j\leq k$, define the map
\begin{equation}\label{eq:fp}
\circ_j:C_{n+d}^\mathit{dR}(\mathcal{L}_{k+1}(a))\otimes C_{n+d'}^\mathit{dR}\left(\mathcal{L}_{k'+1}(a')\right)\rightarrow C_{n+d+d'}^\mathit{dR}\left(\mathcal{L}_{k+k'}(a+a')\right)
\end{equation}
by
\begin{equation}
x\circ_jy:=(-1)^{(\dim(U)-|\omega|-n)|\omega'|}\left(U{{}_{\varphi_j}\times_{\varphi_0'}}U',\mathit{con}_j\circ(\varphi_j\times\varphi_0'),\omega\times\omega'|_{U{{}_{\varphi_j}\times_{\varphi_0'}}U'}\right),
\end{equation}
where $\varphi_j=\mathit{ev}_j^\mathcal{L}\circ\varphi$ and $\varphi_0'=\mathit{ev}_0^\mathcal{L}\circ\varphi'$. One can check that this is a chain map. On the total complex $C_\ast^\mathit{dR}(a)$, it descends to a map
\begin{equation}
H_{n+d}(\mathcal{L}(a)L;\mathbb{R})\otimes H_{n+d'}\left(\mathcal{L}(a')L;\mathbb{R}\right)\rightarrow H_{n+d+d'}\left(\mathcal{L}(a+a')L;\mathbb{R}\right)
\end{equation}
which corresponds to the Chas-Sullivan loop product defined in \cite{cs} under the isomorphism $H_\ast\left(C_\ast^\mathit{dR}(a)\right)\cong H_\ast^\mathit{sing}(\mathcal{L}(a)L;\mathbb{R})$ mentioned above.
\bigskip

There is a relative version of the above construction, whose definition makes use of de Rham chains on $[-1,1]\times\mathcal{L}_{k+1}(a)$ relative to $\{-1,1\}\times\mathcal{L}_{k+1}(a)$. Let $\overline{\mathcal{P}}(\mathcal{L}_{k+1}(a))$ denote the set of tuples $(U,\varphi,\tau_+,\tau_-)$, where
\begin{itemize}
	\item $U\in\mathfrak{U}$ and $\varphi:U\rightarrow\mathbb{R}\times\mathcal{L}_{k+1}(a)$. Write $\varphi=(\varphi_\mathbb{R},\varphi_\mathcal{L})$, and for every interval $I\subset\mathbb{R}$, define $U_I:=(\varphi_\mathbb{R})^{-1}(I)$.
	\item $\varphi_\mathbb{R}$ and $\varphi_\mathcal{L}$ are $C^\infty$ maps. Moreover, the map $U\rightarrow\mathbb{R}\times L$ defined by $u\mapsto\left(\varphi_\mathbb{R}(u),\mathit{ev}_0\circ\varphi_\mathcal{L}(u)\right)$ is a submersion.
	\item $\tau_+:U_{\geq1}\rightarrow\mathbb{R}_{\geq1}\times U_1$ is a diffeomorphism such that
	\begin{equation}
	\varphi|_{U_{\geq1}}=(i_{\geq1}\times\varphi_\mathcal{L}|_{U_1})\circ\tau_+,
	\end{equation}
	where $i_{\geq1}:\mathbb{R}_{\geq1}\hookrightarrow\mathbb{R}$ is the obvious inclusion.
	\item $\tau_-:U_{\leq-1}\rightarrow\mathbb{R}_{\leq-1}\times U_{-1}$ is a diffeomorphism such that
	\begin{equation}
	\varphi|_{U_{\leq-1}}=(i_{\leq-1}\times\varphi_\mathcal{L}|_{U_{-1}})\circ\tau_{-1},
	\end{equation}
	where $i_{\leq-1}:\mathbb{R}_{\leq-1}\hookrightarrow\mathbb{R}$ is the obvious inclusion.
\end{itemize}
Note that the sets $U_{\geq1}$ and $U_{\leq-1}$ can be empty.

For any $(U,\varphi,\tau_+,\tau_-)\in\overline{\mathcal{P}}(\mathcal{L}_{k+1}(a))$ and $N\in\mathbb{Z}$, let $A^N(U,\varphi,\tau_+,\tau_-)$ be the vector space of differential $N$-forms $\omega\in A^N(U)$ on $U$ such that
\begin{itemize}
	\item $\omega|_{U_{[-1,1]}}$ is compactly supported,
	\item $\omega|_{U_{\geq1}}=(\tau_+)^\ast(1\times\omega|_{U_1})$,
	\item $\omega|_{U_{\leq-1}}=(\tau_-)^\ast(1\times\omega|_{U_{-1}})$.
\end{itemize}
Define the space of $N$th degree \textit{relative de Rham chains} to be
\begin{equation}
\overline{C}_N^\mathit{dR}(\mathcal{L}_{k+1}(a)):=\left(\bigoplus_{(U,\varphi,\tau_+,\tau_-)\in\overline{\mathcal{P}}(\mathcal{L}_{k+1}(a))}A^{\dim(U)-N-1}(U,\varphi,\tau_+,\tau_-)\right)/\overline{Z}_N,
\end{equation}
where the subspace $\overline{Z}_N\subset\overline{C}_N^\mathit{dR}(\mathcal{L}_{k+1}(a))$ is generated by
\begin{equation}
(U,\varphi,\tau_+,\tau_-,\omega)-(U',\varphi',\tau_+',\tau_-',\omega'),
\end{equation}
if there exists a submersion $\pi:U'\rightarrow U$ satisfying $\varphi'=\varphi\circ\pi$, $\omega=\pi_!\omega'$, and
\begin{equation}
\tau_+\circ\pi|_{U_{\geq1}'}=(\mathit{id}_{\mathbb{R}_{\geq1}}\times\pi|_{U_1'})\circ\tau_+',
\end{equation}
\begin{equation}
\tau_-\circ\pi|_{U_{\leq-1}'}=(\mathit{id}_{\mathbb{R}_{\leq-1}}\times\pi|_{U_{-1}'})\circ\tau_-',
\end{equation}
where $\mathit{id}_{\mathbb{R}_I}$ is the identity map on $\mathbb{R}_I$. The differential $\overline{\partial}:\overline{C}_\ast^\mathit{dR}(\mathcal{L}_{k+1}(a))\rightarrow\overline{C}_{\ast-1}^\mathit{dR}(\mathcal{L}_{k+1}(a))$ is defined to be
\begin{equation}
\overline{\partial}(U,\varphi,\tau_+,\tau_-,\omega):=(-1)^{|\omega|+1}(U,\varphi,\tau_+,\tau_-,d\omega).
\end{equation}
Again, one can check that $\overline{\partial}$ is well-defined and $\overline{\partial}^2=0$, which gives rise to a relative version of de Rham homology $\overline{H}_\ast^\mathit{dR}(\mathcal{L}_{k+1}(a))$.

The fiber product (\ref{eq:fp}) also exists on the relative de Rham complex. For $k\in\mathbb{N}$, $k'\in\mathbb{Z}_{\geq0}$, $1\leq j\leq k$, $a,a'\in H_1(L;\mathbb{Z})$, and $x=(U,\varphi,\tau_+,\tau_-,\omega)$, $y=(U',\varphi',\tau_+',\tau_-',\omega')$ two relative de Rham chains, define
\begin{equation}\label{eq:fpr}
\circ_j:\overline{C}_{n+d}^\mathit{dR}(\mathcal{L}_{k+1}(a))\otimes\overline{C}^\mathit{dR}_{n+d'}(\mathcal{L}_{k'+1}(a'))\rightarrow\overline{C}^\mathit{dR}_{n+d+d'}(\mathcal{L}_{k+k'}(a+a'))
\end{equation}
by
\begin{equation}
x\circ_jy=(-1)^{(\dim(U)-|\omega|-n-1)|\omega'|+n}\left(U_{\varphi_j}\times_{\varphi_0'}U',\varphi'',\tau_+'',\tau_-'',\omega\times\omega'\right),
\end{equation}
where
\begin{equation}
\varphi_j:=(\varphi_\mathbb{R},\mathit{ev}_j^\mathcal{L}\circ\varphi_\mathcal{L}),\textrm{ }\varphi_0':=(\varphi_\mathbb{R}',\mathit{ev}_0^\mathcal{L}\circ\varphi_\mathcal{L}'),
\end{equation}
and
\begin{equation}
\varphi''(u,u'):=\left(\varphi_\mathbb{R}(u),\mathit{con}_j(\varphi_\mathcal{L}(u),\varphi_\mathcal{L}'(u'))\right),
\end{equation}
\begin{equation}
\tau_+''(u,u'):=\left(\rho_+(u,u'),(\mathit{pr}_{U_1}\circ\tau_+(u),\mathit{pr}_{U_1'}\circ\tau_+'(u'))\right),
\end{equation}
\begin{equation}
\tau_-''(u,u'):=\left(\rho_-(u,u'),(\mathit{pr}_{U_{-1}}\circ\tau_-(u),\mathit{pr}_{U_{-1}'}\circ\tau_-'(u'))\right),
\end{equation}
with the functions $\rho_\pm$ given by
\begin{equation}
\rho_+(u,u'):=\mathit{pr}_{\mathbb{R}_{\geq1}}\circ\tau_+(u)=\mathit{pr}_{\mathbb{R}_{\geq1}}\circ\tau_+'(u'),
\end{equation}
\begin{equation}
\rho_-(u,u'):=\mathit{pr}_{\mathbb{R}_{\leq-1}}\circ\tau_-(u)=\mathit{pr}_{\mathbb{R}_{\leq-1}}\circ\tau_-'(u'),
\end{equation}
where $\mathit{pr}_{\mathbb{R}_I}$ denotes the trivial projection to the $\mathbb{R}_I$ factor.

The de Rham complex $C_\ast^\mathit{dR}(\mathcal{L}_{k+1}(a))$ and its relative version $\overline{C}_\ast^\mathit{dR}(\mathcal{L}_{k+1}(a))$ are related as follows. It is natural to consider the inclusion map $i:C_\ast^\mathit{dR}(\mathcal{L}_{k+1}(a))\rightarrow\overline{C}_\ast^\mathit{dR}(\mathcal{L}_{k+1}(a))$ defined by
\begin{equation}\label{eq:i}
i(U,\varphi,\omega):=(-1)^{\dim(U)}(\mathbb{R}\times U,\mathit{id}_\mathbb{R}\times\varphi,\tau_+,\tau_-,1\times\omega),
\end{equation}
where the diffeomorphisms $\tau_\pm$ are defined in the obvious way, and the projection maps $e_\pm:\overline{C}_\ast^\mathit{dR}(\mathcal{L}_{k+1}(a))\rightarrow C_\ast^\mathit{dR}(\mathcal{L}_{k+1}(a))$ are given by
\begin{equation}\label{eq:e+}
e_+(U,\varphi,\tau_+,\tau_-,\omega):=(-1)^{\dim(U)-1}(U_1,\varphi|_{U_1},\omega|_{U_1}),
\end{equation}
\begin{equation}\label{eq:e-}
e_-(U,\varphi,\tau_+,\tau_-,\omega):=(-1)^{\dim(U)-1}(U_{-1},\varphi|_{U_{-1}},\omega|_{U_{-1}}),
\end{equation}
where $U_1$ (resp. $U_{-1}$) is oriented so that $\tau_+$ (resp. $\tau_-$) is orientation preserving ($\mathbb{R}_{\geq1}$ and $\mathbb{R}_{\leq-1}$ are oriented so that $\frac{\partial}{\partial t}$ is the positive direction). It is easy to see that $e_\pm$ are surjective, and $i,e_+$ and $e_-$ are well-defined chain maps such that $e_+\circ i=e_-\circ i=\mathit{id}_C$, where $\mathit{id}_C$ denotes the identity of the chain complex $C_\ast^\mathit{dR}(\mathcal{L}_{k+1}(a))$. Conversely, one can show that $i\circ e_+$ and $i\circ e_-$ are chain homotopic to $\mathit{id}_{\overline{C}}$, the identity of the relative chain complex $\overline{C}_\ast^\mathit{dR}(\mathcal{L}_{k+1}(a))$. In particular, the projections $e_\pm$ are quasi-isomorphisms, and we have
\begin{equation} 
H_\ast\left(\overline{C}_\ast^\mathit{dR}(a)\right)\cong H_\ast\left(C_\ast^\mathit{dR}(a)\right)\cong H_\ast^\mathit{sing}(\mathcal{L}(a)L;\mathbb{R}),
\end{equation}
where $\overline{C}_\ast^\mathit{dR}(a):=\prod_{k=0}^\infty\overline{C}_\ast^\mathit{dR}\left(\mathcal{L}_{k+1}(a)\right)$ is the total complex in the relative case.

\begin{remark}\label{remark:model}
The de Rham chain models for the free loop space homology of $L$ in this subsection interpolates the constructions of Irie \cite{ki1} and Wang \cite{ywt} in the following sense. For the model $\mathcal{L}_{k+1}$ of the free loop space, we have used Wang's construction in \cite{ywt}, Chapter 2, which realizes $\mathcal{L}_{k+1}$ as a subspace of $(\Pi_1L)^{k+1}$ instead of the space of Moore loops with marked points appeared in Irie's work \cite{ki2}, while for the de Rham complexes $C_\ast^\mathit{dR}$ and $\overline{C}_\ast^\mathit{dR}$, we implemented Irie's construction \cite{ki1} rather than Wang's in \cite{ywt}, Chapter 3, where he also includes manifolds with corners in the definition of plots. Our choice here has the advantage that when pushing forward the virtual fundamental chains of the relevant moduli spaces to get de Rham chains in $C_\ast^\mathit{dR}(\mathcal{L}_{k+1})$ or $\overline{C}_\ast^\mathit{dR}(\mathcal{L}_{k+1})$, the $C^0$-approximation lemma used in Irie's work \cite{ki2} can be avoided. On the other hand, considering the collection of closed submanifolds in $\mathbb{R}^N$ without boundary is enough for the purpose of this paper.
\end{remark}

\subsection{Chain level string bracket}\label{section:BV}

The de Rham chain model introduced in Section \ref{section:de Rham} is good enough for Irie's realization of Fukaya's ideas outlined in \cite{kf}. However, in order to write down the Maurer-Cartan equations for $S^1$-equivariant chains on the free loop space, we need a de Rham chain model on which the chain level homotopy $S^1$-action induced by loop rotations is strict (in the sense of \cite{sg1}, Definition 2.1), and equip such an $S^1$-complex with a chain level Lie bracket, so that it becomes a dg Lie algebra (with degree shift). Unlike the loop bracket, whose definition naturally lifts to the chain level, the original construction of the string bracket by Chas-Sullivan \cite{cs} does not directly apply to the $S^1$-equivariant complex of de Rham chains (see Remark \ref{remark:str-bra} below). To resolve these issues, we will introduce two (mutually quasi-isomorphic) different quotients of Irie-Wang's chain model $\prod_{k=0}^\infty C_\ast^\mathit{dR}(\mathcal{L}_{k+1})$, on the first one the $S^1$-action becomes strict, and a Lie bracket exists on the chain level for the second one.

From now on, we shall further abbreviate the notations by setting
\begin{equation}
C_\ast(a,k)=C^\mathit{dR}_{\ast+n+\mu(a)+k-1}(\mathcal{L}_{k+1}(a)),
\end{equation}
\begin{equation}
\textrm{ }C_\ast(k)=C_{\ast+n}^\mathit{dR}(\mathcal{L}_{k+1}):=\bigoplus_{a\in H_1(L;\mathbb{Z})}C_{\ast+n+\mu(a)+k-1}^\mathit{dR}(\mathcal{L}_{k+1}(a)).
\end{equation}
Consider the total complex
\begin{equation}\label{eq:tot}
\left(C_\ast:=\bigoplus_{a\in H_1(L;\mathbb{Z})}\prod_{k=0}^\infty C_\ast(a,k),\tilde{\partial}\right),
\end{equation}
where the differential $\tilde{\partial}$ can be expressed in terms of the de Rham differential $\partial$ and the cosimplicial structure maps $\delta_{k,i}$ defined below, see \cite{ki1}, Section 2.5.2. Under the assumption that $L\subset M$ is a closed Lagrangian submanifold in the Liouville manifold $M$, it is equipped with the action filtration
\begin{equation}\label{eq:Xi}
F^\Xi C_\ast:=\bigoplus_{\theta_M(a)>\Xi}\prod_{k=0}^\infty C_\ast(a,k)
\end{equation}
mentioned in the introduction, with respect to which one can take the completion
\begin{equation}\label{eq:completion}
\widehat{C}_\ast:=\varprojlim_{\Xi\rightarrow\infty}C_\ast/F^\Xi C_\ast.
\end{equation}

The chain complexes $\left\{C_\ast(k)\right\}_{k\geq0}$ form a non-symmetric dg operad
\begin{equation}
\mathcal{O}_L=\left\{\mathcal{O}_L(k)\right\}_{k\geq0},
\end{equation} 
with a multiplication
\begin{equation}
\mu_L:=(L',i_2\circ\phi,1)\in C_{-1}(0,2),
\end{equation}
where $i_2:L\rightarrow\mathcal{L}_3(0)$ is defined by taking three copies of $i_0:L\rightarrow\mathcal{L}_1(0)$, $p\mapsto[p]=(p,p)$, $L'\in\mathfrak{U}$, and $\phi:L'\rightarrow L$ is an orientation-preserving diffeomorphism, and a unit
\begin{equation} 
e_L:=(L',i_0\circ\phi,1)\in C_1(0,0)
\end{equation}
of $\mu_L$. We have
\begin{equation}
\mu_L\circ_1e_L=\mu_L\circ_2e_L=\mathit{id}_{\mathcal{O}_L},
\end{equation}
where
\begin{equation}
\mathit{id}_{\mathcal{O}_L}:=(L',i_1\circ\phi,1)\in C_0(0,1)
\end{equation}
is the identity, with $i_1:L\rightarrow\mathcal{L}_2(0)$ given by taking two copies of $i_0$. The total complex $C_\ast$ has the structure of an associative dg algebra, with the product $\bullet:C_i\otimes C_j\rightarrow C_{i+j-1}$ defined by\footnote{Note the sign difference from \cite{ki1}, which is due to the fact that we are considering here the total complex $\prod_{k=0}^\infty\mathcal{O}_L(k)_{\ast+k-1}$ with degree shifted up by $1$, instead of the total complex $\widetilde{\mathcal{O}}_\ast:=\prod_{k=0}^\infty\mathcal{O}_L(k)_{\ast+k}$ considered in \cite{ki1}. This sign difference will be inherited by many of the formulas later on.}
\begin{equation}\label{eq:bullet}
(x\bullet y)(a,k):=\sum_{\substack{k_1+k_2=k\\a_1+a_2=a}}(-1)^{k_1(|y|+1)}(\mu_L\circ_1 x(a_1,k_1))\circ_{k_1+1}y(a_2,k_2).
\end{equation}

More interestingly, the dg operad $\mathcal{O}_L$ carries a cyclic structure (\cite{ki1}, Definition 2.9). For $(c_0,\cdots,c_k)\in\mathcal{L}_{k+1}$, cyclic permutation defines a map
\begin{equation}\label{eq:rotation}
R_k:\mathcal{L}_{k+1}\rightarrow\mathcal{L}_{k+1};\textrm{ } (c_0,\cdots,c_k)\mapsto(c_1,\cdots,c_k,c_0),
\end{equation}
whose induced map
\begin{equation}
(R_k)_\ast:C_\ast(k)\rightarrow C_\ast(k)
\end{equation}
on the de Rham chain complex gives the cyclic structure of $\mathcal{O}_L$.

Using the cyclic structure maps $(R_k)_\ast$ on the dg operad $\mathcal{O}_L$, one can define the chain level BV operator
\begin{equation}\label{eq:BV-c}
\begin{split}
\delta_\mathit{cyc}&:C_\ast(a,k+1)\rightarrow C_{\ast+1}(a,k), \\
(\delta_\mathit{cyc} x)(a,k)&:=\sum_{j=1}^{k+1}(-1)^{|x|+k(j-1)}(R_{k+1})_\ast^jx(a,k+1)\circ_{k+2-j}e_L.
\end{split}
\end{equation}
Under the isomorphism $H_\ast(C_\ast)\cong H_\ast(\mathcal{L}L;\mathbb{R})$, it follows from $\cite{ki1}$, Section 8.5 and \cite{ywt}, Lemma 3.3.1 that the cohomology level operation of $\delta_\mathit{cyc}$ coincides with the BV operator $\Delta:H_\ast(\mathcal{L}L;\mathbb{R})\rightarrow H_{\ast+1}(\mathcal{L}L;\mathbb{R})$ defined by loop rotations.

As before, the definition of $\delta_\mathit{cyc}$ is compatible with the decomposition (\ref{eq:tot}), therefore it extends to an operator on the completion $\widehat{C}_\ast$. 

Recall that a \textit{cosimplicial chain complex} consists of a sequence of complexes $\left\{C_\ast(k)\right\}_{k\geq0}$, together with two families of chain maps
\begin{equation}
\delta_{k,i}:C_\ast(k-1)\rightarrow C_\ast(k),\textrm{ }\sigma_{k,i}:C_\ast(k+1)\rightarrow C_\ast(k)
\end{equation}
for each $0\leq i\leq k$ such that
\begin{equation}\label{eq:cod}
\delta_{k+1,j}\circ\delta_{k,i}=\delta_{k+1,i}\circ\delta_{k,j-1}\textrm{ for }i<j,
\end{equation}
\begin{equation}\label{eq:cos}
\sigma_{k-1,j}\circ\sigma_{k,i}=\sigma_{k-1,i}\circ\sigma_{k,j+1}\textrm{ for }i\leq j,
\end{equation}
\begin{equation}\label{eq:sd}
\sigma_{k,j}\circ\delta_{k+1,i}=\left\{\begin{array}{ll}
\delta_{k,i}\circ\sigma_{k-1,j-1} & i<j, \\ \mathit{id}_C & i=j,j+1, \\ \delta_{k,i-1}\circ\sigma_{k-1,j} & i>j+1.
\end{array}\right.
\end{equation}
In our case, the cosimplicial structure on the dg operad $\mathcal{O}_L$ is given by
\begin{equation}\label{eq:cosimp}
\delta_{k,i}(x):=\left\{\begin{array}{ll}
\mu_L\circ_2x & i=0, \\ x\circ_i\mu_L & 1\leq i\leq k-1, \\ \mu_L\circ_1x & i=k,
\end{array}\right.
\end{equation}
\begin{equation}
\sigma_{k,i}(x):=x\circ_{i+1} e_L.
\end{equation}
Just like $(R_k)_\ast$, these operations on the de Rham chain complex are induced by very simple operations on $\mathcal{L}_{k+1}$. For example, $\delta_{k,i}$ is induced by a map $\mathcal{L}_k\rightarrow\mathcal{L}_{k+1}$ defined as
\begin{equation}
(c_0,\cdots,c_{k-1})\mapsto\left\{\begin{array}{ll}
\left(c_0,\cdots,c_{i-1},[s(c_i)],c_i,\cdots,c_{k-1}\right) & 0\leq i\leq k-1, \\
\left(c_0,\cdots,c_{k-1},[t(c_{k-1})]\right) & i=k,
\end{array}\right.
\end{equation}
where $[p]\in\Pi_1L$ denote the class of the constant path at $p\in L$. A chain $x\in C_\ast(k)$ is called \textit{normalized} if $\sigma_{k-1,i}(x)=0$ for every $0\leq i\leq k-1$. It is easy to see that the normalized chains in the total complex $C_\ast$ form a subcomplex, which we will denote by $C_\ast^\mathit{nm}$. The inclusion $C_\ast^\mathit{nm}\hookrightarrow C_\ast$ is a quasi-isomorphism, see \cite{ki1}, Lemma 2.5 for a proof.

There is an alternative, yet equivalent realization of the subcomplex $C_\ast^\mathit{nm}$ of normalized chains, which is standard in simplicial homotopy theory. We spell out the details for the reader's convenience. A chain $x\in C_\ast(k)$ is called \textit{degenerate} if there exist an $i$ with $1\leq i\leq k$ and a $y\in C_\ast(k-1)$ such that $x=\delta_{k,i}(y)$. It is clear that the sums of degenerate chains in $C_\ast$ form a subcomplex $D_\ast$, and we call the quotient complex $C_\ast^\mathit{nd}:=C_\ast/D_\ast$ the complex of \textit{non-degenerate} de Rham chains. We include an elementary proof of the following fact, which is a consequence of the Dold-Kan correspondence (cf. \cite{stack}, Lemma 14.25.1).

\begin{lemma}\label{lemma:DK}
The composition
\begin{equation}
\psi_{DK}:C_\ast^\mathit{nm}\hookrightarrow C_\ast\twoheadrightarrow C_\ast^\mathit{nd}
\end{equation}
is an isomorphism of chain complexes. In particular, the projection map $C_\ast\twoheadrightarrow C_\ast^\mathit{nd}$ is also a quasi-isomorphism.
\end{lemma}
\begin{proof}
It suffices to show that the natural map $C_\ast^\mathit{nm}(k)\rightarrow C_\ast^\mathit{nd}(k)$ is an isomorphism for each $k\in\mathbb{Z}_{\geq0}$. Define the decreasing filtration
\begin{equation}
\mathscr{F}^jC_\ast^\mathit{nm}(k):=\bigcap_{0\leq i\leq j}\ker(\sigma_{k-1,i}),\textrm{ }0\leq j\leq k-1,
\end{equation}
and the increasing filtration
\begin{equation}
\mathscr{F}^{j+1}D_\ast(k):=\mathrm{span}_\mathbb{R}\left\langle\mathrm{im}(\delta_{k,1}),\cdots,\mathrm{im}(\delta_{k,j+1})\right\rangle,\textrm{ }0\leq j\leq k-1.
\end{equation}
We will argue by induction on $j$. First note that for $j=0$ the short exact sequence
\begin{equation}
0\rightarrow\mathscr{F}^1D_\ast(k)\rightarrow C_\ast(k)\rightarrow\mathscr{F}^0C_\ast^\mathit{nm}(k)\rightarrow0
\end{equation}
splits by the cosimplicial identity $\sigma_{k-1,0}\circ\delta_{k,1}=\mathit{id}_C$, where the map from $\mathscr{F}^1D_\ast(k)$ to $C_\ast(k)$ is the obvious inclusion, and the map from $C_\ast(k)$ to $\mathscr{F}^0C_\ast^\mathit{nm}(k)$ is defined by $x\mapsto x-\delta_{k,1}\circ\sigma_{k-1,0}(x)$.

Assume that we already have an isomorphism $\mathscr{F}^iC_\ast^\mathit{nm}(k)\cong\frac{C_\ast(k)}{\mathscr{F}^{i+1}D_\ast(k)}$ for $0\leq i\leq j-1$ and all $k\in\mathbb{N}$, there is a commutative diagram
\begin{equation}\label{eq:diagram}
\begin{tikzcd}
	0 \arrow[r] &\mathscr{F}^{j-1}C_\ast^\mathit{nm}(k-1) \arrow[d,"\cong"] \arrow[r,"\delta_{k,j+1}"] &\mathscr{F}^{j-1}C_\ast^\mathit{nm}(k) \arrow[d, "\cong"] \arrow[r,"p_{k,j}"] &\mathscr{F}^jC_\ast^\mathit{nm}(k) \arrow[d] \arrow[r] &0 \\
	0 \arrow[r] &\frac{C_\ast(k-1)}{\mathscr{F}^{j}D_\ast(k-1)} \arrow[r, "\delta_{k,j+1}"] &\frac{C_\ast(k)}{\mathscr{F}^{j}D_\ast(k)} \arrow[r,"\mathit{pr}"] &\frac{C_\ast(k)}{\mathscr{F}^{j+1}D_\ast(k)} \arrow[r] &0
\end{tikzcd}
\end{equation}
where the map $p_{k,j}$ is given by $x\mapsto x-\delta_{k,j+1}\circ\sigma_{k-1,j}(x)$ and $\mathit{pr}$ is the standard projection. 

To see this, we first check that the upper row is exact. By (\ref{eq:sd}), $\sigma_{k-1,i}\circ\delta_{k,j+1}(x)=\delta_{k-1,j}\circ\sigma_{k-2,i}(x)=0$ for $0\leq i\leq j-1$ and $x\in\mathscr{F}^{j-1}C_\ast^\mathit{nm}(k-1)$. Moreover, $\sigma_{k-1,i}(x)-\sigma_{k-1,i}\circ\delta_{k,j+1}\circ\sigma_{k-1,j}(x)=0$ for $0\leq i\leq j$ and $x\in\mathscr{F}^{j-1}C_\ast^\mathit{nm}(k)$ by the cosimplicial identities (\ref{eq:cos}) and (\ref{eq:sd}). Thus the upper row is well-defined. Its exactness follows from
\begin{equation}
\delta_{k,j+1}(x)-\delta_{k,j+1}\circ\sigma_{k-1,j}\circ\delta_{k,j+1}(x)=\delta_{k,j+1}(x)-\delta_{k,j+1}(x)=0,
\end{equation}
where $x\in\mathscr{F}^{j-1}C_\ast^\mathit{nm}(k-1)$ and the first equality follows from (\ref{eq:sd}). The exactness of the lower row is automatic once we show the map $\delta_{k,j+1}$ is well-defined and injective. To see this, let $x=\delta_{k-1,i}(y)$ for some $1\leq i\leq j$ and $y\in\mathscr{F}^{j}D_\ast(k-2)$, then by the cosimplicial identity (\ref{eq:cod}) we have $\delta_{k,j+1}\circ\delta_{k-1,i}(y)=\delta_{k,i}\circ\delta_{k-1,j}(y)\in\mathscr{F}^jD_\ast(k)$. For the injectivity, note that if $\delta_{k,j+1}(x)\in\mathscr{F}^jD_\ast(k)$ for some $x\in C_\ast(k-1)$, then $x=\sigma_{k-1,j}\circ\delta_{k,j+1}(x)\in\mathscr{F}^{j}D_\ast(k-1)$ by (\ref{eq:sd}). Now the first square of the diagram (\ref{eq:diagram}) commutes since it consists of the natural inclusions and projections. For the second square, note that the extra term $\delta_{k,j+1}\circ\sigma_{k-1,j}$ in the definition of $p_{k,j}$ does not affect things modulo $\mathscr{F}^jD_\ast(k)$, so it commutes as well.

Since the horizontal lines of (\ref{eq:diagram}) are short exact sequences, and the first two vertical arrows are isomorphisms, so is the third vertical arrow.
\end{proof}

\begin{remark}
One should, however, take care of the differences in terminologies. The chain complex $C_\ast^\mathit{nd}$ is called the normalized chain complex for cosimplicial objects in \cite{stack} in the sense that it is dual to the normalized chain complex for simplicial objects. Here, to be consistent with the convention of \cite{ki1}, we refer to $C_\ast^\mathit{nm}$ as the normalized chain complex.
\end{remark}

For later purposes, it would be more convenient for us to work with the quotient complex $C_\ast^\mathit{nd}$ of $C_\ast$ rather than its subcomplex $C_\ast^\mathit{nm}$, and we shall use the former to perform our constructions below. It follows from \cite{ki1}, Theorem 2.10 that the chain level BV operator $\delta_\mathit{cyc}$ restricts to an anti-chain map on the subcomplex $C_\ast^\mathit{nm}\subset C_\ast$. Define
\begin{equation}\label{eq:dnd}
\delta_\mathit{cyc}^\mathit{nd}:=\psi_\mathit{DK}\circ(\delta_\mathit{cyc}|_{C_\ast^\mathit{nm}})\circ\psi_\mathit{DK}^{-1}
\end{equation}
using the isomorphism $\psi_\mathit{DK}:C_\ast^\mathit{nm}\rightarrow C_\ast^\mathit{nd}$ established in Lemma \ref{lemma:DK}. It follows that $\delta_\mathit{cyc}^\mathit{nd}:C_\ast^\mathit{nd}\rightarrow C_{\ast+1}^\mathit{nd}$ is an anti-chain map which gives the BV operator after passing to the cohomology $H_\ast(C_\ast^\mathit{nd})\cong H_\ast(\mathcal{L}L;\mathbb{R})$. We observe the following.

\begin{lemma}\label{lemma:delta}
Write $x\in C_\ast=C_\ast^\mathit{nd}\oplus D_\ast$ as $x=x_0+x_D$, where $x_0\in C_\ast^\mathit{nd}$ and $x_D\in D_\ast$. If $x_D=0$, then under the projection $\Pi:C_\ast\twoheadrightarrow C_\ast^\mathit{nd}$, $x\mapsto x_0$, we have
\begin{equation}
\Pi\left(\delta_\mathit{cyc}(x)\right)=\delta_\mathit{cyc}^\mathit{nd}\left(x_0\right).
\end{equation}
\end{lemma}
\begin{proof}
This follows directly from Lemma \ref{lemma:DK} and the definition of $\delta_\mathit{cyc}^\mathit{nd}$.
\end{proof}

As has been mentioned before, the dg operad $\mathcal{O}_L$ carries an additional piece of structure---a \textit{cocyclic chain complex}, which is a cosimplicial chain complex together with an additional family of chain maps
\begin{equation}
\tau_k:C_\ast(k)\rightarrow C_\ast(k)
\end{equation}
such that $\tau_k^{k+1}=\mathit{id}_C$ and
\begin{equation}\label{eq:td}
\tau_k\circ\delta_{k,i}=\left\{\begin{array}{ll}
\delta_{k,k} & i=0, \\ \delta_{k,i-1}\circ\tau_{k-1} & 1\leq i\leq k,
\end{array}\right.
\end{equation}
\begin{equation}\label{eq:cocy2}
\tau_k\circ\sigma_{k,i}=\left\{\begin{array}{ll}
\sigma_{k,k}\circ\tau_{k+1}^2 & i=0, \\ \sigma_{k,i-1}\circ\tau_{k+1} & 1\leq i\leq k.
\end{array}\right.
\end{equation}
In our case, the maps $\tau_k$ are given by $(R_k)_\ast$. See \cite{ki1}, Remark 7.5. By abuse of notations, we will continue to denote the induced differential on the quotient complex $C_\ast^\mathit{nd}$ by $\tilde{\partial}$.

\begin{lemma}\label{lemma:strict}
We have $(\delta_\mathit{cyc}^\mathit{nd})^2=0$ in the complex $C_\ast^\mathit{nd}$ of non-degenerate de Rham chains. In other words, $(C_\ast^\mathit{nd},\tilde{\partial},\delta_\mathit{cyc}^\mathit{nd})$ is a strict $S^1$-complex.
\end{lemma}
\begin{proof}
Starting from a cocyclic chain complex $\left\{C_\ast(k)\right\}_{k\geq0}$, one can equip the associated total complex $C_\ast=\prod_{k\geq0}C_\ast(k)$ with the structure of a strict $S^1$-complex $(C_\ast,\tilde{\partial},\mathbb{B})$, where $\mathbb{B}:=N\circ s\circ(1-\lambda)$ is Connes' operator, where $N$ is the norm of the cyclic operator $\lambda$, and $s$ is the degeneracy operator, see \cite{jl}. More precisely, on $C_\ast(k+1)$ it is given by
\begin{equation}
\mathbb{B}_k=N_k\circ s_k\circ (1-\lambda_{k+1}):C_\ast(k+1)\rightarrow C_{\ast+1}(k),
\end{equation}
where
\begin{equation}
\lambda_k(x)=(-1)^k\tau_k(x),\textrm{ }N_k(x)=(1+\lambda+\cdots+\lambda^k)(x)
\end{equation}
for $x\in C_\ast(k)$, and
\begin{equation}
s_i(x)=(-1)^{|x|}\sigma_{k,i-1}\circ\tau_{k+1}(x)
\end{equation}
for $x\in C_\ast(k+1)$ and $1\leq i\leq k+1$. It is straightforward to verify that $\mathbb{B}^2=0$, see for example \cite{yw}, Example 2.6. By Lemma \ref{lemma:DK} and the definition of $\delta_\mathit{cyc}^\mathit{nd}$, in order to prove that $(\delta_\mathit{cyc}^\mathit{nd})^2=0$, it is enough to show that $\delta_\mathit{cyc}=\mathbb{B}$ when restricted to the subcomplex $C_\ast^\mathit{nm}$. It is observed in \cite{yw}, Example 2.6 that $\mathbb{B}_k|_{C_\ast^\mathit{nm}(k+1)}=N_k\circ s_0$, with $s_0:=\lambda_k\circ s_1\circ\lambda_{k+1}^{-1}$. On the other hand, it is explained in \cite{ki1}, Section 2.5.4 that $N_k\circ s_0$ coincides with the BV operator $\delta_\mathit{cyc}$ defined by (\ref{eq:BV-c}).
\end{proof}

By Lemma \ref{lemma:strict}, various versions of $S^1$-equivariant homology theories can be constructed in terms of the strict $S^1$-complex $(C_\ast^\mathit{nd},\tilde{\partial},\delta_\mathit{cyc}^\mathit{nd})$. What is relevant for us here is the (positive) $S^1$-equivariant free loop space homology $H_\ast^{S^1}(\mathcal{L}L;\mathbb{R})$, also known as the string homology of $L$. It is proved by Wang in \cite{yw}, Proposition 4.8 that $H_\ast^{S^1}(\mathcal{L}L;\mathbb{R})$ is the homology of the chain complex
\begin{equation}\label{eq:cpx-S1}
C_\ast^{S^1}:=\left(C_\ast^\mathit{nd}\otimes_\mathbb{R}\mathbb{R}(\!(u)\!)/u\mathbb{R}[\![u]\!],\partial^{S^1}:=\tilde{\partial}+u\delta_\mathit{cyc}^\mathit{nd}\right),
\end{equation}
where $u$ is a formal variable of degree $-2$.\footnote{Note that this is different from the case of $S^1$-equivariant symplectic cohomology, where $u$ has degree $2$, because of the homological grading convention that is imposed here.} Similarly, one can define the completed version $\widehat{C}_\ast^{S^1}$ of the $S^1$-equivariant chain complex with respect to the action filtration, and its homology $\widehat{H}_\ast^{S^1}(\mathcal{L}L;\mathbb{R})$\footnote{To avoid confusions, we remark that this is not the \textit{periodic $S^1$-equivariant homology} of the free loop space, which is denoted using the same notation in \cite{jz}.}, which is just the $\widehat{\mathbb{H}}_\ast^{S^1}$ defined by (\ref{eq:H-c}). It is clear that $C_\ast^{S^1}$ decomposes according to $(a,k)\in H_1(L;\mathbb{Z})\times\mathbb{Z}_{\geq0}$, and we will use the notations $C_\ast^{S^1}(a,k)$ and $C_\ast^{S^1}(k)$ for their obvious meanings.

For the string homology $H_\ast^{S^1}(\mathcal{L}L;\mathbb{R})$, Chas-Sullivan \cite{cs} defined a Lie bracket, known as \textit{string bracket}, by applying the erasing map to the loop product pre-composed with the marking maps (cf. (\ref{eq:bracket})). This equips $H_\ast^{S^1}(\mathcal{L}L;\mathbb{R})$ with the structure of a graded Lie algebra of degree $2-n$ (under the usual grading convention). Inspired by the construction of the Lie bracket for cyclic operads in \cite{kwz} and \cite{bw}, we introduce here a Lie bracket on a chain complex that is quasi-isomorphic to $C_\ast^{S^1}$, which is supposed to play the role of a chain level refinement of the string bracket. We shall first define it as a bilinear operation
\begin{equation}\label{eq:bra}
\{\cdot,\cdot\}:C_i(\!(u)\!)/C_i[\![u]\!]\otimes C_j(\!(u)\!)/C_j[\![u]\!]\rightarrow C_{i+j+1}(\!(u)\!)/C_{i+j+1}[\![u]\!]
\end{equation}
on the larger vector space $C_\ast(\!(u)\!)/C_\ast[\![u]\!]$ as follows. Let $\tilde{x}=\sum_{i=0}^\infty x_i\otimes u^{-i}\in C_\ast(\!(u)\!)/C_\ast[\![u]\!]$ and $\tilde{y}=\sum_{i=0}^\infty y_i\otimes u^{-i}\in C_\ast(\!(u)\!)/C_\ast[\![u]\!]$, then
\begin{equation}\label{eq:bracket-pr}
\begin{split}
\left\{\tilde{x},\tilde{y}\right\}(a,k)&:=\sum_{\substack{a_1+a_2=a\\k_1+k_2=k+1}}\sum_{i=1}^{k_1}\sum_{j=1}^{k_2+1}(-1)^{\maltese_{ij}}x_0(a_1,k_1)\circ_i\left((R_{k_2+1})_\ast^jy_0(a_2,k_2+1)\circ_{k_2+2-j}e_L\right)\otimes1 \\
&-\sum_{\substack{a_1+a_2=a\\k_1+k_2=k+1}}\sum_{i=1}^{k_1}\sum_{j=1}^{k_1+1}(-1)^{\maltese_{ij}+(|x_0|+1)(|y_0|+1)}\left((R_{k_1+1})_\ast^jy_0(a_1,k_1+1)\circ_{k_1+2-j}e_L\right) \\
&\;\;\;\;\;\;\;\;\;\;\;\;\;\;\;\;\;\;\;\;\;\;\;\;\;\;\;\;\;\;\;\;\circ_ix_0(a_2,k_2)\otimes1,
\end{split}
\end{equation}
where
\begin{equation}
\maltese_{ij}=(i-1)(k_2-1)+(k_1-1)(|y_0|+k_2)+|y_0|+k_2(j-1).
\end{equation}
It is not hard to see that $\{\cdot,\cdot\}$ does \textit{not} define a Lie bracket after passing to the quotient space $C_\ast^{S^1}$ of $C_\ast(\!(u)\!)/C_\ast[\![u]\!]$. However,  we will see that it becomes a Lie bracket after passing to another quotient $C_\ast^\lambda$ of $C_\ast(\!(u)\!)/C_\ast[\![u]\!]$, where
\begin{equation}\label{eq:Connes}
C_\ast^\lambda:=C_\ast/\mathrm{im}(1-\lambda)
\end{equation}
is known as the \textit{Connes' complex} \cite{ac}, with $\lambda:C_\ast\rightarrow C_\ast$ being the cyclic operator defined in the proof of Lemma \ref{lemma:strict}. We first show that the complex $(C_\ast^\lambda,\tilde{\partial})$ is well-defined.

\begin{lemma}
The differential $\tilde{\partial}$ on the total complex $C_\ast$ descends to one on its quotient $C_\ast^\lambda$.
\end{lemma}
\begin{proof}
In the classical case of an endomorphism operad of an associative algebra, this is a standard fact. See for example \cite{jl}, Lemma 2.1.1. In our case, define $\tilde{\partial}':C_\ast\rightarrow C_{\ast-1}$ by
\begin{equation}
(\tilde{\partial}'x)(a,k):=(\tilde{\partial}x)(a,k)-(-1)^{|x|+1}\delta_{k,k}(x(a,k-1)).
\end{equation}
It is straightforward to check that $(1-\lambda)\tilde{\partial}'=\tilde{\partial}(1-\lambda)$.
\end{proof}

In fact, the quotient complex $(C_\ast^\lambda,\tilde{\partial})$ provides an alternative chain model for the string homology of $L$.

\begin{lemma}\label{lemma:qis}
There is a quasi-isomorphism
\begin{equation}
\phi_\lambda:\left(C_\ast^\mathit{nd}\otimes_\mathbb{R}\mathbb{R}(\!(u)\!)/u\mathbb{R}[\![u]\!],\tilde{\partial}+u\delta_\mathit{cyc}^\mathit{nd}\right)\rightarrow(C_\ast^\lambda,\tilde{\partial}).
\end{equation}
\end{lemma}
\begin{proof}
Since we are working over the field $\mathbb{R}\supset\mathbb{Q}$, it follows from Lemma \ref{lemma:DK} and \cite{cv8}, Lemma 4.8 that the compositions
\begin{equation}\label{eq:qis}
\ker(1-\lambda)\hookrightarrow C_\ast\otimes_\mathbb{R}\mathbb{R}[v]\twoheadrightarrow C_\ast^\mathit{nd}\otimes_\mathbb{R}\mathbb{R}[v]\hookrightarrow C_\ast^\mathit{nd}\otimes_\mathbb{R}\mathbb{R}[\![v]\!]
\end{equation}
defines a quasi-isomorphism of chain complexes, where $v$ is a formal variable of degree $2$. Note also that the grading convention used in \cite{cv8} for Connes' complex $C_\ast^\lambda$ differs from the one inherited from the quotient (\ref{eq:Connes}) by exactly $1$. This is dual to \cite{jl}, Theorem 2.1.5, which deals with cyclic chain complexes instead of cocyclic chain complexes. Changing the variable $u=v^{-1}$ and combining with the isomorphism $\ker(1-\lambda)\cong C_\ast/\mathrm{im}(1-\lambda)$ completes the proof.
\end{proof}

As we have remarked above, the advantage of using the quotient chain model $C_\ast^\lambda$  of $C_\ast$ is that the binary operation $\{\cdot,\cdot\}$ defined above now gives rise to a well-defined Lie bracket. By abuse of notations, we will still use $\{\cdot,\cdot\}$ to denote the induced operation on the quotient complex $C_\ast^\lambda$. More precisely, for any $\underline{x},\underline{y}\in C_\ast^\lambda$ which are projections of $\tilde{x},\tilde{y}\in C_\ast(\!(u)\!)/uC[\![u]\!]$, we have
\begin{equation}
\left\{\underline{x},\underline{y}\right\}=\underline{\left\{\tilde{x},\tilde{y}\right\}},
\end{equation}
where the right-hand side is the projection of $\{\tilde{x},\tilde{y}\}$ in $C_\ast^\lambda$. Note that this definition depends on non-canonical choices of the lifts $\tilde{x}$ and $\tilde{y}$ of $\underline{x}$ and $\underline{y}$, respectively, but different choices only differ by a cyclic permutation, hence the induced operation on $C_\ast^\lambda$ is well-defined.

We use $C_\ast^\lambda(a,k)$ to denote $(a,k)$-component in the decomposition
\begin{equation}
C_\ast^\lambda=\bigoplus_{a\in H_1(L;\mathbb{Z})}\prod_{k=0}^\infty C_\ast^\lambda(a,k).
\end{equation} 

In this paper, we will use both of the chain complexes $C_\ast^{S^1}$ and $C_\ast^\lambda$ as de Rham models of $S^1$-equivariant chains on the free loop space. The former complex has the advantage that its differential $\partial^{S^1}$ (involving the BV operator $\delta_\mathit{cyc}^\mathit{nd}$) has a more convenient form, while the latter complex is useful for the algebraic arguments in Sections \ref{section:chain} and \ref{section:approximation} since it carries the structure of an odd dg Lie algebra, as we will show below.

\begin{lemma}
$\left(C_\ast^\lambda,\tilde{\partial},\{\cdot,\cdot\}\right)$ is a dg Lie algebra of degree $1$. In particular, for $\underline{x},\underline{y},\underline{z}\in C_\ast^\lambda$, the bilinear operation $\{\cdot,\cdot\}$ satisfies the Jacobi identity
\begin{equation}\label{eq:Jacobi}
\left\{\underline{x},\{\underline{y},\underline{z}\}\right\}=\left\{\{\underline{x},\underline{y}\},\underline{z}\right\}+(-1)^{(|\underline{x}|+1)(|\underline{y}|+1)}\left\{\underline{y},\{\underline{x},\underline{z}\}\right\},
\end{equation}
and is graded anti-symmetric in the sense that
\begin{equation}\label{eq:anti-sym}
\{\underline{x},\underline{y}\}=-(-1)^{(|\underline{x}|+1)(|\underline{y}|+1)}\{\underline{y},\underline{x}\}.
\end{equation}
\end{lemma}
\begin{proof}
It is straightforward from the definition that the bracket $\{\cdot,\cdot\}$ satisfies the graded Leibniz rule. In order to verify the anti-symmetric property and the Jacobi identity, it is convenient to use the ``edge-grafting" operations
\begin{equation}
{{}_i\circ_j}:C_\ast(k_1)\otimes C_\ast(k_2)\rightarrow C_\ast(k_1+k_2-1), 
\end{equation}
where $0\leq i\leq k_1$ and $0\leq j\leq k_2$ introduced in \cite{kwz}, given explicitly by
\begin{equation}
x(a_1,k_1)\textrm{ }{{}_i\circ_j}\textrm{ }y(a_2,k_2):=\left\{\begin{array}{ll}
x(a_1,k_1)\circ_i\lambda_{k_2}^jy(a_2,k_2) & i\geq1, \\
\lambda_{k_1}x(a_1,k_1)\circ_{k_1}\lambda_{k_2}^jy(a_2,k_2) & i=0,
\end{array}\right.
\end{equation}
where we have used the cyclic operator $\lambda_k:C_\ast(k)\rightarrow C_\ast(k)$ instead of the rotation operator $(R_k)_\ast$ since it simplifies the signs. These operations are graded (anti-)commutative up to cyclic permutations, in the sense that (cf. \cite{bw}, Lemma B.13)
\begin{equation}\label{eq:com}
x(a_1,k_1)\textrm{ }{{}_i\circ_j}\textrm{ }y(a_2,k_2)=(-1)^{(|x|+1)(|y|+1)}\lambda_{k_1+k_2-1}^{k_2-i+j}\left(y(a_2,k_2)\textrm{ }{{}_j\circ_i}\textrm{ }x(a_1,k_1)\right).
\end{equation}

With the identity (\ref{eq:com}), we can write the bracket $\{\tilde{x},\tilde{y}\}$ in a more compact way:
\begin{equation}
\begin{split}
\left\{\tilde{x},\tilde{y}\right\}(a,k)&=\sum_{\substack{a_1+a_2=a\\k_1+k_2=k+1}}\sum_{i=0}^{k_1}\sum_{j=0}^{k_2}(-1)^{\maltese_i+(|x_0|+1)(|y_0|+1)+1}\lambda_{k_1+k_2-1}^{k_2-i+j} \\
&\;\;\;\;\;\;\;\;\;\;\;\;\;\;\;\;\;\;\;\;\;\;\;\;\;\;\;\;\;\;\;\left((y_0(a_2,k_2+1)\textrm{ }{{}_0\circ_0}\textrm{ }e_L)\textrm{ }{{}_j\circ_i}\textrm{ }x_0(a_1,k_1)\right)\otimes1,
\end{split}
\end{equation}
where
\begin{equation}
\maltese_i=(i-1)(k_2-1)+(k_1-1)(|y_0|+k_2).
\end{equation}
See \cite{ki1}, Section 2.5.4 for the sign conventions, especially the appearance of $\maltese_i$ in the expression above. It follows that
\begin{equation}
\begin{split}
\{\underline{x},\underline{y}\}(a,k)=\sum_{\substack{a_1+a_2=a\\k_1+k_2=k+1}}\sum_{i=0}^{k_1}\sum_{j=0}^{k_2}&(-1)^{\maltese_i+(|x_0|+1)(|y_0|+1)+1}\underline{(y_0(a_2,k_2+1)\textrm{ }{{}_0\circ_0}\textrm{ }e_L)} \\
&\underline{{{}_j\circ_i}\textrm{ }x_0(a_1,k_1)\otimes1},
\end{split}
\end{equation}
where the underline on the right-hand side means projecting to the quotient complex $C_\ast^\lambda$ of $C_\ast$. Compare with the definition of the cyclic bracket in the proof of \cite{bw}, Theorem 3.2. Now (\ref{eq:anti-sym}) follows directly from the identity (\ref{eq:com}). The verification of the Jacobi identity (\ref{eq:Jacobi}) is also straightforward, which follows from the identity (\ref{eq:com}), the associativity of the fiber product (\ref{eq:fp}), and the cyclic invariance of chains in $C_\ast^\lambda$. Since this is almost identical to \cite{kwz}, Proposition 3.11, we omit the details.
\end{proof}

\begin{remark}\label{remark:str-bra}
Following Chas-Sullivan \cite{cs}, it is most natural to define the chain level string bracket in the following way. Consider the chain level erasing map
\begin{equation}
\mathbf{I}_c:C_\ast(a,k)\rightarrow C_\ast(\!(u)\!)/uC_\ast[\![u]\!](a,k),
\end{equation}
and the marking map
\begin{equation}
\mathbf{B}_c:C_\ast(\!(u)\!)/uC_\ast[\![u]\!](a,k+1)\rightarrow C_{\ast+1}(a,k),
\end{equation}
given respectively by (cf. Lemma \ref{lemma:delta})
\begin{equation}
(\mathbf{I}_cx)(a,k):=x(a,k)\otimes1
\end{equation}
and
\begin{equation}\label{eq:mark}
(\mathbf{B}_c\tilde{x})(a,k):=\sum_{j=1}^{k+1}(-1)^{|\tilde{x}|+k(j-1)}(R_{k+1})_\ast^jx_0(a,k+1)\circ_{k+2-j}e_L,
\end{equation}
where $x\in C_\ast$ and $\tilde{x}=\sum_{i=0}^\infty x_i\otimes u^{-i}\in C_\ast(\!(u)\!)/uC_\ast[\![u]\!]$. One can check that on the homology level, the chain maps $\mathbf{I}_c$ and $\mathbf{B}_c$ induce the erasing map $\mathbf{I}:H_\ast(\mathcal{L}L;\mathbb{R})\rightarrow H_\ast^{S^1}(\mathcal{L}L;\mathbb{R})$ and the marking map $\mathbf{B}:H_\ast^{S^1}(\mathcal{L}L;\mathbb{R})\rightarrow H_{\ast+1}(\mathcal{L}L;\mathbb{R})$ in string topology, respectively. It is also obvious from the definitions that the composition $\mathbf{B}_c\circ\mathbf{I}_c$ recovers the chain level BV operator $\delta_\mathit{cyc}^\mathit{nd}$. One can then define
\begin{equation}\label{eq:bracket}
\{\tilde{x},\tilde{y}\}_\mathit{str}:=(-1)^{|\tilde{x}|}\mathbf{I}_c\left(\mathbf{B}_c(\tilde{x})\bullet\mathbf{B}_c(\tilde{y})\right)
\end{equation}
for $\tilde{x},\tilde{y}\in C_\ast(\!(u)\!)/uC_\ast[\![u]\!]$, where the product $\bullet$ is given by (\ref{eq:bullet}). This is closely related to the Lie bracket $\{\cdot,\cdot\}$ defined by (\ref{eq:bracket-pr}) as both of them involve cyclic permutations of the marked points on the $u^0$-part of $\tilde{y}$. It is not difficult to show that $\{\cdot,\cdot\}_\mathit{str}$ induces on $H_\ast^{S^1}(\mathcal{L}L;\mathbb{R})$ the string bracket. However, $\{\cdot,\cdot\}_\mathit{str}$ does not give a well-defined Lie bracket on the chain complex $C_\ast^{S^1}$. We conjecture that $\{\cdot,\cdot\}_\mathit{str}$ coincides with $\{\cdot,\cdot\}$ defined above on the homology level, but will not try to prove it in this paper.
\end{remark}

Since the Lie bracket $\{\cdot,\cdot\}$ is compatible with the decomposition of $C_\ast^\lambda$ according to $(a,k)\in H_1(L;\mathbb{Z})\times\mathbb{Z}_{\geq0}$, it naturally extends to the completion $\widehat{C}_\ast^\lambda$, and equips it with the structure of a dg Lie algebra $\left(\widehat{C}_\ast^\lambda,\tilde{\partial},\{\cdot,\cdot\}\right)$ of degree $1$ as well.

As in the non-equivariant case, we can introduce the relative versions of the $S^1$-equivariant de Rham chain complexes $C_\ast^{S^1}$ and $C_\ast^\lambda$. To do this, first notice that the cyclic permutation map (\ref{eq:rotation}) also induces a chain map
\begin{equation}
(\overline{R}_k)_\ast:\overline{C}_\ast(k)\rightarrow\overline{C}_\ast(k)
\end{equation}
on the relative de Rham complex, using which we can define a BV operator
\begin{equation}\label{eq:bar-BV}
\begin{split}
\bar{\delta}_\mathit{cyc}&:\overline{C}_\ast(a,k+1)\rightarrow\overline{C}_{\ast+1}(a,k) \\
(\bar{\delta}_\mathit{cyc}\bar{x})(a,k)&:=\sum_{j=1}^{k+1}(-1)^{|\bar{x}|+k(j-1)}(\overline{R}_{k+1})_\ast^j\bar{x}(a,k+1)\circ_{k+2-j}\bar{e}_L,
\end{split}
\end{equation}
where
\begin{equation}
\bar{e}_L:=\left(\mathbb{R}\times L',\mathit{id}_\mathbb{R}\times(i_0\circ\phi),\mathit{id}_{\mathbb{R}_{\geq1}}\times\phi,\mathit{id}_{\mathbb{R}_{\leq-1}}\times\phi,1_\mathbb{R}\times1_{L'}\right)
\end{equation}
is the unit in $\overline{C}_1(0,0)$, with the diffeomorphism $\phi:L'\rightarrow L$ and the inclusion of constant loops $i_0:L\rightarrow\mathcal{L}_1(0)$ defined as above, and $1_\mathbb{R}$ and $1_{L'}$ are constant functions with value $1$ on $\mathbb{R}$ and $L'$, respectively. Since the relative de Rham complexes $\left\{\overline{C}_\ast(k)\right\}_{k\geq0}$ also carries a cocyclic structure, the analogue of Lemma \ref{lemma:strict} holds in the relative case, and after passing to the quotient complex $\overline{C}_\ast^\mathit{nd}$ of $\overline{C}_\ast$ of non-degenerate relative chains, $\bar{\delta}_\mathit{cyc}$ becomes an anti-chain map $\bar{\delta}_\mathit{cyc}^\mathit{nd}$ with $(\bar{\delta}_\mathit{cyc}^\mathit{nd})^2=0$. Define the $S^1$-equivariant relative de Rham chain complex as
\begin{equation}
\overline{C}_\ast^{S^1}:=\left(\overline{C}_\ast^\mathit{nd}\otimes_\mathbb{R}\mathbb{R}((u))/u\mathbb{R}[[u]],\bar{\partial}^{S^1}:=\tilde{\bar{\partial}}+u\bar{\delta}_\mathit{cyc}^\mathit{nd}\right),
\end{equation}
with $|u|=-2$. The chain level Lie bracket is defined in the same way as above. By abuse of notations, we shall still denote it by $\{\cdot,\cdot\}$. As before, it is first defined as an operation on the vector space $C_\ast(\!(u)\!)/uC_\ast[\![u]\!]$, and then descends to a Lie bracket of degree $1$ on the quotient complex
\begin{equation}
\overline{C}_\ast^\lambda:=\overline{C}_\ast/\mathrm{im}(1-\bar{\lambda}),
\end{equation}
where $\bar{\lambda}$ denotes the cyclic operator on the relative de Rham complex $\overline{C}_\ast$. Moreover, $\{\cdot,\cdot\}$ extends to the completion $\widehat{\overline{C}}_\ast^\lambda$ of $\overline{C}_\ast^\lambda$ with respect to the action filtration.

To relate the chain complexes $C_\ast$ and $\overline{C}_\ast$, we have the inclusion map $i$ defined by (\ref{eq:i}), and the projection maps $e_\pm$ to both ends given by (\ref{eq:e+}) and (\ref{eq:e-}). Since these maps are compatible with the cyclic permutations of $c_0,\cdots,c_k\in\Pi_1L$, and the fiber products with the units $e_L$ and $\bar{e}_L$, they induce morphisms between the strict $S^1$-complexes $C_\ast^\mathit{nd}$ and $\overline{C}_\ast^\mathit{nd}$. We denote the induced chain maps on the $S^1$-equivariant de Rham complexes by
\begin{equation}
\tilde{i}:C_\ast^{S^1}\rightarrow\overline{C}_\ast^{S^1}\textrm{ and }\tilde{e}_\pm:\overline{C}_\ast^{S^1}\rightarrow C_\ast^{S^1},
\end{equation}
respectively. One can check that $\tilde{e}_+\circ\tilde{i}=\tilde{e}_-\circ\tilde{i}=\mathit{id}_C$. Note that we have abused the notations here and used $\mathit{id}_C$ for the identity map of $C_\ast^{S^1}$ as well. Similarly, the identity endomorphism of $\overline{C}_\ast^{S^1}$ will still be denoted by $\mathit{id}_{\overline{C}}$. The maps $i$ and $e_\pm$ also descend to the quotient complexes $C_\ast^\lambda$ and $\overline{C}_\ast^\lambda$. We denote these chain maps by
\begin{equation}
\underline{i}:C_\ast^\lambda\rightarrow\overline{C}_\ast^\lambda\textrm{ and }\underline{e}_\pm:\overline{C}_\ast^\lambda\rightarrow C_\ast^\lambda.
\end{equation}
It is also clear that $\underline{e}_+\circ\underline{i}=\underline{e}_-\circ\underline{i}=\mathit{id}_C$, where the $\mathit{id}_C$ here is the identity map on $C_\ast^\lambda$.

\begin{lemma}\label{lemma:surjective}
The map $(\tilde{e}_+,\tilde{e}_-):\overline{C}_\ast^{S^1}\rightarrow C_\ast^{S^1}\oplus C_\ast^{S^1}$ is surjective. Moreover, $\tilde{i}\circ\tilde{e}_+$ and $\tilde{i}\circ\tilde{e}_-$ are chain homotopic to $\mathit{id}_{\overline{C}}$.
\end{lemma}
\begin{proof}
The surjectivity of $(\tilde{e}_+,\tilde{e}_-)$ follows from \cite{ki2}, Lemma 4.7. 

In \cite{ki2}, Lemma 4.8, a chain homotopy $K:\overline{C}_\ast\rightarrow\overline{C}_{\ast+1}$ from $i\circ e_+$ to $\mathit{id}_{\overline{C}}$ has been defined. It is clear that the map $K$ induces a map on the $S^1$-equivariant chain complex $\overline{C}_\ast^{S^1}$. By abuse of notations, we will still denote it by $K:\overline{C}_\ast^{S^1}\rightarrow\overline{C}_{\ast+1}^{S^1}$. Here we need to check that $K\circ\bar{\delta}_\mathit{cyc}+\bar{\delta}_\mathit{cyc}\circ K=0$, so the map $K$ satisfies
\begin{equation}
K(\bar{\partial}+u\bar{\delta}_\mathit{cyc}^\mathit{nd})+(\bar{\partial}+u\bar{\delta}_\mathit{cyc}^\mathit{nd})K=\mathit{id}_{\overline{C}}-\tilde{i}\circ\tilde{e}_+,
\end{equation}
which gives a chain homotopy between $\tilde{i}\circ\tilde{e}_+$ and $\mathit{id}_{\overline{C}}$. To see this, recall that
\begin{equation}
K(U,\varphi,\tau_+,\tau_-,\omega)=(-1)^{|\omega|+1}(\mathbb{R}\times U,\bar{\varphi},\bar{\tau}_+,\bar{\tau}_-,\bar{\omega}),
\end{equation}
where $\bar{\varphi}:\mathbb{R}\times U\rightarrow\mathbb{R}\times\mathcal{L}_{k+1}$ is given by $\bar{\varphi}=\left(\alpha\left(r,\varphi_\mathbb{R}(u)\right),\varphi_\mathcal{L}(u)\right)$, with $\alpha:\mathbb{R}^2\rightarrow\mathbb{R}$ being some fixed choice of smooth function (cf. Step 1 in the proof of \cite{ki2}, Lemma 4.8). It is clear from the definition that the map $K$ commutes with cyclic permutations of $c_0,\cdots,c_k\in\Pi_1L$ and the fiber product with $\bar{e}_L$, which implies that $K\circ\bar{\delta}_\mathit{cyc}^\mathit{nd}=(-1)^{\varepsilon}\bar{\delta}_\mathit{cyc}^\mathit{nd}\circ K$. To determine the sign $(-1)^{\varepsilon}$, recall that the definition of the BV operator $\bar{\delta}_\mathit{cyc}^\mathit{nd}$ involves a sign $(-1)^{|\bar{x}|+k(j-1)}$ before each term (cf. (\ref{eq:bar-BV})). Since the map $K$ has degree 1, it follows that $\varepsilon=1$. The argument for the $\tilde{i}\circ\tilde{e}_-$ case is similar.
\end{proof}

Similarly, for the quotient complexes $C_\ast^\lambda$ and $\overline{C}_\ast^\lambda$, we have the following.

\begin{corollary}\label{corollary:surjective}
The map $(\underline{e}_+,\underline{e}_-):\overline{C}_\ast^\lambda\rightarrow C_\ast^\lambda\oplus C_\ast^\lambda$ is surjective. Moreover, $\underline{i}\circ\underline{e}_+$ and $\underline{i}\circ\underline{e}_-$ are chain homotopic to $\mathit{id}_{\overline{C}}$.
\end{corollary}

Note that by the same argument as in the proof of Lemma \ref{lemma:surjective}, one can prove the compatibility of $\underline{e}_+$ and $\underline{e}_-$ with the odd Lie brackets $\{\cdot,\cdot\}$ on $\overline{C}_\ast^\lambda$ and $C_\ast^\lambda$, respectively.

\subsection{Chain level statement of Theorem \ref{theorem:main}}\label{section:chain}

We provide a chain level statement of Theorem \ref{theorem:main} using the de Rham models $(\widehat{C}_\ast^\lambda,\tilde{\partial})$ and $(\widehat{C}_\ast^{S^1},\partial^{S^1})$ of the $S^1$-equivariant free loop space homology defined in Section \ref{section:BV}. Recall that $\widehat{C}_\ast^\lambda$ carries the structure of a dg Lie algebra of degree $1$, with the Lie bracket given by $\{\cdot,\cdot\}$ defined in (\ref{eq:bracket-pr}).

\begin{theorem}\label{theorem:chain}
Let $M$ be a Liouville manifold admitting a cyclic dilation, and $L\subset M$ a closed Lagrangian submanifold that is oriented and $\mathit{Spin}$ relative to $\alpha$. There exist $S^1$-equivariant de Rham chains $\underline{x}\in\widehat{C}_{-2}^\lambda$, $\underline{y}\in\widehat{C}_2^\lambda$, $\underline{z}\in\widehat{C}_1^\lambda$, and a real number $\varepsilon\in\mathbb{R}_{>0}$ such that
\begin{itemize}
	\item[(i)] $\tilde{\partial}(\underline{x})-\frac{1}{2}\left\{\underline{x},\underline{x}\right\}=0$. In other words, $\underline{x}$ is a Maurer-Cartan element.
	\item[(ii)] $\tilde{\partial}(\underline{y})-\left\{\underline{x},\underline{y}\right\}=\underline{z}$. In other words, $\underline{z}$ has a primitive $\underline{y}$ with respect to the $\underline{x}$-deformed differential on $\widehat{C}_\ast^\lambda$.
	\item[(iii)] $\underline{x}(a,k)\neq 0$ only if $\theta_M(a)\geq2\varepsilon$ or $a=0$, $k\geq3$. Moreover, the chain $\underline{x}(0,3)$ admits a lift $\tilde{x}(0,3)\in C_{-2}(\!(u)\!)/uC_{-2}[\![u]\!](0,3)$, whose image under the marking map $\mathbf{B}_c$ defined by (\ref{eq:mark}) gives a cycle $x(0,2)\in C_n^\mathit{dR}(\mathcal{L}_3(0))$, and therefore a cycle in the total complex $C_n^\mathit{dR}(0)$ by setting the other $k$-components ($k\neq2$) to be zero, which we still denote by $x(0,2)$ by abuse of notations. Under the isomorphism between de Rham and singular homologies, $\left[x(0,2)\right]\in H_n\left(C_\ast^\mathit{dR}(0)\right)$ goes to $(-1)^{n+1}[L]\in H_n(\mathcal{L}(0)L;\mathbb{R})$.
	\item[(iv)] $\underline{z}(a,k)\neq0$ only if $\theta_M(a)\geq2\varepsilon$ or $a=0$. Moreover, $\underline{z}(0,0)$ lifts along the natural projection $\widehat{C}_\ast\rightarrow\widehat{C}_\ast^\lambda$ to a cycle $z(0,0)\in C_n^\mathit{dR}(\mathcal{L}_1(0))$, whose homology class $[z(0,0)]\in H_n\left(C_\ast^\mathit{dR}(0)\right)$ corresponds to $(-1)^{n+1}[L]\in H_n(\mathcal{L}(0)L;\mathbb{R})$ under the isomorphism between de Rham and singular homologies.
\end{itemize}
\end{theorem}

\begin{remark}\label{remark:lift1}
In order to produce the chains $\underline{x}$, $\underline{y}$ and $\underline{z}$ in Theorem \ref{theorem:chain}, it is enough to find chains $\tilde{x}\in\widehat{C}_{-2}^{S^1}$, $\tilde{y}\in\widehat{C}_2^{S^1}$ and $\tilde{z}\in\widehat{C}_1^{S^1}$ such that the equations
\begin{equation}\label{eq:lifts}
\partial^{S^1}(\tilde{x})-\frac{1}{2}\left\{\tilde{x},\tilde{x}\right\}=0\textrm{ and }\partial^{S^1}(\tilde{y})-\left\{\tilde{x},\tilde{y}\right\}=\tilde{z}
\end{equation}
hold, and the corresponding conditions in (iii) and (iv) of Theorem \ref{theorem:chain} are satisfied. In particular, $\mathbf{B}_c\left(\tilde{x}(0,3)\right)=(-1)^{n+1}\mu_L$ and $\tilde{z}(0,0)=e_L\otimes1$. Here, we have abused notations so that $\tilde{x},\tilde{y},\tilde{z}$ denotes chains in the quotient $C_\ast^{S^1}$ of $C_\ast(\!(u)\!)/uC_\ast[\![u]\!]$. The operation $\{\cdot,\cdot\}$ in the above isn't a Lie bracket, but the equations in (\ref{eq:lifts}) still make sense. In fact, for our geometric argument in Section \ref{section:moduli}, we will construct chains in the $S^1$-equivariant complex $C_\ast^{S^1}$ instead of Connes' complex $C_\ast^\lambda$, by pushing forward virtual fundamental chains of moduli spaces of holomorphic curves.
\end{remark}

We then show that Theorem \ref{theorem:main} is a corollary of its chain level refinement---Theorem \ref{theorem:chain}. First note that (i) and (iii) above imply that the element
\begin{equation}
\underline{x}(0):=\sum_{k=3}^\infty\underline{x}(0,k)\in\widehat{C}_{-2}^\lambda
\end{equation}
is also a Maurer-Cartan element for the odd dg Lie algebra $\left(\widehat{C}_\ast^\lambda,\tilde{\partial},\{\cdot,\cdot\}\right)$, so we can use it to deform the differential $\tilde{\partial}$. Introduce the notation
\begin{equation}
C_\ast^\lambda(a):=\prod_{k=0}^\infty C_\ast^\lambda(a,k),
\end{equation}
and denote by $\tilde{\partial}_{\underline{x}(0)}:C_\ast^\lambda(a)\rightarrow C_{\ast-1}^\lambda(a)$ the deformed differential. By definition,
\begin{equation}
\tilde{\partial}_{\underline{x}(0)}(\underline{w}):=\tilde{\partial}\underline{w}-\left\{\underline{x}(0),\underline{w}\right\}
\end{equation}
for $\underline{w}\in C_\ast^\lambda(a)$. We show that the deformation induced by the ``low energy" Maurer-Cartan element $\underline{x}(0)$ is actually trivial.

\begin{lemma}\label{lemma:def-0}
There is an isomorphism
\begin{equation}\label{eq:iso}
H_\ast\left(C_\ast^\lambda(a),\tilde{\partial}_{\underline{x}(0)}\right)\cong H_{\ast+n+\mu(a)-1}^{S^1}(\mathcal{L}(a)L;\mathbb{R}).
\end{equation}
\end{lemma}
\begin{proof}
For this proof we temporarily use Irie's model of $\mathcal{L}_{k+1}$ as the space of Moore loops with $k+1$ marked points on $L$, so that the complexes $C_\ast^\lambda(a,k)$ are now constructed using Moore loops. Consider the short exact sequence
\begin{equation}
0\rightarrow\prod_{k=1}^\infty C_\ast^\lambda(a,k)\rightarrow C_\ast^\lambda(a)\rightarrow C_\ast^\lambda(a,0)\rightarrow0.
\end{equation}
The advantage of using Irie's model of Moore loops is that since $\underline{x}(0,0)=0$ by Theorem \ref{theorem:chain} (iii), we have the isomorphism
\begin{equation}
H_\ast\left(C_\ast^\lambda(a,0),\partial\right)\cong H_\ast\left(C_\ast^{S^1}(a,0),\partial^{S^1}\right)\cong H_{\ast+n+\mu(a)-1}^{S^1}(\mathcal{L}(a)L;\mathbb{R}),
\end{equation} 
which otherwise doesn't hold for Wang's model of the space $\mathcal{L}_1$. To prove the isomorphism (\ref{eq:iso}) for Irie's model, it suffices to show that the subcomplex $\prod_{k=1}^\infty C_\ast^\lambda(a,k)\subset C_\ast^\lambda(a)$ is acyclic.

For every $N\in\mathbb{N}$, observe that the differential $\tilde{\partial}_{\underline{x}(0)}$ preserves the subcomplex
\begin{equation}
\prod_{k>2N}C_\ast^\lambda(a,k)\subset\prod_{k=1}^\infty C_\ast^\lambda(a,k), 
\end{equation} 
so it descends to the quotient $\prod_{1\leq k\leq2N}C_\ast^\lambda(a,k)$, and it is compatible with the filtration $F^\bullet$ defined by
\begin{equation}
F^i\left(\prod_{1\leq k\leq2N}C_\ast^\lambda(a,k)\right):=\prod_{i\leq k\leq2N}C_\ast^\lambda(a,k).
\end{equation}
Consider the spectral sequence associated to this filtration, whose $E_1$-term is
\begin{equation}
H_\ast\left(C_\ast^\lambda(a,k),\partial\right)\cong H_{\ast+n+\mu(a)+k-1}^{S^1}(\mathcal{L}(a)L;\mathbb{R}),
\end{equation}
where the isomorphism holds since $\mathcal{L}_{k+1}$ is now the space of Moore loops with $k+1$ marked points. Fix lifts $\tilde{x}(a,k)=\sum_{i=0}^\infty x_i(a,k)\otimes u^{-i}$ and $\tilde{w}(a,k)=\sum_{i=0}^\infty w_i(a,k)\otimes u^{-i}$ in $C_\ast(\!(u)\!)/uC_\ast[\![u]\!]$ for $\underline{x}(a,k)$ and $\underline{w}(a,k)$, respectively. The differential $d_1:H_\ast\left(C_\ast^\lambda(a,k)\right)\rightarrow H_{\ast-1}\left(C_\ast^\lambda(a,k+1)\right)$ is given by
\begin{equation}
	\begin{split}
    \left\{\underline{x}(0),\underline{w}\right\}(a,k+1)&=(-1)^{|w_0|}\sum_{i=1}^k\sum_{j=1}^3(-1)^{\maltese_{ij}}\underline{w_0(a,k)\circ_i\left((R_3)_\ast^jx_0(0,3)\circ_{4-j}e_L\right)} \\
    &-\sum_{i=1}^2\sum_{j=1}^3(-1)^{\maltese_{ij}}\underline{\left((R_3)_\ast^jx_0(0,3)\circ_{4-j}e_L\right)\circ_iw_0(a,k)} \\
    &=(-1)^{|w_0|+n+1}\left(\underline{\mu_L\circ_1w_0(a,k)}+\sum_{1\leq i\leq k}\underline{w_0(a,k)\circ_i\mu_L}\right. \\
    &\left.+(-1)^{k+1}\underline{\mu_L\circ_2w_0(a,k)}\right),
    \end{split}
\end{equation}
where the first equality follows from the anti-symmetric property (\ref{eq:com}) of the Lie bracket, and the second identity follows from Theorem \ref{theorem:chain}, (iii). We conclude that $d_1=\pm\sum_{j=0}^{k+1}(-1)^i$. Thus all $E_2$-terms vanish and the complex $\left(\prod_{k=1}^\infty C_\ast^\lambda(a,k),\tilde{\partial}_{\underline{x}(0)}\right)$ is acyclic.

Finally, to switch back to Wang's model of $\mathcal{L}_{k+1}$ and finish the proof, we apply \cite{ywt}, Theorem 2.2.1.
\end{proof}

Straightforward computations imply the following.

\begin{lemma}\label{lemma:+}
Define $\underline{x}^+:=\underline{x}-\underline{x}(0)$, then it satisfies
\begin{equation}
\tilde{\partial}_{\underline{x}(0)}(\underline{x}^+)-\frac{1}{2}\left\{\underline{x}^+,\underline{x}^+\right\}=0\textrm{ and } \tilde{\partial}_{\underline{x}(0)}(\underline{y})-\left\{\underline{x}^+,\underline{y}\right\}=\underline{z}.
\end{equation}
\end{lemma}

The following is now a consequence of the two lemmas above and the homotopy transfer lemma for $L_\infty$-structures.

\begin{proposition}
Theorem \ref{theorem:chain} implies Theorem \ref{theorem:main}.
\end{proposition}
\begin{proof}
It follows from our construction that the homology of the total complex $C_\ast^\lambda$ is $\mathbb{H}_\ast^{S^1}$ defined by (\ref{eq:StrH}). By Lemma \ref{lemma:def-0}, there exist linear maps
\begin{equation}
\iota:\mathbb{H}_\ast^{S^1}\rightarrow C_\ast^\lambda,\textrm{ }\pi:C_\ast^\lambda\rightarrow\mathbb{H}_\ast^{S^1},\textrm{ }\kappa:C_\ast^\lambda\rightarrow C_{\ast+1}^\lambda,
\end{equation}
such that
\begin{equation}
\tilde{\partial}_{\underline{x}(0)}\circ\iota=0,\textrm{ }\pi\circ\tilde{\partial}_{\underline{x}(0)}=0,\textrm{ }\pi\circ\iota=\mathit{id}_\mathbb{H},
\end{equation}
where $\mathit{id}_\mathbb{H}$ denotes the identity of $\mathbb{H}_\ast^{S^1}$, and
\begin{equation}
\kappa\circ\tilde{\partial}_{\underline{x}(0)}+\tilde{\partial}_{\underline{x}(0)}\circ\kappa=\mathit{id}_C-\iota\circ\pi,
\end{equation}
where $\mathit{id}_C$ is the identity of $C_\ast^\lambda$. These maps can be chosen to be compatible with the decompositions of $\mathbb{H}_\ast^{S^1}$ and $C_\ast^\lambda$ over $H_1(L;\mathbb{Z})$, therefore extend to the completions $\widehat{\mathbb{H}}^{S^1}_\ast$ and $\widehat{C}_\ast^\lambda$. Moreover, it follows from Theorem \ref{theorem:chain}, (iv) that we can choose $\pi$ so that the chain $\sum_{k=0}^\infty\underline{z}(0,k)\in\widehat{C}_1^\lambda$ gets mapped to $(-1)^{n+1}[\![L]\!]\in H_n^{S^1}(\mathcal{L}(0)L;\mathbb{R})$.

By \cite{jlf}, Proposition 4.9, there exist an $L_\infty$-structure $(\tilde{\ell}_k)_{k\geq1}$ on $\mathbb{H}_\ast^{S^1}$ and an $L_\infty$-homomorphism
\begin{equation}
p=(p_k)_{k\geq1}:\left(C_\ast^\lambda,\tilde{\partial}_{\underline{x}(0)},\{\cdot,\cdot\}\right)\rightarrow\left(\mathbb{H}_\ast^{S^1},(\tilde{\ell}_k)_{k\geq1}\right)
\end{equation}
such that $\ell_1=0$ and $p_1=\pi$. The $L_\infty$-structure $(\tilde{\ell}_k)_{k\geq1}$ and the $L_\infty$-homomorphism $p$ can be taken so that they respect the decompositions over $H_1(L;\mathbb{Z})$, therefore also extend to the relevant completions. Moreover, by Lemma \ref{lemma:+}, the elements
\begin{equation}\label{eq:MC}
\underline{X}:=\sum_{k=1}^\infty\frac{1}{k!}p_k(\underline{x}^+,\cdots,\underline{x}^+)\in\widehat{\mathbb{H}}_{-2}^{S^1},
\end{equation}
\begin{equation}
\underline{Y}:=\sum_{k=1}^\infty\frac{1}{(k-1)!}p_k(\underline{y},\underline{x}^+,\cdots,\underline{x}^+)\in\widehat{\mathbb{H}}_2^{S^1},
\end{equation}
\begin{equation}
\underline{Z}:=\sum_{k=1}^\infty\frac{1}{(k-1)!}p_k(\underline{z},\underline{x}^+,\cdots,\underline{x}^+)\in\widehat{\mathbb{H}}_1^{S^1}
\end{equation}
satisfy
\begin{equation}
\sum_{k=2}^\infty\frac{1}{k!}\tilde{\ell}_k(\underline{X},\cdots,\underline{X})=0,
\end{equation}
\begin{equation}
\sum_{k=2}^\infty\frac{1}{(k-1)!}\tilde{\ell}_k(\underline{Y},\underline{X},\cdots,\underline{X})=\underline{Z}.
\end{equation}
Note that the infinite sums in the definitions of $\underline{X}$, $\underline{Y}$, and $\underline{Z}$ make sense since $\underline{x}^+(a)\neq0$ only when $\theta_M(a)\geq2\varepsilon$. Since $\underline{X}(a)\neq0$ only if $\theta_M(a)\geq2\varepsilon$, (iii) of Theorem \ref{theorem:main} holds with $c=2\varepsilon$. To complete the proof, it remains to show that $\underline{Z}(0)=(-1)^{n+1}[\![L]\!]$. Since $p$ respects the decompositions over $H_1(L;\mathbb{Z})$, and $\underline{z}(a,k)\neq0$ only if $\theta_M(a)\geq2\varepsilon$ or $a=0$, we obtain
\begin{equation}
\underline{Z}(0)=\pi\left(\sum_{k=0}^\infty\underline{z}(0,k)\right)=(-1)^{n+1}[\![L]\!].
\end{equation}
\end{proof}

\subsection{Approximating the solutions}\label{section:approximation}

To prove Theorem \ref{theorem:chain}, we need to find chains $\underline{x},\underline{y},\underline{z}\in\widehat{C}_\ast^\lambda$ satisfying the equations
\begin{equation}\label{eq:eqn}
\tilde{\partial}(\underline{x})-\frac{1}{2}\{\underline{x},\underline{x}\}=0,\textrm{ }\tilde{\partial}(\underline{y})-\{\underline{x},\underline{y}\}=\underline{z}.
\end{equation}
In principle, these chains should be defined by pushing forward the virtual fundamental chains of the moduli spaces introduced in Section \ref{section:CG}. However, as Irie noticed in his situation \cite{ki2}, it is difficult to get such chains all at once as it will involve simultaneous perturbations of the Kuranishi maps of infinitely many moduli spaces. To overcome this difficulty, one can instead define a sequence of triple chains $(\underline{x}_i,\underline{y}_i,\underline{z}_i)\in C_{-2}^\lambda\times C_2^\lambda\times C_1^\lambda$ such that
\begin{itemize}
	\item[(i)] the triple $(\underline{x}_i,\underline{y}_i,\underline{z}_i)$ satisfies (\ref{eq:eqn}) up to certain energy level, which goes to infinity as $i\rightarrow\infty$;
    \item[(ii)] the triples $(\underline{x}_i,\underline{y}_i,\underline{z}_i)$ and $(\underline{x}_{i+1},\underline{y}_{i+1},\underline{z}_{i+1})$ are gauge equivalent up to certain energy level, which also goes to infinity as $i\rightarrow\infty$.
\end{itemize}
(ii) implies that the limits $\underline{x}$, $\underline{y}$ and $\underline{z}$ of the sequences of chains $(\underline{x}_i)$, $(\underline{y}_i)$ and $(\underline{z}_i)$ exist, and (i) implies that they satisfy (\ref{eq:eqn}). We show in this section that having such sequences of chains is enough for the validity of Theorem \ref{theorem:chain}.

We start by describing how to specify the value of $\varepsilon$ in the statement of Theorem \ref{theorem:chain}. Let $M$ be a Liouville manifold with $c_1(M)=0$, and $L\subset M$ a closed Lagrangian submanifold. We say that an almost complex structure $J_M$ on $M$ is of \textit{contact type} if it is compatible with $d\theta_M$ and $dr\circ J_M=-\theta_M$ on the cylindrical end $[1,\infty)\times\partial\overline{M}$. From now on, fix an almost complex structure $J_M$ on $M$ which is of contact type. We take $\varepsilon>0$ so that $2\varepsilon$ is less than the minimal symplectic area of $J_M$-holomorphic discs with boundary on $L$.

For the purpose of approximating $\underline{x},\underline{y},\underline{z}$ using finite energy de Rham chains, it would be convenient for us to consider an alternative completion of the chain complex $C_\ast^\lambda$. For each $m\in\mathbb{Z}$, define a filtration on $C_\ast$ by
\begin{equation}\label{eq:filt}
F^mC_\ast:=\bigoplus_{\substack{a\in H_1(L;\mathbb{Z}) \\ k\in\mathbb{Z}_{\geq0}\\ \theta_M(a)\geq\varepsilon(m+1-k)}}C_\ast(a,k),
\end{equation}
which induces a similar filtration on its quotient complex $C_\ast^\lambda$. By abuse of notations, the filtration on $C_\ast^\lambda$ will still be denoted by $F^m$. It follows from the definition that $\tilde{\partial}F^m\subset F^m$ and $\left\{F^m,F^{m'}\right\}\subset F^{m+m'}$. In the same manner, one can define a filtration on the relative chain complexes $\overline{C}_\ast$ and $\overline{C}_\ast^\lambda$, which we denote by $\overline{F}^m$. The chains $\underline{x}$, $\underline{y}$ and $\underline{z}$ satisfying (\ref{eq:eqn}) will then be defined as limits in the completion of $C_\ast^\lambda$ with respect to the filtration $F^m$. 

Note that a priori, the completions of $C_\ast$ (resp. $C_\ast^\lambda$) with respect to the filtration $F^m$ defined above can be different from $\widehat{C}_\ast$ (resp. $\widehat{C}_\ast^\lambda$) considered in Section \ref{section:BV} defined using the action filtration $F^\Xi$, see (\ref{eq:Xi}). In order to show that the chains $\underline{x}$, $\underline{y}$ and $\underline{z}$ actually lie in $\widehat{C}_\ast^\lambda$, we will follow the strategy of \cite{ki3} below.

We show in this subsection that in order to prove Theorem \ref{theorem:chain}, it suffices to prove the following.

\begin{theorem}\label{theorem:approximate}
Assume that $M$ admits a cyclic dilation, and $L\subset M$ is a closed oriented Lagrangian submanifold that is $\mathit{Spin}$ relative to $\alpha$. There exist integers $I,U\geq3$ and a sequence $(\underline{x}_i,\underline{y}_i,\underline{z}_i,\bar{\underline{x}}_i,\bar{\underline{y}}_i,\bar{\underline{z}}_i)_{i\geq I}$ of chains with $\underline{x}_i,\underline{y}_i,\underline{z}_i\in C_\ast^\lambda$, and $\bar{\underline{x}}_i,\bar{\underline{y}}_i,\bar{\underline{z}}_i\in\overline{C}_\ast^\lambda$, such that the following conditions hold.
\begin{itemize}
\item[(i)]
$\underline{x}_i\in F^1C_{-2}^\lambda$, $\bar{\underline{x}}_i\in\overline{F}^1\overline{C}_{-2}^\lambda$, $\underline{y}_i\in F^{-U}C_2^\lambda$, $\bar{\underline{y}}_i\in\overline{F}^{-U}\overline{C}_2^\lambda$, $\underline{z}_i\in F^{-1}C_1^\lambda$, $\bar{\underline{z}}_i\in\overline{F}^{-1}\overline{C}_1^\lambda$, where $\overline{F}^m$ is the filtration on $\overline{C}_\ast^\lambda$ defined above.
\item[(ii)]
$\underline{x}_i=\underline{e}_-(\bar{\underline{x}}_i)$, $\underline{y}_i=\underline{e}_-(\bar{\underline{y}}_i)$, $\underline{z}_i=\underline{e}_-(\bar{\underline{z}}_i)$.
\item[(iii)]
$\tilde{\bar{\partial}}(\bar{\underline{x}}_i)-\frac{1}{2}\left\{\bar{\underline{x}}_i,\bar{\underline{x}}_i\right\}\in \overline{F}^i\overline{C}_{-3}^\lambda$, $\tilde{\bar{\partial}}(\bar{\underline{y}}_i)-\left\{\bar{\underline{x}}_i,\bar{\underline{y}}_i\right\}-\bar{\underline{z}}_i\in \overline{F}^{i-U-1}\overline{C}_1^\lambda$, $\tilde{\bar{\partial}}(\bar{\underline{z}}_i)-\left\{\bar{\underline{x}}_i,\bar{\underline{z}}_i\right\}\in \overline{F}^{i-2}\overline{C}_0^\lambda$, where $\tilde{\bar{\partial}}$ is the boundary operator on $\overline{C}_\ast^\lambda$.
\item[(iv)]
$\underline{x}_{i+1}-\underline{e}_+(\bar{\underline{x}}_i)\in F^iC_{-2}^\lambda$, $\underline{y}_{i+1}-\underline{e}_+(\bar{\underline{y}}_i)\in F^{i-U-1}C_2^\lambda$, $\underline{z}_{i+1}-\underline{e}_+(\bar{\underline{z}}_i)\in F^{i-2}C_1^\lambda$.
\item[(v)]$\underline{x}_i(a,k)\neq0$ only if $\theta_M(a)\geq2\varepsilon$ or $a=0$, $k\geq3$. Moreover, $\underline{x}_i(0,3)$ admits a lift $\tilde{x}_i(0,3)\in C_{-2}(\!(u)\!)/uC_{-2}[\![u]\!](0,3)$ such that $\mathbf{B}_c\left(\tilde{x}_i(0,3)\right)=x_i(0,2)\in C_{-1}(0,2)$ is a cycle, regarded as one in the total complex, whose homology class coincides with $(-1)^{n+1}[L]$ under the isomorphism between de Rham and singular homologies of the free loop space. Here, the map $\mathbf{B}_c$ is defined by (\ref{eq:mark}).
\item[(vi)] $\underline{z}_i(a,k)\neq0$ only if $\theta_M(a)\geq2\varepsilon$ or $a=0$. Moreover, $\underline{z}_i(0,0)$ lifts to a cycle $z_i(0,0)\in C_1(0,0)$, regarded as a cycle in the total complex, whose homology class $\left[{z}_i(0,0)\right]$ corresponds to $(-1)^{n+1}[L]$ under the isomorphism between de Rham and singular homologies of the free loop space.
\item[(vii)] We introduce the following subsets of $H_1(L;\mathbb{Z})$.
\begin{equation}\label{eq:A1}
\underline{A}_x:=\left\{a\in H_1(L;\mathbb{Z})\vert\exists(i,k)\textrm{ such that }\bar{\underline{x}}_i(a,k)\neq0\right\},
\end{equation}
\begin{equation}\label{eq:A2}
\underline{A}_x^+:=\left\{a_1+\cdots+a_m\vert m\geq1, a_1,\cdots,a_m\in\underline{A}_x\right\},
\end{equation}
\begin{equation}\label{eq:A3}
\underline{A}_{y,z}:=\left\{a\in H_1(L;\mathbb{Z})\vert\exists(i,k)\textrm{ such that }\left(\bar{\underline{y}}_i(a,k),\bar{\underline{z}}_i(a,k)\right)\neq(0,0)\right\},
\end{equation}
\begin{equation}\label{eq:A4}
\underline{A}_{y,z}^+:=\left\{a_1+\cdots+a_m\vert m\geq1, a_1\in\underline{A}_{y,z},a_2\cdots,a_m\in\underline{A}_x\right\}.
\end{equation}
Then for any $\Xi>0$,
\begin{equation}
\underline{A}_x^+(\Xi):=\left\{a\in\underline{A}_x^+\vert\theta_M(a)<\Xi\right\}\textrm{ and }\underline{A}_{y,z}^+(\Xi):=\left\{a\in\underline{A}_{y,z}^+\vert\theta_M(a)<\Xi\right\}
\end{equation}
are finite sets.
\end{itemize}
\end{theorem}

\begin{remark}\label{remark:lift2}
As we have noticed in Remark \ref{remark:lift1}, in order to find the chains in $C_\ast^\lambda$ and $\overline{C}_\ast^\lambda$ satisfying the conditions of Theorem \ref{theorem:approximate}, it is enough to have the chains $(\tilde{x}_i,\tilde{y}_i,\tilde{z}_i,\bar{\tilde{x}}_i,\bar{\tilde{y}}_i,\bar{\tilde{z}}_i)_{i\geq I}$, where $\tilde{x}_i,\tilde{y}_i,\tilde{z}_i\in C_\ast^{S^1}$, and $\bar{\tilde{x}}_i,\bar{\tilde{y}}_i,\bar{\tilde{z}}_i\in\overline{C}_\ast^{S^1}$, satisfying the following conditions corresponding to (i)---(vii) above.
\begin{itemize}
	\item[(i')] $\tilde{x}_i\in F^1C_{-2}^{S^1}$, $\bar{\tilde{x}}_i\in\overline{F}^1\overline{C}_{-2}^{S^1}$, $\tilde{y}_i\in F^{-U}C_2^{S^1}$, $\bar{\tilde{y}}_i\in\overline{F}^{-U}\overline{C}_2^{S^1}$, $\tilde{z}_i\in F^{-1}C_1^{S^1}$, $\bar{\tilde{z}}_i\in\overline{F}^{-1}\overline{C}_1^{S^1}$.
	\item[(ii')] $\tilde{x}_i=\tilde{e}_-(\bar{\tilde{x}}_i)$, $\tilde{y}_i=\tilde{e}_-(\bar{\tilde{y}}_i)$, $\tilde{z}_i=\tilde{e}_-(\bar{\tilde{z}}_i)$.
	\item[(iii')] $\bar{\partial}^{S^1}(\bar{\tilde{x}}_i)-\frac{1}{2}\left\{\bar{\tilde{x}}_i,\bar{\tilde{x}}_i\right\}\in \overline{F}^i\overline{C}_{-3}^{S^1}$, $\bar{\partial}^{S^1}(\bar{\tilde{y}}_i)-\left\{\bar{\tilde{x}}_i,\bar{\tilde{y}}_i\right\}-\bar{\tilde{z}}_i\in \overline{F}^{i-U-1}\overline{C}_1^{S^1}$, $\bar{\partial}^{S^1}(\bar{\tilde{z}}_i)-\left\{\bar{\tilde{x}}_i,\bar{\tilde{z}}_i\right\}\in \overline{F}^{i-2}\overline{C}_0^{S^1}$.
	\item[(iv')] $\tilde{x}_{i+1}-\tilde{e}_+(\bar{\tilde{x}}_i)\in F^iC_{-2}^{S^1}$, $\tilde{y}_{i+1}-\tilde{e}_+(\bar{\tilde{y}}_i)\in F^{i-U-1}C_2^{S^1}$, $\tilde{z}_{i+1}-\tilde{e}_+(\bar{\tilde{z}}_i)\in F^{i-2}C_1^{S^1}$.
	\item[(v')] $\tilde{x}_i(a,k)\neq0$ only if $\theta_M(a)\geq2\varepsilon$ or $a=0$, $k\geq3$. Moreover, $\mathbf{B}_c\left(\tilde{x}_i(0,3)\right)=x_i(0,2)$.
	\item[(vi')] $\tilde{z}_i(a,k)\neq0$ only if $\theta_M(a)\geq2\varepsilon$ or $a=0$. Moreover, $\tilde{z}_i(0,0)\in C_1^{S^1}$ is a cycle, whose homology class $\left[\tilde{z}_i(0,0)\right]=(-1)^{n+1}[\![L]\!]$.
	\item[(vii')] Parallel to Theorem \ref{theorem:approximate}, (vii), one can also introduce the subsets of $H_1(L;\mathbb{Z})$ corresppnding to (\ref{eq:A1})--(\ref{eq:A4}) for chains in the $S^1$-equivariant complex $\overline{C}_\ast^{S^1}$, and we will denote these subsets by $A_x$, $A_x^+$, $A_{x,y}$ and $A_{x,y}^+$, respectively. For example,
	\begin{equation}
	A_x:=\left\{a\in H_1(L;\mathbb{Z})\vert\exists(i,k)\textrm{ such that }\bar{\tilde{x}}_i(a,k)\neq0\right\}.
	\end{equation}
	Then for any $\Xi>0$,
	\begin{equation}
	A_x^+(\Xi):=\left\{a\in A_x^+\vert\theta_M(a)<\Xi\right\}\textrm{ and }A_{y,z}^+(\Xi):=\left\{a\in A_{y,z}^+\vert\theta_M(a)<\Xi\right\}
	\end{equation}
	are finite sets.
\end{itemize}
We remark that due to the cyclic invariance for chains in $\overline{C}_\ast^\lambda$, the requirement $\bar{\partial}^{S^1}(\bar{\tilde{x}}_i)-\frac{1}{2}\left\{\bar{\tilde{x}}_i,\bar{\tilde{x}}_i\right\}\in \overline{F}^i\overline{C}_{-3}^{S^1}$ in (iii') can be weakened. See our proof of Theorem \ref{theorem:approximate} in Section \ref{section:proof}.
\end{remark}

The following is an $S^1$-equivariant analogue of \cite{ki2}, Lemma 6.4 (except for the item (ix), see \cite{ki3}). Its proof is almost identical to the non-equivariant case, except for some changes of notations, gradings and signs. We record the details here for the readers' convenience.

\begin{lemma}\label{lemma:induction}
Let $I,U\in\mathbb{Z}_{\geq3}$ and $\underline{x}_i,\underline{y}_i,\underline{z}_i,\bar{\underline{x}}_i,\bar{\underline{y}}_i,\bar{\underline{z}}_i$ be as in Theorem \ref{theorem:approximate}. Then there exists a sequence
\begin{equation}
(\underline{x}_{i,j},\underline{y}_{i,j},\underline{z}_{i,j},\bar{\underline{x}}_{i,j},\bar{\underline{y}}_{i,j},\bar{\underline{z}}_{i,j})_{i\geq I,j\geq0}
\end{equation}
of (relative) de Rham chains satisfying the following conditions:
\begin{itemize}
	\item[(i)] $\underline{x}_{i,0}=\underline{x}_i$, $\underline{y}_{i,0}=\underline{y}_i$, $\underline{z}_{i,0}=\underline{z}_i$, $\bar{\underline{x}}_{i,0}=\bar{\underline{x}}_i$,  $\bar{\underline{y}}_{i,0}=\bar{\underline{y}}_i$, $\bar{\underline{z}}_{i,0}=\bar{\underline{z}}_i$.
	\item[(ii)] $\underline{x}_{i,j}\in F^1C_{-2}^\lambda$, $\bar{\underline{x}}_{i,j}\in\overline{F}^1\overline{C}_{-2}^\lambda$, $\underline{y}_{i,j}\in F^{-U}C_2^\lambda$, $\bar{\underline{y}}_{i,j}\in \overline{F}^{-U}\overline{C}_2^\lambda$, $\underline{z}_{i,j}\in F^{-1}C_1^\lambda$, $\bar{\underline{z}}_{i,j}\in\overline{F}^{-1}\overline{C}_1^\lambda$.
	\item[(iii)] $\underline{x}_{i,j}=\underline{e}_-(\underline{\bar{x}}_{i,j})$, $\underline{y}_{i,j}=\underline{e}_-(\bar{\underline{y}}_{i,j})$, $\underline{z}_{i,j}=\underline{e}_-(\bar{\underline{z}}_{i,j})$.
	\item[(iv)] $\tilde{\bar{\partial}}(\bar{\underline{x}}_{i,j})-\frac{1}{2}\left\{\bar{\underline{x}}_{i,j},\bar{\underline{x}}_{i,j}\right\}\in\overline{F}^{i+j}\overline{C}_{-3}^\lambda$, $\tilde{\bar{\partial}}(\bar{\underline{y}}_{i,j})-\left\{\bar{\underline{x}}_{i,j},\bar{\underline{y}}_{i,j}\right\}-\bar{\underline{z}}_{i,j}\in\overline{F}^{i+j-U-1}\overline{C}_1^\lambda$, $\tilde{\bar{\partial}}(\bar{\underline{z}}_{i,j})-\left\{\bar{\underline{x}}_{i,j},\bar{\underline{z}}_{i,j}\right\}\in\overline{F}^{i+j-2}\overline{C}_0^\lambda$.
	\item[(v)] $\underline{x}_{i+1,j}-\underline{e}_+(\bar{\underline{x}}_{i,j})\in F^{i+j}C_{-2}^\lambda$, $\underline{y}_{i+1,j}-\underline{e}_+(\bar{\underline{y}}_{i,j})\in F^{i+j-U-1}C_2^\lambda$, $\underline{z}_{i+1,j}-\underline{e}_+(\bar{\underline{z}}_{i,j})\in F^{i+j-2}C_1^\lambda$.
	\item[(vi)] $\bar{\underline{x}}_{i,j+1}-\bar{\underline{x}}_{i,j}\in\overline{F}^{i+j}\overline{C}_{-2}^\lambda$, $\bar{\underline{y}}_{i,j+1}-\bar{\underline{y}}_{i,j}\in\overline{F}^{i+j-U-1}\overline{C}_2^\lambda$, $\bar{\underline{z}}_{i,j+1}-\bar{\underline{z}}_{i,j}\in\overline{F}^{i+j-2}\overline{C}_1^\lambda$.
	\item[(vii)] $\bar{\underline{x}}_{i,j}(a,k)\neq0$ only if $\theta_M(a)\geq2\varepsilon$ or $a=0$, $k\geq3$. Moreover, $\bar{\underline{x}}_{i,j}(0,3)$ lifts to a chain $\bar{\tilde{x}}_{i,j}(0,3)\in\overline{C}_{-2}(\!(u)\!)/u\overline{C}_{-2}[\![u]\!]$ satisfying $\overline{\mathbf{B}}_c\left(\bar{\tilde{x}}_{i,j}(0,3)\right)=\bar{x}_{i,j}(0,2)$, where $\overline{\mathbf{B}}_c:\overline{C}_\ast(\!(u)\!)/uC_\ast[\![u]\!](a,k+1)\rightarrow\overline{C}_{\ast+1}(a,k)$ is the marking map defined in the same way as (\ref{eq:mark}), and $\bar{x}_{i,j}(0,2)\in\overline{C}_{-1}(0,2)$ is a cycle whose homology class $\left[\bar{x}_{i,j}(0,2)\right]$ in the total complex coincides with $(-1)^{n+1}[L]$ under the isomorphism between relative de Rham and singular homologies of the free loop space.
	\item[(viii)] $\bar{\underline{z}}_{i,j}(a,k)\neq0$ only if $\theta_M(a)\geq2\varepsilon$ or $a=0$. Moreover, $\bar{\underline{z}}_{i,j}(0,0)\in\overline{C}_1^\lambda$ lifts to a cycle $\bar{z}_{i,j}(0,0)\in\overline{C}_1(0,0)$ whose homology class $\left[\bar{z}_{i,j}(0,0)\right]$ in the total complex corresponds to $(-1)^{n+1}[L]$ under the isomorphism between relative de Rham and singular homologies of the free loop space.
	\item[(ix)] If there exists $(i,j,k)\in\mathbb{Z}_{\geq I}\times\mathbb{Z}_{\geq0}\times\mathbb{Z}_{\geq0}$ such that 
	\begin{itemize}
		\item $\bar{\underline{x}}_{i,j}(a,k)\neq0$, then $a\in\underline{A}_x^+$.
		\item $\left(\bar{\underline{y}}_{i,j}(a,k),\bar{\underline{z}}_{i,j}(a,k)\right)\neq(0,0)$, then $a\in\underline{A}_{y,z}^+$.
	\end{itemize}
\end{itemize}
\end{lemma}
\begin{proof}
We prove the lemma by induction on $j$. Define the chains $(\underline{x}_{i,0},\underline{y}_{i,0},\underline{z}_{i,0},\bar{\underline{x}}_{i,0},\bar{\underline{y}}_{i,0},\bar{\underline{z}}_{i,0})$ as in (i). Assume that we have defined a sequence of chains $(\underline{x}_{i,j},\underline{y}_{i,j},\underline{z}_{i,j},\bar{\underline{x}}_{i,j},\bar{\underline{y}}_{i,j},\bar{\underline{z}}_{i,j})_{i\geq I}$ which satisfies the conditions (i)-(ix) above, we need to define the sequence 
\begin{equation}\label{eq:j+1}
(\underline{x}_{i,j+1},\underline{y}_{i,j+1},\underline{z}_{i,j+1},\bar{\underline{x}}_{i,j+1},\bar{\underline{y}}_{i,j+1},\bar{\underline{z}}_{i,j+1})_{i\geq I}.
\end{equation}

Set
\begin{equation}
\Delta_x^i:=\underline{x}_{i+1,j}-\underline{e}_+(\bar{\underline{x}}_{i,j})\in F^{i+j}C_{-2}^\lambda,
\end{equation}
\begin{equation}
\Delta_y^i:=\underline{y}_{i+1,j}-\underline{e}_+(\bar{\underline{y}}_{i,j})\in F^{i+j-U-1}C_2^\lambda,
\end{equation}
\begin{equation}
\Delta_z^i:=\underline{z}_{i+1,j}-\underline{e}_+(\bar{\underline{z}}_{i,j})\in F^{i+j-2}C_1^\lambda.
\end{equation}
Since $\underline{e}_-$ preserves the differential $\tilde{\bar{\partial}}$, the Lie bracket $\{\cdot,\cdot\}$, and the filtration $\overline{F}^m$, applying it to the expressions in (iv) gives
\begin{equation}
\tilde{\partial}(\underline{x}_{i+1,j})-\frac{1}{2}\left\{\underline{x}_{i+1,j},\underline{x}_{i+1,j}\right\}\in F^{i+j+1}C_{-3}^\lambda,
\end{equation}
\begin{equation}
\tilde{\partial}(\underline{y}_{i+1,j})-\left\{\underline{x}_{i+1,j},\underline{y}_{i+1,j}\right\}-\underline{z}_{i+1,j}\in F^{i+j-U}C_1^\lambda,
\end{equation}
\begin{equation}
\tilde{\partial}(\underline{z}_{i+1,j})-\left\{\underline{x}_{i+1,j},\underline{z}_{i+1,j}\right\}\in F^{i+j-1}C_0^\lambda.
\end{equation}
Combining with the definitions of $\Delta_x^i$, $\Delta_y^i$ and $\Delta_z^i$ and using (v), we obtain
\begin{equation}
\tilde{\partial}(\Delta_x^i)+\underline{e}_+\left(\tilde{\bar{\partial}}(\bar{\underline{x}}_{i,j})-\frac{1}{2}\left\{\bar{\underline{x}}_{i,j},\bar{\underline{x}}_{i,j}\right\}\right)\in F^{i+j+1}C_{-3}^\lambda,
\end{equation}
\begin{equation}
\tilde{\partial}(\Delta_y^i)+\underline{e}_+\left(\tilde{\bar{\partial}}(\bar{\underline{y}}_{i,j})-\left\{\bar{\underline{x}}_{i,j},\bar{\underline{y}}_{i,j}\right\}-\bar{\underline{z}}_{i,j}\right)\in F^{i+j-U}C_1^\lambda,
\end{equation}
\begin{equation}
\tilde{\partial}(\Delta_z^i)+\underline{e}_+\left(\tilde{\bar{\partial}}(\bar{\underline{z}}_{i,j})-\left\{\bar{\underline{x}}_{i,j},\bar{\underline{z}}_{i,j}\right\}\right)\in F^{i+j-1}C_0^\lambda.
\end{equation}
Applying the Leibniz rule and the Jacobi identity we get
\begin{equation}
\begin{split}
\tilde{\bar{\partial}}\left(\tilde{\bar{\partial}}(\bar{\underline{x}}_{i,j})-\frac{1}{2}\left\{\bar{\underline{x}}_{i,j},\bar{\underline{x}}_{i,j}\right\}\right)=&-\frac{1}{2}\left(\left\{\tilde{\bar{\partial}}(\bar{\underline{x}}_{i,j})-\frac{1}{2}\left\{\bar{\underline{x}}_{i,j},\bar{\underline{x}}_{i,j}\right\},\bar{\underline{x}}_{i,j}\right\}\right. \\
&\left.-\left\{\bar{\underline{x}}_{i,j},\tilde{\bar{\partial}}(\bar{\underline{x}}_{i,j})-\frac{1}{2}\left\{\bar{\underline{x}}_{i,j},\bar{\underline{x}}_{i,j}\right\}\right\}\right)\in\overline{F}^{i+j+1}\overline{C}_{-4}^\lambda,
\end{split}
\end{equation}
\begin{equation}
\begin{split}
\tilde{\bar{\partial}}\left(\tilde{\bar{\partial}}(\bar{\underline{y}}_{i,j})-\left\{\bar{\underline{x}}_{i,j},\bar{\underline{y}}_{i,j}\right\}-\bar{\underline{z}}_{i,j}\right)=&-\left\{\tilde{\bar{\partial}}(\bar{\underline{x}}_{i,j})-\frac{1}{2}\left\{\bar{\underline{x}}_{i,j},\bar{\underline{x}}_{i,j}\right\},\bar{\underline{y}}_{i,j}\right\} \\
&+\left\{\bar{\underline{x}}_{i,j},\tilde{\bar{\partial}}(\bar{\underline{y}}_{i,j})-\left\{\bar{\underline{x}}_{i,j},\bar{\underline{y}}_{i,j}\right\}-\bar{\underline{z}}_{i,j}\right\} \\
&-\left(\tilde{\bar{\partial}}(\bar{\underline{z}}_{i,j})-\left\{\bar{\underline{x}}_{i,j},\bar{\underline{z}}_{i,j}\right\}\right)\in\overline{F}^{i+j-U}\overline{C}_0^\lambda,
\end{split}
\end{equation}
\begin{equation}
\begin{split}
\tilde{\bar{\partial}}\left(\tilde{\bar{\partial}}(\bar{\underline{z}}_{i,j})-\left\{\bar{\underline{x}}_{i,j},\bar{\underline{z}}_{i,j}\right\}\right)=&-\left\{\tilde{\bar{\partial}}(\bar{\underline{x}}_{i,j})-\frac{1}{2}\left\{\bar{\underline{x}}_{i,j},\bar{\underline{x}}_{i,j}\right\},\bar{\underline{z}}_{i,j}\right\} \\
&+\left\{\bar{\underline{x}}_{i,j},\tilde{\bar{\partial}}(\bar{\underline{z}}_{i,j})-\left\{\bar{\underline{x}}_{i,j},\bar{\underline{z}}_{i,j}\right\}\right\}\in\overline{F}^{i+j-1}\overline{C}_{-1}^\lambda,
\end{split}
\end{equation}
where the levels of the filtration above are determined using the fact that $\left\{\overline{F}^m,\overline{F}^{m'}\right\}\subset\overline{F}^{m+m'}$.

It follows from Corollary \ref{corollary:surjective} that we can apply \cite{ki2}, Lemma 6.3 to the surjective quasi-isomorphism
\begin{equation}
\underline{e}_+:\overline{F}^D\overline{C}_\ast^\lambda/\overline{F}^{D+1}\overline{C}_\ast^\lambda\rightarrow F^DC_\ast^\lambda/F^{D+1}C_\ast^\lambda,
\end{equation}
where $D$ stands for the integer $i+j$, $i+j-U-1$, or $i+j-2$. As a consequence, there exist relative chains
\begin{equation}\label{eq:rc}
\overline{\Delta}_x^i\in\overline{F}^{i+j}\overline{C}_{-2}^\lambda,\textrm{ }\overline{\Delta}_y^i\in\overline{F}^{i+j-U-1}\overline{C}_2^\lambda,\textrm{ }\overline{\Delta}_z^i\in\overline{F}^{i+j-2}\overline{C}_1^\lambda
\end{equation}
such that
\begin{equation}\label{eq:diff}
\underline{e}_+(\overline{\Delta}_x^i)-\Delta_x^i\in F^{i+j+1}C_{-2}^\lambda,\textrm{ }\underline{e}_+(\overline{\Delta}_y^i)-\Delta_y^i\in F^{i+j-U}C_2^\lambda,\textrm{ }\underline{e}_+(\overline{\Delta}_z^i)-\Delta_z^i\in F^{i+j-1}C_1^\lambda
\end{equation}
and
\begin{equation}
\tilde{\bar{\partial}}(\overline{\Delta}_x^i)+\left(\tilde{\bar{\partial}}(\bar{\underline{x}}_{i,j})-\frac{1}{2}\left\{\bar{\underline{x}}_{i,j},\bar{\underline{x}}_{i,j}\right\}\right)\in\overline{F}^{i+j+1}\overline{C}_{-3}^\lambda,
\end{equation}
\begin{equation}
\tilde{\bar{\partial}}(\overline{\Delta}_y^i)+\left(\tilde{\bar{\partial}}(\bar{\underline{y}}_{i,j})-\left\{\bar{\underline{x}}_{i,j},\bar{\underline{y}}_{i,j}\right\}-\bar{\underline{z}}_{i,j}\right)\in\overline{F}^{i+j-U}\overline{C}_1^\lambda,
\end{equation}
\begin{equation}
\tilde{\bar{\partial}}(\overline{\Delta}_z^i)+\left(\tilde{\bar{\partial}}(\bar{\underline{z}}_{i,j})-\left\{\bar{\underline{x}}_{i,j},\bar{\underline{z}}_{i,j}\right\}\right)\in\overline{F}^{i+j-1}\overline{C}_0^\lambda.
\end{equation}
We now define the chains in (\ref{eq:j+1}) as
\begin{equation}
\bar{\underline{x}}_{i,j+1}:=\bar{\underline{x}}_{i,j}+\overline{\Delta}_x^i,\textrm{ }\bar{\underline{y}}_{i,j+1}:=\bar{\underline{y}}_{i,j}+\overline{\Delta}_y^i,\textrm{ }\bar{\underline{z}}_{i,j+1}:=\bar{\underline{z}}_{i,j}+\overline{\Delta}_z^i,
\end{equation}
and
\begin{equation}
\underline{x}_{i,j+1}:=\underline{e}_-(\bar{\underline{x}}_{i,j+1}),\textrm{ }\underline{y}_{i,j+1}:=\underline{e}_-(\bar{\underline{y}}_{i,j+1}),\textrm{ }\underline{z}_{i,j+1}:=\underline{e}_-(\bar{\underline{z}}_{i,j+1}).
\end{equation}
To complete the induction step, one needs to check the following properties for every $i\in\mathbb{Z}_{\geq I}$:
\begin{equation}\label{eq:1}
\tilde{\bar{\partial}}(\bar{\underline{x}}_{i,j+1})-\frac{1}{2}\left\{\bar{\underline{x}}_{i,j+1},\bar{\underline{x}}_{i,j+1}\right\}\in\overline{F}^{i+j+1}\overline{C}_{-3}^\lambda
\end{equation}
\begin{equation}\label{eq:2}
\tilde{\bar{\partial}}(\bar{\underline{y}}_{i,j+1})-\left\{\bar{\underline{x}}_{i,j+1},\bar{\underline{y}}_{i,j+1}\right\}-\bar{\underline{z}}_{i,j+1}\in\overline{F}^{i+j-U}\overline{C}_1^\lambda,
\end{equation}
\begin{equation}\label{eq:3}
\tilde{\bar{\partial}}(\bar{\underline{z}}_{i,j+1})-\left\{\bar{\underline{x}}_{i,j+1},\bar{\underline{z}}_{i,j+1}\right\}\in\overline{F}^{i+j-1}\overline{C}_0^\lambda,
\end{equation}
\begin{equation}\label{eq:4}
\underline{x}_{i+1,j+1}-\underline{e}_+(\bar{\underline{x}}_{i,j+1})\in F^{i+j+1}C_{-2}^\lambda,
\end{equation}
\begin{equation}\label{eq:5}
\underline{y}_{i+1,j+1}-\underline{e}_+(\bar{\underline{y}}_{i,j+1})\in F^{i+j-U}C_2^\lambda,
\end{equation}
\begin{equation}\label{eq:6}
\underline{z}_{i+1,j+1}-\underline{e}_+(\bar{\underline{z}}_{i,j+1})\in F^{i+j-1}C_1^\lambda.
\end{equation}

To prove (\ref{eq:1}), we use the definition of $\bar{\underline{x}}_{i,j+1}$ to compute
\begin{equation}
\begin{split}
\tilde{\bar{\partial}}(\bar{\underline{x}}_{i,j+1})-\frac{1}{2}\left\{\bar{\underline{x}}_{i,j+1},\bar{\underline{x}}_{i,j+1}\right\}&=\left(\tilde{\bar{\partial}}(\bar{\underline{x}}_{i,j})+\tilde{\bar{\partial}}(\overline{\Delta}_x^i)-\frac{1}{2}\left\{\bar{\underline{x}}_{i,j},\bar{\underline{x}}_{i,j}\right\}\right) \\
&-\frac{1}{2}\left\{\overline{\Delta}_x^i,\overline{\Delta}_x^i\right\}-\left\{\bar{\underline{x}}_{i,j},\overline{\Delta}_x^i\right\}.
\end{split}
\end{equation}
Now (\ref{eq:1}) follows since all three terms on the right-hand side lie in $\overline{F}^{i+j+1}\overline{C}_{-3}^\lambda$.

For (\ref{eq:2}), using the definitions of $\bar{\underline{y}}_{i,j+1}$ and $\bar{\underline{z}}_{i,j+1}$ we have
\begin{equation}
\begin{split}
\tilde{\bar{\partial}}(\bar{\underline{y}}_{i,j+1})-\left\{\bar{\underline{x}}_{i,j+1},\bar{\underline{y}}_{i,j+1}\right\}-\bar{\underline{z}}_{i,j+1}&=\left(\tilde{\bar{\partial}}(\overline{\Delta}_y^i)+\tilde{\bar{\partial}}(\bar{\underline{y}}_{i,j})-\left\{\bar{\underline{x}}_{i,j},\bar{\underline{y}}_{i,j}\right\}-\bar{\underline{z}}_{i,j}\right) \\
&-\left\{\overline{\Delta}_x^i,\bar{\underline{y}}_{i,j}\right\}-\left\{\bar{\underline{x}}_{i,j},\overline{\Delta}_y^i\right\}-\left\{\overline{\Delta}_x^i,\overline{\Delta}_y^i\right\}-\overline{\Delta}_z^i.
\end{split}
\end{equation}
Since all the terms on the right-hand side lie in $\overline{F}^{i+j-U}\overline{C}_1^\lambda$, (\ref{eq:2}) follows.

By definitions of $\bar{\underline{x}}_{i,j+1}$ and $\bar{\underline{z}}_{i,j+1}$, one can compute
\begin{equation}
\begin{split}
\tilde{\bar{\partial}}(\bar{\underline{z}}_{i,j+1})-\left\{\bar{\underline{x}}_{i,j+1},\bar{\underline{z}}_{i,j+1}\right\}&=\left(\tilde{\bar{\partial}}(\bar{\underline{z}}_{i,j})+\tilde{\bar{\partial}}(\overline{\Delta}_z^i)-\left\{\bar{\underline{x}}_{i,j},\bar{\underline{z}}_{i,j}\right\}\right) \\
&-\left\{\overline{\Delta}_x^i,\bar{\underline{z}}_{i,j}\right\}-\left\{\bar{\underline{x}}_{i,j},\overline{\Delta}_z^i\right\}-\left\{\overline{\Delta}_x^i,\overline{\Delta}_z^i\right\}.
\end{split}
\end{equation}
(\ref{eq:3}) follows since all the four terms on the right-hand side lie in $\overline{F}^{i+j-1}\overline{C}_0^\lambda$.

(\ref{eq:4}) follows from the computation
\begin{equation}\label{eq:v4}
\begin{split}
\underline{x}_{i+1,j+1}-\underline{e}_+(\bar{\underline{x}}_{i,j+1})&=(\underline{x}_{i+1,j+1}-\underline{x}_{i+1,j})+\left(\underline{x}_{i+1,j}-\underline{e}_+(\bar{\underline{x}}_{i,j})\right)+e_+(\bar{\underline{x}}_{i,j}-\bar{\underline{x}}_{i,j+1}) \\
&=\underline{e}_-(\overline{\Delta}_x^{i+1})+\left(\Delta_x^i-\underline{e}_+(\overline{\Delta}_x^i)\right),
\end{split}
\end{equation}
where the definitions of $\Delta_x^i$, $\overline{\Delta}_x^i$ are used, and the first term of the second line is obtained by applying $\underline{e}_-$ to $\bar{\underline{x}}_{i+1,j+1}:=\bar{\underline{x}}_{i+1,j}+\overline{\Delta}_x^{i+1}$. Note that by (\ref{eq:rc}) and (\ref{eq:diff}), the right-hand side of (\ref{eq:v4}) lies in $F^{i+j+1}C_{-2}^\lambda$.

Similarly, (\ref{eq:5}) and (\ref{eq:6}) follow from the computations
\begin{equation}
	\begin{split}
	\underline{y}_{i+1,j+1}-\underline{e}_+(\bar{\underline{y}}_{i,j+1})&=(\underline{y}_{i+1,j+1}-\underline{y}_{i+1,j})+\left(\underline{y}_{i+1,j}-\underline{e}_+(\bar{\underline{y}}_{i,j})\right)+\underline{e}_+(\bar{\underline{y}}_{i,j}-\bar{\underline{y}}_{i,j+1}) \\
	&=\underline{e}_-(\overline{\Delta}_y^{i+1})+\left(\Delta_y^i-\underline{e}_+(\overline{\Delta}_y^i)\right),
	\end{split}
\end{equation}
\begin{equation}
	\begin{split}
	\underline{z}_{i+1,j+1}-\underline{e}_+(\bar{\underline{z}}_{i,j+1})&=(\underline{z}_{i+1,j+1}-\underline{z}_{i+1,j})+\left(\underline{z}_{i+1,j}-\underline{e}_+(\bar{\underline{z}}_{i,j})\right)+\underline{e}_+(\bar{\underline{z}}_{i,j}-\bar{\underline{z}}_{i,j+1}) \\
	&=\underline{e}_-(\overline{\Delta}_z^{i+1})+\left(\Delta_z^i-\underline{e}_+(\overline{\Delta}_z^i)\right).
	\end{split}
\end{equation}

By the induction hypothesis, $\bar{\underline{x}}_{i,j}$ and $\bar{\underline{z}}_{i,j}$ satisfy the conditions (vii)---(ix) in the statement of the lemma. Since the $(a,k)$-components of $\Delta_x^i$ and $\tilde{\bar{\partial}}(\bar{\underline{x}}_{i,j})-\frac{1}{2}\left\{\bar{\underline{x}}_{i,j},\bar{\underline{x}}_{i,j}\right\}$ are non-zero only when $\theta_M(a)\geq2\varepsilon$ or $a=0$, we can take $\overline{\Delta}_x^i(a,k)\neq0$ only when $\theta_M(a)\geq2\varepsilon$ or $a=0$. Similarly, we can take $\overline{\Delta}_z^i$ so that $\overline{\Delta}_z^i(a,k)\neq0$ only if $\theta_M(a)\geq2\varepsilon$ or $a=0$. By definitions of $\bar{\underline{x}}_{i,j+1}$ and $\bar{\underline{z}}_{i,j+1}$, it follows that $\bar{\underline{x}}_{i,j+1}(a,k)$ and $\bar{\underline{z}}_{i,j+1}(a,k)$ are non-zero only if $\theta_M(a)\geq2\varepsilon$ or $a=0$. Moreover, since $\overline{\Delta}_x^i\in\overline{F}^3\overline{C}_{-2}^\lambda$ and $\overline{\Delta}_z^i\in\overline{F}^0\overline{C}_1^\lambda$, it follows from the definition of the filtration (\ref{eq:filt}) that $\overline{\Delta}_x^i(0,k)=0$ if $k=0,1,2,3$ and $\overline{\Delta}_z^i(0,k)=0$ if $k=0$. This shows that the corresponding statement holds for $\bar{\underline{x}}_{i,j+1}$ and $\bar{\underline{z}}_{i,j+1}$ as well. Finally, if $a\notin\underline{A}_x^+$, then for any $k\in\mathbb{Z}_{\geq0}$, the $(a,k)$-components of $\Delta_x^i$ and $\tilde{\bar{\partial}}(\bar{\underline{x}}_{i,j})-\frac{1}{2}\left\{\bar{\underline{x}}_{i,j},\bar{\underline{x}}_{i,j}\right\}$ are zero by the induction hypothesis. Thus one can take $\overline{\Delta}_x^i$ so that $\overline{\Delta}_x^i(a,k)=0$ for any $k\in\mathbb{Z}_{\geq0}$. Similarly, if $a\notin\underline{A}_{y,z}^+$, then for any $k\in\mathbb{Z}_{\geq0}$, the $(a,k)$-components of $\Delta_y^i$, $\Delta_z^i$, $\tilde{\bar{\partial}}(\bar{\underline{y}}_{i,j})-\left\{\bar{\underline{x}}_{i,j},\bar{\underline{y}}_{i,j}\right\}-\bar{\underline{z}}_{i,j}$ and $\tilde{\bar{\partial}}(\bar{\underline{z}}_{i,j})-\left\{\bar{\underline{x}}_{i,j},\bar{\underline{z}}_{i,j}\right\}$ are zero. It follows that one can take $\overline{\Delta}_y^i$ and $\overline{\Delta}_z^i$ so that $\overline{\Delta}_y^i(a,k)=\overline{\Delta}_z^i(a,k)=0$ for any $k\in\mathbb{Z}_{\geq0}$. This verifies (ix) for the chains $\bar{\underline{x}}_{i,j+1}$, $\bar{\underline{y}}_{i,j+1}$ and $\bar{\underline{z}}_{i,j+1}$.
\end{proof}

\begin{proposition}
Theorem \ref{theorem:approximate} implies Theorem \ref{theorem:chain}.
\end{proposition}
\begin{proof}
Fix an integer $i\geq I$. For every $j\in\mathbb{Z}_{\geq0}$, applying $\underline{e}_-$ to the equations in Lemma \ref{lemma:induction}, (vi) we obtain
\begin{equation}
\underline{x}_{i,j+1}-\underline{x}_{i,j}\in F^{i+j}C_{-2}^\lambda,\textrm{ }\underline{y}_{i,j+1}-\underline{y}_{i,j}\in F^{i+j-U-1}C_2^\lambda,\textrm{ }\underline{z}_{i,j+1}-\underline{z}_{i,j}\in F^{i+j-2}C_1^\lambda.
\end{equation}
Thus the limits
\begin{equation}
\underline{x}:=\lim_{j\rightarrow\infty}\underline{x}_{i,j},\textrm{ }\underline{y}:=\lim_{j\rightarrow\infty}\underline{y}_{i,j},\textrm{ }\underline{z}:=\lim_{j\rightarrow\infty}\underline{z}_{i,j}
\end{equation}
exist in the completion of $C_\ast^\lambda$ with respect to the filtration $F^m$, and they satisfy the equations
\begin{equation}
\tilde{\partial}(\underline{x})-\frac{1}{2}\left\{\underline{x},\underline{x}\right\}=0,\textrm{ }\tilde{\partial}(\underline{y})-\left\{\underline{x},\underline{y}\right\}=\underline{z}.
\end{equation}
To show that we actually have $\underline{x}\in\widehat{C}_{-2}^\lambda$, $\underline{y}\in\widehat{C}_2^\lambda$ and $\underline{z}\in\widehat{C}_1^\lambda$, we need to verify that for any $\Xi>0$, there are only finitely many classes $a\in H_1(L;\mathbb{Z})$ with $\theta_M(a)<\Xi$ and $\left(x(a,k),y(a,k),z(a,k)\right)\neq0$ for some $k\in\mathbb{Z}_{\geq0}$. This follows from the observation that by Lemma \ref{lemma:induction} (ix), such a class $a$ must satisfy $a\in\underline{A}_x^+\cup\underline{A}_{y,z}^+$. On the other hand, Theorem \ref{theorem:approximate}, (vii) implies that $\underline{A}_x^+(\Xi)\cup\underline{A}_{y,z}^+(\Xi)$ is a finite set. (iii) and (iv) in the statement of Theorem \ref{theorem:chain} follow from the last two conditions of Lemma \ref{lemma:induction}.
\end{proof}

\section{Moduli spaces and Kuranishi structures}\label{section:moduli}

This section contains the main contents of this paper, where we prove Theorem \ref{theorem:main}. In the previous section, we have reduced it to the proof of Theorem \ref{theorem:approximate}, which requires us to produce sequences of cyclic de Rham chains in the complex $C_\ast^\lambda$ (and its relative version $\overline{C}_\ast^\lambda$) satisfying certain conditions. As mentioned in Remarks \ref{remark:lift1} and \ref{remark:lift2}, these chains will be obtained as images of $S^1$-equivariant de Rham chains in $C_\ast^{S^1}$ (and its relative version $\overline{C}_\ast^{S^1}$) under the quasi-isomorphism $\phi_\lambda$ (and its relative version) in Lemma \ref{lemma:qis}. Unlike the case of $\mathbb{C}^n$, for general Liouville manifolds with cyclic dilations, the definitions of these chains will involve new sequences of moduli spaces\footnote{With these moduli spaces, one can actually give a new proof of Fukaya-Irie's result in the $\mathbb{C}^n$ case using our ``$S^1$-equivariant argument". See Remark \ref{remark:subcritical}.}, which we will introduce below in Sections \ref{section:disc} and \ref{section:CG}. The proof follows generally the strategy of Irie \cite{ki2}, which is based on the following underlying principle: given a K-space $(X,\widehat{\mathcal{U}})$ with a CF(continuous family)-perturbation $\widehat{\mathcal{S}}=(\widehat{\mathcal{S}}^\varepsilon)_{0<\varepsilon\leq1}$ and a strongly smooth map $\hat{f}:(X,\widehat{\mathcal{U}})\rightarrow\mathcal{L}_{k+1}$, one can define a de Rham chain $\hat{f}_\ast(X,\widehat{\mathcal{U}},\widehat{\mathcal{S}}^\varepsilon)\in C_\ast$ for sufficiently small $\varepsilon$, by integration along the fiber. We will actually need slight variations of this principle for admissible K-spaces and relative de Rham chains. These facts are briefly recalled in Appendix \ref{section:pushforward}. To apply this principle, we equip the relevant moduli spaces with Kuranishi structures in Section \ref{section:compactify}, and define strongly smooth maps from these spaces to $\mathcal{L}_{k+1}$, so that they are compatible with the boundary strata in Section \ref{section:smooth}.

In this section, $M$ will be a Liouville manifold with $c_1(M)=0$, and $L\subset M$ is a closed Lagrangian submanifold, which we assume to be oriented and relatively \textit{Spin} with respect to some fixed $\mathbb{Z}_2$-gerbe $\alpha\in C^2(M;\mathbb{Z}_2)$. 

\subsection{Moduli spaces of holomorphic discs}\label{section:disc}

We start by recalling the definition of a family of moduli spaces which appears frequently in the study of mirror symmetry \cite{fooo1}, while here they serve as the sources of the chains approximating the (non-equivariant) Maurer-Cartan element $x\in\widehat{C}_{-1}$ mentioned in the introduction. Along the way, we also fix some notations.

Let $\mathcal{R}_{k+1}$ be the moduli space of closed unit discs $D$ with distinct marked points $z_0,\cdots,z_k\in\partial D$ aligned in counterclockwise order, where $k\in\mathbb{Z}_{\geq0}$, modulo the automorphism group $\mathit{Aut}(D)\cong\mathit{PSL}(2,\mathbb{R})$. For any homotopy class $\beta\in\pi_2(M,L)$ with $\beta\neq0$, or $\beta=0$ and $k\geq2$, define\footnote{Using $\mathcal{R}$ to denote the moduli space of holomorphic discs is an unusual notation, but it will be convenient when these moduli spaces appear in the boundary strata of the Cohen-Ganatra moduli spaces defined below.}
\begin{equation}
\mathcal{R}_{k+1}(L,\beta)
\end{equation}
to be the space of pairs
\begin{equation}\label{eq:pair1}
\left((D,z_0,\cdots,z_k),u\right), 
\end{equation}
where the map $u:(D,\partial D)\rightarrow(M,L)$ satisfies $\bar{\partial}u=0$ and $[u]=\beta$. Here, the Cauchy-Riemann operator $\bar{\partial}$ is taken with respect to the almost complex structure $J_M$ fixed at the beginning of Section \ref{section:approximation}. As a convention, we have $\mathcal{R}_1(L,0)=\mathcal{R}_2(L,0)=\emptyset$.

For a general closed Lagrangian submanifold $L\subset M$, the transversality of the moduli spaces $\mathcal{R}_{k+1}(L,\beta)$ cannot be achieved with standard perturbations of the almost complex structure. To study it we need the language of Kuranishi structures introduced by Fukaya-Oh-Ohta-Ono $\cite{fooo1,fooo2,fooo3,fooo4}$, which we will briefly recall in Appendix \ref{section:basic}. $\mathcal{R}_{k+1}(L,\beta)$ admits a compactification $\overline{\mathcal{R}}_{k+1}(L,\beta)$, which is an admissible K-space and is modeled on decorated rooted ribbon trees. See $\cite{ki2}$, Section 7.2.2 for details. 

To describe the compactification $\overline{\mathcal{R}}_{k+1}(L,\beta)$, we recall the following notion.

\begin{definition}[\cite{ki2}, Definition 7.18]\label{definition:tree}
A decorated rooted ribbon tree is a pair $(T,B)$ satisfying the following requirements.
\begin{itemize}
	\item[(i)] $T$ is a connected tree, with the set of vertices $C_0(T)$ and the set of edges $C_1(T)$.
	\item[(ii)] For each $v\in C_0(T)$, a cyclic order of the set of edges is fixed.
	\item[(iii)] A decomposition $C_0(T)=C_{0,\mathit{int}}(T)\sqcup C_{0,\mathit{ext}}(T)$ into the set of interior and exterior vertices. For every $v\in C_{0,\mathit{int}}(T)$, define $k_v$ to be the valency of $v$ minus 1.
	\item[(iv)] A distinguished element in $C_{0,\mathit{ext}}(T)$, which plays the role of the root.
	\item[(v)] The valency of every exterior vertex is 1.
	\item[(vi)] There is a map $B:C_{0,\mathit{int}}(T)\rightarrow\pi_2(M,L)$. For every $v\in C_{0,\mathit{int}}(T)$, either $d\theta_M(B(v))>0$ or $B(v)=0$.
	\item[(vii)] Every $v\in C_{0,\mathit{int}}(T)$ with $B(v)=0$ has valency at least 3.
\end{itemize}
\end{definition}
Denote also by $C_{1,\mathit{int}}(T)$ is the set of interior edges, and by $C_{1,\mathit{ext}}(T)$ the set of exterior edges. An edge is called exterior if it contains an exterior vertex, otherwise it is interior. For every $k\in\mathbb{Z}_{\geq0}$ and $\beta\in\pi_2(M,L)$, let $\mathcal{G}(k+1,\beta)$ be the set of decorated rooted ribbon trees with $\# C_{0,\mathit{ext}}(T)=k+1$ and $\sum_{v\in C_{0,\mathit{int}}(T)}B(v)=\beta$. For every $(T,B)\in\mathcal{G}(k+1,\beta)$, one can define an ``interior" evaluation map
\begin{equation}\label{eq:ev-int}
\mathit{ev}_\mathit{int}:\prod_{v\in C_{0,\mathit{int}}(T)}\mathcal{R}_{k_v+1}(L,B(v))\rightarrow\prod_{e\in C_{1,\mathit{int}}(T)}L^2
\end{equation}
by, roughly speaking, evaluating at the two endpoints of each edge $e\in C_{1,\mathit{int}}(T)$. For details, see \cite{fooo4}, Section 21.1. We also have the ``exterior" evaluation map
\begin{equation}\label{eq:ev-ext}
\mathit{ev}_\mathit{ext}:\prod_{C_{0,\mathit{int}}(T)}\mathcal{R}_{k_v+1}(L,B(v))\rightarrow\prod_{e\in C_{1,\mathit{ext}}(T)}L\cong L^{k+1},
\end{equation}
defined in the obvious way by evaluating at the exterior edges. Define the evaluation map
\begin{equation}
\mathit{ev}^\mathcal{R}=(\mathit{ev}_0^\mathcal{R},\cdots,\mathit{ev}_k^\mathcal{R}):\overline{\mathcal{R}}_{k+1}(L,\beta)\rightarrow L^{k+1}
\end{equation}
on the compactified moduli space by restricting $\mathit{ev}_\mathit{ext}$ to
\begin{equation}
\overline{\mathcal{R}}_{k+1}(L,\beta):=\bigsqcup_{(T,B)\in\mathcal{G}(k+1,\beta)}\left(\prod_{e\in C_{1,\mathit{int}}(T)}L\right)\textrm{ }{{}_\Delta\times_{\mathit{ev}_\mathit{int}}}\textrm{ }\left(\prod_{v\in C_{0,\mathit{int}}(T)}\mathcal{R}_{k_v+1}(L,B(v))\right),
\end{equation}
where $\Delta:\prod_{e\in C_{1,\mathit{int}}(T)}L\rightarrow\prod_{e\in C_{1,\mathit{int}}(T)}L^2$ is the diagonal map.

To construct the chain level Maurer-Cartan element $\tilde{x}\in\widehat{C}_{-2}^{S^1}$ in the equivariant case, we need a slight variation of the moduli spaces $\mathcal{R}_{k+1}(L,\beta)$ defined above. 

For $k\geq3$, define
\begin{equation}
\mathcal{R}_{k+1,\vartheta}:=\left\{(D,z_0,\cdots,z_{k-1},z_k=z_0e^{i\vartheta_k})\right\}/\mathit{Aut}(D),
\end{equation}
where $z_0,\cdots,z_k\in\partial D$ are ordered counterclockwisely and $\vartheta_k\in(0,2\pi)$ is a fixed constant. Equivalently, regarding $z_0$ as the base point, with respect to which every other marked point $z_i$, where $1\leq i\leq k$, has an argument $\vartheta_i\in(0,2\pi)$. Then $\mathcal{R}_{k+1,\vartheta}\subset\mathcal{R}_{k+1}$ is the subspace where every element satisfies
\begin{equation}\label{eq:cons}
\vartheta_1<\cdots<\vartheta_{k-1}<\vartheta_k=\mathrm{const}.
\end{equation}
By convention $\mathcal{R}_{2,\vartheta}=\mathcal{R}_{3,\vartheta}=\emptyset$. For $1\leq i\leq k$, there is a map
\begin{equation}\label{eq:f1}
\pi_{\vartheta,i}:\mathcal{R}_{k+1,\vartheta}\rightarrow\mathcal{R}_k
\end{equation}
defined by applying the cyclic permutation $i$ times to the boundary marked points $z_0,\cdots,z_k$, so that $z_{i+j\textrm{ mod }k+1}$ becomes $z_j$ for $0\leq j\leq k$, and then forgetting the point labeled $z_{k+1-i}$ after the permutation. Since the position of $z_{k+1-i}$ after the permutation is uniquely determined by $z_{k-i}$ by our assumption, $(-1)^{ki}\pi_{\vartheta,i}$ is an orientation-preserving embedding (the sign $(-1)^{ki}$ follows from our orientation convention in Appendix \ref{section:orientation}), which identifies $\mathcal{R}_{k+1,\vartheta}$ as an open sector of $\mathcal{R}_k$. More precisely, we have the following:

\begin{lemma}\label{lemma:erase}
When $k\geq3$, the disjoint union $\bigsqcup_{1\leq i\leq k}(-1)^{ki}\pi_{\vartheta,i}(\mathcal{R}_{k+1,\vartheta})$ covers all but codimension $1$ strata of $\mathcal{R}_k$.
\end{lemma}
\begin{proof}
This follows directly from the definitions. After permuting $i$ times the boundary marked points $z_0,\cdots,z_k\in\partial D$, the point $z_0$ before permutation will become a point lying in the open arc $(z_{k-i},z_{k-i+1\textrm{ mod }k+1})$ after the permutation, therefore performing cyclic permutations $k$ times exhausts the whole circle $\partial D$ (up to isolated points).
\end{proof}

Consider the compactification $\overline{\mathcal{R}}_{k+1,\vartheta}$, which consists of nodal discs with a total number of $k+1$ marked points (excluding the nodes) on the boundaries of the components. For an element $S$ of $\overline{\mathcal{R}}_{k+1,\vartheta}$, we call the component $S_0\subset S$ containing the marked point $z_0$ the \textit{main component}. Let $z_{i_1},\cdots,z_{i_r}\in\partial S_0$ be the marked points on the boundary of the main component, and let $\zeta_1,\cdots,\zeta_s\in\partial S_0$ be the nodes. Denote by $\varphi_i$ the argument of $\zeta_i$ taken with respect to the base point $z_0$. If $i_r=k$, then $S$ is required to satisfy $\vartheta_{i_r}=\mathrm{const}$, otherwise we impose the condition $\varphi_i=\mathrm{const}$ (where the constant here is the same as the value of $\vartheta_k$ before breaking) for the subtree of disc bubbles emanating from $\zeta_i$ containing the marked point $z_k$. For a concrete example of a nodal disc $S\in\overline{\mathcal{R}}_{8,\vartheta}$, see Figure \ref{fig:bubble}. Note that the codimension $1$ boundary of $\overline{\mathcal{R}}_{k+1,\vartheta}$ is covered by the natural inclusions of
\begin{equation}\label{eq:st}
\bigsqcup_{\substack{k_1+k_2=k+1\\1\leq k_1\leq k-1\\1\leq i\leq k_1}}\overline{\mathcal{R}}_{k_1+1,\vartheta}\textrm{ }{{}_i\times_0}\textrm{ }\overline{\mathcal{R}}_{k_2+1},
\end{equation}
where the notation ${{}_i\times_0}$ means that the disc breaking happens at $z_i$, with the nodal point playing the role of $z_0$ on the disc bubble in $\overline{\mathcal{R}}_{k_2+1}$. The map $\pi_{\vartheta,i}$ extends as a map $\overline{\mathcal{R}}_{k+1,\vartheta}\rightarrow\overline{\mathcal{R}}_k$ defined on the compactifications, an example is given in Figure \ref{fig:bubble}.

\begin{figure}
	\centering
	\begin{tikzpicture}
		\filldraw[draw=black,color={black!15},opacity=0.5] (0,0) circle (1.5);
		\filldraw[draw=black,color={black!15},opacity=0.5] (2.5,0) circle (1);
		\filldraw[draw=black,color={black!15},opacity=0.5] (0,2.5) circle (1);
		\filldraw[draw=black,color={black!15},opacity=0.5] (4.3,0) circle (0.8);
		\draw (0,0) circle [radius=1.5];
		\draw (2.5,0) circle [radius=1];
		\draw (4.3,0) circle [radius=0.8];
		\draw (0,2.5) circle [radius=1];
		
		\draw (-1.061,-1.061) node[circle,fill,inner sep=1pt] {};
		\node at (-1.3,-1.1) {$z_0$};
		\draw (1.5,0) node[circle,fill,inner sep=1pt] {};
		\node at (1.25,0) {$\zeta_1$};
		\draw (0,1.5) node[circle,fill,inner sep=1pt] {};
		\node at (0,1.25) {$\zeta_2$};
		\draw (-1,2.5) node[circle,fill,inner sep=1pt] {};
		\node at (-1.25,2.5) {$z_7$};
		\draw (1,2.5) node[circle,fill,inner sep=1pt] {};
		\node at (1.25,2.5) {$z_6$};
		\draw (2.5,1) node[circle,fill,inner sep=1pt] {};
		\node at (2.5,1.2) {$z_5$};
		\draw (2.5,-1) node[circle,fill,inner sep=1pt] {};
		\node at (2.5,-1.25) {$z_1$};
		\draw (3.5,0) node[circle,fill,inner sep=1pt] {};
		\draw (5.1,0) node[circle,fill,inner sep=1pt] {};
		\node at (5.35,0) {$z_3$};
		\draw (4.3,0.8) node[circle,fill,inner sep=1pt] {};
		\node at (4.3,1) {$z_4$};
		\draw (4.3,-0.8) node[circle,fill,inner sep=1pt] {};
		\node at (4.3,-1.05) {$z_2$};
		
		\filldraw[draw=black,color={black!15},opacity=0.5] (0,-6) circle (1.5);
		\filldraw[draw=black,color={black!15},opacity=0.5] (2.5,-6) circle (1);
		\filldraw[draw=black,color={black!15},opacity=0.5] (0,-3.5) circle (1);
		\filldraw[draw=black,color={black!15},opacity=0.5] (4.3,-6) circle (0.8);
		\draw (0,-6) circle [radius=1.5];
		\draw (2.5,-6) circle [radius=1];
		\draw (4.3,-6) circle [radius=0.8];
		\draw (0,-3.5) circle [radius=1];
		
		\draw (1.5,-6) node[circle,fill,inner sep=1pt] {};
		\node at (1.25,-6) {$\zeta_1$};
		\draw (0,-4.5) node[circle,fill,inner sep=1pt] {};
		\node at (0,-4.75) {$\zeta_2$};
		\draw (-1,-3.5) node[circle,fill,inner sep=1pt] {};
		\node at (-1.25,-3.5) {$z_6$};
		\draw (1,-3.5) node[circle,fill,inner sep=1pt] {};
		\node at (1.25,-3.5) {$z_5$};
		\draw (2.5,-5) node[circle,fill,inner sep=1pt] {};
		\node at (2.5,-4.8) {$z_4$};
		\draw (2.5,-7) node[circle,fill,inner sep=1pt] {};
		\node at (2.5,-7.25) {$z_0$};
		\draw (3.5,-6) node[circle,fill,inner sep=1pt] {};
		\draw (5.1,-6) node[circle,fill,inner sep=1pt] {};
		\node at (5.35,-6) {$z_2$};
		\draw (4.3,-5.2) node[circle,fill,inner sep=1pt] {};
		\node at (4.3,-5) {$z_3$};
		\draw (4.3,-6.8) node[circle,fill,inner sep=1pt] {};
		\node at (4.3,-7.05) {$z_1$};
		
		\draw[->] (6,0) to[bend left] (6,-6);
		\node at (7.3,-3) {$\pi_{\vartheta,1}$};
		\end{tikzpicture}
	\caption{An element of the moduli space $\overline{\mathcal{R}}_{8,\vartheta}$, which satisfies $\varphi_2=\mathrm{const}$, and its image under $\pi_{\vartheta,1}$, which gives an element in $\overline{\mathcal{R}}_7$. }
	\label{fig:bubble}
\end{figure}
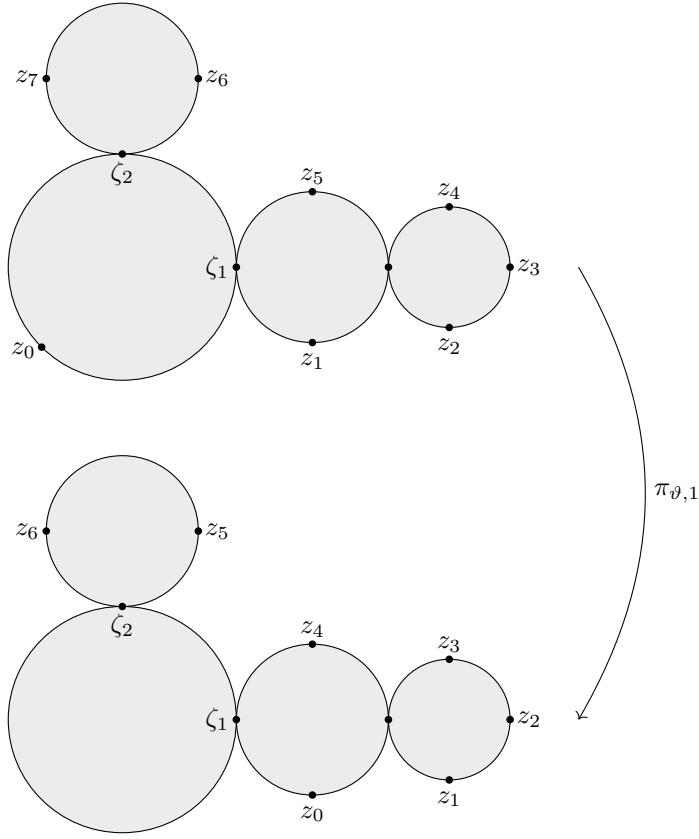

When $\beta=0$, we set $\mathcal{R}_{2,\vartheta}(L,0)=\mathcal{R}_{3,\vartheta}(L,0)=\emptyset$. When $\beta\neq0$ or $k\geq3$, define
\begin{equation}
\mathcal{R}_{k+1,\vartheta}(L,\beta)
\end{equation}
to be the space of $\left\{(u,z_0,\cdots,z_k=z_0e^{i\vartheta_k})\right\}/\mathit{Aut}(D)$, with $u:(D,\partial D)\rightarrow(M,L)$ as above and the position of $z_k$ fixed as in (\ref{eq:cons}). It is clear that $\mathcal{R}_{k+1,\vartheta}(L,\beta)$ is a K-space.  As in the case of $\overline{\mathcal{R}}_{k+1}(L,\beta)$, the compactification $\overline{\mathcal{R}}_{k+1,\vartheta}(L,\beta)$ is also modeled on decorated rooted ribbon trees. In a similar fashion as above, we can define the interior evaluation map
\begin{equation}
\mathit{ev}_\mathit{int}:\prod_{v\in C_{0,\mathit{int}}(T)\setminus\{v_0\}}\mathcal{R}_{k_v+1}(L,B(v))\times\mathcal{R}_{k_{v_0}+1,\vartheta}(L,B(v_0))\rightarrow\prod_{e\in C_{1,\mathit{int}}(T)}L^2
\end{equation}
and the exterior evaluation map
\begin{equation}\label{eq:ext2}
\mathit{ev}_\mathit{ext}:\prod_{v\in C_{0,\mathit{int}}(T)\setminus\{v_0\}}\mathcal{R}_{k_v+1}(L,B(v))\times\mathcal{R}_{k_{v_0}+1,\vartheta}(L,B(v_0))\rightarrow L^{k+1}.
\end{equation}
It follows that the compactified moduli space is given by
\begin{equation}
\begin{split}
\overline{\mathcal{R}}_{k+1,\vartheta}(L,\beta)&:=\bigsqcup_{\substack{(T,B)\in\mathcal{G}(k+1,\beta)\\v_0\in C_{0,\mathit{int}}(T)}}\left(\prod_{e\in C_{1,\mathit{int}}(T)}L\right)\textrm{ }{{}_\Delta\times_{\mathit{ev}_\mathit{int}}}\textrm{ } \\
&\;\;\;\;\left(\prod_{v\in C_{0,\mathit{int}}(T)}\mathcal{R}_{k_v+1}(L,B(v))\times\mathcal{R}_{k_{v_0}+1,\vartheta}(L,B(v_0))\right).
\end{split}
\end{equation}
Define the evaluation map
\begin{equation}
\mathit{ev}_\vartheta^\mathcal{R}=(\mathit{ev}_{\vartheta,0}^\mathcal{R},\cdots,\mathit{ev}_{\vartheta,k}^\mathcal{R}):\overline{\mathcal{R}}_{k+1,\vartheta}(L,\beta)\rightarrow L^{k+1}
\end{equation}
as the restriction of $\mathit{ev}_\mathit{ext}$ to $\overline{\mathcal{R}}_{k+1,\vartheta}(L,\beta)$.

\begin{remark}
In \cite{kf}, Fukaya proposed to use the moduli space of holomorphic discs (without marked points on $\partial D$) modulo $\mathit{Aut}(D,1)$, the automorphism group of $D$ fixing $1\in\partial D$ to define the Maurer-Cartan element in the non-equivariant case, and the same space of maps modulo $\mathit{Aut}(D)$ to define the Maurer-Cartan element in the $S^1$-equivariant case. This is compatible with our constructions with the presence of boundary marked points.
\end{remark}

\subsection{Cohen-Ganatra moduli spaces}\label{section:CG}

We consider a sequence of moduli spaces which are variants of the moduli spaces studied in a different context by Cohen-Ganatra $\cite{cg}$.  They will be used to construct the $S^1$-equivariant chains approximating the primitive $\tilde{y}\in\widehat{C}_2^{S^1}$ (under the deformed differential) of the identity (cf. (\ref{eq:def1})).

To start with, we briefly recall the definition of a simpler moduli space introduced by Ganatra (cf. \cite{sg1}, Section 4.3). An \textit{$l$-point angle-decorated cylinder} is a cylindner $C=\mathbb{R}\times S^1$, equipped with a collection of auxiliary marked points $p_1,\cdots,p_l\in C$, such that their $s\in\mathbb{R}$ coordinates $(p_i)_s$, $1\leq i\leq l$ satisfy
\begin{equation}
(p_1)_s\leq\cdots\leq(p_l)_s.
\end{equation}
Let ${}_l\mathcal{M}$ be the moduli space of $l$-point angle-decorated cylinders $(C,p_1,\cdots,p_l)$, modulo translations in the $s$-direction. It admits a compactification ${}_l\overline{\mathcal{M}}$ as a smooth manifold with corners, by adding broken cylinders with marked point decorations.

To incorporate symplectic geometry, one needs to study smooth maps from the domains $(C,p_1,\cdots,p_l)\in{}_l\mathcal{M}$ to $M$ satisfying Floer's equation with certain asymptotic conditions. For this purpose we need to introduce the Floer data for elements in ${}_l\mathcal{M}$. We say that a time-dependent Hamiltonian function $H_t:S^1\times M\rightarrow\mathbb{R}$ is \textit{admissible} if $H_t=H+F_t$ is the sum of an autonomous Hamiltonian function $H:M\rightarrow\mathbb{R}$ which is equal to $r^2$ on the cylindrical end $[r_0,\infty)\times\partial\overline{M}$ for some $r_0\gg1$, and a time-dependent perturbation $F_t:S^1\times M\rightarrow\mathbb{R}$. Furthermore, we require that for any $r_1\gg r_0$, there exists an $r>r_1$ such that $F_t=0$ in a neighborhood of the hypersurface $\{r\}\times\partial\overline{M}\subset M$. Denote by $\mathcal{H}(M)$ the set of admissible Hamiltonians $H_t$ so that all the 1-periodic orbits of the Hamiltonian vector field $X_{H_t}$ are non-degenerate.

Let $J_t:S^1\times\mathit{TM}\rightarrow\mathit{TM}$ be a time-dependent almost complex structure. It is called \textit{weak contact type} if there exists a sequence $\{r_i\}_{i\in\mathbb{N}}$ of positive real numbers with $\lim_{i\rightarrow\infty}r_i=\infty$ such that near each hypersurface $\{r_i\}\times\partial\overline{M}$ it satisfies $dr\circ J_t=-\theta_M$. Denote by $\mathcal{J}(M)$ the set of $d\theta_M$-compatible almost complex structures on $M$ which are of weak contact type.

A Floer datum for an element of ${}_l\mathcal{M}$ consists of the following:
\begin{itemize}
	\item Choices of positive and negative cylindrical ends
	\begin{equation}
	\varepsilon^+:[0,\infty)\times S^1\rightarrow C\textrm{ and }\varepsilon^-:(-\infty,0]\rightarrow C
	\end{equation}
    such that
    \begin{equation}
    \varepsilon^+(s,t)=(s+(p_l)_s+\eta,t),\textrm{ }\varepsilon^-(s,t)=(s+(p_1)_s+\eta,t+(p_1)_t),
    \end{equation}
    where $(p_1)_t$ is the $t\in S^1$-coordinate of $p_1$ and $\eta\in\mathbb{R}_{>0}$ is fixed.
    \item A domain-dependent Hamiltonian function $H_C:C\times M\rightarrow\mathbb{R}$ satisfying
    \begin{equation}
    (\varepsilon^\pm)^\ast H_C=H_t
    \end{equation}
    for some $H_t\in\mathcal{H}(M)$.
    \item A domain-dependent almost complex structure $J_C:C\times\mathit{TM}\rightarrow\mathit{TM}$ such that
    \begin{equation}
    (\varepsilon^\pm)^\ast J_C=J_t
    \end{equation}
    for some $J_t\in\mathcal{J}(M)$.
\end{itemize}

A \textit{Floer datum for the $S^1$-action} on the cochain complex $\mathit{SC}^\ast(M)$ is an inductive sequence of choices of Floer data for the compactified moduli spaces $\left\{{}_l\overline{\mathcal{M}}\right\}_{l\in\mathbb{N}}$ for each $l\geq 1$ and each element of ${}_l\overline{\mathcal{M}}$, which is compatible with the boundary strata in $\partial{}_l\overline{\mathcal{M}}$, and varies smoothly with respect to gluing. The existence of such a Floer datum is guaranteed by an induction argument. 

From now on, fix a choice of Floer datum for the $S^1$-action on $\mathit{SC}^\ast(M)$. For two 1-periodic orbits $x$ and $y$ of $X_{H_t}$, define
\begin{equation}
{}_l\mathcal{M}(x,y)
\end{equation}
to be the space of pairs $\left((C,p_1,\cdots,p_l),u\right)$, where $(C,p_1,\cdots,p_l)\in{}_l\mathcal{M}$ is an $l$-point angle-decorated cylinder, and $u:C\rightarrow M$ is a map satisfying
\begin{equation}\label{eq:F-eq}
\left\{\begin{array}{l}(du-X_{H_C}\otimes dt)^{0,1}=0, \\ \lim_{s\rightarrow\infty}(\varepsilon^+)^\ast u(s,\cdot)=x, \\ \lim_{s\rightarrow-\infty}(\varepsilon^-)^\ast u(s,\cdot)=y,\end{array}\right.
\end{equation}
where the $(0,1)$-part in the Floer equation is taken with respect to $J_C$. With generic choices of Floer data, transversality for ${}_l\mathcal{M}(x,y)$ can be achieved with standard argument, and the virtual dimension $\leq1$ components of the Gromov compactified moduli space ${}_l\overline{\mathcal{M}}(x,y)$ are compact manifolds with boundary. A signed count of rigid elements in ${}_l\overline{\mathcal{M}}(x,y)$ defines the cochain level operations
\begin{equation}
\delta_l:\mathit{SC}^{\ast+2l-1}(M)\rightarrow\mathit{SC}^\ast(M)
\end{equation}
appearing in (\ref{eq:SC}). Note that if we allow $l=0$, then the moduli space ${}_0\mathcal{M}(x,y)$, which we will abbreviate by $\mathcal{M}(x,y)$, is the moduli space defining the Floer differential on the cochain complex $\mathit{SC}^\ast(M)$.
\bigskip

The Cohen-Ganatra moduli spaces that we introduce below can be thought of as certain interpolations of the moduli spaces $\mathcal{M}_l(x,y)$ considered above and the moduli space $\mathcal{R}_{k+1}(L,\beta)$ recalled in Section \ref{section:disc}. Since this is a parametrized moduli space, we first introduce the moduli space ${}_l\mathcal{R}_{k+1}^1$ of domains
\begin{equation}\label{eq:domain}
(S;z_0,\cdots,z_k,p_1,\cdots,p_l;\ell)
\end{equation}
modulo automorphisms, where $S=D\setminus\{\zeta\}$ is a closed unit disc with an interior puncture $\zeta$, which will serve as an input. At $\zeta$ there is an asymptotic marker, which is a half-line $\ell\in T_\zeta D$. As in the case of $\mathcal{R}_{k+1}(L,\beta)$, there are $k+1$ distinct marked points $z_0,\cdots,z_k\in\partial D$, labeled in counterclockwise order. Moreover, there is a set of auxiliary marked points $p_1,\cdots,p_l\in S$ lying in the interior of $D$. For a representative of an element of ${}_l\mathcal{R}_{k+1}^1$ with $\zeta$ fixed at the origin, and $z_0$ fixed at $1$, these points are required to be strictly radially ordered with norms in $(0,\frac{1}{2})$, i.e.
\begin{equation}\label{eq:radial}
0<|p_l|<\cdots<|p_1|<\frac{1}{2}.
\end{equation}
Finally, we require that the asymptotic marker $\ell$ at $\zeta$ points toward $p_l$. See Figure \ref{fig:domain} for a depiction of a representative of ${}_3\mathcal{R}_4^1$.

\begin{figure}
    \centering
    \begin{tikzpicture}
	\filldraw[draw=black,color={black!15},opacity=0.5] (0,0) circle (1.5);
	\draw (0,0) circle [radius=1.5];
	\draw [orange] [dashed] (0,0) circle [radius=0.75];
	\draw (0,0) node {$\times$};
	\node at (0,0.25) {$\zeta$};
	\draw [orange] (0,-0.3) node[circle,fill,inner sep=1pt] {};
	\draw [orange] (-0.5,0) node[circle,fill,inner sep=1pt] {};
	\draw [orange] (0,0.65) node[circle,fill,inner sep=1pt] {};
	\draw [teal] [->] (0,0) to (0,-0.3);
	\node [orange] at (0,-0.5) {\small $p_3$};
	\node [orange] at (-0.5,-0.2) {\small $p_2$};
	\node [orange] at (0,0.85) {\small $p_1$};
	\draw (1.5,0) node[circle,fill,inner sep=1pt] {};
	\draw (0,1.5) node[circle,fill,inner sep=1pt] {};
	\draw (-1.5,0) node[circle,fill,inner sep=1pt] {};
	\draw (0,-1.5) node[circle,fill,inner sep=1pt] {};
	\node at (1.75,0) {$z_0$};
	\node at (0,1.75) {$z_1$};
	\node at (-1.75,0) {$z_2$};
	\node at (0,-1.75) {$z_3$};
	\end{tikzpicture}
	\caption{An element in the moduli space ${}_3\mathcal{R}_4^1$}
	\label{fig:domain}
\end{figure}
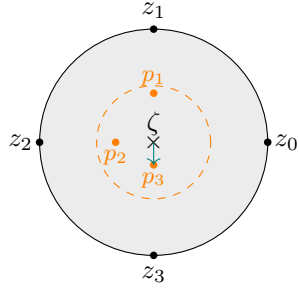

The compactification ${}_l\overline{\mathcal{R}}_{k+1}^1$ of the moduli space ${}_l\mathcal{R}_{k+1}^1$ is a real blow-up of the usual Deligne-Mumford compactification. To describe its boundary strata, we introduce two auxiliary moduli spaces: ${}_l^{j,j+1}\mathcal{R}_{k+1}^1$ and ${}_{l-1}\mathcal{R}_{k+1}^{S^1}$.

${}_l^{j,j+1}\mathcal{R}_{k+1}^1$ is the moduli space of the domains (\ref{eq:domain}), except that the strict radial ordering condition (\ref{eq:radial}) for the auxiliary marked points $p_1,\cdots,p_l$ is now replaced with
\begin{equation}
|p_l|<\cdots<|p_{j+1}|=|p_j|<\cdots<|p_1|<\frac{1}{2}, 
\end{equation}
for some $1\leq j<l$. 

${}_{l-1}\mathcal{R}_{k+1}^{S^1}$ is the moduli space of the same domains, with $p_1,\cdots,p_l$ satisfying the strict radial condition (\ref{eq:radial}), but with $|p_1|=\frac{1}{2}$, i.e.
\begin{equation}
|p_l|<\cdots<|p_1|=\frac{1}{2}.
\end{equation}

For each $1\leq j\leq l-1$, there exists a map
\begin{equation}\label{eq:fj}
\pi_j:{}_l^{j,j+1}\mathcal{R}_{k+1}^1\rightarrow{}_{l-1}\mathcal{R}_{k+1}^1,
\end{equation}
which forgets the marked point $p_j$. Since this amounts to forgetting the argument of $p_j$, $\pi_j$ has 1-dimensional fibers. The map $\pi_j$ extends to a map defined on the compactification ${}_l^{j,j+1}\overline{\mathcal{R}}_{k+1}^1$, which we will still denote by $\pi_j$ by abuse of notations. On the moduli space ${}_{l-1}\mathcal{R}_{k+1}^{S^1}$, there is another map
\begin{equation}\label{eq:f-S1}
\pi_{S^1}:{}_{l-1}\mathcal{R}_{k+1}^{S^1}\rightarrow{}_{l-1}\mathcal{R}_{k+1}^1,
\end{equation}
forgetting the marked point $p_1$. It also extends to a map on the compactification ${}_{l-1}\overline{\mathcal{R}}_{k+1}^{S^1}$.

\begin{proposition}\label{proposition:domain-bdy}
With the notations introduced above, the codimension $1$ boundary components of ${}_l\overline{\mathcal{R}}_{k+1}^1$ are covered by the images of the natural inclusions of the following strata
\begin{equation}\label{eq:dom-bdy1}
{}_j\overline{\mathcal{M}}\times{}_{l-j}\overline{\mathcal{R}}_{k+1}^1, 1\leq j\leq l,
\end{equation}
\begin{equation}\label{eq:dom-bdy2}
{}_l^{j,j+1}\overline{\mathcal{R}}_{k+1}^1, 1\leq j\leq l-1,
\end{equation}
\begin{equation}\label{eq:dom-bdy3}
{}_{l-1}\overline{\mathcal{R}}_{k+1}^{S^1},
\end{equation}
\begin{equation}\label{eq:dom-bdy4}
{}_l\overline{\mathcal{R}}_{k_1+1}^1\textrm{ }{{}_i\times}_0\textrm{ }\overline{\mathcal{R}}_{k_2+1},\textrm{ }k_1\geq1,k_2\geq2,k_1+k_2=k+1,1\leq i\leq k_1,
\end{equation}
\begin{equation}\label{eq:dom-bdy5}
\overline{\mathcal{R}}_{k_1+1}\textrm{ }{{}_i\times}_0\textrm{ }{}_l\overline{\mathcal{R}}_{k_2+1}^1,\textrm{ }k_1\geq2,k_2\geq0,k_1+k_2=k+1, 1\leq i\leq k_1.
\end{equation}
\end{proposition}
\begin{proof}
When there is no marked points on $\partial S$, the boundary strata of the moduli spaces ${}_l\overline{\mathcal{R}}^1_{k+1}$ have been analysed by Cohen-Ganatra \cite{cg}, Section 4.2. See also \cite{jz}, Section 4.4. In particular, we have the boundary strata (\ref{eq:dom-bdy1}), (\ref{eq:dom-bdy2}) and (\ref{eq:dom-bdy3}). Note that the strata in (\ref{eq:dom-bdy1}) are loci created by real blow-ups, see Figure \ref{fig:real-blp} for an illustration. The only difference in our case is that there are now $k+1$ additional marked points $z_0,\cdots,z_k$ on the boundary $\partial S$. When two boundary marked points $z_i$ and $z_j$ come together, for example when $0\leq i<j\leq k$, a disc will break off from the domain $(S;z_0,\cdots,z_k,p_1,\cdots,p_l;\ell)$, carrying the marked points $z_i,\cdots,z_j$, together with the nodal point on its boundary. Such disc bubbles give rise to the strata in (\ref{eq:dom-bdy4}) and (\ref{eq:dom-bdy5}).
\end{proof}

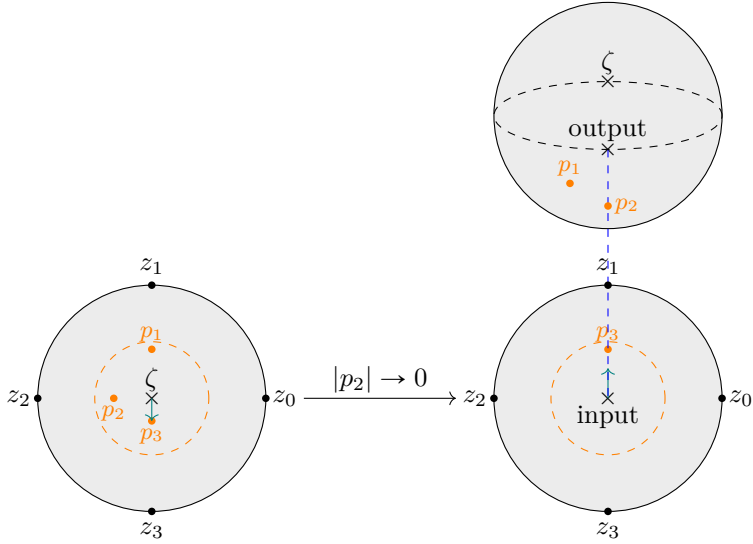
\begin{figure}
	\centering
	\begin{tikzpicture}
	\filldraw[draw=black,color={black!15},opacity=0.5] (0,0) circle (1.5);
	\draw (0,0) circle [radius=1.5];
	\draw [orange] [dashed] (0,0) circle [radius=0.75];
	\draw (0,0) node {$\times$};
	\node at (0,0.25) {$\zeta$};
	\draw [orange] (0,-0.3) node[circle,fill,inner sep=1pt] {};
	\draw [orange] (-0.5,0) node[circle,fill,inner sep=1pt] {};
	\draw [orange] (0,0.65) node[circle,fill,inner sep=1pt] {};
	\draw [teal] [->] (0,0) to (0,-0.3);
	\node [orange] at (0,-0.5) {\small $p_3$};
	\node [orange] at (-0.5,-0.2) {\small $p_2$};
	\node [orange] at (0,0.85) {\small $p_1$};
		
	\filldraw[draw=black,color={black!15},opacity=0.5] (6,0) circle (1.5);
    \draw (6,0) circle [radius=1.5];
    \draw [orange] [dashed] (6,0) circle [radius=0.75];
    \draw (6,0) node {$\times$};
    \node at (6,-0.25) {input};
    \draw [orange] (6.4,0.5) node[circle,fill,inner sep=1pt] {};
    \node [orange] at (6.6,0.75) {\small $p_3$};
    \draw [teal] [->] (6,0) to (6.3,0.375);
		
	\filldraw[draw=black,color={black!15},opacity=0.5] (6,3.75) circle (1.5);
    \draw (6,3.75) circle [radius=1.5];
    \draw (6,3.3) node {$\times$};
    \draw (6,4.2) node {$\times$};
    \draw [blue] [dashed] (6,3.3) to (6,0);
    \draw [orange] (6.2,2.55) node[circle,fill,inner sep=1pt] {};
    \draw [orange] (5.5,2.85) node[circle,fill,inner sep=1pt] {};
    \node [orange] at (6.25,2.75) {\small $p_2$};
    \node [orange] at (5.5,3.05) {\small $p_1$};
    \draw [dashed] (6,3.75) ellipse (1.5 and 0.45);
    \node at (6,4.45) {$\zeta$};
    \node at (6,3.55) {output};
		
	\draw [->] (2,0) to (4,0);
	\node at (3,0.25) {$|p_2|\rightarrow0$};
		
	\draw (1.5,0) node[circle,fill,inner sep=1pt] {};
	\draw (0,1.5) node[circle,fill,inner sep=1pt] {};
	\draw (-1.5,0) node[circle,fill,inner sep=1pt] {};
	\draw (0,-1.5) node[circle,fill,inner sep=1pt] {};
	\node at (1.75,0) {$z_0$};
	\node at (0,1.75) {$z_1$};
	\node at (-1.75,0) {$z_2$};
	\node at (0,-1.75) {$z_3$};
		
	\draw (7.5,0) node[circle,fill,inner sep=1pt] {};
	\draw (6,1.5) node[circle,fill,inner sep=1pt] {};
	\draw (4.5,0) node[circle,fill,inner sep=1pt] {};
	\draw (6,-1.5) node[circle,fill,inner sep=1pt] {};
	\node at (7.75,0) {$z_0$};
	\node at (6,1.75) {$z_1$};
	\node at (4.25,0) {$z_2$};
	\node at (6,-1.75) {$z_3$};
	\end{tikzpicture}
	\caption{A cylinder with marked points $p_1$ and $p_2$ bubbles off at $\zeta$, which belongs to the codimension $1$ boundary stratum ${}_2\overline{\mathcal{M}}\times{}_1\overline{\mathcal{R}}_{4}^1$ of ${}_3\overline{\mathcal{R}}_{4}^1$.}
	\label{fig:real-blp}
\end{figure}
\bigskip

The moduli space ${}_{l-1}\mathcal{R}_{k+1}^{S^1}$ introduced above will play a crucial role for our later purposes, we need to analyze it a little bit more. Via the forgetful map (\ref{eq:f-S1}), we have an abstract identification
\begin{equation}
{}_{l-1}\mathcal{R}_{k+1}^{S^1}\cong{}_{l-1}\mathcal{R}_{k+1}^1\times S^1,
\end{equation}
where the $S^1$ factor is determined by the argument $\theta_1:=\arg(p_1)$. Under this identification, the compactification ${}_{l-1}\overline{\mathcal{R}}_{k+1}^{S^1}$ is abstractly modeled by ${}_{l-1}\overline{\mathcal{R}}_{k+1}^1\times S^1$. In particular, the codimension $1$ boundary stratum ${}_{l-1}^{j,j+1}\overline{\mathcal{R}}_{k+1}^1\subset{}_{l-1}\overline{\mathcal{R}}_{k+1}^1$ for some $2\leq j\leq l-1$ corresponds to a stratum ${}_{l-1}^{j,j+1}\overline{\mathcal{R}}_{k+1}^{S^1}$ in the codimension $1$ boundary of ${}_{l-1}\overline{\mathcal{R}}_{k+1}^{S^1}$, where the $S^1$ factor describes the situation that $|p_1|=|p_2|=\frac{1}{2}$. Denote by
\begin{equation}
\pi_j^{S^1}:{}_{l-1}^{j,j+1}\overline{\mathcal{R}}_{k+1}^{S^1}\rightarrow{}_{l-2}\overline{\mathcal{R}}_{k+1}^{S^1}
\end{equation}
the map which forgets $p_j$. 

For an element $S$ of ${}_{l-1}\overline{\mathcal{R}}_{k+1}^{S^1}$, we say that $p_1$ \textit{points at a boundary point} $z_i$, for some $0\leq i\leq k$, if for a representative of $S$ with $\zeta$ fixed at the origin, the ray from $\zeta$ to $p_1$ points at $z_i$. Denote by ${}_{l-1}\overline{\mathcal{R}}_{k+1}^{S_i^1}\subset{}_{l-1}\overline{\mathcal{R}}_{k+1}^{S^1}$ the codimension $1$ locus where $p_1$ points at $z_i$. There is a bijection
\begin{equation}
\tau_i:{}_{l-1}\overline{\mathcal{R}}_{k+1}^{S_i^1}\rightarrow{}_{l-1}\overline{\mathcal{R}}_{k+1}^1,
\end{equation}
which is defined as follows. When $l\geq2$, $\tau_i$ forgets the point $p_1$ on the circle $|z|=\frac{1}{2}$, and relabels the remaining auxiliary marked points $p_2,\cdots,p_l$ as $p_1,\cdots,p_{l-1}$. When $l=1$, $\tau_i$ is defined by cyclically permuting the boundary marked points (for multiple times), so that the original $z_i$ is now labeled $z_k$, and then forgetting $p_1$. Similarly, we say that $p_1$ \textit{points between $z_i$ and} $z_{i+1\textrm{ mod }k}$ if for such a representative, the ray from $\zeta$ to $p_1$ intersects the arc in $\partial S$ from $z_i$ to $z_{i+1\textrm{ mod }k}$. The locus in ${}_{l-1}\mathcal{R}_{k+1}^{S^1}$ where $p_1$ points between $z_i$ and $z_{i+1\textrm{ mod }k}$ is denoted by ${}_{l-1}\mathcal{R}_{k+1}^{S_{i,i+1}^1}$.

As in the case of $\overline{\mathcal{R}}_k$, where $k\geq3$ (cf. Lemma \ref{lemma:erase}), we can decompose the moduli space ${}_{l-1}\overline{\mathcal{R}}_{k+1}^{S^1}$ into its sectors ${}_{l-1}\overline{\mathcal{R}}_{k+1}^{S_{i,i+1}^1}$, and identify each sector with a moduli space of the form ${}_{l-1}\overline{\mathcal{R}}_{k+1,\tau_i}^1$. Here, ${}_{l-1}\mathcal{R}_{k+1,\tau_i}^1$ is the abstract moduli space of discs with $k+2$ boundary marked points $z_0,\cdots,z_{i},z_f,z_{i+1},\cdots,z_k$ arranged in counterclockwise order, with the point $z_f$ marked as auxiliary, one interior puncture $\zeta$, marked as an input, equipped with an asymptotic marker $\ell$, and $l$ auxiliary marked points $p_1,\cdots,p_l$ in the interior of the disc which are strictly radially ordered with norms in $(0,\frac{1}{2})$ in the sense of (\ref{eq:radial}), for a representative of an element of ${}_{l-1}\mathcal{R}_{k+1,\tau_i}^1$ which fixes $z_0$ at $1$ and $\zeta$ at $0$. Note that the compactification ${}_{l-1}\overline{\mathcal{R}}_{k+1,\tau_i}^1$ is abstractly isomorphic to ${}_{l-1}\overline{\mathcal{R}}_{k+2}^1$, except that $z_f$ is marked as auxiliary, which means that it is forgotten when we consider the boundary $\partial S$ with its marked points as an element of $\mathcal{L}_{k+1}$. We remark that this is an important point which will play a crucial role in our argument in Section \ref{section:proof}. More precisely, at any stratum of ${}_{l-1}\overline{\mathcal{R}}_{k+1,\tau_i}^1$:
\begin{itemize}
	\item we treat the main component (the one containing the interior puncture $\zeta$) as belonging to ${}_{l-1}\overline{\mathcal{R}}_{k'+1,\tau_j}^1$ for some $0\leq k'\leq k$ and $0\leq j\leq k'$ if it contains $z_f$ as a boundary marked point, and ${}_{l-1}\overline{\mathcal{R}}_{k'+2}^1$ if it does not;
	\item if a non-main disc component (the one without the puncture $\zeta$) contains the boundary marked point $z_f$, we view it as an element of $\mathcal{R}_{k'+1,f_i}$ for some $0\leq k'\leq k$, the space of discs with $k'+1$ boundary marked points, where the $i$-th point is marked as forgotten. See \cite{sg1}, Appendix A.2 for the detailed construction of the moduli spaces $\mathcal{R}_{k'+1,f_i}$ when the boundary marked points are treated as punctures.
\end{itemize}
Moreover, the asymptotic marker $\ell$ at $\zeta$ points in the direction $\theta_l$ (or $z_f$ if $l=0$). In order to relate the moduli space ${}_{l-1}\overline{\mathcal{R}}_{k+1,\tau_i}$ to a sector ${}_{l-1}\overline{\mathcal{R}}_{k+1}^{S_{i,i+1}^1}\subset{}_{l-1}\overline{\mathcal{R}}_{k+1}^{S^1}$, we define the \textit{auxiliary-rescaling map}
\begin{equation}\label{eq:aux-res}
\pi_f^i:{}_{l-1}\overline{\mathcal{R}}_{k+1,\tau_i}^1\rightarrow{}_{l-1}\overline{\mathcal{R}}_{k+1}^{S_{i,i+1}^1},
\end{equation}
which, for a representative of an element of ${}_{l-1}\overline{\mathcal{R}}_{k+1,\tau_i}$ with $\zeta=0$, adds a point $p_0$ on the line segment connecting $\zeta$ and $z_f$ with $|p_0|=\frac{1}{2}$ and delete $z_f$. Finally, we relabel the marked points $p_0,\cdots,p_l$ as $p_1,\cdots,p_{l+1}$. For an illustrative example of the definition of $\pi_f^i$, see Figure \ref{fig:auxre}. By our orientation conventions (cf. Appendix \ref{section:orientation}), the map $\pi_f^i$ is an oriented diffeomorphism.

\begin{figure}
	\centering
	\begin{tikzpicture}
		\filldraw[draw=black,color={black!15},opacity=0.5] (0,1.9) circle (1.5);
		\draw (0,1.9) circle [radius=1.5];
		\draw [orange] [dashed] (0,1.9) circle [radius=0.75];
		\draw (0,1.9) node {$\times$};
		\node at (0,2.15) {$\zeta$};
		\draw [orange] (0,1.6) node[circle,fill,inner sep=1pt] {};
		\draw [orange] (-0.5,1.9) node[circle,fill,inner sep=1pt] {};
		\draw [teal] [->] (0,1.9) to (0,1.6);
		\node [orange] at (0,1.4) {\small $p_2$};
		\node [orange] at (-0.5,1.7) {\small $p_1$};
		\draw [blue] (1.06,2.96) node[circle,fill,inner sep=1pt] {};
		\node [blue] at (1.15,3.15) {\small $z_f$};
		
		\draw (1.5,1.9) node[circle,fill,inner sep=1pt] {};
		\draw (0,3.4) node[circle,fill,inner sep=1pt] {};
		\draw (-1.5,1.9) node[circle,fill,inner sep=1pt] {};
		\draw (0,0.4) node[circle,fill,inner sep=1pt] {};
		\node at (1.75,1.9) {$z_0$};
		\node at (0,3.65) {$z_1$};
		\node at (-1.75,1.9) {$z_2$};
		\node at (0,0.15) {$z_3$};
		
		\filldraw[draw=black,color={black!15},opacity=0.5] (5,1.9) circle (1.5);
		\draw (5,1.9) circle [radius=1.5];
		\draw [orange] [dashed] (5,1.9) circle [radius=0.75];
		\draw (5,1.9) node {$\times$};
		\node at (5,2.15) {$\zeta$};
		\draw [orange] (5,1.6) node[circle,fill,inner sep=1pt] {};
		\draw [orange] (4.5,1.9) node[circle,fill,inner sep=1pt] {};
		\draw [orange] (5.53,2.43) node[circle,fill,inner sep=1pt] {};
		\draw [teal] [->] (5,1.9) to (5,1.6);
		\node [orange] at (5,1.4) {\small $p_2$};
		\node [orange] at (4.5,1.7) {\small $p_1$};
		\node [orange] at (5.55,2.65) {\small $p_0$};
		\draw [blue] (6.06,2.96) node[circle,fill,inner sep=1pt] {};
		\node [blue] at (6.15,3.15) {\small $z_f$};
		
		\draw (6.5,1.9) node[circle,fill,inner sep=1pt] {};
		\draw (5,3.4) node[circle,fill,inner sep=1pt] {};
		\draw (3.5,1.9) node[circle,fill,inner sep=1pt] {};
		\draw (5,0.4) node[circle,fill,inner sep=1pt] {};
		\node at (6.75,1.9) {$z_0$};
		\node at (5,3.65) {$z_1$};
		\node at (3.25,1.9) {$z_2$};
		\node at (5,0.15) {$z_3$};
		\draw [blue] [dashed] (5,1.9) to (6.06,2.96);
		
		\filldraw[draw=black,color={black!15},opacity=0.5] (10,1.9) circle (1.5);
		\draw (10,1.9) circle [radius=1.5];
		\draw [orange] [dashed] (10,1.9) circle [radius=0.75];
		\draw (10,1.9) node {$\times$};
		\node at (10,2.15) {$\zeta$};
		\draw [orange] (10,1.6) node[circle,fill,inner sep=1pt] {};
		\draw [orange] (9.5,1.9) node[circle,fill,inner sep=1pt] {};
		\draw [orange] (10.53,2.43) node[circle,fill,inner sep=1pt] {};
		\draw [teal] [->] (10,1.9) to (10,1.6);
		\node [orange] at (10,1.4) {\small $p_3$};
		\node [orange] at (9.5,1.7) {\small $p_2$};
		\node [orange] at (10.55,2.65) {\small $p_1$};
		
		\draw (11.5,1.9) node[circle,fill,inner sep=1pt] {};
		\draw (10,3.4) node[circle,fill,inner sep=1pt] {};
		\draw (8.5,1.9) node[circle,fill,inner sep=1pt] {};
		\draw (10,0.4) node[circle,fill,inner sep=1pt] {};
		\node at (11.75,1.9) {$z_0$};
		\node at (10,3.65) {$z_1$};
		\node at (8.25,1.9) {$z_2$};
		\node at (10,0.15) {$z_3$};
		
		\draw [->] (2,1.9) to (3,1.9);
		\draw [->] (7,1.9) to (8,1.9);
		
	\end{tikzpicture}
	\caption{The definition of the auxiliary-rescaling map $\pi_f^0:{}_2\overline{\mathcal{R}}_{4,\tau_0}^1\rightarrow{}_2\overline{\mathcal{R}}_4^{S_{0,1}^1}$}
	\label{fig:auxre}
\end{figure}
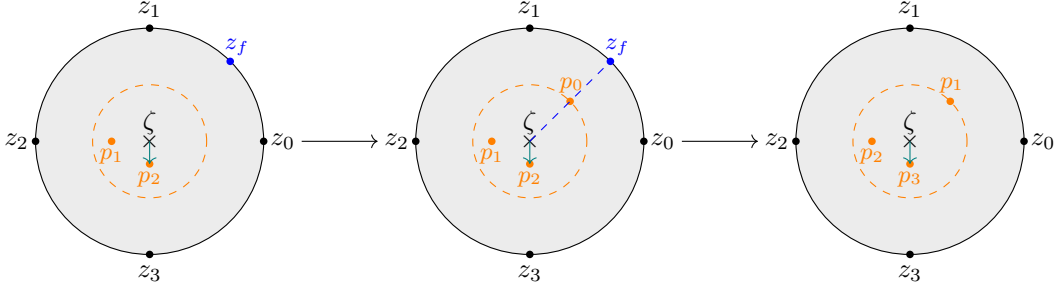

Finally, there is a free $\mathbb{Z}_{k+1}$-action on the moduli space ${}_{l-1}\overline{\mathcal{R}}_{k+1}^{S^1}$ generated by the map
\begin{equation}
\kappa:{}_{l-1}\overline{\mathcal{R}}_{k+1}^{S^1}\rightarrow{}_{l-1}\overline{\mathcal{R}}_{k+1}^{S^1},
\end{equation}
which cyclically permutes the labels of the boundary marked points. Concretely, $\kappa$ changes the label $z_i$ to $z_{i+1}$ if $0\leq i\leq k-1$, and $z_k$ to $z_0$. It can be shown that this $\mathbb{Z}_{k+1}$-action is properly discontinuous. For a similar (in fact, almost identical) situation, see \cite{sg1}, Lemma 12.

\begin{remark}
Similar auxiliary moduli spaces as ${}_{l-1}\mathcal{R}_{k+1}^{S^1}$, ${}_{l-1}\mathcal{R}_{k+1}^{S_i^1}$,  ${}_{l-1}\mathcal{R}_{k+1}^{S_{i,i+1}^1}$ and ${}_{l-1}\mathcal{R}_{k+1,\tau_i}^1$ considered above also play important roles in Ganatra's construction of cyclic open-closed maps \cite{sg1}. The main difference between our set up and his is that the marked points $z_0,\cdots,z_k\in\partial S$ there are treated as punctures, and therefore equipped with strip-like ends (and Floer data) for the purpose of constructing closed-open maps. Here, as we have mentioned, they are regarded as marked points on the boundary loops, so one can study the corresponding forgetful maps. These forgetful maps will be important for the construction of cyclic invariant Kuranishi structures in Section \ref{section:compactify}, see the proof of Theorem \ref{theorem:moduli}.
\end{remark}
\bigskip

In order to study the Floer theory associated to the domains in ${}_l\mathcal{R}_{k+1}^1$, we need the corresponding notion of Floer data. Note that by identifying an element of ${}_l\mathcal{R}_{k+1}^1$ with a half-cylinder $\mathbb{R}_{\geq0}\times S^1$, we can assign each point $p_i$ with an $s$-coordinate $(p_i)_s\in\mathbb{R}_{\geq0}$, and a $t$-coordinate $(p_i)_t\in S^1$.

\begin{definition}\label{definition:data}
A Floer datum for an element of ${}_l\mathcal{R}_{k+1}^1$ consists of the following:
\begin{itemize}
	\item A positive cylindrical end which is compatible with the asymptotic marker, namely an embedding
	\begin{equation}
	\varepsilon^+:[0,\infty)\times S^1\rightarrow S,\textrm{ }(s,t)\mapsto\left(s+(p_l)_s+\eta,t\right)
	\end{equation}
	for some $\eta>0$.
	\item A sub-closed 1-form $\nu_S\in\Omega^1(S)$ such that $\nu_S\equiv0$ near $\partial S$ and $(\varepsilon^+)^\ast\nu_S=dt$.
	\item A domain-dependent Hamiltonian $H_S:S\times M\rightarrow\mathbb{R}$ satisfying
	\begin{equation}
	(\varepsilon^+)^\ast H_S=H_t
	\end{equation}
    for some $H_t\in\mathcal{H}(M)$, and
    \begin{equation}
    H_S\equiv0\textrm{ near }\partial S.
    \end{equation}
    \item A domain-dependent almost complex structure $J_S:S\times\mathit{TM}\rightarrow\mathit{TM}$ such that
    \begin{equation}
    (\varepsilon^+)^\ast J_S=J_t
    \end{equation}
    for some $J_t\in\mathcal{J}(M)$, and
    \begin{equation}
    J_S\equiv J_M\textrm{ near }\partial S,
    \end{equation}
    where $J_M$ is the almost complex structure fixed in Section \ref{section:approximation}.
\end{itemize}
\end{definition}

In a similar fashion, we can define Floer data on the auxiliary moduli spaces ${}_l^{j,j+1}\mathcal{R}_{k+1}^1$ and ${}_{l-1}\mathcal{R}_{k+1}^{S^1}$. We will inductively choose the Floer data on the compactified moduli spaces ${}_l\overline{\mathcal{R}}_{k+1}^1$ so that they satisfy certain consistency conditions. In order to do so, we first choose a Floer datum on the auxiliary moduli space ${}_{l-1}\overline{\mathcal{R}}_{k+1}^{S^1}$ so that
\begin{itemize}
	\item The Floer datum on ${}_{l-1}\overline{\mathcal{R}}_{k+1}^{S^1}$ is $\mathbb{Z}_{k+1}$-equivariant under the cyclic permutation $\kappa$.
	\item On the boundary stratum ${}_{l-1}^{j,j+1}\overline{\mathcal{R}}_{k+1}^{S^1}\subset\partial{}_{l-1}\overline{\mathcal{R}}_{k+1}^{S^1}$, the Floer datum is conformally equivalent to the one pulled back from ${}_{l-2}\overline{\mathcal{R}}_{k+1}^{S^1}$ via the forgetful map $\pi_j^{S^1}$.
\end{itemize}

\begin{definition}
A Cohen-Ganatra Floer datum is a an inductive sequence of choices, for every $k\in\mathbb{Z}_{\geq0}$ and $l\in\mathbb{N}$, of a Floer datum for every representative of an element of ${}_l\overline{\mathcal{R}}_{k+1}^1$ in the sense of Definition \ref{definition:data}, which vary smoothly over the moduli spaces, and is required to satisfy the following :
\begin{itemize}
	\item[(i)] The choice of Floer datum on any boundary stratum should agree with the inductively chosen datum along the boundary stratum for which we have already picked the data.
	\item[(ii)] Near the boundary strata in (\ref{eq:dom-bdy2}), the Floer data are conformally equivalent to the ones obtained by puling back from ${}_{l-1}\overline{\mathcal{R}}_{k+1}^1$ via the forgetful maps $\pi_j$, and fix a conformal equivalence.
	\item[(iii)] On the codimension 1 loci ${}_{l-1}\overline{\mathcal{R}}_{k+1}^{S_i^1}\subset{}_{l-1}\overline{\mathcal{R}}_{k+1}^{S^1}$, where $p_1$ points at $z_i$, the Floer datum should agree with the pullback by $\tau_i$ of the existing Floer datum on ${}_{l-1}\overline{\mathcal{R}}_{k+1}^1$.
\end{itemize}
\end{definition}

\begin{proposition}
A Cohen-Ganatra Floer datum exists.
\end{proposition}
\begin{proof}
This follows from the fact that the choices of Floer data at each stage is contractible, and for a suitable inductive order, the conditions imposed on Floer data at various strata do not contradict each other. For a similar situation, see \cite{sg1}, Proposition 10.
\end{proof}

Fix a Cohen-Ganatra Floer datum, and a 1-periodic orbit $x$ of $X_{H_t}$, which is a generator of the symplectic cochain complex $\mathit{SC}^\ast(M)$. Let $\pi_2(M,x,L)$ be the set of homotopy classes of maps $u:S\rightarrow M$ with boundary on $L$ and asymptotic to $x$ at the puncture $\zeta$. There is a natural map $\partial:\pi_2(M,x,L)\rightarrow H_1(L;\mathbb{Z})$ defined by sending $[u]\in\pi_2(M,x,L)$ to the homology class $[u(\partial S)]\in H_1(L;\mathbb{Z})$. For any $\ring{\beta}\in\pi_2(M,x,L)$, define the \textit{Cohen-Ganatra moduli space}
\begin{equation}
{}_l\mathcal{R}_{k+1}^1(x,L,\ring{\beta})
\end{equation}
to be the space of pairs
\begin{equation}\label{eq:pair}
\left((S;z_0,\cdots,z_k,p_1,\cdots,p_l;\ell),u\right),
\end{equation}
where $(S;z_0,\cdots,z_k,p_1,\cdots,p_l;\ell)\in{}_l\mathcal{R}^1_{k+1}$, and the map $u:S\rightarrow M$ satisfies
\begin{equation}\label{eq:CR}
\left\{\begin{array}{l}
(du-X_{H_S}\otimes\nu_S)^{0,1}=0, \\
u(\partial S)\subset L, \\
\lim_{s\rightarrow\infty}(\varepsilon^+)^\ast u(s,\cdot)=x, \\
{[u]}=\ring{\beta},
\end{array}\right.
\end{equation}
where the $(0,1)$-part in the Floer equation is taken with respect to $J_S$.

Among the Cohen-Ganatra moduli spaces ${}_l\mathcal{R}_{k+1}^1(x,L,\ring{\beta})$, there is one special case that will be particularly important for us, namely the moduli space
\begin{equation}
\mathcal{R}_{k+1}^1(e_M,L,\beta),
\end{equation}
which is defined using domains without the auxiliary interior marked points $p_1,\cdots,p_l$, and the asymptotic condition at the positive puncture is given by the minimum of a $C^2$-small Morse function $f:M^\mathit{in}\rightarrow\mathbb{R}$ defined in the interior $M^\mathit{in}$ of the Liouville domain $\overline{M}$. In this case, the homotopy class $\ring{\beta}$ can be identified with a class $\beta\in\pi_2(M,L)$ under the isomorphism $\pi_2(M,e_M,L)\cong\pi_2(M,L)$, which explains the notation.

Given $x\in\mathit{SC}^\ast(M)$ and $\ring{\beta}\in\pi_2(M,x,L)$, in a similar fashion as above, we can define the moduli spaces
\begin{equation}
{}_{l-1}\mathcal{R}_{k+1}^{S^1}(x,L,\ring{\beta}),
\end{equation}
\begin{equation}\label{eq:sector}
{}_{l-1}\mathcal{R}_{k+1,\tau_i}^1(x,L,\ring{\beta}),\textrm{ }0\leq i\leq k,
\end{equation}
\begin{equation}
{}_l^{j,j+1}\mathcal{R}_{k+1}^1(x,L,\ring{\beta}),\textrm{ }1\leq j\leq l-1,
\end{equation}
which parametrize the same maps $u:S\rightarrow M$ as in (\ref{eq:CR}), but with the domains in the moduli spaces ${}_{l-1}\mathcal{R}_{k+1}^{S^1}$, ${}_{l-1}\mathcal{R}_{k+1,\tau_i}^1$ and ${}_l^{j,j+1}\mathcal{R}_{k+1}^1$, respectively. 

When defining (\ref{eq:sector}) (and its compactification), we need to choose Floer data on the moduli spaces ${}_{l-1}\overline{\mathcal{R}}_{k+1,\tau_i}^1$. In this case, there is an oriented diffeomorphism
\begin{equation}\label{eq:permu}
\tau_i:{}_{l-1}\overline{\mathcal{R}}_{k+1,\tau_i}^1\rightarrow{}_{l-1}\overline{\mathcal{R}}_{k+1,\tau_0}^1
\end{equation}
induced by cyclic permutations of the boundary labels $z_0,\cdots,z_k$. The Floer datum for ${}_{l-1}\mathcal{R}_{k+1,\tau_i}^1$ shall be chosen so that it coincides with the pullback of the Floer datum on ${}_{l-1}\mathcal{R}_{k+1,\tau_0}^1$ by (\ref{eq:permu}).

The Gromov compactification ${}_l\overline{\mathcal{R}}_{k+1}^1(x,L,\ring{\beta})$ of the Cohen-Ganatra moduli space is an admissible K-space. Its detailed description will be postponed to the next subsection. 

\subsection{Kuranishi structures on compactified moduli spaces}\label{section:compactify}

The purpose of this subsection is to state Theorem \ref{theorem:moduli}, where we describe the compactifications of the moduli spaces $\mathcal{R}_{k+1,\vartheta}(L,\beta)$ and ${}_l\mathcal{R}_{k+1}^1(x,L,\ring{\beta})$ introduced in Sections \ref{section:disc} and \ref{section:CG} as admissible K-spaces. We start with the following lemma which summarizes the basic properties of these moduli spaces.

\begin{lemma}\label{lemma:basic}
The following properties of the moduli spaces hold.
\begin{itemize}
	\item[(i)] When $\beta=0$, the moduli space $\mathcal{R}_{k+1,\vartheta}(L,0)$ consists of constant maps for every $k\geq3$, and $\mathcal{R}_{k+1}^1(e_M,L,0)$ consists of constant maps for every $k\geq0$, where $e_M\in\mathit{SC}^0(M)$ represents the identity.
	\item[(ii)] If $\theta_M(\partial\ring{\beta})+|A_{H_t}(x)|<0$, then ${}_l\mathcal{R}_{k+1}^1(x,L,\ring{\beta})=\emptyset$, where $H_t\in\mathcal{H}(M)$ is a fixed choice of some admissible Hamiltonian, which serves as part of our Floer datum, and
	\begin{equation}
	A_{H_t}(x)=-\int_{S^1}x^\ast\theta_M+\int_{S^1}H_t(x(t))dt
	\end{equation}
	is the action of the orbit $x$ of $X_{H_t}$.
	\item[(iii)] For every $k\in\mathbb{N}$ we have
	\begin{equation}
	\mathcal{R}_{k+1,\vartheta}(L,\beta)=\emptyset\Leftrightarrow\mathcal{R}_{2,\vartheta}(L,\beta)=\emptyset.
	\end{equation}
    For every $k,l\in\mathbb{Z}_{\geq0}$, we have
    \begin{equation}
    {}_l\mathcal{R}_{k+1}^1(x,L,\ring{\beta})=\emptyset\Leftrightarrow{}_l\mathcal{R}_1^1(x,L,\ring{\beta})=\emptyset.
    \end{equation}
    Moreover, for every $c>0$ and fixed $l\in\mathbb{Z}_{\geq0}$, the sets
    \begin{equation}
    \left\{\beta\in\pi_2(M,L)\vert\mathcal{R}_{2,\vartheta}(L,\beta)\neq\emptyset,\theta_M(\partial\beta)<c\right\},
    \end{equation}
    \begin{equation}
    \left\{\ring{\beta}\in\pi_2(M,x,L)\left\vert{}_l\mathcal{R}_1^1(x,L,\ring{\beta})\neq\emptyset,\theta_M(\partial\ring{\beta})<c\right.\right\}
    \end{equation}
    are both finite.
\end{itemize}
\end{lemma}
\begin{proof}
(i) is obvious for $\mathcal{R}_{k+1,\vartheta}(L,0)$. For any map $u:(S,\partial S)\rightarrow(M,L)$ parametrized by $\mathcal{R}_{k+1}^1(e_M,L,\beta)$, the removable singularity theorem for pseudoholomorphic maps implies that there exists a $J_M$-holomorphic map $\bar{u}:(D,\partial D)\rightarrow(M,L)$, with $\left[\bar{u}\right]=\beta$ and $\bar{u}(0)$ mapped to the minimum of the $C^2$-small Morse function $f:M^\mathit{in}\rightarrow\mathbb{R}$. If $\bar{u}$ is constant, so must be $u$.

(ii) is an energy estimate. For any element (which is represented by a pair (\ref{eq:pair})) of ${}_l\mathcal{R}_{k+1}(L,x,\ring{\beta})$, we have
\begin{equation}\label{eq:energy}
0\leq E^\mathit{geom}(u)=\int_S\frac{1}{2}||du-X_{H_S}\otimes\nu_S||^2\leq E^\mathit{top}(u)=\theta_M(\partial\ring{\beta})+A_{H_t}(x),
\end{equation}
where $E^\mathit{geom}(u)$ and $E^\mathit{top}(u)$ are the geometric and the topological energies of $u$, respectively (cf. \cite{ase}, Section 7.2). Note that in the above computations, we have used the fact that $H_S\equiv0$ near $\partial S$.

(iii) is a consequence of Gromov compactness.
\end{proof}

To describe the compactifications of the moduli spaces ${}_l\mathcal{R}_{k+1}^1(x,L,\ring{\beta})$, ${}_l^{j,j+1}\mathcal{R}_{k+1}^1(x,L,\ring{\beta})$, ${}_{l-1}\mathcal{R}_{k+1}^{S^1}(x,L,\ring{\beta})$ and ${}_{l-1}\mathcal{R}_{k+1,\tau_i}^1(x,L,\ring{\beta})$ introduced in Section \ref{section:CG}, we need a slight variation of the notion of a decorated rooted ribbon tree.

\begin{definition}
A decorated rooted ribbon tree with a single puncture is a triple $(T,\ring{B},v_0)$ such that the tree $T$ satisfies (i)-(v) in the definition of a decorated rooted ribbon tree, but now there is a distinguished interior vertex $v_0\in C_{0,\mathit{int}}(T)$ which is called a puncture. The conditions imposed on the map $B:C_{0,\mathit{int}}(T)\rightarrow\pi_2(M,L)$ in Definition \ref{definition:tree}, (vi) and (vii) are replaced with the following:
\begin{itemize}
	\item[(vi)'] There exists a map $\ring{B}:C_{0,\mathit{int}}(T)\rightarrow\pi_2(M,x,L)$, such that its restrictions to $C_{0,\mathit{int}}(T)\setminus\{v_0\}$ and $\{v_0\}$ give rise to maps $C_{0,\mathit{int}}(T)\setminus\{v_0\}\rightarrow\pi_2(M,L)$ and $\{v_0\}\rightarrow\pi_2(M,x,L)$, respectively, where $x$ is a 1-periodic orbit of $X_{H_t}$. For every $v\in C_{0,\mathit{int}}(T)$, either $d\theta_M\left(\ring{B}(v)\right)>0$ or $\ring{B}(v)=0$. Note that for $v=v_0$, this means that $x$ is a constant orbit, so $\ring{B}(v_0)$ can be regarded as a class in $\pi_2(M,L)$.
	\item[(vii)'] Every $v\in C_{0,\mathit{int}}(T)\setminus\{v_0\}$ with $\ring{B}(v)=0$ has valency at least $3$. When $x$ is a constant orbit, we also require that $v_0$ has valency at least $3$ if $\ring{B}(v_0)=0$ in $\pi_2(M,L)$.
\end{itemize}

\end{definition}

Note that when $x$ is a constant orbit, this reduces to the notion of a decorated rooted ribbon tree. For $k\in\mathbb{Z}_{\geq0}$ and $\ring{\beta}\in\pi_2(M,x,L)$, denote by $\mathcal{G}(k+1,\ring{\beta})$ the set of decorated rooted ribbon trees with a single puncture $(T,\ring{B},v_0)$, such that $\#C_{0,\mathit{ext}}(T)=k+1$ and $\sum_{v\in C_{0,\mathit{int}}(T)}\ring{B}(v)=\ring{\beta}$.

We also have the following extension of a reduction of decorated rooted ribbon trees (cf. \cite{ki2}, Definition 7.19).

\begin{definition}\label{definition:reduction}
Let $(T,\ring{B},v_0)\in\mathcal{G}(k+1,\ring{\beta})$ and $e\in C_{1,\mathit{int}}(T)$, with $v_0,v_1$ being vertices of $e$. By collapsing $e$ to a new vertex $v_{01}$, we get another $(T',\ring{B}',v_0)\in\mathcal{G}(k+1,\ring{\beta})$ such that
\begin{equation}
C_0(T')=\left(C_0(T)\setminus\{v_0,v_1\}\right)\cup\{v_{01}\},
\end{equation}
\begin{equation}
C_1(T')=C_1(T)\setminus\{e\},
\end{equation}
\begin{equation}
\ring{B}'(v)=\left\{\begin{array}{ll}
\ring{B}(v) & v\neq v_{01}, \\ \ring{B}(v_0)+\ring{B}(v_1) & v=v_{01}.
\end{array}\right.
\end{equation}
An element of $\mathcal{G}(k+1,\ring{\beta})$ which can be obtained from $(T,\ring{B},v_0)\in\mathcal{G}(k+1,\ring{\beta})$ by repeating the above procedure is called a reduction of $(T,\ring{B},v_0)$.
\end{definition}

Let $X$ be an admissible K-space. For $r\in\mathbb{N}$, denoted by $\widehat{S}_rX$ the normalized codimension $r$ corner of $X$. When $r=1$, we shall also write $\partial X$ for the normalized boundary.

From now on, let $\varepsilon>0$ be chosen as in the beginning of Section \ref{section:approximation}, so that every non-constant $J_M$-holomorphic disc bounded by $L\subset M$ has area more than $2\varepsilon$, where $J_M$ is a convex almost complex structure. Take $U\in\mathbb{N}$ so that $\varepsilon(U-1)\geq|A_{H_t}(x)|$.

\begin{theorem}\label{theorem:moduli}
For every $k,m,l\in\mathbb{Z}_{\geq0}$, and $P=\{m\}$ or $[m,m+1]$, there exist the following data.
\begin{itemize}
	\item[(i)] (Moduli spaces) Compact, oriented, admissible K-spaces (when $l=0$, the moduli spaces (\ref{eq:m3}), (\ref{eq:m4}) and (\ref{eq:m5}) are empty)
	\begin{equation}\label{eq:m0}
	\overline{\mathcal{R}}_{k+1}(L,\beta;P),\textrm{ where }\beta\in\pi_2(M,L)\textrm{ and }\theta_M(\partial\beta)<(m-k+1)\varepsilon,
	\end{equation}
	\begin{equation}\label{eq:m1}
	\overline{\mathcal{R}}_{k+2,\vartheta}(L,\beta;P),\textrm{ where }\beta\in\pi_2(M,L)\textrm{ and }\theta_M(\partial\beta)<(m-k+1)\varepsilon,
	\end{equation}
    \begin{equation}\label{eq:m2}
    {}_l\overline{\mathcal{R}}_{k+1}^1(x,L,\ring{\beta};P),\textrm{ where }\ring{\beta}\in\pi_2(M,x,L)\textrm{ and }\theta_M(\partial\ring{\beta})<(m-k-U)\varepsilon,
    \end{equation}
    \begin{equation}\label{eq:m3}
    {}_l^{j,j+1}\overline{\mathcal{R}}_{k+1}^1(x,L,\ring{\beta};P), \textrm{ where }\ring{\beta}\in\pi_2(M,x,L)\textrm{ and }\theta_M(\partial\ring{\beta})<(m-k-U)\varepsilon,
    \end{equation}
    \begin{equation}\label{eq:m4}
    {}_{l-1}\overline{\mathcal{R}}_{k+1}^{S^1}(x,L,\ring{\beta};P),\textrm{ where }\ring{\beta}\in\pi_2(M,x,L)\textrm{ and }\theta_M(\partial\ring{\beta})<(m-k-U)\varepsilon,
    \end{equation}
    \begin{equation}\label{eq:m5}
    {}_{l-1}\overline{\mathcal{R}}_{k+1,\tau_i}^1(x,L,\ring{\beta};P),\textrm{ where }\ring{\beta}\in\pi_2(M,x,L)\textrm{ and }\theta_M(\partial\ring{\beta})<(m-k-U)\varepsilon,
    \end{equation}
    whose underlying topological spaces are $P\times\overline{\mathcal{R}}_{k+1}(L,\beta)$, $P\times\overline{\mathcal{R}}_{k+2,\vartheta}(L,\beta)$, $P\times{}_l\overline{\mathcal{R}}_{k+1}^1(x,L,\ring{\beta})$, $P\times{}_l^{j,j+1}\overline{\mathcal{R}}_{k+1}^1(x,L,\ring{\beta})$, $P\times{}_{l-1}\overline{\mathcal{R}}_{k+1}^{S^1}(x,L,\ring{\beta})$, and $P\times{}_{l-1}\overline{\mathcal{R}}_{k+1,\tau_i}^1(x,L,\ring{\beta})$, respectively. Dimensions of these K-spaces are
    \begin{equation}
    \dim\overline{\mathcal{R}}_{k+2,\vartheta}(L,\beta;P)=\dim\overline{\mathcal{R}}_{k+1}(L,\beta;P)=\mu(\beta)+n+k-2+\dim P,
    \end{equation}
    \begin{equation}
    \dim{}_l\overline{\mathcal{R}}_{k+1}^1(x,L,\ring{\beta};P)=\mu(\ring{\beta})+k+2l+\mathit{CZ}(x)+\dim P,
    \end{equation}
    \begin{equation}
    \begin{split}
    \dim{}_l^{j,j+1}\overline{\mathcal{R}}_{k+1}^1(x,L,\ring{\beta};P)&=\dim{}_{l-1}\overline{\mathcal{R}}_{k+1}^{S^1}(x,L,\ring{\beta};P)=\dim{}_{l-1}\overline{\mathcal{R}}_{k+1,\tau_i}^1(x,L,\ring{\beta};P) \\
    &=\mu(\ring{\beta})+k+2l-1+\mathit{CZ}(x)+\dim P,
    \end{split}
    \end{equation}
    where $\mathit{CZ}(x)$ is the Conley-Zehnder index of $x$.
    \item[(ii)] (Evaluation maps) Corner-stratified strongly smooth maps
    \begin{equation}\label{eq:eva0}
    \mathit{ev}^{\mathcal{R},P}:\overline{\mathcal{R}}_{k+1}(L,\beta;P)\rightarrow P\times L^{k+1},
    \end{equation}
    \begin{equation}\label{eq:eva1}
    \mathit{ev}_\vartheta^{\mathcal{R},P}:\overline{\mathcal{R}}_{k+2,\vartheta}(L,\beta;P)\rightarrow P\times L^{k+2},
    \end{equation}
    \begin{equation}\label{eq:eva2}
    {}_l\mathit{ev}^{\mathcal{R},P}:{}_l\overline{\mathcal{R}}_{k+1}^1(x,L,\ring{\beta};P)\rightarrow P\times L^{k+1},
    \end{equation}
    \begin{equation}\label{eq:eva3}
    {}_l^{j,j+1}\mathit{ev}^{\mathcal{R},P}:{}_l^{j,j+1}\overline{\mathcal{R}}_{k+1}^1(x,L,\ring{\beta};P)\rightarrow P\times L^{k+1},
    \end{equation}
    \begin{equation}\label{eq:eva4}
    {}_{l-1}\mathit{ev}^{S^1,P}:{}_{l-1}\overline{\mathcal{R}}_{k+1}^{S^1}(x,L,\ring{\beta};P)\rightarrow P\times L^{k+1},
    \end{equation}
    \begin{equation}\label{eq:eva5}
	{}_{l-1}\mathit{ev}_i^{\mathcal{R},P}:{}_{l-1}\overline{\mathcal{R}}_{k+1,\tau_i}^1(x,L,\ring{\beta};P)\rightarrow P\times L^{k+1},
    \end{equation}
    such that their underlying set-theoretic maps are $\mathit{id}_P\times\mathit{ev}^\mathcal{R}$, $\mathit{id}_P\times\mathit{ev}_\vartheta^\mathcal{R}$, $\mathit{id}_P\times{}_l\mathit{ev}^\mathcal{R}$, $\mathit{id}_P\times{}_l^{j,j+1}\mathit{ev}^\mathcal{R}$, $\mathit{id}_P\times{}_{l-1}\mathit{ev}^{S^1}$, and $\mathit{id}_P\times{}_{l-1}\mathit{ev}_i^\mathcal{R}$, respectively. Moreover, the maps
    \begin{equation}
    (\mathit{id}_P\times\mathit{pr}_0)\circ\mathit{ev}^{\mathcal{R},P}:\overline{\mathcal{R}}_{k+1}(L,\beta;P)\rightarrow P\times L,
    \end{equation}
    \begin{equation}
    (\mathit{id}_P\times\mathit{pr}_0)\circ \mathit{ev}_\vartheta^{\mathcal{R},P}:\overline{\mathcal{R}}_{k+2,\vartheta}(L,\beta;P)\rightarrow P\times L,
    \end{equation}
    \begin{equation}
    (\mathit{id}_P\times\mathit{pr}_0)\circ{}_l\mathit{ev}^{\mathcal{R},P}:{}_l\overline{\mathcal{R}}_{k+1}^1(x,L,\ring{\beta};P)\rightarrow P\times L,
    \end{equation}
    \begin{equation}
    (\mathit{id}_P\times\mathit{pr}_0)\circ{}_l^{j,j+1}\mathit{ev}^{\mathcal{R},P}:{}_l^{j,j+1}\overline{\mathcal{R}}_{k+1}^1(x,L,\ring{\beta};P)\rightarrow P\times L,
    \end{equation}
    \begin{equation}
    (\mathit{id}_P\times\mathit{pr}_0)\circ\mathit{ev}^{S^1,P}:{}_{l-1}\overline{\mathcal{R}}_{k+1}^{S^1}(x,L,\ring{\beta};P)\rightarrow P\times L,
    \end{equation}
    \begin{equation}
    (\mathit{id}_P\times\mathit{pr}_0)\circ\mathit{ev}_i^{\mathcal{R},P}:{}_{l-1}\overline{\mathcal{R}}_{k+1,\tau_i}^1(x,L,\ring{\beta};P)\rightarrow P\times L
    \end{equation}
    are corner-stratified weak submersions, where $\mathit{pr}_0:L^{k+1}\rightarrow L$ is the projection to the first factor.
    \item[(iii)] (Energy zero part) An isomorphism of admissible Kuranishi structures
    \begin{equation}\label{eq:e0}
    \overline{\mathcal{R}}_{k+2,\vartheta}(L,0;P)\cong P\times L\times D^{k-2}
    \end{equation}
    for every $k\geq2$, so that $\mathit{ev}_\vartheta^{\mathcal{R},P}: \overline{\mathcal{R}}_{k+2,\vartheta}(L,0;P)\rightarrow P\times L^{k+2}$ coincides with $\mathit{pr}_P\times(\mathit{pr}_L)^{k+2}$, where $\mathit{pr}_P$ is the projection to $P$ and $\mathit{pr}_L$ is the projection to $L$. Here, the disc $D^{k-2}$ on the right-hand side of (\ref{eq:e0}) is given by the Stasheff cell.
    \item[(iv)] (Compatibility at boundaries) Orientation-preserving isomorphisms of admissible K-spaces:
    \begin{equation}\label{eq:bd1}
    \partial\overline{\mathcal{R}}_{k+2,\vartheta}(L,\beta;\{m\})\cong\bigsqcup_{\substack{k_1+k_2=k+1\\1\leq i\leq k_1+1\\ \beta_1+\beta_2=\beta}}(-1)^{\varepsilon_1}\overline{\mathcal{R}}_{k_1+2,\vartheta}(L,\beta_1;\{m\})\textrm{ }{{}_i\times_0}\textrm{ }\overline{\mathcal{R}}_{k_2+1}(L,\beta_2;\{m\}),
    \end{equation}
    \begin{equation}\label{eq:bd2}
    \begin{split}
    \partial{}_l\overline{\mathcal{R}}_{k+1}^1(x,L,\ring{\beta};\{m\})&\cong\bigsqcup_{\substack{k_1+k_2=k+1\\1\leq i\leq k_1\\ \ring{\beta}_1+\beta_2=\ring{\beta}}}(-1)^{\varepsilon_2}{}_l\overline{\mathcal{R}}_{k_1+1}^1(x,L,\ring{\beta}_1;\{m\})\textrm{ }{{}_i\times_0}\textrm{ }\overline{\mathcal{R}}_{k_2+1}(L,\beta_2;\{m\}) \\
    &\sqcup\bigsqcup_{\substack{k_1+k_2=k+1\\1\leq i\leq k_1\\ \beta_1+\ring{\beta}_2=\ring{\beta}}}(-1)^{\varepsilon_3}\overline{\mathcal{R}}_{k_1+1}(L,\beta_1;\{m\})\textrm{ }{{}_i\times_0}\textrm{ }{}_l\overline{\mathcal{R}}^1_{k_2+1}(x,L,\ring{\beta}_2;\{m\}) \\
    &\sqcup\bigsqcup_{0\leq j\leq l}(-1)^{\varepsilon_{4,j}}{}_j\overline{\mathcal{M}}(x,y_j;\{m\})\times{}_{l-j}\overline{\mathcal{R}}_{k+1}^1(y_j,L,\ring{\beta};\{m\}) \\
    &\sqcup\bigsqcup_{1\leq j\leq l-1}(-1)^{\varepsilon_5}{}_l^{j,j+1}\overline{\mathcal{R}}_{k+1}^1(x,L,\ring{\beta};\{m\}) \\
    &\sqcup(-1)^{\varepsilon_6}{}_{l-1}\overline{\mathcal{R}}_{k+1}^{S^1}(x,L,\ring{\beta};\{m\}),
    \end{split}
    \end{equation}
    \begin{equation}\label{eq:bd4}
    \begin{split}
    &\partial\overline{\mathcal{R}}_{k+2,\vartheta}(L,\beta;[m,m+1])\cong(-1)^{\varepsilon_7}\overline{\mathcal{R}}_{k+2,\vartheta}(L,\beta;\{m\})\sqcup(-1)^{\varepsilon_{8}}\overline{\mathcal{R}}_{k+2,\vartheta}(L,\beta;\{m+1\}) \\
    &\sqcup\bigsqcup_{\substack{k_1+k_2=k+1\\1\leq i\leq k_1+1\\ \beta_1+\beta_2=\beta}}(-1)^{\varepsilon_{9}}\overline{\mathcal{R}}_{k_1+2,\vartheta}(L,\beta_1;[m,m+1])\textrm{ }{{}_i\times_0}\textrm{ }\overline{\mathcal{R}}_{k_2+1}(L,\beta_2;[m,m+1]),
    \end{split}
    \end{equation}
    \begin{equation}\label{eq:bd5}
    \begin{split}
    &\partial{}_l\overline{\mathcal{R}}_{k+1}^1(x,L,\ring{\beta};[m,m+1]) \\
    &\cong(-1)^{\varepsilon_{10}}{}_l\overline{\mathcal{R}}_{k+1}^1(x,L,\ring{\beta};\{m\})\sqcup(-1)^{\varepsilon_{11}}{}_l\overline{\mathcal{R}}_{k+1}^1(x,L,\ring{\beta};\{m+1\}) \\
    &\sqcup\bigsqcup_{\substack{k_1+k_2=k+1\\1\leq i\leq k_1\\ \ring{\beta}_1+\beta_2=\ring{\beta}}}(-1)^{\varepsilon_{12}}{}_l\overline{\mathcal{R}}_{k_1+1}^1(x,L,\ring{\beta}_1;[m,m+1])\textrm{ }{{}_i\times_0}\textrm{ }\overline{\mathcal{R}}_{k_2+1}(L,\beta_2;[m,m+1]) \\
    &\sqcup\bigsqcup_{\substack{k_1+k_2=k+1\\1\leq i\leq k_1\\ \beta_1+\ring{\beta}_2=\ring{\beta}}}(-1)^{\varepsilon_{13}}\overline{\mathcal{R}}_{k_1+1}(L,\beta_1;[m,m+1])\textrm{ }{{}_i\times_0}\textrm{ }{}_l\overline{\mathcal{R}}_{k_2+1}(x,L,\ring{\beta}_2;[m,m+1])\\
    &\sqcup\bigsqcup_{0\leq j\leq l}(-1)^{\varepsilon_{14,j}}{}_j\overline{\mathcal{M}}(x,y_j;[m,m+1])\times{}_{l-j}\overline{\mathcal{R}}_{k+1}^1(y_j,L,\ring{\beta};[m,m+1]) \\
    &\sqcup\bigsqcup_{1\leq j\leq l-1}(-1)^{\varepsilon_{15}}{}_l^{j,j+1}\overline{\mathcal{R}}_{k+1}^1(x,L,\ring{\beta};[m,m+1])\sqcup(-1)^{\varepsilon_{16}}{}_{l-1}\overline{\mathcal{R}}_{k+1}^{S^1}(x,L,\ring{\beta};[m,m+1]),
    \end{split}
    \end{equation}
    where the notation ${{}_i\times_0}$ in the above is an abbreviation for the fiber product ${{}_{\mathit{ev}_i}\times_{\mathit{ev}_0}}$, with $\mathit{ev}_i$ given by the composition $(\mathit{id}_P\times\mathit{pr}_i)\circ\mathit{ev}^{\bullet,P}$, where $\mathit{ev}^{\bullet,P}$ stands for one of the evaluations maps (\ref{eq:eva0}), (\ref{eq:eva1}) and (\ref{eq:eva2}) in (ii), and $\mathit{pr}_i:L^{k+1}\rightarrow L$ is the projection to the $(i+1)$th factor. The signs $\varepsilon_1,\cdots,\varepsilon_{16}$ above are given as follows:
    \begin{equation}
    \begin{split}
    &\varepsilon_1=(k_1-i)(k_2-1)+n+k_1-1, \\
    &\varepsilon_2=(k_1-i)(k_2-1)+n+k,\textrm{ }\varepsilon_3=(k_1-i)(k_2-1)+n+1, \\
    &\varepsilon_{4,j}=n+|y_j|,\textrm{ }0\leq j\leq l,\textrm{ }\varepsilon_5=0,\textrm{ }\varepsilon_6=0, \\
    &\varepsilon_7=1,\textrm{ }\varepsilon_{8}=0,\textrm{ }\varepsilon_{9}=(k_1-i)(k_2-1)+n+k_1, \\
    &\varepsilon_{10}=1,\textrm{ }\varepsilon_{11}=0, \\
    &\varepsilon_{12}=(k_1-i)(k_2-1)+n+k+1,\textrm{ }\varepsilon_{13}=(k_1-i)(k_2-1)+n, \\
    &\varepsilon_{14,j}=n+|y_j|+1,\textrm{ }0\leq j\leq l,\textrm{ }\varepsilon_{15}=1,\textrm{ }\varepsilon_{16}=1.
    \end{split}
    \end{equation}
    The compatibility of the Kuranishi structures at the boundaries of ${}_l^{j,j+1}\overline{\mathcal{R}}_{k+1}^1(x,L,\ring{\beta};P)$, ${}_{l-1}\overline{\mathcal{R}}_{k+1}^{S^1}(x,L,\ring{\beta};P)$ and ${}_{l-1}\overline{\mathcal{R}}_{k+1,\tau_i}^1(x,L,\ring{\beta};P)$ are similar to that of ${}_l\overline{\mathcal{R}}_{k+1}^1(x,L,\ring{\beta};P)$ described above, we do not write down the details for these moduli spaces here since they will not be needed for our main argument in Section \ref{section:proof}.
    \item[(v)] (Compatibility at corners, I) Isomorphisms of admissible K-spaces\footnote{For simplicity, we do not consider orientations of the moduli spaces for these isomorphisms. The signs can be computed inductively using the methods described in Appendix \ref{section:orientation}.}
    \begin{equation}\label{eq:corner1}
    \begin{split}
    &\widehat{S}_r\overline{\mathcal{R}}_{k+2,\vartheta}(L,\beta;P)\cong\bigsqcup_{\substack{(T,B)\in\mathcal{G}(k+2,\beta)\\ \#C_{1,\mathit{int}}(T)+d=r \\ v_0\in C_{0,\mathit{int}}(T)}}\left(\prod_{e\in C_{1,\mathit{int}}(T)}\widehat{S}_dP\times L\right)\textrm{ }{{}_\Delta\times\mathit{ev}_\mathit{int}}\textrm{ } \\
    &\left(\prod_{v\in C_{0,\mathit{int}}(T)\setminus\{v_0\}}\overline{\mathcal{R}}_{k_v+1}\left(L,B(v);\widehat{S}_dP\right)\times\overline{\mathcal{R}}_{k_{v_0}+1,\vartheta}\left(L,B(v_0);\widehat{S}_dP\right)\right),
    \end{split}
    \end{equation}
    where the fiber product on the right-hand side is taken over $\prod_{e\in C_{1,\mathit{int}}(T)}\left(\widehat{S}_dP\times L\right)^2$.
    \begin{equation}\label{eq:corner2}
    \begin{split}
    &\widehat{S}_r\left({}_l\overline{\mathcal{R}}_{k+1}^1(x,L,\ring{\beta};P)\right)\cong\bigsqcup_{\substack{(T,\ring{B},v_0)\in\mathcal{G}(k+1,\ring{\beta})\\ \#C_{1,\mathit{int}}(T)+d+r_1+r_2=r \\ v_0\in C_{0,\mathit{int}}(T)}}\bigsqcup_{\substack{j_1+\cdots+j_{r_2}=l_2\\l_1+l_2=l}}\left(\prod_{e\in C_{1,\mathit{int}}(T)}\widehat{S}_dP\times L\right) \\
    &\textrm{ }{{}_\Delta\times\mathit{ev}_\mathit{int}}\textrm{ }\left(\prod_{v\in C_{0,\mathit{int}}(T)\setminus\{v_0\}}\overline{\mathcal{R}}_{k_v+1}\left(L,\ring{B}(v);\widehat{S}_dP\right)\right.\\
    &\left.\times\left(\prod_{i=1}^{r_2}{}_{j_i}\overline{\mathcal{M}}(y_{j_i},x;\widehat{S}_dP)\times\widehat{S}_{r_1}({}_{l_1}\overline{\mathcal{R}}_{k_{v_0}+1}^1)\left(y_{j_i},L,\ring{B}(v_0);\widehat{S}_dP\right)\right)\right), \\
    \end{split}
    \end{equation}
    where $\widehat{S}_{r_1}({}_{l_1}\overline{\mathcal{R}}_{k_{v_0}+1}^1)$ denotes the codimension $r_1<r$ corner of the moduli space of domains ${}_{l_1}\overline{\mathcal{R}}_{k_{v_0}+1}^1$, and $y_{j_i}$ are 1-periodic orbits of $X_{H_t}$. The identifications of the codimension $r$ corners of the moduli spaces ${}_l^{j,j+1}\overline{\mathcal{R}}_{k+1}^1(x,L,\ring{\beta};P)$, ${}_{l-1}\overline{\mathcal{R}}_{k+1}^{S^1}(x,L,\ring{\beta};P)$ and ${}_{l-1}\overline{\mathcal{R}}_{k+1,\tau_i}^1(x,L,\ring{\beta};P)$ are similar to that of ${}_l\overline{\mathcal{R}}_{k+1}^1(x,L,\ring{\beta};P)$ above, and are therefore omitted.
    \item[(vi)] (Compatibility at corners, II) Let $X$ be either $\overline{\mathcal{R}}_{k+1}(L,\beta;P)$, ${}_l\overline{\mathcal{R}}_{k+1}^1(x,L,\ring{\beta};P)$, $\overline{\mathcal{R}}_{k+2,\vartheta}(L,\beta;P)$, ${}_l^{j,j+1}\overline{\mathcal{R}}_{k+1}^1(x,L,\ring{\beta};P)$, ${}_{l-1}\overline{\mathcal{R}}_{k+1}^{S^1}(x,L,\ring{\beta};P)$ and ${}_{l-1}\overline{\mathcal{R}}_{k+1,\tau_i}^1(x,L,\ring{\beta};P)$. Then for every $l,l'\in\mathbb{N}$, the canonical covering map $\pi_{l,l'}:\widehat{S}_{l'}(\widehat{S}_lX)\rightarrow\widehat{S}_{l+l'}(X)$ coincides with the map defined from the fiber product presentation in (v).
    \item[(vii)] (Cyclic invariance) Let $X$ be one of the moduli spaces $\overline{\mathcal{R}}_{k+1}(L,\beta;P)$, $\overline{\mathcal{R}}_{k+2,\vartheta}(L,\beta;P)$, ${}_l\overline{\mathcal{R}}_{k+1}^1(x,L,\ring{\beta};P)$, ${}_l^{j,j+1}\overline{\mathcal{R}}_{k+1}^1(x,L,\ring{\beta};P)$, and ${}_{l-1}\overline{\mathcal{R}}_{k+1}^{S^1}(x,L,\ring{\beta};P)$. Then the Kuranishi structure on $X$ is invariant under the $\mathbb{Z}_{k+1}$-action induced by the cyclic permutations of boundary marked points $z_0,\cdots,z_k$.
\end{itemize}
\end{theorem}

\begin{proof}
We shall explain here some important points in the statement of the theorem which are not self-evident.

The constructions of the Kuranishi structures on the moduli spaces ${}_l\overline{\mathcal{R}}_{k+1}^1(x,L,\ring{\beta};P)$, ${}_l^{j,j+1}\overline{\mathcal{R}}_{k+1}^1(x,L,\ring{\beta};P)$, ${}_{l-1}\overline{\mathcal{R}}_{k+1}^{S^1}(x,L,\ring{\beta};P)$, $\overline{\mathcal{R}}_{k+1}^1(x,L,\ring{\beta};P)$ and ${}_{l-1}\overline{\mathcal{R}}^1_{k+1,\tau_i}(x,L,\ring{\beta};P)$ are combinations of the constructions of the Kuranishi structures for the moduli spaces of pseudoholomorphic discs (tree-like K-systems, cf. \cite{fooo5,fooo6}) and that of the moduli spaces of solutions to Floer's equation associated to a time-dependent Hamiltonian function (linear K-systems, cf. \cite{fooo4,fon}). In fact, since in our case the transversality of the moduli spaces ${}_j\overline{\mathcal{M}}(x,y)$ of Floer cylinders can be achieved by carefully choosing Floer data, the constructions of Kuranishi structures on these moduli spaces are essentially reduced to that of $J_M$-holomorphic discs. However, as we will discuss shortly, in order to achieve the invariance property (vii), we need to be a little bit more careful when constructing these Kuranishi structures. The dimension computations are straightforward. The explicit description of a Kuranishi chart of ${}_l\overline{\mathcal{R}}_{k+1}^1(x,L,\ring{\beta};P)$ can be found in Lemma \ref{lemma:chart} below.

Next, we explain how to define the evaluation map
\begin{equation}
{}_l\mathit{ev}^\mathcal{R}=({}_l\mathit{ev}^\mathcal{R}_0,\cdots,{}_l\mathit{ev}^\mathcal{R}_k):{}_l\overline{\mathcal{R}}_k^1(x,L,\ring{\beta})\rightarrow L^{k+1}.
\end{equation} 
For each $(T,\ring{B},v_0)\in\mathcal{G}(k+1,\ring{\beta})$, each $0\leq l_1\leq l$, and each decomposition $l-l_1=\sum_{i=1}^{r_2}j_i$, with $j_i\in\mathbb{Z}_{\geq0}$ for each $1\leq i\leq r_2$, we can define the exterior evaluation map
\begin{equation}
\begin{split}
\mathit{ev}_\mathit{ext}:&\prod_{v\in C_{0,\mathit{int}}(T)\setminus\{v_0\}}\mathcal{R}_{k_v+1}\left(L,\ring{B}(v)\right)\times\prod_{i=1}^{r_2}{}_{j_i}\mathcal{M}(y_{j_i},x)\times{}_{l_1}\mathcal{R}_{k_{v_0}+1}^1\left(y_{j_i},L,\ring{B}(v_0)\right)\rightarrow L^{k+1}
\end{split}
\end{equation}
in the same manner as the maps (\ref{eq:ev-ext}) and (\ref{eq:ext2}) in Section \ref{section:disc}. The map ${}_l\mathit{ev}^\mathcal{R}$ is given by restricting $\mathit{ev}_\mathit{ext}$ to the compactified moduli space
\begin{equation}
\begin{split}
&{}_l\overline{\mathcal{R}}_{k+1}^1(x,L,\ring{\beta})\cong\bigsqcup_{\substack{(T,\ring{B},v_0)\in\mathcal{G}(k+1,\ring{\beta})\\v_0\in C_{0,\mathit{int}}(T)}}\bigsqcup_{\substack{j_1+\cdots+j_{r_2}=l_2\\l_1+l_2=l}}\left(\prod_{e\in C_{1,\mathit{int}}(T)}L\right)\textrm{ }{{}_\Delta\times\mathit{ev}_\mathit{int}}\textrm{ } \\
&\left(\prod_{v\in C_{0,\mathit{int}}(T)\setminus\{v_0\}}\mathcal{R}_{k_v+1}\left(L,\ring{B}(v)\right)\times\prod_{i=1}^{r_2}{}_{j_i}\mathcal{M}(y_{j_i},x)\times{}_{l_1}\mathcal{R}_{k_{v_0}+1}^1\left(y_{j_i},L,\ring{B}(v_0)\right)\right). \\
\end{split}
\end{equation}
This definition directly extends to the moduli spaces ${}_l\overline{\mathcal{R}}_{k+1}^1(x,L,\ring{\beta};P)$ with $P$-coefficients and gives rise to the evaluation map ${}_l\mathit{ev}^{\mathcal{R},P}$ in (ii). The constructions of the evaluation maps (\ref{eq:eva3}), (\ref{eq:eva4}) and (\ref{eq:eva5}) are similar, and we omit the details.
	
The isomorphisms (\ref{eq:bd1}) and (\ref{eq:bd4}) follow from (\ref{eq:st}). Note that in (\ref{eq:bd1}) and (\ref{eq:bd4}), there is no boundary strata of the form
\begin{equation}
\overline{\mathcal{R}}_{k_1+1}(L,\beta_1;P)\textrm{ }{{}_i\times_0}\textrm{ }\overline{\mathcal{R}}_{k_2+2,\vartheta}(L,\beta_2;P).
\end{equation}
This is because according to the construction, the constraint imposed on the boundary marked point with ``the largest label" only makes sense on the disc component with which carries $z_0$.
	
For the isomorphisms (\ref{eq:bd2}) and (\ref{eq:bd4}), note that the codimension $1$ boundary strata of the compactified moduli space ${}_l\overline{\mathcal{R}}_{k+1}^1(x,L,\ring{\beta})$ are covered by the natural inclusions of
\begin{equation}\label{eq:de1}
{}_j\overline{\mathcal{M}}(x,y_j)\times{}_{l-j}\overline{\mathcal{R}}_{k+1}^1(y_j,L,\ring{\beta}),\textrm{ where }1\leq j\leq l,
\end{equation}
\begin{equation}\label{eq:de2}
{}_l^{j,j+1}\overline{\mathcal{R}}_{k+1}^1(x,L,\ring{\beta}),\textrm{ where }1\leq j\leq l-1,
\end{equation}
\begin{equation}\label{eq:de3}
{}_{l-1}\overline{\mathcal{R}}_{k+1}^{S^1}(x,L,\ring{\beta}),
\end{equation}
which come from the degenerations of domains,
\begin{equation}\label{eq:br1}
\overline{\mathcal{M}}(x,y)\times{}_l\overline{\mathcal{R}}_{k+1}^1(y,L,\ring{\beta}),
\end{equation}
which comes from semi-stable breaking, and
\begin{equation}\label{eq:bubb1}
{}_l\overline{\mathcal{R}}_{k_1+1}^1(x,L,\ring{\beta}_1)\textrm{ }{{}_i\times}_0\textrm{ }\overline{\mathcal{R}}_{k_2+1}(L,\beta_2),\textrm{ where }k_1+k_2=k+1,1\leq i\leq k_1\textrm{ and }\ring{\beta}_1+\beta_2=\ring{\beta},
\end{equation}
\begin{equation}\label{eq:bubb2}
\overline{\mathcal{R}}_{k_1+1}(L,\beta_1)\textrm{ }{{}_i\times}_0\textrm{ }{}_l\overline{\mathcal{R}}_{k_2+1}^1(x,L,\ring{\beta}_2),\textrm{ where }k_1+k_2=k+1,1\leq i\leq k_1\textrm{ and }\beta_1+\ring{\beta}_2=\ring{\beta},
\end{equation}
which are disc bubbles (the cases when $k_2\geq2$ and $k_1\geq2$ in (\ref{eq:bubb1}) and (\ref{eq:bubb2}) respectively also correspond to degenerations of domains, see Proposition \ref{proposition:domain-bdy}). \\
In the above, the strata (\ref{eq:de1})-(\ref{eq:de3}) correspond respectively to the strata (\ref{eq:dom-bdy1})-(\ref{eq:dom-bdy4}) in Proposition \ref{proposition:domain-bdy}. Due to the possible non-exactness of $L$ in our case, there can be disc bubbles. By our choice of Floer data as specified in Definition \ref{definition:data}, their contributions give the boundary strata (\ref{eq:bubb1}) and (\ref{eq:bubb2}).

The sign computations in (iv) will be carried out in Appendix \ref{section:orientation}. The compatibility conditions in (iv) and (v) will be explained at the end of this subsection.

For the moduli spaces $\overline{\mathcal{R}}_{k+1}(L,\beta;P)$ and $\overline{\mathcal{R}}_{k+2,\vartheta}(L,\beta;P)$, the invariance property (vii) of the Kuranishi structures follows from the construction of Fukaya (cf.\cite{kfc}, Corollary 3.1). The arguments for the moduli spaces ${}_l\overline{\mathcal{R}}_{k+1}^1(x,L,\ring{\beta};P)$, ${}_l^{j,j+1}\overline{\mathcal{R}}_{k+1}^1(x,L,\ring{\beta};P)$ and ${}_{l-1}\overline{\mathcal{R}}_{k+1}^{S^1}(x,L,\ring{\beta};P)$ are similar. For clarity, we breifly explain here Fukaya's construction. The key point is to observe that \cite{kfc}, Lemma 3.1 holds not only for the case when the domain is a disc, but a punctured disc as well. For simplicity, we consider here only the case of ${}_l\overline{\mathcal{R}}_{k+1}^1(x,L,\ring{\beta};P)$. Let ${}_l\overline{\mathcal{R}}^1(x,L,\ring{\beta};P)$ be the Cohen-Ganatra moduli space constructed as in Section \ref{section:CG}, but using the domains $(S;p_1,\cdots,p_l;\ell)\in{}_l\mathcal{R}^1$ without boundary marked points. This enables us to construct Kuranishi structures on ${}_l\overline{\mathcal{R}}_{k+1}^1(x,L,\ring{\beta};P)$ for various $k\in\mathbb{Z}_{\geq0}$ by pulling back the Kuranishi structure on the moduli space ${}_l\overline{\mathcal{R}}^1(x,L,\ring{\beta};P)$, via the map
\begin{equation}\label{eq:forget}
f_{k+1}:{}_l\overline{\mathcal{R}}_{k+1}^1(x,L,\ring{\beta};P)\rightarrow{}_l\overline{\mathcal{R}}^1(x,L,\ring{\beta};P)
\end{equation}
which forgets all the boundary marked points $z_0,\cdots,z_k$. Here, what is slightly different from the cases of the moduli spaces $\overline{\mathcal{R}}_{k+1}(L,\beta;P)$ and $\overline{\mathcal{R}}_{k+2,\vartheta}(L,\beta;P)$ is that additional care must be taken when $\partial\ring{\beta}=0\in H_1(L;\mathbb{Z})$, in which case there is an extra boundary stratum in $\partial{}_l\overline{\mathcal{R}}^1(x,L,\ring{\beta};P)$ corresponding to a configuration of Floer cylinders (with additional marked points) in ${}_l\overline{\mathcal{M}}(x,\mathit{pt})$, where $\mathit{pt}\in L$ is regarded as a constant orbit of $X_{H_t}$, attached to a constant disc at $\mathit{pt}$. Equivalently, denote by ${}_l\overline{\mathcal{M}}(x)$ the moduli space of pairs $\left((\mathbb{C},p_1,\cdots,p_l),u\right)$, where $p_1,\cdots,p_l\in\mathbb{C}$ are auxiliary marked points satisfying
\begin{equation}
0<|p_1|<\cdots<|p_l|,
\end{equation}
and $u:\mathbb{C}\rightarrow M$ is a solution to the Floer equation with the asymptotic condition at the positive cylindrical end given by $x$. Evaluation at the origin $0\in\mathbb{C}$ defines a map $\mathit{ev}_0^\mathcal{M}:{}_l\overline{\mathcal{M}}(x)\rightarrow M$, with which the extra boundary stratum can be written as the fiber product
\begin{equation}\label{eq:extra}
{}_l\overline{\mathcal{M}}(x)\textrm{ }{{}_{\mathit{ev}_0^\mathcal{M}}\times_M}\textrm{ }L.
\end{equation}
In this case, we need a Kuranishi structure on ${}_l\overline{\mathcal{M}}(x)$ which is compatible with the gluings of Floer trajectories, such that $\mathit{ev}_0^\mathcal{M}$ is strongly continuous and weakly submersive. Since transversalities of (\ref{eq:extra}) can be achieved by appropriately choosing Floer data, this is obviously true with trivial obstruction bundle. Taking a manifold (with corners) chart $U_p$ of ${}_l\overline{\mathcal{M}}(x)$, where $p\in{}_l\overline{\mathcal{M}}(x)$, the restriction of $\mathit{ev}_0^\mathcal{M}$ defines a map $U_p\rightarrow M$. Define $V_p:=(\mathit{ev}_0^\mathcal{M}|_{U_p})^{-1}(L)$, then its product with the compactification of the configuration space
\begin{equation}
\mathit{Conf}_{k+1}:=\left\{(z_0,\cdots,z_k)\in(\partial S)^{k+1}|\textrm{the points }z_i\textrm{'s respect cyclic order}\right\}/S^1
\end{equation}
of $k+1$ cyclically ordered points on the circle can be regarded as a stratum of a Kuranishi neighborhood of the pullback of $(p,q)\in{}_l\overline{\mathcal{M}}(x)\textrm{ }{{}_{\mathit{ev}_0^\mathcal{M}}\times_M}\textrm{ }L\subset\partial{}_l\overline{\mathcal{R}}^1(x,L,\ring{\beta};P)$ in ${}_l\overline{\mathcal{R}}_{k+1}^1(x,L,\ring{\beta};P)$. We can then extend it to a Kuranishi neighborhood of $f_{k+1}^{-1}(p,q)$ in ${}_l\overline{\mathcal{R}}_{k+1}^1(x,L,\ring{\beta};P)$.
\end{proof}

\begin{remark}\label{remark:invariance}
Theorem \ref{theorem:moduli} above is actually parallel to \cite{ki2}, Theorem 7.20. A remarkable difference is that the invariance property (vii) of the Kuranishi structures for certain moduli spaces under cyclic permutations of boundary marked points is not required there. However, this is crucial for our purposes since we are working in the $S^1$-equivariant case, where both the $S^1$-equivariant differential and the chain level string bracket involve cyclic permutations of the marked points on the loops (or equivalently, cyclic permutations of the representatives of the smooth paths $c_0,\cdots,c_k\in\Pi_1L$ whose end points are evaluations of $z_0,\cdots,z_k$). Note also that the moduli space $\overline{\mathcal{R}}_{k+1,\tau_i}^1(x,L,\ring{\beta};P)$ does not admit a Kuranishi structure which is invariant under the cyclic permutations of the boundary marked points. Instead, by our choices of Floer data, the group $\mathbb{Z}_{k+1}$ of cyclic permutations acts transitively on the set of admissible K-spaces $\left\{\overline{\mathcal{R}}_{k+1,\tau_i}^1(x,L,\ring{\beta};P)\right\}_{i=0}^k$.
\end{remark}

From now on, we shall fix Kuranishi structures on the moduli spaces (\ref{eq:m0})---(\ref{eq:m5}) satisfying (ii)-(vii) of Theorem \ref{theorem:moduli}.

We describe explicitly the Kuranishi charts of the compactified Cohen-Ganatra moduli space ${}_l\overline{\mathcal{R}}_{k+1}^1(x,L,\ring{\beta};P)$. Let
\begin{equation}
{}_l\overline{\mathcal{K}}_{k+1}^1(x,L,\ring{\beta};P)
\end{equation}
be the space of pairs (\ref{eq:pair}), where $(S;z_0,\cdots,z_k,p_1,\cdots,p_l;\ell)\in{}_l\overline{\mathcal{R}}_{k+1}^1$, but now $u:S\rightarrow M$ is a $C^\infty$-map satisfying $u(\partial S)\subset L$, $[u]=\ring{\beta}\in\pi_2(M,x,L)$, $\bar{\partial} u=0$ (with respect to $J_M$) in a neighborhood of $\partial S$, and
\begin{equation}\label{eq:Floer}
\partial_su+J_t(\partial_tu-X_{H_t})=0
\end{equation}
on the (positive) cylindrical end, if there is no cylinders breaking off at $\zeta$. When cylinders bubble off at $\zeta$, $u$ is required to satisfy (\ref{eq:Floer}) on the cylindrical end of the main component, and solves the corresponding Floer equations of the form (\ref{eq:F-eq}) on the cylinder components (which may carry some of the auxiliary marked points $p_1,\cdots,p_l$). Moreover, we have the space
\begin{equation}
\overline{\mathcal{K}}_{k+1}(L,\beta;P)
\end{equation}
introduced in \cite{ki2}, which parametrizes pairs the (\ref{eq:pair1}), but now $u:(D,\partial D)\rightarrow(M,L)$ is a $C^\infty$-map satisfying $\bar{\partial}u=0$ on a small neighborhood of $\partial D$ and $[u]=\beta\in\pi_2(M,L)$. Since we can take the obstruction bundles $\mathcal{E}$ in both cases so that every section of it is supported away from $\partial S$ (resp. $\partial D$) and the cylindrical end near $\zeta$ (in fact, the punctured disc of radius $\frac{1}{2}$), the following lemma is straightforward.

\begin{lemma}\label{lemma:chart}
Let $p\in{}_l\overline{\mathcal{R}}_{k+1}^1(x,L,\ring{\beta};P)$ and $\mathcal{U}_p=(U_p,\mathcal{E}_p,s_p,\psi_p)$ be a Kuranishi chart at $p$. Let $(T,\ring{B},v_0)\in\mathcal{G}(k+1,\ring{\beta})$, $0\leq l_1\leq l$, $l-l_1=\sum_{i=1}^{r_2}j_i$, where $j_i\in\mathbb{Z}_{\geq0}$ for each $1\leq i\leq r_2$, such that
\begin{equation}\label{eq:stratum}
\begin{split}
p\in P&\times\left(\prod_{e\in C_{1,\mathit{int}}(T)}L\right)\textrm{ }{{}_\Delta\times\mathit{ev}_\mathit{int}}\textrm{ }\left(\prod_{v\in C_{0,\mathit{int}}(T)\setminus\{v_0\}}\mathcal{R}_{k_v+1}\left(L,\ring{B}(v)\right)\right. \\
&\left.\times\prod_{i=1}^{r_2}{}_{j_i}\mathcal{M}(y_{j_i},x)\times{}_{l_1}\mathcal{R}_{k_{v_0}+1}^1\left(y_{j_i},L,\ring{B}(v_0)\right)\right)
\end{split}
\end{equation}
Then set theoretically $U_p$ can be embedded into
\begin{equation}
\begin{split}
\bigsqcup_{(T',B',v_0)}P&\times\left(\prod_{e\in C_{1,\mathit{int}}(T')}L\right)\textrm{ }{{}_\Delta\times_{\mathit{ev}_\mathit{int}}}\textrm{ }\left(\prod_{v\in C_{0,\mathit{int}}(T')\setminus\{v_0\}}\overline{\mathcal{K}}_{k_v+1}\left(L,B'(v)\right)\right. \\
&\left.\times{}_l\overline{\mathcal{K}}^1_{k_{v_0}+1}\left(x,L,\ring{B}'(v_0)\right)\right),
\end{split}
\end{equation}
where $(T',\ring{B}',v_0)$ runs over all reductions of $(T,\ring{B},v_0)$.
\end{lemma}

Let $\mathcal{C}$ be the collection of all the moduli spaces appeared in Theorem \ref{theorem:moduli}, except for those of the form ${}_{l-1}\overline{\mathcal{R}}_{k+1,\tau_i}^1(x,L,\ring{\bar{a}};P)$, where $\ring{\bar{a}}\in\pi_2(M,x,L)$ and $a=\partial\ring{\bar{a}}\in H_1(L;\mathbb{Z})$. To achieve the compatibility conditions (iv)-(vi) in Theorem \ref{theorem:moduli}, we need a total order on $\mathcal{C}$ with the following property: if a moduli space $X\in\mathcal{C}$ is part of the normalized boundary $\partial Y$ of another moduli space $Y\in\mathcal{C}$, then $X<Y$. Concretely, we choose our total order on $\mathcal{C}$ as follows:
\begin{itemize}
	\item[(i)] $X<Y$ if the $P$-part of $X$ is $\{m\}$ and the $P$-part of $Y$ is $[m',m'+1]$, for any $m,m'\in\mathbb{Z}_{\geq0}$.
	\item[(ii)] $X<Y$ if the $P$-part of $X$ is $\{m\}$ and the $P$-part of $Y$ is $\{m+1\}$.
	\item[(iii)] $X<Y$ if the $P$-part of $X$ is $[m,m+1]$ and the $P$-part of $Y$ is $[m+1,m+2]$.
	\item[(iv)] For any $k_1,k_2,k_3,k_4,k_5,l_1,l_2,l_3,l_4\in\mathbb{Z}_{\geq0}$, $a_1,a_2,a_3,a_4,a_5\in H_1(L;\mathbb{Z})$, $1\leq j\leq l_2-1$ and 1-periodic orbits $x_1,x_2,x_3,x_4,x_5$ of $X_{H_t}$,
	\begin{equation}
	\begin{split}
	{}_{l_1}\overline{\mathcal{M}}(x_1,x_2;\{m\})&<\overline{\mathcal{R}}_{k_1+1}(L,\bar{a}_1;\{m\})<\overline{\mathcal{R}}_{k_2+2,\vartheta}(L,\bar{a}_2;\{m\}) \\
	&<{}_{l_2}^{j,j+1}\overline{\mathcal{R}}_{k_3+1}^1(x_3,L,\ring{\bar{a}}_3;\{m\})<{}_{l_3}\overline{\mathcal{R}}_{k_4+1}^{S^1}(x_4,L,\ring{\bar{a}}_4;\{m\}) \\
	&<{}_{l_4}\overline{\mathcal{R}}_{k_5+1}^1(x_5,L,\ring{\bar{a}}_5;\{m\})
	\end{split}
	\end{equation}
    \item[(v)] If $k_1,k_2,l_1,l_2\in\mathbb{Z}_{\geq0}$, $1\leq j_1\leq l_1-1$, $1\leq j_2\leq l_2-1$ and $a_1,a_2\in H_1(L;\mathbb{Z})$ satisfy $\theta_M(a_1)+(k_1-1)\varepsilon<\theta_M(a_2)+(k_2-1)\varepsilon$, then
    \begin{equation}
    \overline{\mathcal{R}}_{k_1+1}(L,\bar{a}_1;\{m\})<\overline{\mathcal{R}}_{k_2+1}(L,\bar{a}_2;\{m\}),
    \end{equation}
    \begin{equation}
    \overline{\mathcal{R}}_{k_1+2,\vartheta}(L,\bar{a}_1;\{m\})<\overline{\mathcal{R}}_{k_2+2,\vartheta}(L,\bar{a}_2;\{m\}),
    \end{equation}
    \begin{equation}
    {}_{l_1}^{j_1,j_1+1}\overline{\mathcal{R}}_{k_1+1}^1(x,L,\ring{\bar{a}}_1;\{m\})<{}_{l_2}^{j_2,j_2+1}\overline{\mathcal{R}}_{k_2+1}^1(x,L,\ring{\bar{a}}_2;\{m\}),
    \end{equation}
    \begin{equation}
    {}_{l_1}\overline{\mathcal{R}}_{k_1+1}^{S^1}(x,L,\ring{\bar{a}}_1;\{m\})<{}_{l_2}\overline{\mathcal{R}}_{k_2+1}^{S^1}(x,L,\ring{\bar{a}}_2;\{m\}),
    \end{equation}
    \begin{equation}
    {}_{l_1}\overline{\mathcal{R}}_{k_1+1}^1(x,L,\ring{\bar{a}}_1;\{m\})<{}_{l_2}\overline{\mathcal{R}}_{k_2+1}^1(x,L,\ring{\bar{a}}_2;\{m\}).
    \end{equation}
    \item[(vi)] Let $k,l_1,l_2\in\mathbb{Z}_{\geq0}$, $1\leq j_1\leq l_1-1$, $1\leq j_2\leq l_2-1$, $a\in H_1(L;\mathbb{Z})$ and $x_1$, $x_2$, $x_3$, $x_4$ be 1-periodic orbits of $X_{H_t}$. If $A_{H_t}(x_1)<A_{H_t}(x_2)$, then
    \begin{equation}
    {}_{l_1}^{j_1,j_1+1}\overline{\mathcal{R}}_{k+1}^1(x_1,L,\ring{\bar{a}};\{m\})<{}_{l_2}^{j_2,j_2+1}\overline{\mathcal{R}}_{k+1}^1(x_2,L,\ring{\bar{a}};\{m\}),
    \end{equation}
    \begin{equation}
    {}_{l_1}\overline{\mathcal{R}}_{k+1}^{S^1}(x_1,L,\ring{\bar{a}};\{m\})<{}_{l_2}\overline{\mathcal{R}}_{k+1}^{S^1}(x_2,L,\ring{\bar{a}};\{m\}),
    \end{equation}
    \begin{equation}
    {}_{l_1}\overline{\mathcal{R}}_{k+1}^1(x_1,L,\ring{\bar{a}};\{m\})<{}_{l_2}\overline{\mathcal{R}}_{k+1}^1(x_2,L,\ring{\bar{a}};\{m\}).
    \end{equation}
    If $A_{H_t}(x_1)-A_{H_t}(x_2)<A_{H_t}(x_3)-A_{H_t}(x_4)$, then
    \begin{equation}
    {}_{l_1}\overline{\mathcal{M}}(x_1,x_2)<{}_{l_2}\overline{\mathcal{M}}(x_3,x_4).
    \end{equation}
    \item[(vii)] For $c\in\mathbb{R}$, we choose arbitrary total orders on sets
    \begin{equation}
    \left\{{}_l\overline{\mathcal{M}}(x_1,x_2;\{m\})|A_{H_t}(x_1)-A_{H_t}(x_2)=c\right\},
    \end{equation}
    \begin{equation}
    \left\{\overline{\mathcal{R}}_{k+1}(L,\bar{a};\{m\})|\theta_M(a)+(k-1)\varepsilon=c\right\},
    \end{equation}
    \begin{equation}
    \left\{\overline{\mathcal{R}}_{k+2,\vartheta}(L,\bar{a};\{m\})|\theta_M(a)+(k-1)\varepsilon=c\right\},
    \end{equation}
    \begin{equation}
    \left\{{}_l^{j,j+1}\overline{\mathcal{R}}_{k+1}^1(x,L,\ring{\bar{a}};\{m\})|\theta_M(a)+(k-1)\varepsilon=c\right\},
    \end{equation}
    \begin{equation}
    \left\{{}_{l-1}\overline{\mathcal{R}}_{k+1}^{S^1}(x,L,\ring{\bar{a}};\{m\})|\theta_M(a)+(k-1)\varepsilon=c\right\},
    \end{equation}
    \begin{equation}
    \left\{{}_l\overline{\mathcal{R}}_{k+1}^1(x,L,\ring{\bar{a}};\{m\})|\theta_M(a)+(k-1)\varepsilon=c\right\}.
    \end{equation}
    \item[(viii)] For each $m\in\mathbb{Z}_{\geq0}$, we take a total order on the moduli spaces with $P=[m,m+1]$ in a similar manner as (iv)---(vii) above.
\end{itemize}

We haven't dealt with the moduli spaces ${}_{l-1}\overline{\mathcal{R}}_{k+1,\tau_i}^1(x,L,\ring{\bar{a}};P)$ when defining the total ordering above. In this case, we have a (codimension 0) embedding of admissible K-spaces
\begin{equation}
{}_{l-1}\overline{\mathcal{R}}_{k+1,\tau_i}^1(x,L,\ring{\bar{a}};P)\hookrightarrow{}_{l-1}\overline{\mathcal{R}}_{k+1}^{S^1}(x,L,\ring{\bar{a}};P)
\end{equation}
induced from the auxiliary-rescaling map (\ref{eq:aux-res}). See the proof of Lemma \ref{lemma:ana} below for details. Thus the compatibility at corners for ${}_{l-1}\overline{\mathcal{R}}_{k+1,\tau_i}^1(x,L,\ring{\bar{a}};P)$ follows essentially from the compatibility at corners for ${}_{l-1}\overline{\mathcal{R}}_{k+1}^{S^1}(x,L,\ring{\bar{a}};P)$, which is guaranteed by the existence of the total order on $\mathcal{C}$ specified above.

\subsection{Strongly smooth evaluation maps}\label{section:smooth}

In this subsection, we show the existence of strongly smooth maps from the moduli spaces $\overline{\mathcal{R}}_{k+2,\vartheta}(L,\beta)$, ${}_l\overline{\mathcal{R}}_{k+1}^1(x,L,\ring{\beta})$, ${}_{l-1}\overline{\mathcal{R}}_{k+1}^{S^1}(x,L,\ring{\beta})$, ${}_{l-1}\overline{\mathcal{R}}_{k+1,\tau_i}^1(x,L,\ring{\beta})$ and ${}_l^{j,j+1}\overline{\mathcal{R}}_{k+1}^1(x,L,\ring{\beta})$ to the spaces $\mathcal{L}_{k+2}$ and $\mathcal{L}_{k+1}$ introduced in Section \ref{section:de Rham}, such that natural compatibility conditions are satisfied. 

We start by observing that for $\beta\in\pi_2(M,L)$ and $(T,B)\in\mathcal{G}(k+1,\beta)$, there is a smooth evaluation map
\begin{equation}
\mathit{ev}_\mathit{int}^P:\prod_{v\in C_{0,\mathit{int}}(T)}P\times\mathcal{L}_{k_v+1}(\partial B(v))\rightarrow\prod_{e\in C_{1,\mathit{int}}(T)}(P\times L)^2
\end{equation}
defined in the same way as (\ref{eq:ev-int}). Using the concatenation map (\ref{eq:con}), which is also smooth, we obtain a smooth map
\begin{equation}\label{eq:ev-con}
\left(\prod_{e\in C_{1,\mathit{int}}(T)}P\times L\right)\textrm{ }{{}_\Delta\times_{\mathit{ev}_\mathit{int}^P}}\textrm{ }\left(\prod_{v\in C_{0,\mathit{int}}(T)}P\times\mathcal{L}_{k_v+1}(\partial B(v))\right)\rightarrow P\times\mathcal{L}_{k+1}(\partial\beta).
\end{equation}

As we have mentioned in Remark \ref{remark:model}, with Wang's model of $\mathcal{L}_{k+1}$ in place of the space of Moore loops, there is no essential difficulties in defining strongly smooth maps from these moduli spaces to $\mathcal{L}_{k+1}$, therefore the proof of the following proposition is greatly simplified. Compare with \cite{ki2}, Proposition 7.26, where Irie constructed strongly continuous map from these moduli spaces to the space of continuous Moore loops.

\begin{proposition}\label{proposition:smev}
For $k,m\in\mathbb{Z}_{\geq0}$, and $P=\{m\}$ or $[m,m+1]$, there are strongly smooth maps
\begin{equation}\label{eq:sc1}
\mathit{Ev}^\mathcal{R}:\overline{\mathcal{R}}_{k+1}(L,\beta;P)\rightarrow P\times\mathcal{L}_{k+1}(\partial\beta),\textrm{ where }\theta_M(\partial\beta)<(m+1-k)\varepsilon,
\end{equation}
\begin{equation}\label{eq:sc2}
\mathit{Ev}^{\mathcal{R}}_\vartheta:\overline{\mathcal{R}}_{k+2,\vartheta}(L,\beta;P)\rightarrow P\times\mathcal{L}_{k+1}(\partial\beta),\textrm{ where }\theta_M(\partial\beta)<(m+1-k)\varepsilon,
\end{equation}
\begin{equation}\label{eq:sc3}
{}_l\mathit{Ev}^\mathcal{R}:{}_l\overline{\mathcal{R}}_{k+1}^1(x,L,\ring{\beta};P)\rightarrow P\times\mathcal{L}_{k+1}(\partial\ring{\beta}),\textrm{ where }\theta_M(\partial\ring{\beta})<(m-k-U)\varepsilon,
\end{equation}
\begin{equation}\label{eq:sc4}
{}_{l-1}\mathit{Ev}^{S^1}:{}_{l-1}\overline{\mathcal{R}}_{k+1}^{S^1}(x,L,\ring{\beta};P)\rightarrow P\times\mathcal{L}_{k+1}(\partial\ring{\beta}),\textrm{ where }\theta_M(\partial\ring{\beta})<(m-k-U)\varepsilon,
\end{equation}
\begin{equation}\label{eq:sc5}
{}_l^{j,j+1}\mathit{Ev}^\mathcal{R}:{}_l^{j,j+1}\overline{\mathcal{R}}_{k+1}^1(x,L,\ring{\beta};P)\rightarrow P\times\mathcal{L}_{k+1}(\partial\ring{\beta}),\textrm{ where }\theta_M(\partial\ring{\beta})<(m-k-U)\varepsilon,
\end{equation}
\begin{equation}\label{eq:sc6}
{}_{l-1}\mathit{Ev}_i:{}_{l-1}\overline{\mathcal{R}}_{k+1,\tau_i}^1(x,L,\ring{\beta};P)\rightarrow P\times\mathcal{L}_{k+1}(\partial\ring{\beta}),\textrm{ where }\theta_M(\partial\ring{\beta})<(m-k-U)\varepsilon,
\end{equation}
such that the following diagram commutes for every $(T,B)\in\mathcal{G}(k+2,\beta)$:
\begin{equation}
    \begin{tikzcd}[font=\small]
    &(\prod_eP\times L){{}_\Delta\times_{\mathit{ev}_\mathit{int}^P}}\left(\prod_{v\neq v_0}\mathcal{R}_{k_v+1}(B(v);P)\times\mathcal{R}_{k_{v_0}+1,\vartheta}(B(v_0);P)\right) \arrow[r] \arrow[d] &\overline{\mathcal{R}}_{k+2,\vartheta}(\beta;P) \arrow[d] \\
    &(\prod_eP\times L){{}_\Delta\times_{\mathit{ev}_\mathit{int}^P}}\left(\prod_vP\times\mathcal{L}_{k_v+1}(\partial B(v))\right) \arrow[r,"(\ref{eq:ev-con})"] &P\times\mathcal{L}_{k+2}(\partial\beta)
    \end{tikzcd}
\end{equation}
where the first horizontal map is defined from (\ref{eq:corner1}) by setting $d=0$, and the vertical maps are given by the $\mathit{Ev}^\mathcal{R}$ and $\mathit{Ev}_\vartheta^\mathcal{R}$ above; and the diagram
\begin{equation}\label{eq:cd2}
	\begin{tikzcd}[font=\small]
	&(\prod_eP\times L){{}_\Delta\times_{\mathit{ev}_\mathit{int}^P}}\left(\prod_{v\neq v_0}\mathcal{R}_{k_v+1}\left(\ring{B}(v);P\right)\times{}_l\mathcal{R}^1_{k_{v_0}+1}\left(\ring{B}(v_0);P\right)\right) \arrow[r] \arrow[d] &{}_l\overline{\mathcal{R}}_{k+1}^1(\ring{\beta};P) \arrow[d] \\
	&(\prod_eP\times L){{}_\Delta\times_{\mathit{ev}_\mathit{int}^P}}\left(\prod_vP\times\mathcal{L}_{k_v+1}\left(\partial \ring{B}(v)\right)\right) \arrow[r,"(\ref{eq:ev-con})"] &P\times\mathcal{L}_{k+1}(\partial\ring{\beta})
	\end{tikzcd}
\end{equation}
commutes for every $(T,\ring{B},v_0)\in\mathcal{G}(k+1,\ring{\beta})$, where the first horizontal map is defined from (\ref{eq:corner2}) by setting $d=0$, and the vertical maps are given by ${}_l\mathit{Ev}^\mathcal{R}$. 

In the above, we have abbreviated the notations of the moduli spaces so that the boundary conditions specified by the Lagrangian submanifold $L$ and the asymptotic conditions specified by a Hamiltonian orbit $x$ are omitted. In the commutative diagram (\ref{eq:cd2}) above, we can also include cylinder bubbles in $\prod_{i=1}^{r_2}{}_{j_i}\mathcal{M}(y_{j_i},x)$, just as in (\ref{eq:stratum}). This will not affect the compatibility. There are similar compatibility diagrams for the strongly smooth maps $\mathit{Ev}^\mathcal{R}$, ${}_{l-1}\mathit{Ev}^{S^1}$, ${}_l^{j,j+1}\mathit{Ev}^\mathcal{R}$ and ${}_{l-1}\mathit{Ev}_i$, which we will not write down since they are not explicitly used in our argument in Section \ref{section:proof}.
\end{proposition}
\begin{proof}
Once the strongly smoothmaps (\ref{eq:sc1})-(\ref{eq:sc6}) are defined, the commutativity of the corresponding diagrams are straightforward. For completeness, we will sketch the definition of the strongly smooth map ${}_l\mathit{Ev}^\mathcal{R}:{}_l\overline{\mathcal{R}}_{k+1}^1(x,L,\ring{\beta};P)\rightarrow P\times\mathcal{L}_{k+1}(\partial\ring{\beta})$. For each point $p\in{}_l\overline{\mathcal{R}}_{k+1}(x,L,\ring{\beta};P)$, let $\mathcal{U}_p=(U_p,\mathcal{E}_p,s_p,\psi_p)$ be a Kuranishi chart at $p$, we need to define a smooth map ${}_l\mathit{Ev}^\mathcal{R}_p:U_p\rightarrow P\times\mathcal{L}_{k+1}$, which is compatible with coordinate changes. By Lemma \ref{lemma:chart}, $x\in U_p$ can be identified with an element of
\begin{equation}
P\times\left(\prod_{e\in C_{1,\mathit{int}}(T')}L\right)\textrm{ }{{}_\Delta\times_{\mathit{ev}_\mathit{int}}}\textrm{ }\left(\prod_{v\in C_{0,\mathit{int}}(T')\setminus\{v_0\}}\overline{\mathcal{K}}_{k_v+1}\left(L,\ring{B}'(v)\right)\times{}_l\overline{\mathcal{K}}_{k_{v_0}+1}^1\left(x,L,\ring{B}'(v_0)\right)\right),
\end{equation}
with $(T',\ring{B}',v_0)$ a reduction of $(T,\ring{B},v_0)\in\mathcal{G}(k+1,\ring{\beta})$. Using this description, we can define smooth maps $\mathit{ev}_{v_0}:{}_l\overline{\mathcal{K}}_{k_{v_0}+1}^1\left(x,L,\ring{B}'(v_0)\right)\rightarrow\mathcal{L}_{k_{v_0}+1}\left(\partial\ring{B}'(v_0)\right)$ and $\mathit{ev}_v:\overline{\mathcal{K}}_{k_v+1}\left(L,\ring{B}'(v)\right)\rightarrow\mathcal{L}_{k_{v_0}+1}\left(\partial\ring{B}'(v_0)\right)$ for each $v\in C_{0,\mathit{int}}(T')\setminus\{v_0\}$ in a straightforward way by evaluating at the boundary marked points. The map ${}_l\mathit{Ev}_p^\mathcal{R}$ is defined by taking concatenations of the images of $\mathit{ev}_v$ for different $v\in C_{0,\mathit{int}}(T')$, therefore it is also smooth. The compatibility of the family of maps $({}_l\mathit{Ev}_p^\mathcal{R})_{p\in{}_l\overline{\mathcal{R}}^1_{k+1}(x,L,\ring{\beta};P)}$ with coordinate changes follows directly from the definition.
\end{proof}

\subsection{Proof of the main theorem}\label{section:proof}

We prove Theorem \ref{theorem:approximate} in this subsection, therefore complete the proof of Theorem \ref{theorem:main}, the main result of this paper. 

From now on, let $M$ be a Liouville manifold with $c_1(M)=0$ admitting a cyclic dilation $\tilde{b}\in\mathit{SH}_{S^1}^1(M)$, whose cochain level representative is $\tilde{\beta}=\sum_{l=0}^\infty\beta_l\otimes u^{-l}\in\mathit{SC}_{S^1}^1(M)$, where $\beta_l\in\mathit{SC}^{2l+1}(M)$. Note that there are only finitely many $\beta_l$ that are non-vanishing. By definition of a cyclic dilation, we have
\begin{equation}\label{eq:cdilation}
B_c(\tilde{\beta})=e_M-\partial\beta_{-1}
\end{equation}
for some $\beta_{-1}\in\mathit{SC}^{-1}(M)$, where $B_c$ is the cochain level marking map
\begin{equation}\label{eq:marking}
B_c:=\sum_{l=0}^\infty\delta_{l+1}u^l:\mathit{SC}_{S^1}^\ast(M)\rightarrow\mathit{SC}^{\ast-1}(M),
\end{equation}
and $e_M\in\mathit{SC}^0(M)$ is the cochain level representative of the identity. We want to fix the cochain $x$ in the definition of Cohen-Ganatra moduli spaces so that it comes from the cochains $\{\beta_l\}_{l\in\mathbb{Z}_{\geq0}}$ in the expression of the cyclic dilation $\tilde{\beta}$, together with $\beta_{-1}$. For a closed Lagrangian submanifold $L\subset M$ which is oriented and $\mathit{Spin}$ relative to the gerbe $\alpha$, fix classes $a_{l-1}\in H_1(L;\mathbb{Z})$ for each $l\in\mathbb{Z}_{\geq0}$ such that $a_{l-1}=0$ if $\beta_{l-1}=0$, and $\sum_{l=0}^\infty a_{l-1}=a$. For $\ring{\bar{a}}_{l-1}\in\pi_2(M,\beta_{l-1},L)$ (when $\beta_l$ consists of a linear combination of Hamiltonian orbits, $\ring{\bar{a}}_{l-1}$ is a linear combination of the corresponding homotopy classes) with $\partial\ring{\bar{a}}_{l-1}=a_{l-1}$, and $P\in\left\{\{m\},[m,m+1]\right\}$ for some $m\in\mathbb{Z}_{\geq0}$, consider the moduli spaces
\begin{equation}
{}_l\mathcal{R}_{k+1}^1(\beta_{l-1},L,\ring{\bar{a}}_{l-1};P),\textrm{ where }l\in\mathbb{Z}_{\geq0}.
\end{equation}
Set $\ring{\bar{a}}=\sum_{l=0}^\infty\ring{\bar{a}}_{l-1}$. It follows that $\partial\ring{\bar{a}}=a$. 

\begin{remark}\label{remark:subcritical}
When $M$ is a subcritical Weinstein manifold, the symplectic cohomology vanishes and so is the cyclic dilation, therefore we can pick $\tilde{\beta}\equiv0$. Thus we only need to consider the moduli space $\mathcal{R}_{k+1}^1(\beta_{-1},L,\ring{\bar{a}};P)$, where $\beta_{-1}$ is the primitive of the identity $e_M$. This moduli space is closely related to the moduli space $\mathcal{N}_{k+1}^{\geq0}(L,\bar{a};P)$ studied by Fukaya and Irie (see \cite{ki2}, Section 7.2, the definition of the unmarked case has been recalled in Section \ref{section:Lag}). In fact, one can show that there is an identification (as compact K-spaces) between the boundary component $\overline{\mathcal{R}}_{k+1}^1(e_M,L,\bar{a};P)\subset\partial\overline{\mathcal{R}}_{k+1}^1(\beta_{-1},L,\ring{\bar{a}};P)$ arising from a cylinder breaking off at the puncture $\zeta$, and the boundary component $\overline{\mathcal{N}}(L,\bar{a};P)\subset\partial\overline{\mathcal{N}}_{k+1}^{\geq0}(L,\bar{a};P)$ defined by setting the inhomogeneous term to be $0$.
\end{remark}

Recall that we have the following data for every $k,m,l\in\mathbb{Z}_{\geq0}$ and $P\in\{\{m\},[m,m+1]\}$.
\begin{itemize}
	\item[(i)] Compact admissible K-spaces (when $l=0$, the moduli spaces (\ref{eq:S1}), (\ref{eq:tau}) and (\ref{eq:jj}) are empty; when $l=1$, the moduli space (\ref{eq:jj}) is empty)
	\begin{equation}\label{eq:d}
	\overline{\mathcal{R}}_{k+1}(L,\bar{a};P),\textrm{ }\theta_M(a)<(m+1-k)\varepsilon,
	\end{equation}
	\begin{equation}\label{eq:theta}
    \overline{\mathcal{R}}_{k+2,\vartheta}(L,\bar{a};P),\textrm{ }\theta_M(a)<(m+1-k)\varepsilon,
	\end{equation}
    \begin{equation}\label{eq:CG}
    {}_l\overline{\mathcal{R}}_{k+1}^1(\beta_{l-1},L,\ring{\bar{a}}_{l-1};P),\textrm{ }\theta_M(a)<(m-k-U)\varepsilon,
    \end{equation}
    \begin{equation}\label{eq:S1}
    {}_{l-1}\overline{\mathcal{R}}^{S^1}_{k+1}(\beta_{l-1},L,\ring{\bar{a}}_{l-1};P),\textrm{ }\theta_M(a)<(m-k-U)\varepsilon,
    \end{equation}
    \begin{equation}\label{eq:tau}
    {}_{l-1}\overline{\mathcal{R}}_{k+1,\tau_i}^1(\beta_{l-1},L,\ring{\bar{a}}_{l-1};P),\textrm{ }\theta_M(a)<(m-k-U)\varepsilon,
    \end{equation}
    \begin{equation}\label{eq:jj}
    {}_l^{j,j+1}\overline{\mathcal{R}}_{k+1}^1(\beta_{l-1},L,\ring{\bar{a}}_{l-1};P),\textrm{ }\theta_M(a)<(m-k-U)\varepsilon,
    \end{equation}
    and admissible CF-perturbations on these moduli spaces.
    
    Moreover, consider the $\mathbb{Z}_{k+1}$-action on these moduli spaces induced by cyclic permutations of the boundary marked points $z_0,\cdots,z_k$, we also require that the Kuranishi structures and CF-perturbations on (\ref{eq:d}), (\ref{eq:theta}), (\ref{eq:CG}), (\ref{eq:S1}) and (\ref{eq:jj}) to be $\mathbb{Z}_{k+1}$-invariant. For the moduli spaces (\ref{eq:tau}), we require $\mathbb{Z}_{k+1}$ acts transitively on the set $\left\{{}_{l-1}\overline{\mathcal{R}}_{k+1,\tau_i}^1(\beta_{l-1},L,\ring{\bar{a}}_{l-1};P),\widehat{\mathcal{S}}_i^\varepsilon\right\}_{i=0}^k$ of moduli spaces together with their admissible CF-perturbations.
    \item[(ii)] Admissible maps (cf. Proposition \ref{proposition:smev})
    \begin{equation}\label{eq:ev1}
    \mathit{Ev}^\mathcal{R}:\overline{\mathcal{R}}_{k+1}(L,\bar{a};P)\rightarrow P\times\mathcal{L}_{k+1}(a),
    \end{equation}
    \begin{equation}\label{eq:ev2}
    \mathit{Ev}_\vartheta^\mathcal{R}:\overline{\mathcal{R}}_{k+2,\vartheta}(L,\bar{a};P)\rightarrow P\times\mathcal{L}_{k+2}(a),
    \end{equation}
    \begin{equation}\label{eq:ev3}
    {}_l\mathit{Ev}^\mathcal{R}:{}_l\overline{\mathcal{R}}_{k+1}^1(\beta_{l-1},L,\ring{\bar{a}}_{l-1};P)\rightarrow P\times\mathcal{L}_{k+1}(a_{l-1}),
    \end{equation}
    \begin{equation}\label{eq:ev4}
    {}_{l-1}\mathit{Ev}^{S^1}: {}_{l-1}\overline{\mathcal{R}}^{S^1}_{k+1}(\beta_{l-1},L,\ring{\bar{a}}_{l-1};P)\rightarrow P\times\mathcal{L}_{k+1}(a_{l-1}),
    \end{equation}
    \begin{equation}\label{eq:ev5}
    {}_{l-1}\mathit{Ev}_i:{}_{l-1}\overline{\mathcal{R}}_{k+1,\tau_i}^1(\beta_{l-1},L,\ring{\bar{a}}_{l-1};P)\rightarrow P\times\mathcal{L}_{k+1}(a_{l-1}),
    \end{equation}
    \begin{equation}\label{eq:ev6}
    {}_l^{j,j+1}\mathit{Ev}^\mathcal{R}:{}_l^{j,j+1}\overline{\mathcal{R}}_{k+1}^1(\beta_{l-1},L,\ring{\bar{a}}_{l-1};P)\rightarrow P\times\mathcal{L}_{k+1}(a_{l-1}),
    \end{equation}
    such that their compositions with $\mathit{id}_P\times\mathit{ev}_0^\mathcal{L}$ are corner stratified strong submersions with respect to the CF-perturbations fixed in (i).
    \item[(iii)] Isomorphisms of admissible K-spaces in (\ref{eq:bd1})---(\ref{eq:bd5}). We require these isomorphisms to be compatible with the CF-perturbations in (i) and the evaluation maps in (ii).
\end{itemize}

\begin{remark}\label{remark:CF}
As the cyclic invariance of Kuranishi structures on $X$ (cf. the proof of Theorem \ref{theorem:moduli}), the existence of cyclically invariant CF-perturbations for the moduli spaces (\ref{eq:m0}), (\ref{eq:m1}), (\ref{eq:m2}), (\ref{eq:m3}) and (\ref{eq:m4}) follows essentially from Fukaya's argument in \cite{kfc}, Section 5.
\end{remark}

Applying Theorem \ref{theorem:pushforward} in Appendix \ref{section:pushforward} to the moduli spaces in (i) above, we have well-defined (relative) de Rham chains (where we have abbreviated the evaluation maps (\ref{eq:ev1})---(\ref{eq:ev6}) above as $\mathit{Ev}$)
\begin{equation}\label{eq:c1}
x_m(k):=\sum_{\theta_M(a)<(m+1-k)\varepsilon}(-1)^{n+1}\mathit{Ev}_\ast\left(\overline{\mathcal{R}}_{k+1}(L,\bar{a};\{m\})\right)\in C_{-1},
\end{equation}
\begin{equation}\label{eq:c2}
\bar{x}_m(k):=\sum_{\theta_M(a)<(m+1-k)\varepsilon}(-1)^{k+1}\mathit{Ev}_\ast\left(\overline{\mathcal{R}}_{k+1}(L,\bar{a};[m,m+1])\right)\in\overline{C}_{-1},
\end{equation}
\begin{equation}\label{eq:c3}
x_{m,0}(k+1):=\sum_{\theta_M(a)<(m+1-k)\varepsilon}(-1)^{n+1}\mathit{Ev}_\ast\left(\overline{\mathcal{R}}_{k+2,\vartheta}(L,\bar{a};\{m\})\right)\in C_{-2},
\end{equation}
\begin{equation}\label{eq:c4}
\bar{x}_{m,0}(k+1):=\sum_{\theta_M(a)<(m+1-k)\varepsilon}(-1)^k\mathit{Ev}_\ast\left(\overline{\mathcal{R}}_{k+2,\vartheta}(L,\bar{a};[m,m+1])\right)\in\overline{C}_{-2},
\end{equation}
\begin{equation}\label{eq:c5}
y_{m,0}(k):=\sum_{\theta_M(a)<(m-U-k)\varepsilon}\sum_{l=0}^\infty(-1)^{n+k+1}\mathit{Ev}_\ast\left({}_l\overline{\mathcal{R}}_{k+1}^1(\beta_{l-1},L,\ring{\bar{a}}_{l-1};\{m\})\right)\in C_2,
\end{equation}
\begin{equation}\label{eq:c6}
y_{m,1}(k+1):=\sum_{\theta_M(a)<(m-U-k-1)\varepsilon}\sum_{l=1}^\infty(-1)^{n+k+1}\mathit{Ev}_\ast\left({}_{l-1}\overline{\mathcal{R}}_{k+2}^1(\beta_{l-1},L,\ring{\bar{a}}_{l-1};\{m\})\right)\in C_0,
\end{equation}
\begin{equation}\label{eq:c7}
\bar{y}_{m,0}(k):=\sum_{\theta_M(a)<(m-U-k)\varepsilon}\sum_{l=0}^\infty\mathit{Ev}_\ast\left({}_l\overline{\mathcal{R}}_{k+1}^1(\beta_{l-1},L,\ring{\bar{a}}_{l-1};[m,m+1])\right)\in\overline{C}_2,
\end{equation}
\begin{equation}\label{eq:c8}
\bar{y}_{m,1}(k+1):=\sum_{\theta_M(a)<(m-U-k-1)\varepsilon}\sum_{l=1}^\infty\mathit{Ev}_\ast\left({}_{l-1}\overline{\mathcal{R}}_{k+2}^1(\beta_{l-1},L,\ring{\bar{a}}_{l-1};[m,m+1])\right)\in\overline{C}_0,
\end{equation}
\begin{equation}\label{eq:c9}
z_{m}(k):=\sum_{\theta_M(a)<(m-1-k)\varepsilon}(-1)^{n+k+1}\mathit{Ev}_\ast\left(\overline{\mathcal{R}}_{k+1}^1(e_M,L,\bar{a};\{m\})\right)\in C_1,
\end{equation}
\begin{equation}\label{eq:c10}
\bar{z}_{m}(k):=-\sum_{\theta_M(a)<(m-1-k)\varepsilon}\mathit{Ev}_\ast\left(\overline{\mathcal{R}}_{k+1}^1(e_M,L,\bar{a};[m,m+1])\right)\in\overline{C}_1,
\end{equation}
where the chains (\ref{eq:c1}) and (\ref{eq:c2}) were defined in \cite{ki2}, Section 7.6. In the above, the sums are taken over $a\in H_1(L;\mathbb{Z})$, which give the components of the chains (resp. relative chains) in $C_\ast(k)$ (resp. $\overline{C}_\ast(k)$). One can further take sums over all $k\in\mathbb{Z}_{\geq0}$ to obtain the full chains $x_m,\cdots,\bar{z}_m$, or restrict to a particular $a\in H_1(L;\mathbb{Z})$ to obtain the (relative) chains $x_m(a,k),\cdots,\bar{z}_m(a,k)$. Note that these chains give rise to chains in the quotient complexes $C_\ast^\mathit{nd}$ and $\overline{C}_\ast^\mathit{nd}$ under the natural projections $C_\ast\rightarrow C_\ast^\mathit{nd}$ and $\overline{C}_\ast\rightarrow\overline{C}_\ast^\mathit{nd}$. By abuse of notations, we shall denote the corresponding chains in $C_\ast^\mathit{nd}$ and $\overline{C}_\ast^\mathit{nd}$ by $x_m,\cdots,\bar{z}_m$ as well. 

Using the chains (\ref{eq:c3})---(\ref{eq:c10}) (realized as chains in $C_\ast^\mathit{nd}$ and $\overline{C}_\ast^\mathit{nd}$), we can form the $S^1$-equivariant (relative) de Rham chains
\begin{equation}\label{eq:t1}
\tilde{x}_m(k):=x_{m,0}(k)\otimes1\in C_{-2}^{S^1},
\end{equation}
\begin{equation}\label{eq:t2}
\bar{\tilde{x}}_m(k):=\bar{x}_{m,0}(k)\otimes1\in\overline{C}_{-2}^{S^1},
\end{equation}
\begin{equation}\label{eq:t3}
\tilde{y}_m(k,k+1):=y_{m,0}(k)\otimes1+y_{m,1}(k+1)\otimes u^{-1}\in C_2^{S^1},
\end{equation}
\begin{equation}\label{eq:t4}
\bar{\tilde{y}}_m(k,k+1):=\bar{y}_{m,0}(k)\otimes1+\bar{y}_{m,1}(k+1)\otimes u^{-1}\in \overline{C}_2^{S^1},
\end{equation}
\begin{equation}\label{eq:t5}
\tilde{z}_m(k):=z_m(k)\otimes1\in C_1^{S^1},
\end{equation}
\begin{equation}\label{eq:t6}
\bar{\tilde{z}}_m(k):=\bar{z}_m(k)\otimes1\in\overline{C}_1^{S^1}.
\end{equation}
In the above definitions, we have specified the $k$-components of the $x$ and $z$-type chains, but this is not the case for the $y$-type chains, where we have combined the chain $y_{m,0}$ (resp. $\bar{y}_{m,0}$) over the $k$-component and the chain $y_{m,1}$ (resp. $\bar{y}_{m,1}$) over the $(k+1)$-component. This is because we will be dealing with the $S^1$-equivariant differentials of these chains, so $\partial\left(y_{m,0}(k)\right)$ (resp. $\partial\left(\bar{y}_{m,0}(k)\right)$) and $\delta_\mathit{cyc}^\mathit{nd}\left(y_{m,1}(k+1)\right)$ (resp. $\delta_\mathit{cyc}^\mathit{nd}\left(\bar{y}_{m,1}(k+1)\right)$) both lie in the $(a,k)$-component of $C_\ast^\mathit{nd}$ (resp. $\overline{C}_\ast^\mathit{nd}$).

For later purposes, we also introduce the following auxiliary chains (as before, they can be regarded as chains in $\overline{C}_1^\mathit{nd}$):
\begin{equation}\label{eq:c11}
\bar{y}^{j,j+1}_m(k):=\sum_{\theta_M(a)<(m-k-U)\varepsilon}\sum_{l=2}^\infty\mathit{Ev}_\ast\left({}_l^{j,j+1}\overline{\mathcal{R}}_{k+1}^1(\beta_{l-1},L,\ring{\bar{a}}_{l-1};[m,m+1])\right)\in\overline{C}_1,
\end{equation}
\begin{equation}\label{eq:c12}
\bar{y}_m^{S^1}(k):=\sum_{\theta_M(a)<(m-k-U)\varepsilon}\sum_{l=1}^\infty\mathit{Ev}_\ast\left({}_{l-1}\overline{\mathcal{R}}_{k+1}^{S^1}(\beta_{l-1},L,\ring{\bar{a}}_{l-1};[m,m+1])\right)\in\overline{C}_1.
\end{equation}
\begin{equation}\label{eq:c13}
\bar{y}_{m,1}^i(k):=\sum_{\theta_M(a)<(m-k-U)\varepsilon}\sum_{l=1}^\infty\mathit{Ev}_\ast\left({}_{l-1}\overline{\mathcal{R}}_{k+1,\tau_i}^1(\beta_{l-1},L,\ring{\bar{a}}_{l-1};[m,m+1])\right)\in\overline{C}_1.
\end{equation}

\begin{remark}
In the definitions of the de Rham chains (\ref{eq:c1})---(\ref{eq:c10}), (\ref{eq:c11}), (\ref{eq:c12}), and (\ref{eq:c13}) above, we have omitted the differential form $\hat{\omega}\equiv1$ and CF-perturbations (compare with Theorem \ref{theorem:pushforward}). The superscripts and subscripts for various evaluation maps (\ref{eq:ev1})---(\ref{eq:ev6}) have also been left out for simplicity of notations. For the precise choices of the parameters of CF-perturbations, we refer the reader to \cite{ki2}, Remark 7.36.

The moduli spaces ${}_j\mathcal{M}(\beta_{l-1},y;P)$, where $y$ is an orbit of $X_{H_t}$, will also play an essential role in the argument below, but we didn't include them in the above discussions, since their transversalities can be achieved with standard inhomogeneous perturbations, therefore do not involve choices of Kuranishi structures or CF-perturbations.
\end{remark}

To prove Theorem \ref{theorem:approximate}, we start with the following observations.

\begin{lemma}\label{lemma:vartheta}
We have the following relation for (relative) de Rham chains in $\overline{C}_\ast^\mathit{nd}$:
\begin{equation}
\bar{x}_m(k)=\sum_{i=1}^{k+1}(-1)^{|\bar{x}_{m,0}|+k(i-1)}(\overline{R}_{k+1})^i_\ast\bar{x}_{m,0}(k+1)\circ_{k+2-i}\bar{e}_L,
\end{equation}
for any $k\in\mathbb{Z}_{\geq0}$.
\end{lemma}
\begin{proof}
It follows from Theorem \ref{theorem:fp} (its variation for admissible K-spaces and relative de Rham chains) and Theorem \ref{theorem:cyclic} in Appendix \ref{section:pushforward} that the diagram
\begin{equation}
\begin{tikzcd}
    \overline{\mathcal{R}}_{k+2,\vartheta}(L,\bar{a};[m,m+1]) \arrow[rrrr, "\phi_{\vartheta,i}"] \arrow[d,"\mathit{Ev}_\ast"']
	&&&& \overline{\mathcal{R}}_{k+1}(L,\bar{a};[m,m+1]) \arrow[d, "\mathit{Ev}_\ast"] \\
	\overline{C}_\ast(a,k+1) \arrow[rrrr, "(-1)^{|\bar{x}_{m,0}|+k(i-1)}\sigma_{k,k+1-i}\circ\tau_{k+1}^i"]
	&&&& \overline{C}_\ast(a,k)
\end{tikzcd}
\end{equation}
commutes for every $a\in H_1(L;\mathbb{Z})$ and $k\in\mathbb{Z}_{\geq0}$, where the map $\phi_{\vartheta,i}$ is induced from the map $(-1)^{(k+1)i}\pi_{\vartheta,i}$ on domains considered in Section \ref{section:disc}, which identifies $\mathcal{R}_{k+2,\vartheta}$ with an open sector of the moduli space $\mathcal{R}_{k+1}$. Here the operations $\tau_{k+1}$ and $\sigma_{k,k+1-i}$ are defined in Section \ref{section:BV}. It follows from Lemma \ref{lemma:erase} that the union of the images of $\phi_{\vartheta,i}$ covers $\overline{\mathcal{R}}_{k+1}(L,\bar{a};[m,m+1])$ up to codimension 1 strata, which do not affect the de Rham chain in the quotient complex $\overline{C}_\ast^\mathit{nd}$ defined via the pushforward $\mathit{Ev}_\ast$.

It remains to explain the signs before the maps $\sigma_{k,k+1-i}\circ\tau_{k+1}^i$. Recall the operators $s_1:\overline{C}_\ast(k+1)\rightarrow\overline{C}_{\ast+1}(k)$ and $\lambda_k:\overline{C}_\ast(k)\rightarrow\overline{C}_\ast(k)$ from the proof of Lemma \ref{lemma:strict}. We claim that the identity
\begin{equation}\label{eq:lam}
\begin{split}
\bar{x}_m(k)&=N_k\circ s_0\left(\bar{x}_{m,0}(k+1)\right) \\
&=\left(\sum_{i=0}^{k}\lambda_{k}^i\right)(\lambda_{k}\circ s_1\circ\lambda_{k+1}^{-1})\left(\bar{x}_{m,0}(k+1)\right)
\end{split}
\end{equation}
holds. To see this, we need to identify the effect on the de Rham chain $\bar{x}_{m,0}(k+1)$ induced by cyclically permuting the boundary marked points $z_0,\cdots,z_{k+1}$ with the algebraic operation $\lambda_{k+1}$, and identify the effect on the de Rham chain $\lambda_{k+1}^{-1}\left(\bar{x}_{m,0}(k+1)\right)$ induced by forgetting the marked point $z_1$ with the algebraic operation $s_1$. The former will be done in the proof of Theorem \ref{theorem:cyclic} in Appendix \ref{section:pushforward}, where it is shown that with our choices of Kuranishi structures and CF-perturbations on the moduli space $\overline{\overline{\mathcal{R}}}_{k+2,\vartheta}(L,\bar{a};[m,m+1])$, changing the labeling of the boundary marked points from $z_0,z_1\cdots,z_{k+1}$ to $z_{k+1},z_0,\cdots,z_k$ has the effect of applying $(-1)^{k+1}(\overline{R}_{k+1})_\ast$ to $\bar{x}_{m,0}(k+1)$, which is by definition $\lambda_{k+1}$. The latter follows from a comparison of the orientations of $\overline{\mathcal{R}}_{k+2,f_1}(L,\bar{a};[m,m+1])$ (the moduli space of pseudoholomorphic discs bounded by $L$ with $z_1\in\partial D$ marked as forgotten, cf. Section \ref{section:CG}) with the fiber product $\overline{\mathcal{R}}_{k+2,\vartheta}(L,\bar{a};[m,m+1])\textrm{ }{{}_1\times_0}\textrm{ }\overline{\mathcal{R}}_1^1(e_M,L,0;[m,m+1])$, which gives the sign $|\bar{x}_{m,0}|+1$, where $\overline{\mathcal{R}}_1^1(e_M,L,0;[m,m+1])$ is the moduli space defining the unit $\bar{e}_L$. Now the total sign $|\bar{x}_{m,0}|+k(i-1)$ is the sum of the sign $|\bar{x}_{m,0}|+1$ with a sign $k+1$ coming from the comparison between $\lambda_{k+1}^{-1}$ and $\tau_{k+1}^{-1}$, a sign $k$ from $\lambda_k$, and a sign $ki$ from $\lambda_k^i$.
\end{proof}
	
We then analyze the codimension $1$ boundary strata
\begin{equation}\label{eq:st1}
\bigsqcup_{0\leq j\leq l}{}_j\overline{\mathcal{M}}(\beta_{l-1},y_{j,l};[m,m+1])\times{}_{l-j}\overline{\mathcal{R}}_{k+1}^1(y_{j,l},L,\ring{\bar{a}}_{l-1};[m,m+1]),
\end{equation}
where the $y_{j,l}$'s are 1-periodic orbits of $X_{H_t}$,
\begin{equation}\label{eq:st2}
\bigsqcup_{1\leq j\leq l-1}{}_l^{j,j+1}\overline{\mathcal{R}}_{k+1}^1(\beta_{l-1},L,\ring{\bar{a}}_{l-1};[m,m+1]),
\end{equation}
and
\begin{equation}\label{eq:st3}
{}_{l-1}\overline{\mathcal{R}}_{k+1}^{S^1}(\beta_{l-1},L,\ring{\bar{a}}_{l-1};[m,m+1])
\end{equation}
of the admissible K-space ${}_l\overline{\mathcal{R}}_{k+1}^1(\beta_{l-1},L,\ring{\bar{a}}_{l-1};[m,m+1])$ (cf. (\ref{eq:bd5})), and identify their contributions in terms of chain level string topology operations.

\begin{lemma}\label{lemma:ana}
The following identities hold in $\overline{C}_\ast^\mathit{nd}$ for any $k,l\in\mathbb{Z}_{\geq0}$:
\begin{equation}\label{eq:id1}
\sum_{l=0}^\infty\sum_{j=0}^ly_{j,l}=\sum_{l=1}^\infty\sum_{j=0}^l\delta_j(\beta_{l-1})=e_M,
\end{equation}
\begin{equation}\label{eq:id2}
\bar{y}_m^{j,j+1}(k)=0,
\end{equation}
\begin{equation}\label{eq:id3}
\bar{y}_m^{S^1}(k)=\delta_\mathit{cyc}^\mathit{nd}\left(\bar{y}_{m,1}(k+1)\right),
\end{equation}
where the $\delta_j$'s are structure maps on the $S^1$-complex $\mathit{SC}^\ast(M)$, with $\delta_0=\partial$, and $\delta_\mathit{cyc}^\mathit{nd}$ is the BV operator on the strict $S^1$-complex $\overline{C}_\ast^\mathit{nd}$ defined by (\ref{eq:dnd}).
\end{lemma}
\begin{proof}
First, observe that the moduli spaces ${}_j\overline{\mathcal{M}}(\beta_{l-1},y_{j,l})$ in (\ref{eq:st1}) are rigid due to dimension reasons. Since counting rigid elements of ${}_j\overline{\mathcal{M}}(\beta_{l-1},y_{j,l})$ gives the image of $\beta_{l-1}$ under the operator $\delta_j$, it follows that $\sum_{l=0}^\infty\sum_{j=0}^ly_{j,l}=\sum_{l=0}^\infty\sum_{j=0}^l\delta_j(\beta_{l-1})$. The $S^1$-equivariant cocycle condition satisfied by $\tilde{\beta}=\sum_{l=0}^\infty\beta_l\otimes u^{-l}$ implies that
\begin{equation}
\sum_{l=1}^\infty\sum_{j=1}^{l-1}\delta_j(\beta_{l-1})=0,
\end{equation}
so
\begin{equation}
\sum_{l=0}^\infty\sum_{j=0}^ly_{j,l}=\sum_{l=0}^\infty y_{l,l}=\sum_{l=0}^\infty\delta_l(\beta_{l-1}).
\end{equation}
By (\ref{eq:marking}), the right-hand side gives $B_c(\tilde{\beta})+\partial\beta_{-1}$, which is equal to the identity $e_M$ by (\ref{eq:cdilation}). This proves (\ref{eq:id1}).

For the second identity, observe that the forgetful map (\ref{eq:fj}) induces an isomorphism
\begin{equation}\label{eq:split}
{}_l^{j,j+1}\overline{\mathcal{R}}_{k+1}^1(\beta_{l-1},L,\ring{\bar{a}}_{l-1};[m,m+1])\cong{}_{l-1}\overline{\mathcal{R}}_{k+1}^1(\beta_{l-1},L,\ring{\bar{a}}_{l-1};[m,m+1])\times S^1
\end{equation}
as admissible K-spaces. Since the admissible map ${}_l^{j,j+1}\mathit{Ev}^\mathcal{R}$ factors as the projection onto the first factor followed by ${}_{l-1}\mathit{Ev}^\mathcal{R}$, and the virtual fundamental chain of the moduli space ${}_l^{j,j+1}\overline{\mathcal{R}}_{k+1}^1(\beta_{l-1},L,\ring{\bar{a}}_{l-1};[m,m+1])$ is product-like, with non-trivial degree in the $S^1$-factor, we conclude that
\begin{equation}
\mathit{Ev}_\ast\left({}_l^{j,j+1}\overline{\mathcal{R}}_{k+1}^1(\beta_{l-1},L,\ring{\bar{a}}_{l-1};[m,m+1])\right)=0
\end{equation}
in $\overline{C}_1^\mathit{nd}$. This proves (\ref{eq:id2}).
	
The proof of the third identity is inspired by that of \cite{sg1}, Proposition 5.9. We use the decomposition (up to some codimension $1$ strata) of the moduli space ${}_{l-1}\mathcal{R}_{k+1}^{S^1}$ into sectors ${}_{l-1}\mathcal{R}_{k+1,\tau_i}^1$ considered in Section \ref{section:CG}, which induces a decomposition of the moduli space ${}_{l-1}\overline{\mathcal{R}}_{k+1}^{S^1}(\beta_{l-1},L,\ring{\bar{a}}_{l-1};[m,m+1])$ into the corresponding sectors ${}_{l-1}\overline{\mathcal{R}}_{k+1,\tau_i}^1(\beta_{l-1},L,\ring{\bar{a}}_{l-1};[m,m+1])$. More precisely, there is an embedding of abstract moduli spaces
\begin{equation}
\bigsqcup_{i=0}^k{}_{l-1}\mathcal{R}_{k+1,\tau_i}^1\xrightarrow{\bigsqcup_i\pi_f^i}\bigsqcup_{i=0}^k{}_{l-1}\mathcal{R}_{k+1}^{S_{i,i+1}^1}\hookrightarrow{}_{l-1}\mathcal{R}_{k+1}^{S^1},
\end{equation}
which is compatible with our choices of Floer data, and covers all but a codimension $1$ locus in the target (compare with Lemma \ref{lemma:erase}). Since codimension $1$ strata do not affect the de Rham chain defined by the pushforward $\mathit{Ev}_\ast$ in the quotient complex $\overline{C}_\ast^\mathit{nd}$, the chain defined by ${}_{l-1}\overline{\mathcal{R}}_{k+1}^{S^1}(\beta_{l-1},L,\ring{\bar{a}}_{l-1};[m,m+1])$ is a sum of the chains defined by the sectors
\begin{equation}
\bigsqcup_{i=0}^k{}_{l-1}\overline{\mathcal{R}}_{k+1,\tau_i}^1(\beta_{l-1},L,\ring{\bar{a}}_{l-1};[m,m+1])\xrightarrow{\cong}\bigsqcup_{i=0}^k{}_{l-1}\overline{\mathcal{R}}_{k+1}^{S_{i,i+1}^1}(\beta_{l-1},L,\ring{\bar{a}}_{l-1};[m,m+1]),
\end{equation}
where the identification is induced by the auxiliary-rescaling maps $\left\{\pi_f^i\right\}_{i=0}^k$. It follows that
\begin{equation}
\begin{split}
\bar{y}_m^{S^1}(k)&=\sum_{\theta_M(a)<(m-k-U)\varepsilon}\sum_{l=1}^\infty\mathit{Ev}_\ast\left({}_{l-1}\overline{\mathcal{R}}_{k+1}^{S^1}(\beta_{l-1},L,\ring{\bar{a}}_{l-1};[m,m+1])\right) \\
&=\sum_{i=0}^k\sum_{\theta_M(a)<(m-k-U)\varepsilon}\sum_{l=1}^\infty\mathit{Ev}_\ast\left({}_{l-1}\overline{{\mathcal{R}}}_{k+1,\tau_i}^1(\beta_{l-1},L,\ring{\bar{a}}_{l-1};[m,m+1])\right) \\
&=\sum_{i=0}^k\bar{y}_{m,1}^i(k).
\end{split}
\end{equation}
In the definitions of the moduli spaces contributing to the chains $\bar{y}_{m,1}(k+1)\in\overline{C}_0$, we have required that the boundary marked points $z_0,\cdots,z_{k+1}\in\partial S$ being distinct, therefore by definitions of $\delta_{k,j}:C_\ast(k-1)\rightarrow C_\ast(k)$ these chains do not have degenerate part. Applying Lemma \ref{lemma:delta} (more precisely, its variation for the relative de Rham complex) to these chains we see that in order to deduce (\ref{eq:st3}), it remains to show that
\begin{equation}\label{eq:BV}
\bar{y}_{m,1}^i(k)=(-1)^{k(k-i)}(\overline{R}_{k+1})_\ast^{k+1-i}\bar{y}_{m,1}(k+1)\circ_{i+1}\bar{e}_L
\end{equation}
in $\overline{C}_\ast^\mathit{nd}$ for any $0\leq i\leq k$. By cyclically permuting the labels of the boundary marked points $z_0,\cdots,z_k$ for an element of the moduli space ${}_{l-1}\overline{\mathcal{R}}_{k+1,\tau_i}^1(\beta_{l-1},L,\ring{\bar{a}}_{l-1};[m,m+1])$, we can achieve that the auxiliary marked point $z_f\in\partial S$ lies between $z_0$ and $z_1$. By our choices of Floer data (cf. Section \ref{section:CG}) and CF-perturbations (cf. (i) at the beginning of this section), the chain $\bar{y}_{m,1}^i(k)$ is obtained by applying the algebraic operator $\lambda^{-i}$ to the de Rham chain $\bar{y}_{m,1}^0(k)$ defined by the moduli space ${}_{l-1}\overline{\mathcal{R}}_{k+1,\tau_0}^1(\beta_{l-1},L,\ring{\bar{a}}_{l-1};[m,m+1])$, see our arguments for Lemma \ref{lemma:vartheta} and Theorem \ref{theorem:cyclic}. Thus it is enough to prove (\ref{eq:BV}) for $i=0$. It is a principle explained in the proof of \cite{sg1}, Lemma 5.8 that marking $z_f$ as auxiliary is equivalent to putting it back in as an ordinary marked point and then erasing it by taking fiber product with the unit $\bar{e}_L$. Since treating $z_f$ as an ordinary marked point gives the moduli space ${}_{l-1}\overline{\mathcal{R}}_{k+2}^1(\beta_{l-1},L,\ring{\bar{a}};[m,m+1])$, rotated so that $z_f$ now play the role of $z_1$ under the new labeling, it follows that
\begin{equation}
\bar{y}_{m,1}^0(k)=(s_1\circ\lambda^{-1}_{k+1})\left(\bar{y}_{m,1}(k+1)\right),
\end{equation}
which is exactly the identity (\ref{eq:BV}) when $i=0$. The justification for signs is the same as in the proof of Lemma \ref{lemma:vartheta}.
\end{proof}

\begin{remark}
We describe here an alternative approach to prove the identity (\ref{eq:BV}). To do this, we equip the Cohen-Ganatra moduli space ${}_l\mathcal{R}_{k+1}^1(x,L,\ring{\bar{a}})$ with the Floer data so that a disc bubble develops (at the boundary point determined by the line segment connecting the origin and $p_1$) when $|p_1|\rightarrow\frac{1}{2}$. See the right-hand side of Figure \ref{fig:alter}. This gives an alternative compactification ${}_l\widehat{\mathcal{R}}_{k+1}^1(x,L,\ring{\bar{a}})$ of the moduli space ${}_l\mathcal{R}_{k+1}^1(x,L,\ring{\bar{a}})$, which is also an admissible K-space. When $l=1$ and $L\subset M$ is exact, such a compactification is used in \cite{ss} to deduce the identity (2.14) there. When $l\geq2$, a little bit more care must be taken when choosing the Floer data if we want to keep the splitting (\ref{eq:split}) induced by the forgetful map. Here, we choose the Floer data so that when $|p_1|$ is sufficiently close to $\frac{1}{2}$, it creates a ``neck region" with gluing parameter
\begin{equation}
r=\log\left(\frac{1-2|p_2|}{1-2|p_1|}\right).
\end{equation}
In other words, when $|p_1|\rightarrow\frac{1}{2}$, $r\rightarrow\infty$, a disc bubble appears, but when moving towards the codimension 2 stratum corresponding to $|p_1|=|p_2|=\frac{1}{2}$, the disc bubble shrinks back to a point.

\begin{figure}
	\centering
	\begin{tikzpicture}
		\filldraw[draw=black,color={black!15},opacity=0.5] (0,0) circle (1.5);
		\draw (0,0) circle [radius=1.5];
		\draw [orange] [dashed] (0,0) circle [radius=0.75];
		
		\draw (1.5,0) node[circle,fill,inner sep=1pt] {};
		\draw (0,1.5) node[circle,fill,inner sep=1pt] {};
		\draw (-1.5,0) node[circle,fill,inner sep=1pt] {};
		\draw (0,-1.5) node[circle,fill,inner sep=1pt] {};
		
		\draw [blue] (1.06,1.06) node[circle,fill,inner sep=1pt] {};
		\node [blue] at (1.15,1.25) {\small $z_f$};
		\draw [blue,dashed] (0,0)--(1.06,1.06);
		
		\draw [orange] (0.53,0.53) node[circle,fill,inner sep=1pt] {};
		\draw [orange] (0,-0.3) node[circle,fill,inner sep=1pt] {};
		\draw [orange] (-0.5,0) node[circle,fill,inner sep=1pt] {};
		\draw [teal] [->] (0,0) to (0,-0.3);
		\node [orange] at (0,-0.5) {\small $p_3$};
		\node [orange] at (-0.5,-0.2) {\small $p_2$};
		
		\node at (1.75,0) {$z_0$};
		\node at (0,1.75) {$z_1$};
		\node at (-1.75,0) {$z_2$};
		\node at (0,-1.75) {$z_3$};
		
		\draw (0,0) node {$\times$};
		\node at (0,0.25) {$\zeta$};
		
		\filldraw[draw=black,color={black!15},opacity=0.5] (5,0) circle (1.5);
		\draw (5,0) circle [radius=1.5];
		\draw [orange] [dashed] (5,0) circle [radius=0.75];
		\filldraw[draw=black,color={black!15},opacity=0.5] (6.4136,1.4136) circle (0.5);
		\draw (6.4136,1.4136) circle [radius=0.5];
		
		\draw (6.5,0) node[circle,fill,inner sep=1pt] {};
		\draw (5,1.5) node[circle,fill,inner sep=1pt] {};
		\draw (3.5,0) node[circle,fill,inner sep=1pt] {};
		\draw (5,-1.5) node[circle,fill,inner sep=1pt] {};
		\node at (6.75,0) {$z_0$};
		\node at (5,1.75) {$z_2$};
		\node at (3.25,0) {$z_3$};
		\node at (5,-1.75) {$z_4$};
		
		\draw (6.06,1.06) node[circle,fill,inner sep=1pt] {};
		\node at (5.85,1.05) {\small $z_1$};
		
		\draw [blue] (6.767,1.767) node[circle,fill,inner sep=1pt] {};
		\node [blue] at (7,1.8) {\small $z_f$};
		\draw [blue,dashed] (5,0)--(6.06,1.06);
		
		\draw [orange] (5.53,0.53) node[circle,fill,inner sep=1pt] {};
		\draw [orange] (5,-0.3) node[circle,fill,inner sep=1pt] {};
		\draw [orange] (4.5,0) node[circle,fill,inner sep=1pt] {};
		\draw [teal] [->] (5,0) to (5,-0.3);
		\node [orange] at (5,-0.5) {\small $p_3$};
		\node [orange] at (4.5,-0.2) {\small $p_2$};
		
		\draw (5,0) node {$\times$};
		\node at (5,0.25) {$\zeta$};
		
	\end{tikzpicture}
	\caption{The boundary stratum ${}_2\overline{\mathcal{R}}_{4,\tau_0}^1(x,L,\ring{\bar{a}})$ in the compactification ${}_3\overline{\mathcal{R}}_4^1(x,L,\ring{\bar{a}})$ (left) and the corresponding stratum ${}_2\widehat{\mathcal{R}}_{4,\tau_0}^1(x,L,\ring{\bar{a}})$ in ${}_3\widehat{\mathcal{R}}_4^1(x,L,\ring{\bar{a}})$ (right)}
	\label{fig:alter}
\end{figure}
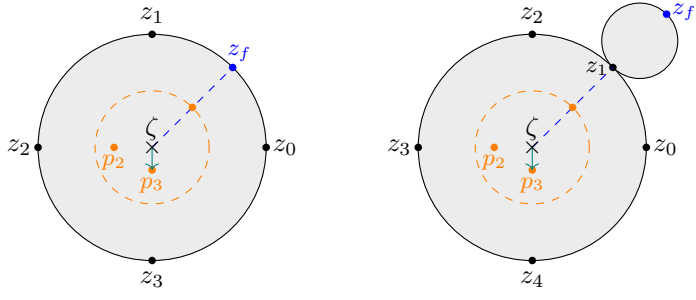

There is a cobordism relating the two compactifications ${}_l\overline{\mathcal{R}}_{k+1}^1(x,L,\ring{\bar{a}})$ and ${}_l\widehat{\mathcal{R}}_{k+1}^1(x,L,\ring{\bar{a}})$, under which the boundary stratum ${}_{l-1}\overline{\mathcal{R}}_{k+1,\tau_0}^1(x,L,\ring{\bar{a}})\subset\partial{}_l\overline{\mathcal{R}}_{k+1}^1(x,L,\ring{\bar{a}})$ goes to
\begin{equation}
{}_{l-1}\widehat{\mathcal{R}}_{k+1,\tau_0}^1(x,L,\ring{\bar{a}}):=\bigsqcup_{\ring{\bar{a}}_1+\bar{a}_2=\ring{\bar{a}}}\kappa^{-1}\left({}_{l-1}\overline{\mathcal{R}}_{k+2}^1(x,L,\ring{\bar{a}})\right)\textrm{ }{{}_1\times_0}\textrm{ }\overline{\mathcal{R}}_{2,f_1}(e_M,L,0),
\end{equation}
where $\overline{\mathcal{R}}_{2,f_1}^1$ is abstractly the moduli space $\overline{\mathcal{R}}_2^1$ defined in Section \ref{section:disc}, but with $z_1\in\partial D$ marked as forgotten, and $\kappa$ is the $\mathbb{Z}_{k+2}$-action on ${}_{l-1}\overline{\mathcal{R}}_{k+2}^1(x,L,\ring{\bar{a}})$ induced by cyclic permutations of the boundary marked points. Since pushing forward the virtual fundamental chain of $\overline{\mathcal{R}}_{2,f_1}^1(e_M,L,0)$ defines the identity $e_L\in C_1(0,0)$, up to sign this gives the right-hand side of (\ref{eq:BV}) when $i=1$.
\end{remark}

With the preparations above, we are now at a position to complete the proof of our main result.

\begin{proof}[Proof of Theorem \ref{theorem:approximate}]
Following Remark \ref{remark:lift2}, we shall construct chains in the $S^1$-equivariant de Rham complex $C_\ast^{S^1}$, and then map them to chains in Connes' complex $C_\ast^\lambda$ which satisfy the conditions (i)-(vii) of Theorem \ref{theorem:approximate}, via the quasi-isomorphism $\phi_\lambda$ provided by Lemma \ref{lemma:qis}. We have defined the chains $\tilde{x}_m$, $\bar{\tilde{x}}_m$, $\tilde{y}_m$, $\bar{\tilde{y}}_m$, $\tilde{z}_m$ and $\bar{\tilde{z}}_m$ in (\ref{eq:t1})---(\ref{eq:t6}). 

Let us first check that the requirements (i')---(vi') in Remark \ref{remark:lift2} are satisfied. The relations $\tilde{x}_m=\tilde{e}_-(\bar{\tilde{x}}_m)$, $\tilde{y}_m=\tilde{e}_-(\bar{\tilde{y}}_m)$, and $\tilde{z}_m=\tilde{e}_-(\bar{\tilde{z}}_m)$ follow from the definitions above and (\ref{eq:epm}). Moreover, the implication
\begin{equation}
\left(\tilde{x}_{m+1}-\tilde{e}_+(\bar{\tilde{x}}_m)\right)(a,k)\neq0\Rightarrow\theta_m(a)\geq(m-k)\varepsilon
\end{equation}
shows that $\tilde{x}_{m+1}-\tilde{e}_+(\bar{\tilde{x}}_m)\in F^m$. Similarly, one can show that $\tilde{y}_{m+1}-\tilde{e}_+(\bar{\tilde{y}}_m)\in F^{m-U-1}$ and $\tilde{z}_{m+1}-\tilde{e}_+(\bar{\tilde{z}}_m)\in F^{m-2}$.

By Theorem \ref{theorem:pushforward} (and its variations described in Remark \ref{remark:pushforward}), the isomorphism (\ref{eq:bd1}) implies that
\begin{equation}\label{eq:dx}
\begin{split}
&\;\;\;\;\;\;\bar{\partial}\bar{x}_{m,0}(a,k+1) \\
&=\sum_{\substack{k_1+k_2=k+1\\a_1+a_2=a\\1\leq i\leq k_1}}(-1)^{(k_1-i)(k_2-1)+k_1}\bar{x}_{m,0}(a_1,k_1+1)\circ_i\bar{x}_m(a_2,k_2) \\
&=\sum_{\substack{k_1+k_2=k+1\\a_1+a_2=a\\1\leq i\leq k_1}}\sum_{j=1}^{k_2+1}(-1)^{\maltese_{ij}^1}\bar{x}_{m,0}(a_1,k_1+1)\circ_i\left((\overline{R}_{k_2+1})^j_\ast\bar{x}_{m,0}(a_2,k_2+1)\circ_{k_2+2-j}e_L\right),
\end{split}
\end{equation}
where
\begin{equation}
\maltese_{ij}^1=(i-1)(k_2-1)+(k_1-1)k_2+k_2(j-1),
\end{equation}
for every $(a,k)$ with $\theta_M(a)<(m+k-1)\varepsilon$. In the above, the second equality follows from Lemma \ref{lemma:vartheta}. Note that the right hand side of (\ref{eq:dx}) isn't equal to $\frac{1}{2}\left\{\bar{\tilde{x}}_m,\bar{\tilde{x}}_m\right\}(a,k+1)$ after projecting to $\overline{C}_\ast^\mathit{nd}$ (cf. (\ref{eq:bracket-pr})), so the condition (iii') in Remark \ref{remark:lift2} isn't satisfied. However, by definition of $\{\cdot,\cdot\}$ it is clear that after passing to the complex $\overline{C}_\ast^\lambda$ quasi-isomorphic to $\overline{C}_\ast^\mathit{nd}$, the right-hand side of (\ref{eq:dx}) becomes $\frac{1}{2}\left\{\bar{\underline{x}}_m,\bar{\underline{x}}_m\right\}(a,k+1)$. It is therefore strong enough to replace the requirement $\bar{\partial}^{S^1}(\bar{\tilde{x}}_m)-\frac{1}{2}\left\{\bar{\tilde{x}}_m,\bar{\tilde{x}}_m\right\}\in\overline{F}^m$ in Remark \ref{remark:lift2}, (iii') for the purpose of proving Theorem \ref{theorem:approximate}. We shall refer to this weaker condition as (iii''). Note that here the equivariant differential $\bar{\partial}^{S^1}$ is actually given by the ordinary de Rham differential $\bar{\partial}$ since the $S^1$-equivariant chain $\bar{\tilde{x}}_m$ only has $u^0$-part. Similar argument shows that $\bar{\partial}^{S^1}(\bar{\tilde{z}}_m)-\left\{\bar{\tilde{x}}_m,\bar{\tilde{z}}_m\right\}\in\overline{F}^{m-2}$.

In order to confirm that $\bar{\partial}^{S^1}(\bar{\tilde{y}}_m)-\left\{\bar{\tilde{x}}_m,\bar{\tilde{y}}_m\right\}-\bar{\tilde{z}}_m\in\overline{F}^{m-U-1}$, we first note that it follows from the isomorphism (\ref{eq:bd5}) that
\begin{equation}\label{eq:ana}
\begin{split}
\bar{\partial}{\bar{y}}_{m,0}(a,k)&=\sum_{\substack{k_1+k_2=k+1\\a_1+a_2=a\\1\leq i\leq k_1}}(-1)^{(k_1-i)(k_2-1)+k_1-1}\bar{y}_{m,0}(a_1,k_1)\circ_i\bar{x}_m(a_2,k_2) \\
&+\sum_{\substack{k_1+k_2=k+1\\a_1+a_2=a\\1\leq i\leq k_1}}(-1)^{(k_1-i)(k_2-1)+k_1}\bar{x}_m(a_1,k_1)\circ_i\bar{y}_{m,0}(a_2,k_2) \\
&-\sum_{l=1}^\infty\sum_{j=0}^l\mathit{Ev}_\ast\left({}_{l-j}\overline{\mathcal{R}}_{k+1}^1\left(\delta_j(\beta_{l-1}),L,\ring{\bar{a}}_{l-1};[m,m+1]\right)\right) \\
&-\bar{y}_m^{j,j+1}(a,k)-\bar{y}_m^{S^1}(a,k).
\end{split}
\end{equation}
By the identity (\ref{eq:id1}) proved in Lemma \ref{lemma:ana}, we have
\begin{equation}\label{eq:fun}
\sum_{l=0}^\infty\sum_{j=0}^l\mathit{Ev}_\ast\left({}_{l-j}\overline{\mathcal{R}}_{k+1}^1\left(\delta_j(\beta_{l-1}),L,\ring{\bar{a}}_{l-1};[m,m+1]\right)\right)=\mathit{Ev}_\ast\left(\overline{\mathcal{R}}_{k+1}^1(e_M,L,\bar{a};[m,m+1])\right),
\end{equation}
where the right-hand side of (\ref{eq:fun}) is by definition the de Rham chain $\bar{z}_m(a,k)$. By the identities (\ref{eq:id2}) and (\ref{eq:id3}) from Lemma \ref{lemma:ana}, we have
\begin{equation}\label{eq:ana1}
\begin{split}
\bar{\partial}{\bar{y}}_{m,0}(a,k)&=\sum_{\substack{k_1+k_2=k+1\\a_1+a_2=a\\1\leq i\leq k_1}}(-1)^{(k_1-i)(k_2-1)+k_1-1}\bar{y}_{m,0}(a_1,k_1)\circ_i\bar{x}_m(a_2,k_2) \\
&+\sum_{\substack{k_1+k_2=k+1\\a_1+a_2=a\\1\leq i\leq k_1}}(-1)^{(k_1-i)(k_2-1)+k_1}\bar{x}_m(a_1,k_1)\circ_i\bar{y}_{m,0}(a_2,k_2) \\
&+\bar{z}_m(a,k)-\bar{\delta}_\mathit{cyc}^\mathit{nd}\left(\bar{y}_{m,1}(a,k+1)\right).
\end{split}
\end{equation}
When combining with $\bar{\partial}\bar{y}_{m,0}(a,k)$, the term $\bar{\delta}_\mathit{cyc}^\mathit{nd}\left(\bar{y}_{m,1}(a,k+1)\right)$ gives the $(a,k)$-part of the $S^1$-equivariant differential $\bar{\partial}^{S^1}(\bar{\tilde{y}}_m)(a,k)$. Note that $\bar{\partial}^{S^1}\left(\bar{\tilde{y}}_m(a,k,k+1)\right)$ contains an additional term $\bar{\partial}\bar{y}_{m,1}(a,k+1)\otimes u^{-1}$, which belongs to the $(a,k+1)$-part of $\overline{C}_\ast^{S^1}$, so it does not appear in (\ref{eq:ana1}). The identification of the first two sums on the right-hand side of (\ref{eq:ana1}) with the $(a,k)$-part of the Lie bracket $\left\{\bar{\tilde{x}}_m,\bar{\tilde{y}}_m\right\}$ is the same as above, which relies on Lemma \ref{lemma:vartheta}. All together, we have proved $\bar{\partial}^{S^1}(\bar{\tilde{y}}_m)-\left\{\bar{\tilde{x}}_m,\bar{\tilde{y}}_m\right\}-\bar{\tilde{z}}_m\in\overline{F}^{m-U-1}$.

For (v') and (vi'), $\tilde{x}_{m,0}(a,k+1)\neq0$ implies that $\mathcal{R}_{k+2,\vartheta}(L,\bar{a};\{m\})\neq\emptyset$, thus $\theta_m(a)\geq2\varepsilon$ or $a=0$, $k\geq2$. Moreover, it follows from Lemma \ref{lemma:vartheta} that
\begin{equation}
\mathbf{B}_c\left(\tilde{x}_{m,0}(0,3)\right)=\sum_{j=1}^{k+1}(-1)^{j-1}(R_3)_\ast^jx_{m,0}(0,3)\circ_{4-j}e_L=x_m(0,2).
\end{equation}
Similarly, $\tilde{z}_m(a,k)\neq0$ implies that $\mathcal{R}_{k+1}^1(e_M,L,\bar{a};\{m\})\neq\emptyset$, thus $\theta_M(a)\geq2\varepsilon$ or $a=0$ (cf. the proof of Lemma \ref{lemma:basic}, (i)), and $\left[\tilde{z}_m(0,0)\right]=(-1)^{n+1}\left[\overline{\mathcal{R}}_1^1(e_M,L,0;\{m\})\right]\otimes1=(-1)^{n+1}[\![L]\!]$.

Finally, we show that Remark \ref{remark:lift2}, (vii') is satisfied. Note that $\left(\bar{\tilde{y}}_m(a,k),\bar{\tilde{z}}_m(a,k)\right)\neq(0,0)$ implies that there exists an $l\in\mathbb{Z}_{\geq0}$ such that one of the following two moduli spaces
\begin{equation}
{}_l\overline{\mathcal{R}}_{k+1}^1(\beta_{l-1},L,\ring{\bar{a}}_{l-1})\textrm{ and }{}_{l-1}\overline{\mathcal{R}}_{k+2}^1(\beta_{l-1},L,\ring{\bar{a}}_{l-1})
\end{equation}
is non-empty. (When $l=0$, then $\overline{\mathcal{R}}_{k+1}^1(\beta_{-1},L,\ring{\bar{a}}_{-1})\neq\emptyset$.) By Lemma \ref{lemma:basic} (iii), this implies that for such an $l$, we have
\begin{equation}\label{eq:l}
\left({}_{l}\overline{\mathcal{R}}_1^1(\beta_{l-1},L,\ring{\bar{a}}_{l-1}),{}_{l-1}\overline{\mathcal{R}}_1^1(\beta_{l-1},L,\ring{\bar{a}}_{l-1})\right)\neq(\emptyset,\emptyset).
\end{equation}
Hence, $a\in A_x$ implies that $\overline{\mathcal{R}}_{2,\vartheta}(L,\bar{a};\{m\})\neq\emptyset$, while $a\in A_{y,z}$ implies there is an $l$ such that (\ref{eq:l}) holds. From now on, fix the choice of such an $l$. We claim that the set
\begin{equation}\label{eq:s}
\left\{a\in H_1(L;\mathbb{Z})\left\vert\theta_M(a)<\Xi,\left(\overline{\mathcal{R}}_{2,\vartheta}(L,\bar{a};\{m\}),{}_{l}\overline{\mathcal{R}}_1^1(\beta_{l-1},L,\ring{\bar{a}}),{}_{l-1}\overline{\mathcal{R}}_1^1(\beta_{l-1},L,\ring{\bar{a}})\right)
\neq(\emptyset,\emptyset,\emptyset)\right.\right\}
\end{equation}
is finite for any $\Xi>0$. In fact, assume that this isn't the case, then there exists a sequence $(a_\ell)_{\ell\in\mathbb{N}}$ of distinct elements $a_\ell\in H_1(L;\mathbb{Z})$ such that $\theta_M(a_\ell)<\Xi+|A_{H_t}(\beta_l)|$ for every $\ell\in\mathbb{N}$, and at least one of the following conditions holds:
\begin{itemize}
	\item $\overline{\mathcal{R}}_{2,\vartheta}(L,\bar{a}_\ell;\{m\})\neq\emptyset$ for every $\ell\in\mathbb{N}$,
	\item ${}_{l}\overline{\mathcal{R}}_1^1(\beta_{l-1},L,\ring{\bar{a}}_\ell)\neq\emptyset$ for every $\ell\in\mathbb{N}$,
	\item ${}_{l-1}\overline{\mathcal{R}}_1^1(\beta_{l-1},L,\ring{\bar{a}}_\ell)\neq\emptyset$ for every $\ell\in\mathbb{N}$.
\end{itemize}
Consider the first case. Pick a $u_\ell\in\overline{\mathcal{R}}_{2,\vartheta}(L,\bar{a}_\ell;\{m\})$ for each $\ell$. By possibly passing to a subsequence we may assume that $(u_\ell)_{\ell\in\mathbb{N}}$ is Gromov convergent, so $(a_\ell)_{\ell\in\mathbb{N}}$ is constant for $\ell\gg0$, which contradicts the assumption. Similar arguments in the second and the third case will lead to the same contradiction. Thus we conclude that (\ref{eq:s}) is a finite set, which shows that both $A_x^+(\Xi)$ and $A_{y,z}^+(\Xi)$ are finite sets for any $\Xi>0$. This completes the verification of the conditions (i')---(vii') in Remark \ref{remark:lift2} (with (iii') replaced with a weaker condition (iii'') described above). Thus after mapping the chains $\tilde{x}_m$, $\bar{\tilde{x}}_m$, $\tilde{y}_m$, $\bar{\tilde{y}}_m$, $\tilde{z}_m$ and $\bar{\tilde{z}}_m$ to $C_\ast^\lambda$ and $\overline{C}_\ast^\lambda$, we obtain the chains $\underline{x}_m$, $\underline{y}_m$, $\underline{z}_m$, $\bar{\underline{x}}_m$, $\bar{\underline{y}}_m$ and $\bar{\underline{z}}_m$ satisfying Theorem \ref{theorem:approximate} (with the index $m$ replaced with $i$).
\end{proof}

\section{Miscellanies}\label{section:mix}

We discuss some implications and variations of the main results of this paper. In this section, $\mathbb{K}$ will be a field with $\mathrm{char}(\mathbb{K})=0$.

\subsection{Cyclic coordinate functions}\label{section:coordinate}

Recall that the symplectic cohomology $\mathit{SH}^\ast(M)$ of a Liouville manifold carries the structure of a Gerstenhaber algebra, with the Gerstenhaber bracket
\begin{equation}
[\cdot,\cdot]:\mathit{SH}^i(M)\otimes\mathit{SH}^j(M)\rightarrow\mathit{SH}^{i+j-1}(M)
\end{equation}
defined by counting rigid points in the moduli space of solutions to the perturbed Cauchy-Riemann equation, with domains given by 1-parameter families of pair-of-pants, see $\cite{ps4}$, Section (4a) for details. When $M=T^\ast Q$ is the cotangent bundle of a compact smooth manifold $Q$, the Gerstenhaber bracket coincides with the loop bracket in string topology $\cite{cs}$ under the Viterbo isomorphism.

Inspired by the work of Fukaya $\cite{kf}$, Ganatra and Pomerleano introduced the following notion (cf. $\cite{gp1}$, Remark 5.5). From now on, let $M$ be a Liouville manifold with $c_1(M)=0$.

\begin{definition}
A pair $(c,\phi)\in\mathit{SH}^0(M)\times\mathit{SH}^1(M)$ is called a coordinate function if
\begin{equation}\label{eq:coordinate}
[c,\phi]=1.
\end{equation}
\end{definition}

It is straightforward to check that $T^\ast S^1$ admits a coordinate function. By the K\"{u}nneth formula for symplectic cohomology \cite{ao}, a coordinate function exists for $T^\ast(S^1\times K)$, where $K$ is any closed oriented manifold.

The relation between coordinate functions and Fukaya-Irie's result is as follows. One expects that the $\ell_2$ in the identity (\ref{eq:def}) should coincide with the loop bracket $[\cdot,\cdot]$ (up to sign, cf. \cite{ki2}, Remark 3.4), which means that the $L_\infty$-structure $(\ell_k)_{k\geq1}$ on $H_{n-\ast}(\mathcal{L}L;\mathbb{R})$ is a deformation of the formal $L_\infty$-structure given by $[\cdot,\cdot]$. If this deformation is trivial, and $y(a)\neq0$ for some $a\in H_1(L;\mathbb{Z})$ with $\mu(a)=2$, then (\ref{eq:def}) reduces to (\ref{eq:coordinate}), and we conclude that $\mathit{SH}^\ast(T^\ast L)$ admits a coordinate function. 

On the other hand, when restricting ourselves to dimension 3, Theorem \ref{theorem:3} provides evidence to the following conjecture made in $\cite{gp1}$.

\begin{conjecture}[Ganatra-Pomerleano]\label{conjecture:GP}
If the cotangent bundle $T^\ast Q$ of a closed, oriented 3-manifold $Q$ admits a coordinate function, then $Q$ is diffeomorphic to $S^1\times\Sigma_g$ for some $g\geq0$.
\end{conjecture}

More generally, the discussions at the end of Section \ref{section:GH} suggest that coordinate functions exist for many Liouville domains which admit symplectic embeddings into $\mathbb{C}^n$, or more generally, into flexible Weinstein manifolds. In general, it seems quite difficult to construct interesting examples of Liouville manifolds with coordinate functions. 

Inspired by Theorem \ref{theorem:main}, we introduce the following $S^1$-equivariant analogue of a coordinate function.

\begin{definition}
We call a pair $(\tilde{c},\tilde{\phi})\in\mathit{SH}_{S^1}^1(M)\times\mathit{SH}_{S^1}^1(M)$ a cyclic coordinate function if they satisfy
\begin{equation}
\left\{\tilde{c},\tilde{\phi}\right\}=1.
\end{equation}
\end{definition}

Here $\{\cdot,\cdot\}$ is the degree $-2$ Lie bracket on $\mathit{SH}_{S^1}^\ast(M)$, which serves as part of the gravity algebra structure on $\mathit{SH}_{S^1}^\ast(M)$ (cf. \cite{ps2}, Section (8b)). One can obtain non-trivial examples of Liouville manifolds with cyclic coordinate functions by relating it with the following notion introduced in $\cite{yl2}$.

\begin{definition}\label{definition:cyclic-quasi}
We say that a Liouville manifold $M$ admits a cyclic quasi-dilation if there is a pair $(\tilde{b},h)\in\mathit{SH}_{S^1}^1(M)\times\mathit{SH}^0(M)^\times$ such that $B(\tilde{b})=h$.
\end{definition}

It follows from the definition of the Lie bracket $\{\cdot,\cdot\}$ that if $M$ admits a cyclic coordinate function, then
\begin{equation}\label{eq:B}
B(\tilde{c})\cdot B(\tilde{\phi})=1.
\end{equation}
In particular, $M$ admits a cyclic quasi-dilation. On the other hand, if $M$ admits a cyclic dilation $\tilde{b}\in\mathit{SH}_{S^1}^1(M)$, then the pair $(\tilde{b},\tilde{b})$ gives a cyclic coordinate function. Parallel to Conjecture \ref{conjecture:GP}, it seems reasonable to expect the following:

\begin{conjecture}
If the cotangent bundle $T^\ast Q$ of a closed, oriented 3-manifold $Q$ admits a cyclic coordinate function, then $Q$ is diffeomorphic to $S^1\times\Sigma_g$, where $g\geq0$, or a spherical space form.
\end{conjecture}

One of the motivations of introducing the notion of a cyclic coordinate function is that we can prove its existence for an important class of Liouville manifolds---punctured $A_m$ Milnor fibers, which are hypersurfaces $M\subset\mathbb{C}^n\times\mathbb{C}^\ast$ defined by the equation
\begin{equation}
\left\{(x_1,\cdots,x_n,z)\in\mathbb{C}^n\times\mathbb{C}^\ast\left|x_1^2+\cdots+x_n^2=z^{m+1}-1\right.\right\},
\end{equation}
where $m\geq0$. Note that here we do not exclude the possibility $m=0$, in which case $M$ isn't a punctured Milnor fiber, but a self-plumbing of $T^\ast S^n$.

From the point of view of Theorem \ref{theorem:main}, there are good reasons to believe that cyclic coordinate functions should exist for these punctured Milnor fibers, since adding a copy of $T^\ast S^{n-1}$ at the origin to the total space of $\pi:M\rightarrow\mathbb{C}^\ast$ defines a symplectic embedding from $M$ to the $A_m$ Milnor fiber, and the latter space admits a cyclic dilation (when $n\geq3$, a dilation). On the other hand, it is clear that these punctured Milnor fibers cannot be symplectically embedded into $\mathbb{C}^n$ (except when $m=0$), or more generally, any Liouville manifold with vanishing symplectic cohomology, so one doesn't expect the existence of coordinate functions for a punctured $A_m$ Milnor fiber when $m\geq1$.

\begin{proposition}
Let $M$ be a punctured $A_m$ Milnor fiber of dimension $n\geq2$, then $M$ admits a cyclic coordinate function.
\end{proposition}
\begin{proof}
$M$ is the total space of a Lefschetz fibration $\pi:M\rightarrow\mathbb{C}^\ast$, whose smooth fibers are symplectomorphic to $T^\ast S^{n-1}$, and the monodromy at the origin is trivial. It follows that there is a long exact sequence
\begin{equation}
\cdots\rightarrow\mathbb{K}^{\mathit{Crit}(\pi)}[-n]\rightarrow\mathit{SH}_\mathit{vert}^\ast(M)\rightarrow\mathit{SH}^\ast(F)\otimes H^\ast(S^1;\mathbb{K})\rightarrow\cdots
\end{equation}
relating the symplectic cohomology of the fibers $F$ and that of the total space, where $\mathit{SH}_\mathit{vert}^\ast(M)$ is the ``vertical" symplectic cohomology defined as the direct limit of a system of Hamiltonians which are small in the base, but linear with growing slopes in the fiber direction of $\pi$.

When $n\geq3$, it follows from the argument of $\cite{ss}$, Section 7 (in particular, Example 7.5) that $M$ admits a dilation, so in particular, a cyclic coordinate function.

When $n=2$, a smooth fiber $F$ of $\pi$ is symplectomorphic to $T^\ast S^1$, in which case we have an algebra isomorphism $\mathit{SH}^\ast(F)\cong\mathbb{K}[w,w^{-1},\partial_w]$, where $|w|=0$ and $|\partial_w|=1$. Both of the generators $w$ and $w^{-1}$ lie in the image of the BV operator, more precisely,
\begin{equation}\label{eq:vf}
\Delta(-\partial_w)=w^{-1}, \Delta(w^2\partial_w)=w,
\end{equation}
and they can be lifted to elements of $\mathit{SH}^0(M)^\times$. In fact, consider the composition
\begin{equation}\label{eq:compo}
\mathit{SH}^0(F)\xrightarrow{\cong}\mathit{SH}^0_\mathit{vert}(M)\rightarrow\mathit{SH}^0(M),
\end{equation}
where the isomorphism follows from the fact that the contributions of the critical points of $\pi$ are supported in degree 2, and the second map is the continuation map. Denote by $h^{\pm1}\in\mathit{SH}^0(M)$ the images of $w^{\pm1}$ under (\ref{eq:compo}). 

To show that the classes $-\partial_w,w^2\partial_w\in\mathit{SH}^1(F)$ can be lifted to $\mathit{SH}^1_\mathit{vert}(M)$, we will apply an argument of Abouzaid-Smith (cf. \cite{asm1}, Sections 3.3 and 4.3). Consider the fiberwise (partial) compactification $\bar{\pi}:\overline{M}\rightarrow\mathbb{C}^\ast$ of $\pi$, whose generic fiber $\overline{F}\cong\mathbb{C}$, and the singular fibers consist of a $\mathbb{C}$ intersecting transversely with a $\mathbb{P}^1$ at a unique point. Let $\mathcal{M}_1(\overline{M}|1)$ be the moduli space of stable rational curves in $\overline{M}$ with one marked point. The Gromov-Witten invariant $\mathit{GW}_1:=\left[\mathit{ev}_1\mathcal{M}_1(\overline{M}|1)\right]\in H^2(\overline{M};\mathbb{Z})$ defined by evaluating at the unique marked point is given by the union of $\mathbb{P}^1$'s in the singular fibers of $\bar{\pi}$ in our case. When restricting to a basis of Lagrangian matching spheres in the punctured Milnor fiber $M$, the contributions of these $\mathbb{P}^1$'s vanish since for any fixed matching sphere, one of them has intersection number $+1$ with the sphere, while the other one has intersection number $-1$. It follows that $\mathit{GW}_1|_M\in H^2(M;\mathbb{Z})$ vanishes. Since the lift of the cochain representative of the fiberwise class $-\partial_w$ (or $w^2\partial_w$) is given by counting thimbles $\mathbb{C}\rightarrow\overline{M}$ which pass through the fiberwise divisor $\overline{M}\setminus M$, this proves that the lifts of the cochain representatives of $-\partial_w$ and $w^2\partial_w$ in $\mathit{SC}^1_\mathit{vert}(M)$ can be corrected to cocycles by adding topological cochains in $C^2(M;\mathbb{Z})$. In particular, they have preimages under the map $\mathit{SH}^1_\mathit{vert}(M)\rightarrow\mathit{SH}^1(F)\otimes H^0(S^1;\mathbb{K})\cong\mathit{SH}^1(F)$ in the long exact sequence. Denote by $b_-$ and $b_+$ their lifts in $\mathit{SH}^1(M)$, respectively. Since both the map $\mathit{SH}^\ast_\mathit{vert}(M)\rightarrow\mathit{SH}^\ast(F)\otimes H^\ast(S^1;\mathbb{K})$ in the long exact sequence and the continuation map $\mathit{SH}^\ast_\mathit{vert}(M)\rightarrow\mathit{SH}^\ast(M)$ are homomorphisms of BV algebras, we have $\Delta(b_-)=h^{-1}$ and $\Delta(b_+)=h$. Denote by $\tilde{b}_-,\tilde{b}_+\in\mathit{SH}_{S^1}^1(M)$ the images of $b_-$ and $b_+$ under the erasing map $I$, it follows that $B(\tilde{b}_-)\cdot B(\tilde{b}_+)=1$, so $(\tilde{b}_-,\tilde{b}_+)$ gives a cyclic coordinate function for the 4-dimensional punctured Milnor fiber.
\end{proof}

We mention here one potential application of the notions of coordinate functions and cyclic coordinate functions, namely proving the non-formality of the $L_\infty$-structures on the cochain complexes $\mathit{SC}^\ast(M)$ and $\mathit{SC}^\ast_{S^1}(M)$.

\begin{definition}
An $L_\infty$-algebra $\mathcal{A}$ is called formal if it has a  minimal model $H^\ast(\mathcal{A})$ with $\ell_k=0$ for $k\geq3$. It is called intrinsically formal if for any $L_\infty$-structure on $H^\ast(\mathcal{A})$, we necessarily have $\ell_k=0$ for $k\geq3$.
\end{definition}

Let $M_d$ be the complement of $d\geq n+2$ generic hyperplanes in $\mathbb{CP}^n$, then it admits a symplectic embedding into $\mathbb{C}^n$. A generalization of our argument should lead to the identity (\ref{eq:def1}), with the string homology $\mathbb{H}_\ast^{S^1}$ replaced with the $S^1$-equivariant symplectic cohomology $\mathit{SH}_{S^1}^\ast(M_d)$. On the other hand, it follows from \cite{yl2}, Theorem 12 that $M_d$ does not admit a cyclic quasi-dilation, so in particular it does not admit a cyclic coordinate function. It follows from (\ref{eq:def1}) that there must be some $k\geq3$ such that $\tilde{\ell}_k\neq0$ on $\mathit{SH}_{S^1}^\ast(M_d)$. In particular, this shows that the $L_\infty$-algebra $\mathit{SH}_{S^1}^\ast(M_d)$ cannot be intrinsically formal. Since it is reasonable to expect that the $L_\infty$-structure $(\tilde{\ell}_k)_{k\geq1}$ is homotopy equivalent to the geometric $L_\infty$-structure on $\mathit{SC}_{S^1}^\ast(M_d)$ constructed by counting holomorphic curves (cf. \cite{ks}), we conjecture that $\mathit{SC}_{S^1}^\ast(M_d)$ isn't formal as an $L_\infty$-algebra. Since $\mathbb{C}^n$ has vanishing symplectic cohomology, similar argument can be carried out in the non-equivariant case to show the non-intrinsic formality (and conjecturally non-formality) of the $L_\infty$-structure $(\ell_k)_{k\geq1}$ on $\mathit{SC}^\ast(M_d)$. The non-existence of a coordinate function follows from similar argument as in the proof of \cite{yl2}, Theorem 12, since $M_d$ cannot be uniruled by $\mathbb{C}^\ast$. Note that these non-formality results fit with the geometric intuition that the affine variety $M_d$ is uniruled by $d$-punctured spheres.

\subsection{Compact symplectic manifolds}\label{section:compact}

We briefly discuss in this subsection how to combine the original ideas of Fukaya $\cite{kf}$ and the new ingredients in this paper to obtain restrictions on Lagrangian embedding in certain compact symplectic manifolds.

In $\cite{kf}$, Fukaya applied similar ideas as described in the introduction to the compact symplectic manifold $\mathbb{CP}^n$ and proved the following:

\begin{theorem}[$\cite{kf}$, Theorem 14.1]\label{theorem:compact}
Let $L\subset\mathbb{CP}^n$ be a closed Lagrangian submanifold, which is a $K(\pi,1)$ space and $\mathit{Spin}$. Then there exists $\bar{a}\in\pi_2(\mathbb{CP}^n,L)$ such that $L$ bounds a holomorphic disc in the class $\bar{a}$ with Maslov index 2. Moreover, $a=\partial\bar{a}\in\pi_1(L)$ is non-zero, and its centralizer $Z_a\subset\pi_1(L)$ has finite index.
\end{theorem}

\begin{remark}\label{remark:Damian}
When $X=Y\times Z$, where $Y$ is a closed K\"{a}hler manifold with a subcritical polarization in the sense of Biran-Cieliebak \cite{bc}, and $Z$ is chosen so that $X$ is monotone, Theorem \ref{theorem:compact} is proved by Damian \cite{md} for monotone Lagrangian submanifolds $L\subset X$ which are $K(\pi,1)$ spaces.
\end{remark}

Theorem \ref{theorem:compact} is parallel to Corollary \ref{corollary:n} of this paper. In particular, it implies that a Lagrangian $L\subset\mathbb{CP}^n$ cannot be a hyperbolic manifold. This is a special case of the \textit{Viterbo-Eliashberg theorem} $\cite{egh}$, which states that if $L\subset X$ is a closed Lagrangian submanifold in a \textit{uniruled} smooth algebraic variety $X$, then $L$ cannot be hyperbolic.

It is therefore desirable to have generalizations of Theorem \ref{theorem:compact} for closed Lagrangian submanifolds $L$ in a more general class of uniruled smooth algebraic varieties $X$. Inspired by the work of Biran-Cieliebak \cite{bc}, we introduce the following notion.

\begin{definition}
Let $X$ be a closed K\"{a}hler manifold. If there is a complex hypersurface $\Sigma\subset X$ such that the complement $X\setminus\Sigma$ is a Weinstein manifold which admits a cyclic dilation, we call the pair $(X,\Sigma)$ a polarization with finite first Gutt-Hutchings capacity.
\end{definition}

In view of Remark \ref{remark:Damian}, it is natural to expect the following.

\begin{conjecture}\label{conjecture:compact}
Let $X$ be a K\"{a}hler manifold which admits a polarization $(X,\Sigma)$ with finite first Gutt-Hutchings capacity. Then the conclusion of Theorem \ref{theorem:compact} holds for any closed Lagrangian submanifold $L\subset X$ which is a $K(\pi,1)$ space and \textit{Spin}.
\end{conjecture}

It is not clear to the author how to prove Conjecture \ref{conjecture:compact} in general. Here let us take a look at a special case, which is inspired by the works of Abouzaid-Smith \cite{asm1} and Ganatra-Pomerleano \cite{gp1}. Let $X$ be a smooth projective variety, and $\Sigma\subset X$ is a normal crossing divisor, with smooth components $\Sigma_1,\cdots,\Sigma_m$. Assume that
\begin{equation}
K_X^{-1}\cong\mathcal{O}(\sum_{i=1}^ma_i\Sigma_i).
\end{equation}
Let $I\subset\{1,\cdots,m\}$ be a subset, we introduce the notations $\Sigma_I:=\bigcap_{i\in I}\Sigma_i$, $\Sigma_I^\circ:=\Sigma_I\setminus\bigcup_{j\notin I}\Sigma_j$. Locally near $\Sigma_I\subset X$, there is an $|I|$-fold intersection $U_I$ of the tubular neighborhoods $U_i$ of each $\Sigma_i$, and the iterated projection $U_I\rightarrow\Sigma_I$ is a symplectic fibration with structure group $U(1)^{|I|}$, with fibers a product of standard symplectic 2-discs. Denote by $S_I\rightarrow\Sigma_I$ the associated $T^{|I|}$-bundle over $\Sigma_I$, and by $S_I^\circ$ its restriction to $\Sigma_I^\circ$.

Let $J_X$ be an almost complex structure on $X$ which is compatible with the symplectic form $\omega_X$, and let $L\subset X$ be a closed Lagrangian submanifold disjoint from $\Sigma$. For a $J_X$-holomorphic map $u:(D,\partial D)\rightarrow(X,L)$ with $u(0)\in\Sigma_I^\circ$, one can choose a local complex coordinate $z$ near the origin so that the normal component $u_i$ of $u$ has an expansion
\begin{equation}
u_i(z)=a_iz^{v_i}+O(|z|^{v_i+1}),
\end{equation}
where $v_i\in\mathbb{N}$ and the coefficient $a_i\neq0$ is the $v_i$-jet normal to $\Sigma_i$ modulo higher order terms, which takes value in $(T_0^\ast D)^{\otimes v_i}\otimes N_{u(0)}\Sigma_i$, where $N\Sigma_i$ is the normal bundle of $\Sigma_i\subset X$. Now fix an asymptotic marker $\ell$ at $0\in D$, it specifies via $a_i$ a non-zero vector in $N_{u(0)}\Sigma_i\setminus\{0\}$. Denote its image in the real projectivization $\mathbb{P}N_{u(0)}\Sigma_i$ by $[a_i]$. Taking the direct sum over $i\in I$ gives a point $\left[\bigoplus_{i\in I}a_i\right]$ in the fiber $S_{I,u(0)}$ of the torus bundle $S_I\rightarrow D_I$. Let $\mathbf{v}=(v_1,\cdots,v_m)\in\mathbb{Z}_{\geq0}^m$, where $v_i=0$ for $i\notin I$, the \textit{enhanced evaluation map} of $u$ at $0$ is defined as
\begin{equation}
\mathit{Ev}_0^\mathbf{v}(u):=\left(u(0),\left[\bigoplus_{i\in I}a_i\right]\right),
\end{equation}
which keeps track of both the incidence condition and the tangency condition of the holomorphic curve $u$ at $0$ with the divisors $\Sigma_i$, $i\in I$.

For $\bar{a}\in H_2(X,L)$ and $\mathbf{v}_l=(v_{l,1},\cdots,v_{l,m})\in\mathbb{Z}_{\geq0}^m$ supported on $I_l\subset\{1,\cdots,m\}$, let
\begin{equation}
{}_l\mathcal{R}_{k+1}(L,\mathbf{v}_l,\bar{a})
\end{equation}
be the moduli space of pairs
\begin{equation}
\left((D;z_0,\cdots,z_k,p_1,\cdots,p_l;\ell),u\right),
\end{equation}
where $z_0,\cdots,z_k\in\partial D$, the auxiliary marked points $p_1,\cdots,p_l$ satisfy (\ref{eq:radial}), and the asymptotic marker $\ell$ at $0\in D$ points towards $p_l$. The map $u:(D,\partial D)\rightarrow(X,L)$ is $J_X$-holomorphic, such that $u(0)\cdot\Sigma_i=v_{l,i}$ for $i\in I_l$, and $[u]=\bar{a}$. The enhanced evaluation map gives rise to a map
\begin{equation}
\mathit{Ev}_0^{\mathbf{v}_l}:{}_l\mathcal{R}_{k+1}(L,\mathbf{v}_l,\bar{a})\rightarrow S_{I_l}^\circ.
\end{equation}
For any locally finite chain $\alpha_l\in C_\ast^\mathit{BM}(S_{I_l}^\circ;\mathbb{K})$, we form the fiber product
\begin{equation}
{}_l\mathcal{R}_{k+1}(L,\mathbf{v}_l,\alpha_l,\bar{a}):={}_l\mathcal{R}_{k+1}(L,\mathbf{v}_l,\bar{a})\times_{\mathit{Ev}_0^{\mathbf{v}_l}}\alpha_l.
\end{equation}

The compactification ${}_l\overline{\mathcal{R}}_{k+1}(L,\mathbf{v}_l,\alpha_l,\bar{a})$ is an admissible K-space, with its codimension 1 boundary covered by the following strata:
\begin{equation}\label{eq:s1}
\bigsqcup_{\substack{k_1+k_2=k+1\\1\leq i\leq k_1\\ \bar{a}_1+\bar{a}_2=\bar{a}}}\left({}_l\overline{\mathcal{R}}_{k_1+1}(L,\mathbf{v}_l,\alpha_l,\bar{a}_1)\textrm{ }{{}_i\times_0}\textrm{ }\overline{\mathcal{R}}_{k_2+1}(L,\bar{a}_2)\sqcup\overline{\mathcal{R}}_{k_1+1}(L,\bar{a}_1)\textrm{ }{{}_i\times_0}\textrm{ }{}_l\overline{\mathcal{R}}_{k_2+1}(L,\mathbf{v}_l,\alpha_l,\bar{a}_2)\right),
\end{equation}
\begin{equation}\label{eq:s2}
{}_j\overline{\mathcal{P}}(\mathbf{v}_l,\alpha_l,x_j)\times{}_{l-j}\overline{\mathcal{R}}^1_{k+1}(x_j,L,\bar{a}), 0\leq j\leq l,
\end{equation}
\begin{equation}\label{eq:s3}
{}_{l-1}\overline{\mathcal{R}}^{S^1}_{k+1}(L,\mathbf{v}_l,\alpha_l,\bar{a}),
\end{equation}
\begin{equation}\label{eq:s4}
{}_l^{j,j+1}\overline{\mathcal{R}}^{S^1}_{k+1}(L,\mathbf{v}_l,\alpha_l,\bar{a}), 1\leq j\leq l-1,
\end{equation}
where ${}_j\mathcal{P}(\mathbf{v}_l,\alpha_l,x_j)$ are moduli spaces of (parametrized) logarithmic Piunikhin-Salamon-Schwarz (PSS) maps studied in \cite{gp1}, with $x_j$ being an orbit of the vector field $X_{H_t}$ of some admissible Hamiltonian function (defined using the radial coordinate on the collar neighborhood of the divisor complement $X\setminus\Sigma$). When $j=0$, this is defined as the fiber product
\begin{equation}
\mathcal{P}(\mathbf{v}_l,\alpha_l,x_j):=\mathcal{P}(\mathbf{v}_l,x_j)\times_{\mathit{Ev}_0^{\mathbf{v}_l}}\alpha_l,
\end{equation}
where $\mathcal{P}(\mathbf{v}_l,x_j)$ is the moduli space of maps $u:\mathbb{C}\rightarrow X$ satisfying the Floer equation $(du-X_{H_t}\otimes dt)^{0,1}=0$, with the obvious asymptotic and tangency conditions. In general, ${}_j\mathcal{P}(\mathbf{v}_l,\alpha_l,x_j)$ is defined as the space of the same maps $u:\mathbb{C}\rightarrow X$, but from a varying family of domains obtained by equipping $\mathbb{C}$ with the auxiliary marked points $p_{l-j+1},\cdots,p_l$, which are strictly radially ordered in the sense that
\begin{equation}
0<|p_l|<\cdots<|p_{l-j+1}|.
\end{equation}
As before, the boundary strata (\ref{eq:s1}) come from disc bubbles along the Lagrangian boundary $L$, (\ref{eq:s2}) arise from breakings at the origin $0\in D$ when $|p_{l-j+1}|\rightarrow0$, (\ref{eq:s3}) corresponds to the locus where $|p_1|=\frac{1}{2}$, and (\ref{eq:s4}) are the loci defined by $|p_j|=|p_{j+1}|$.

A count of rigid elements in the moduli spaces ${}_j\overline{\mathcal{P}}(\mathbf{v},\alpha,x)$ for varying $\alpha\in C_\ast^\mathit{BM}(S_I^\circ;\mathbb{K})$ and $\mathbf{v}\in\mathbb{Z}_{\geq0}^m$ supported on $I\subset\{1,\cdots,m\}$ defines a map
\begin{equation}
\mathit{PSS}^\mathit{log}_j:C_\mathit{log}^{\ast+2j}(X,\Sigma)\rightarrow\mathit{SC}^\ast(M)
\end{equation}
on the logarithmic cochain complex
\begin{equation}
C_\mathit{log}^\ast(X,\Sigma):=\bigoplus_{I\subset\{1,\cdots,m\}}t^\mathbf{v}C_{2n-|I|-\ast}^\mathit{BM}(S_I^\circ;\mathbb{K}),
\end{equation}
When $l=0$, the positive energy part of this map is a chain map on \textit{admissible} cochains $C_\mathit{log}^\ast(X,\Sigma)_\mathit{ad}\subset C_\mathit{log}^\ast(X,\Sigma)$ (cf. \cite{gp1}, Definition 3.20), which induces on the cohomology level the logarithmic PSS map $H_\mathit{log}^\ast(X,\Sigma)_\mathit{ad}\rightarrow\mathit{SH}^\ast_+(M)$ studied in \cite{gp1}.

Now suppose that there exists a family of locally finite chains $(\alpha_l)_{l\geq0}$ and vectors $(\mathbf{v}_l)_{l\geq0}$ such that $\sum_{l=0}^\infty\mathit{PSS}_l^\mathit{log}(\alpha_lt^{\mathbf{v}_l})$ hits the identity $e_M\in\mathit{SC}^0(M)$, and $\mathit{PSS}_j^\mathit{log}(\alpha_lt^{\mathbf{v}_l})=0$ for $j<l$, then by pushing forward the virtual fundamental chains of the moduli spaces in (\ref{eq:s1})---(\ref{eq:s4}), similar arguments as in Section \ref{section:proof} would imply the identity (\ref{eq:def1}), and therefore confirm Conjecture \ref{conjecture:compact} for Lagrangian submanifolds $L\subset X$ which are disjoint from $\Sigma$. 

When $X\setminus\Sigma$ admits a dilation $b\in\mathit{SH}^1(X\setminus\Sigma)$, and the dilation comes from an admissible cochain $\alpha t^\mathbf{v}$, i.e. $\mathit{PSS}^\mathit{log}(\alpha t^\mathbf{v},\mathfrak{b})=b$, where $\mathfrak{b}\in C^\ast(X\setminus\Sigma;\mathbb{K})$ is a bounding cochain for $\alpha t^\mathbf{v}$, it follows from \cite{gp1}, Lemmas 4.29 and 4.30 that there exists an $\tilde{\alpha}\in C_\ast^\mathit{BM}(S_I^\circ;\mathbb{K})$ such that $\mathit{PSS}^\mathit{log}(\tilde{\alpha}t^\mathbf{v})$ hits the identity $e_M\in\mathit{SC}^0(M)$. In particular, this confirms Conjecture \ref{conjecture:compact} for the pairings $(X,\Sigma)$ studied in \cite{gp1}, Sections 6.1 and 6.4, where a dilation can be constructed via the logarithmic PSS map.

\appendix

\section{Kuranishi structures and virtual fundamental chains}\label{section:Kuranishi}

For the reader's convenience, we collect in this appendix some basic notions and useful facts in the theory of Kuranishi structures. Our references here are \cite{fooo1,fooo2,fooo3,fooo4} and \cite{ki2}.

\subsection{Basic notions}\label{section:basic}

Instead of writing down all the definitions, we will only recall here the most basic notions and provide the references for their variations and generalizations, so the reader may use this section as a dictionary for the notions in Kuranishi theory used in the main content of our paper. Following \cite{ki2}, we shall define Kuranish charts using manifolds instead of orbifolds. This is sufficient for our purposes because sphere bubbles do not appear for Lagrangian submanifolds in Liouville manifolds. Let $X$ be a separable, metrizable topological space.

\begin{definition}
A Kuranishi chart of $X$ is a quadruple $\mathcal{U}=(U,\mathcal{E},s,\psi)$ such that
    \begin{itemize}
		\item[(i)] $U$ is a smooth manifold,
		\item[(ii)] $\mathcal{E}$ is a smooth vector bundle on $U$,
		\item[(iii)] $s$ is a smooth section of $\mathcal{E}$,
		\item[(iv)] $\psi:s^{-1}(0)\rightarrow X$ is a homeomorphism onto an open subset of $X$.
	\end{itemize}
We call $U$ a Kuranishi neighborhood, $\mathcal{E}$ an obstruction bundle, $s$ a Kuranishi map, and $\psi$ a parametrization. $\dim\mathcal{U}:=\dim U-\mathrm{rank}\;E$ is called the dimension of $\mathcal{U}$. An orientation of $\mathcal{U}$ is a pair of orientations of $U$ and $E$.
\end{definition}

To globalize the notion of a Kuranishi chart, we need the following two definitions.

\begin{definition}
	Let $\mathcal{U}_i=(U_i,\mathcal{E}_i,s_i,\psi_i)$, $i=1,2$ be Kuranishi charts of $X$. An embedding of Kuranishi charts $\Phi:\mathcal{U}_1\rightarrow\mathcal{U}_2$ is a pair $\Phi=(\phi,\hat{\phi})$ such that
	\begin{itemize}
		\item[(i)] $\phi:U_1\rightarrow U_2$ is an embedding of smooth manifolds,
		\item[(ii)] $\hat{\phi}:\mathcal{E}_1\rightarrow\mathcal{E}_2$ is an embedding of smooth vector bundles over $\phi$,
		\item[(iii)] $\hat{\phi}\circ s_1=s_2\circ\phi$,
		\item[(iv)] $\psi_2\circ\phi=\psi_1$ on $s_1^{-1}(0)$,
		\item[(v)] for any $x\in s_1^{-1}(0)$, the covariant derivative
		\begin{equation}\label{eq:co-der}
			D_{\phi(x)}s_2:\frac{T_{\phi(x)}U_2}{(D_x\phi)(T_xU_1)}\rightarrow\frac{(\mathcal{E}_2)_{\phi(x)}}{\hat{\phi}((\mathcal{E}_1)_x)}
		\end{equation}
		is an isomorphism.
	\end{itemize}
\end{definition}

\begin{definition}
Let $\mathcal{U}_i=(U_i,\mathcal{E}_i,s_i,\psi_i)$, $i=1,2$ be Kuranishi charts of $X$. A coordinate change in the weak sense (resp. strong sense) from $\mathcal{U}_1$ to $\mathcal{U}_2$ is a triple $\Phi_{21}=(U_{21},\phi_{21},\hat{\phi}_{21})$ satisfying the first two (resp. all three) of the requirements below:
	\begin{itemize}
		\item[(i)] $U_{21}$ is an open subset of $U_1$;
		\item[(ii)] $(\phi_{21},\hat{\phi}_{21})$ is an embedding of Kuranishi charts $\mathcal{U}_1|_{U_{21}}\rightarrow\mathcal{U}_2$;
		\item[(iii)] $\psi_1\left(s_1^{-1}(0)\cap U_{21}\right)=\mathrm{im}\psi_1\cap\mathrm{im}\psi_2$.
	\end{itemize}
\end{definition}

We can now define a Kuranishi structure, which is key for the description of the geometric structure of the moduli spaces cut out by non-linear partial differential equations.

\begin{definition}\label{definition:K-structure}
A Kuranishi structure $\widehat{\mathcal{U}}$ of $X$ of dimension $d$ consists of
	\begin{itemize}
		\item a Kuranishi chart of dimension $d$ $\mathcal{U}_p=(U_p,\mathcal{E}_p,s_p,\psi_p)$ at $p$ for every $p\in X$,
		\item a coordinate change in the weak sense $\Phi_{pq}=(U_{pq},\phi_{pq},\hat{\phi}_{pq}):\mathcal{U}_q\rightarrow\mathcal{U}_p$ for every $p\in X$ and $q\in\mathrm{im}\psi_p$ so that $\Phi_{pp}=(U_p,\mathit{id},\mathit{id})$,
	\end{itemize}
such that
	\begin{itemize}
		\item[(i)] $o_q\in U_{pq}$ for every $q\in\mathrm{im}\psi_p$,
		\item[(ii)] for every $p\in X$, $q\in\mathrm{im}\psi_p$ and $r\in\psi_q\left(s^{-1}(0)\cap U_{pq}\right)$, there holds $\Phi_{pr}|_{U_{pqr}}=\Phi_{pq}\circ\Phi_{qr}|_{U_{pqr}}$, where $U_{pqr}:=\phi_{qr}^{-1}(U_{pq})\cap U_{pr}$.
	\end{itemize}
The pair $(X,\widehat{\mathcal{U}})$ is called a K-space of dimension $d$. We say the Kuranishi structure $\widehat{\mathcal{U}}$ is oriented if each Kuranishi chart $\mathcal{U}_p$ is oriented and $\phi_{pq}$ preserves orientations for every $p\in X$ and $q\in\mathrm{im}\psi_p$.
\end{definition}

The analogy of the notion of a manifold with corners in the theory of Kuranishi structures is called an \textit{admissible K-space}, see \cite{fooo4}, Chapters 17 and 25 for details.

\begin{definition}
Let $\widehat{\mathcal{U}}$ be a Kuranishi structure of $X$, and $Y$ a topological space.
\begin{itemize}
	\item[(i)] A strongly continuous map $\hat{f}:(X,\widehat{\mathcal{U}})\rightarrow Y$ assigns a continuous map $f_p:U_p\rightarrow Y$ for each $p\in X$ such that $f_p\circ\phi_{pq}=f_q$ holds on $U_{pq}$.
	\item[(ii)] When $Y$ is a smooth manifold, we say that $\hat{f}$ is strongly smooth if each $f_p$ is smooth.
\end{itemize}
\end{definition}

The notion of a strongly smooth map generalizes to that of an \textit{admissible map} for admissible K-spaces.

In order to define the pushforward of virtual fundamental chains of K-spaces via strongly continuous and submersive maps, we need to equip the K-spaces with \textit{CF-perturbations}. Roughly speaking, these are triples $\mathcal{S}_x=\left(W_x,\omega_x,\{\mathfrak{s}_x^\varepsilon\}\right)$ for every point $x\in V_x\subset U_p$ satisfying certain conditions, where $\mathfrak{s}_x^\varepsilon$ is a family of sections of $\mathcal{E}_p$ over the manifold chart $V_x$ of $U_p$, parametrized by an open neighborhood $W_x$ of $0$ in some finite-dimensional vector space, and $\omega_x$ is a top degree differential form on $W_x$. See \cite{fooo4}, Definition 7.4 for details. 

\begin{definition}
Let $(X,\widehat{\mathcal{U}})$ be a K-space and $N$ a manifold with corners.
\begin{itemize}
	\item[(i)] A strongly continuous map $\hat{f}: (X,\widehat{\mathcal{U}})\rightarrow N$ is said to be a corner stratified smooth map if $f_p:U_p\rightarrow N$ is a corner stratified smooth map (\cite{fooo4}, Definition 26.1) for any $p\in X$.
	\item[(ii)] A corner stratified smooth map $\hat{f}: (X,\widehat{\mathcal{U}})\rightarrow N$ is a corner stratified weak submersion if $f_p:U_p\rightarrow N$ is a corner stratified submersion for any $p\in X$.
	\item[(iii)] Let $\widehat{\mathcal{S}}$ be a CF-perturbation of $X$. We say that a corner stratified smooth map $\hat{f}: (X,\widehat{\mathcal{U}})\rightarrow N$ is a corner stratified strong submersion with respect to $\widehat{\mathcal{S}}$ if the following holds. Let $p\in X$ and $(U_p,\mathcal{E}_p,s_p,\psi_p)$ is a Kuranishi chart at $p$. Let $(V_x,\phi_x)$ be a manifold chart of $U_p$ and $\left(W_x,\omega_x,\{\mathfrak{s}_x^\varepsilon\}\right)$ be a representative of $\widehat{\mathcal{S}}$ in the chart $V_x$, then
	\begin{equation}
	f\circ\psi_p\circ\phi_x\circ\mathit{pr}_V:(\mathfrak{s}_x^\varepsilon)^{-1}(0)\rightarrow N
	\end{equation}
    is a corner stratified submersion, where $\mathit{pr}_V$ is the projection $V_x\times W_x\rightarrow V_x$.
\end{itemize}
\end{definition}

To achieve the compatibility of CF-perturbations, we need the notion of a \textit{thickening} of Kuranishi structures, see \cite{fooo4}, Definition 5.3.

\subsection{Integration along the fiber}\label{section:pushforward}

We recall the integration along the fiber associated to a strongly smooth map, and the compatibility properties satisfied by the de Rham chains defined in this way. The main references here are \cite{fooo4}, Chapter 7 and \cite{ki2}, Section 7.1.

\begin{theorem}[\cite{ki2}, Theorem 7.2]\label{theorem:pushforward}
Let $(X,\widehat{\mathcal{U}})$ be a compact, oriented K-space of dimension $d$ without boundary, equipped with a strongly smooth map $\hat{f}:(X,\widehat{\mathcal{U}})\rightarrow\mathcal{L}_{k+1}$, a differential form $\hat{\omega}$, and a CF-perturbation $\widehat{\mathcal{S}}=(\widehat{\mathcal{S}}^\varepsilon)_{0<\varepsilon\leq1}$. We assume that $\widehat{\mathcal{S}}$ is transversal to $0$, and $\mathit{ev}_0\circ\hat{f}:(X,\widehat{\mathcal{U}})\rightarrow L$ is strongly submersive with respect to $\widehat{\mathcal{S}}$. Then one can define a de Rham chain
\begin{equation}
\hat{f}_\ast(X,\widehat{\mathcal{U}},\hat{\omega},\widehat{\mathcal{S}}^\varepsilon)\in C_{d-|\hat{\omega}|}^\mathit{dR}(\mathcal{L}_{k+1})
\end{equation}
for sufficiently small $\varepsilon$, so that the Stokes' formula and the fiber product formula hold.
\end{theorem}

\begin{remark}
In \cite{ki2}, Theorem 7.2, $\mathcal{L}_{k+1}$ is the space of Moore loops with $k+1$ marked points in $L$, which is a differentiable space in the sense of \cite{ki1}, Section 4.2. Since our model of $\mathcal{L}_{k+1}$ defined in Section \ref{section:de Rham} is a smooth finite-dimensional manifold, the above theorem clearly holds.
\end{remark}

The Stokes' formula and the fiber product formula in Theorem \ref{theorem:pushforward} are stated as follows.

\begin{theorem}[Stokes' formula]\label{theorem:Stokes}
For sufficiently small $\varepsilon>0$, there holds
\begin{equation}
\partial\left(\hat{f}_\ast(X,\widehat{\mathcal{U}},\hat{\omega},\widehat{\mathcal{S}}^\varepsilon)\right)=(-1)^{|\hat{\omega}|+1}\hat{f}_\ast(X,\widehat{\mathcal{U}},d\hat{\omega},\widehat{\mathcal{S}}^\varepsilon).
\end{equation}
\end{theorem}

Suppose for $i=1,2$, we are given the following data:
\begin{itemize}
	\item a compact oriented K-space $(X_i,\widehat{\mathcal{U}}_i)$ of dimension $d_i$ such that $d=d_1+d_2$;
	\item a strongly smooth map $\hat{f}_i:(X_i,\widehat{\mathcal{U}}_i)\rightarrow\mathcal{L}_{k_i+1}$;
	\item a differential form $\hat{\omega}_i$ on $(X_i,\widehat{\mathcal{U}}_i)$;
	\item a CF-perturbation $\widehat{\mathcal{S}}_i$ on $(X_i,\widehat{\mathcal{U}}_i)$ such that $\mathit{ev}_0\circ\hat{f}_i:(X_i,\widehat{\mathcal{U}}_i)\rightarrow L$ is strongly submersive with respect to $\widehat{\mathcal{S}}_i$.
\end{itemize}
For each $1\leq j\leq k_1$, we can take fiber product of K-spaces and define
\begin{equation}
(X_{12},\widehat{\mathcal{U}}_{12}):=(X_1,\widehat{\mathcal{U}}_1)\textrm{ }{{}_{\mathit{ev}_j\circ\hat{f}_1}\times_{\mathit{ev}_0\circ\hat{f}_2}}\textrm{ }(X_2,\widehat{\mathcal{U}}_2).
\end{equation}
On can also define fiber product of CF-perturbations $\widehat{\mathcal{S}}_{12}:=\widehat{\mathcal{S}}_1\times\widehat{\mathcal{S}}_2$ on $(X_{12},\widehat{\mathcal{U}}_{12})$ (\cite{fooo4}, \S10.2). Finally, define a differential form $\hat{\omega}_{12}$ on $(X_{12},\widehat{\mathcal{U}}_{12})$ by
\begin{equation}
\hat{\omega}_{12}:=(-1)^{(d-|\hat{\omega}|_1-n)|\hat{\omega}_2|}\hat{\omega}_1\times\hat{\omega}_2,
\end{equation}
and a strongly smooth map $\hat{f}_{12}:(X_{12},\widehat{\mathcal{U}}_{12})\rightarrow\mathcal{L}_{k_1+k_2}$ by
\begin{equation}
(f_{12})_{p_1,p_2}(x_1,x_2):=\mathit{con}_j\left((f_1)_{p_1}(x_1),(f_2)_{p_2}(x_2)\right),
\end{equation}
where $x_1\in U_{p_1}$, $x_2\in U_{p_2}$, and $\mathit{ev}_j\circ f_{p_1}(x_1)=\mathit{ev}_0\circ f_{p_2}(x_2)$.

\begin{theorem}[Fiber product formula]\label{theorem:fp}
We have
\begin{equation}
(\hat{f}_{12})_\ast\left(X_{12},\widehat{\mathcal{U}}_{12},(-1)^{|\hat{\omega}_{12}|+n}\hat{\omega}_{12},\widehat{\mathcal{S}}_{12}^\varepsilon\right)=(\hat{f}_1)_\ast(X_1,\widehat{\mathcal{U}}_1,\hat{\omega}_1,\widehat{\mathcal{S}}_1^\varepsilon)\circ_j(\hat{f}_2)_\ast(X_2,\widehat{\mathcal{U}}_2,\hat{\omega}_2,\widehat{\mathcal{S}}_2^\varepsilon)
\end{equation}
for sufficiently small $\varepsilon>0$.
\end{theorem}

\begin{remark}\label{remark:pushforward}
There are versions of Theorem \ref{theorem:pushforward} for admissible K-spaces $(X,\widehat{\mathcal{U}})$, in which case the strongly smooth map $\hat{f}$ is replaced with an admissible map, and the strong submersion $\mathit{ev}_0\circ\hat{f}$ is replaced with a corner-stratified strong submersion, see \cite{ki2}, Theorem 7.9. The Stokes' formula now takes the form
\begin{equation}
\partial\left(\hat{f}_\ast(X,\widehat{\mathcal{U}},\hat{\omega},\widehat{\mathcal{S}}^\varepsilon)\right)=(-1)^{|\hat{\omega}|}(\hat{f}|_{\partial X})_\ast(\partial X,\widehat{\mathcal{U}}|_{\partial X},\hat{\omega}|_{\partial X},\widehat{\mathcal{S}}^\varepsilon|_{\partial X})+(-1)^{|\hat{\omega}|+1}\hat{f}_\ast(X,\widehat{\mathcal{U}},d\hat{\omega},\widehat{\mathcal{S}}^\varepsilon),
\end{equation}
where $\partial X$ is the normalized boundary of $X$, which is itself a K-space.
	
As another variation, we can consider admissible maps $\hat{f}:(X,\widehat{\mathcal{U}})\rightarrow[a,b]\times\mathcal{L}_{k+1}$, in which case $\partial X=\partial_hX\sqcup\partial_vX$ is decomposed as horizontal and vertical boundaries, where $\partial_hX=\hat{f}^{-1}\left(\{a,b\}\times\mathcal{L}_{k+1}\right)$. Let $\partial_-X=\hat{f}^{-1}(\{a\}\times\mathcal{L}_{k+1})$ and $\partial_+X=\hat{f}^{-1}(\{b\}\times\mathcal{L}_{k+1})$. In this case, we require
\begin{equation}
\left(\mathit{pr}_{[a,b]}\circ\hat{f},\mathit{ev}_0\circ\mathit{pr}_{\mathcal{L}_{k+1}}\circ\hat{f}\right):(X,\widehat{\mathcal{U}})\rightarrow[a,b]\times L
\end{equation}
to be a corner-stratified strong submersion, and the pushforward defines a relative de Rham chain
\begin{equation}
\hat{f}_\ast(X,\widehat{\mathcal{U}},\hat{\omega},\widehat{\mathcal{S}}^\varepsilon)\in \overline{C}_{d-|\hat{\omega}|-1}^\mathit{dR}(\mathcal{L}_{k+1})
\end{equation}
for sufficiently small $\varepsilon>0$, which satisfies
\begin{equation}\label{eq:epm}
e_\pm\left(\hat{f}_\ast(X,\widehat{\mathcal{U}},\hat{\omega},\widehat{\mathcal{S}}^\varepsilon)\right)=(-1)^{d-1}(\hat{f}|_{\partial_\pm X})_\ast\left(\partial_\pm X,\widehat{\mathcal{U}}|_{\partial_\pm X},\hat{\omega}|_{\partial_\pm X},\widehat{\mathcal{S}}^\varepsilon|_{\partial_\pm X}\right)\in C_{d-|\hat{\omega}|-1}^\mathit{dR}(\mathcal{L}_{k+1}).
\end{equation}
See \cite{ki2}, Theorem 7.14. The Stokes' formula in this case is
\begin{equation}
	\begin{split}
	\partial\left(\hat{f}_\ast(X,\widehat{\mathcal{U}},\hat{\omega},\widehat{\mathcal{S}}^\varepsilon)\right)&=(-1)^{|\hat{\omega}|}(\hat{f}|_{\partial_h X})_\ast\left(\partial X,\widehat{\mathcal{U}}|_{\partial_h X},\hat{\omega}|_{\partial_h X},\widehat{\mathcal{S}}^\varepsilon|_{\partial_h X}\right)\\
	&+(-1)^{|\hat{\omega}|+1}\hat{f}_\ast(X,\widehat{\mathcal{U}},d\hat{\omega},\widehat{\mathcal{S}}^\varepsilon).
	\end{split}
\end{equation}
\end{remark}

For the purposes of this paper (cf. Remark \ref{remark:invariance}), we need an additional property of the de Rham chains defined by integration along the fiber. For any $a\in H_1(L;\mathbb{Z})$, and $x$ a 1-periodic orbit of the Hamiltonian vector field $X_{H_t}$, let $(X,\widehat{\mathcal{U}})$ be one of admissible K-spaces $\overline{\mathcal{R}}_{k+1}(L,\bar{a};P)$, $\overline{\mathcal{R}}_{k+1,\vartheta}(L,\bar{a};P)$, ${}_l\overline{\mathcal{R}}_{k+1}^1(x,L,\ring{\bar{a}};P)$, ${}_{l-1}\overline{\mathcal{R}}_{k+1}^{S^1}(x,L,\ring{\bar{a}};P)$ or ${}_l^{j,j+1}\overline{\mathcal{R}}_{k+1}^1(x,L,\ring{\bar{a}};P)$ considered in Section \ref{section:proof}. Cyclic permutations of the labels of the boundary marked points $z_0,\cdots,z_k$ defines a $\mathbb{Z}_{k+1}$-action on $X$, whose generator is a map $\kappa:X\rightarrow X$. It follows from Theorem \ref{theorem:moduli}, (vii) that $X$ can be equipped with a Kuranishi structure $\widehat{\mathcal{U}}$ and a CF-perturbation $\widehat{\mathcal{S}}^\varepsilon$ that are cyclic invariant, therefore the map $\kappa$ extends to a map
\begin{equation}
\kappa:(X,\widehat{\mathcal{U}},\hat{\omega},\widehat{\mathcal{S}}^\varepsilon)\rightarrow(X,\widehat{\mathcal{U}},\hat{\omega},\widehat{\mathcal{S}}^\varepsilon),
\end{equation}
if the differential form $\hat{\omega}$ is cyclic invariant (e.g. when it is constant). The following theorem is a straightforward consequence of the cyclic invariance of these data. For simplicity, we assume below that $P=\{m\}$. The statement for the $P=[m,m+1]$ case is similar, simply replace $R_k$ with $\overline{R}_k$.

\begin{theorem}[Cyclic permutation formula]\label{theorem:cyclic}
Let the quadruple $(X,\widehat{\mathcal{U}},\hat{\omega},\widehat{\mathcal{S}}^\varepsilon)$ be as above, and assume that $\varepsilon>0$ is sufficiently small. Then for $1\leq i\leq k$,
\begin{equation}
\hat{f}_\ast\left(\kappa^i(X,\widehat{\mathcal{U}},\hat{\omega},\widehat{\mathcal{S}}^\varepsilon)\right)=(-1)^{ki}(R_k)^i_\ast\left(\hat{f}_\ast(X,\widehat{\mathcal{U}},\hat{\omega},\widehat{\mathcal{S}}^\varepsilon)\right).
\end{equation}
Here, the admissible map $\hat{f}:(X,\widehat{\mathcal{U}})\rightarrow P\times\mathcal{L}_{k+1}$ is given by one of the maps (\ref{eq:ev1}), (\ref{eq:ev2}), (\ref{eq:ev3}), (\ref{eq:ev4}) or (\ref{eq:ev6}).
\end{theorem}
\begin{proof}
We only need to justify the signs. Note that cyclic permutation of the boundary marked points from $z_0,z_1,\cdots,z_k$ to $z_k,z_0\cdots,z_{k-1}$ changes the orientations of the abstract moduli spaces from $-dz_1\wedge\cdots\wedge dz_k$ to $dz_2\wedge\cdots\wedge dz_k\wedge dz_1$ (with $z_0$ fixed at $1$, and neglecting the choices of orientations for the interior marked points $p_1,\cdots,p_l$, which all agree, cf. \cite{sg1}, Remark 46). Thus a sign $(-1)^k$ needs to be introduced for each such permutation.
\end{proof}

\section{Orientations of moduli spaces}\label{section:orientation}

In this appendix, we discuss the orientations of the moduli spaces considered in this paper and compute the signs $\varepsilon_1,\cdots,\varepsilon_{16}$ in Theorem \ref{theorem:moduli}, (iv). As before, we assume that $M$ is a Liouville manifold with $c_1(M)=0$ and $L\subset M$ is a Lagrangian submanifold which is equipped with some fixed choice of a $\mathit{Spin}$ structure relative to the $\mathbb{Z}_2$-gerbe $\alpha$.

The orientation of the moduli space $\overline{\mathcal{R}}_{k+1}(L,\beta)$ has been fixed in \cite{ki2}, Section 7.2.3. We shall orient $\overline{\mathcal{R}}_{k+1,\vartheta}(L,\beta)$ in the same way as $\overline{\mathcal{R}}_{k+1}(L,\beta)$. More precisely, let $\mathcal{R}(L,\beta)$ be the moduli space of pseudoholomorphic discs without boundary marked points, which carries its canonical orientation defined in \cite{fooo1}, Section 8.1. $\overline{\mathcal{R}}_{k+1,\vartheta}(L,\beta)$ is oriented so that the natural isomorphism
\begin{equation}
T\left(\overline{\mathcal{R}}_{k+1,\vartheta}(L,\beta)\right)\oplus T\left(\mathit{Aut}(D)\right)\cong T\left(\overline{\mathcal{R}}(L,\beta)\right)\oplus T(\partial D)^{\oplus k}
\end{equation}
is orientation-preserving, where $\partial D$ is equipped with the anticlockwise orientation, and $\mathit{Aut}(D)$ is oriented via the diffeomorphism $\mathit{Aut}(D)\cong(\partial D)^3$. It follows from the sign computation in the case of $\overline{\mathcal{R}}_k(L,\beta)$ that
\begin{equation}
\varepsilon_1=(k_1-i)(k_2-1)+n+k_1-1
\end{equation}
in (\ref{eq:bd1}).

To compute the signs $\varepsilon_2,\varepsilon_3,\varepsilon_{4,0},\cdots,\varepsilon_{4,l},\varepsilon_5,\varepsilon_6$ in (\ref{eq:bd2}), we first pick orientations for the abstract moduli spaces. On a representative of the moduli spaces ${}_l\mathcal{R}_{k+1}^1$ or ${}_{l-1}\mathcal{R}_{k+1}^{S^1}$ which fixes $z_0=1$ and $\zeta=0$, we pick the volume forms
\begin{equation}
-r_1\cdots r_ldz_1\wedge\cdots\wedge dz_k\wedge dr_l\wedge d\theta_l\wedge\cdots\wedge dr_1\wedge d\theta_1
\end{equation}
and
\begin{equation}
r_1\cdots r_ldz_1\wedge\cdots\wedge dz_k\wedge d\theta_1\wedge dr_l\wedge d\theta_l\wedge\cdots\wedge dr_2\wedge d\theta_2
\end{equation}
for ${}_l\mathcal{R}_{k+1}^1$ and ${}_{l-1}\mathcal{R}_{k+1}^{S^1}$, respectively, where $(r_j,\theta_j)$ is the polar coordinates for the auxiliary marked point $p_j$. For the moduli spaces ${}_l\mathcal{R}_{k+1,\tau_i}^1$, we shall equip them with orientations which are compatible with the auxiliary-rescaling map (\ref{eq:aux-res}). More precisely, for a representative with $z_0=1$ and $\zeta=0$, we pick the volume form
\begin{equation}
r_1\cdots r_ldz_1\wedge\cdots\wedge dz_k\wedge dz_f\wedge dr_l\wedge d\theta_l\wedge\cdots\wedge dr_1\wedge d\theta_1.
\end{equation}
The space of $l$-point angle decorated cylinders admits a canonical complex orientation, and we trivialize the $\mathbb{R}$-action on it by choosing $\partial_s$ as the vector field inducing the action. The quotient gives the moduli space ${}_l\mathcal{M}$. Using this convention, we have an isomorphism
\begin{equation}\label{eq:M}
\langle\partial_s\rangle\otimes\lambda^\mathit{top}\left(T{}_l\overline{\mathcal{M}}(x,y)\right)\cong o_x\otimes o_y^{-1},
\end{equation}
where $o_x$ and $o_y$ are orientation lines associated to the Hamiltonian orbits $x$ and $y$, respectively, and $\lambda^\mathit{top}$ stands for the top degree exterior power. As in the definition of the Floer differential, when counting rigid elements of ${}_l\overline{\mathcal{M}}(x,y)$ in the definition of the operations $\delta_l:\mathit{SC}^\ast(M)\rightarrow\mathit{SC}^{\ast+1-2l}(M)$, we twist by the sign $(-1)^{|y|}$.

The moduli space ${}_l\overline{\mathcal{R}}_{k+1}^1(x,L,\ring{\beta})$ is oriented by combining the ideas of \cite{ma}, Appendix A and \cite{fooo1}, Chapter 8. More precisely, let $D_u$ be the linearization of the perturbed Cauchy-Riemann operator, then we have an isomorphism of line bundles
\begin{equation}\label{eq:R}
\det(D_u)\otimes o_x\otimes\kappa_x^\alpha\cong\lambda^{\mathit{top}}(\mathit{TL}),
\end{equation}
where $\kappa_x^\alpha$ is the \textit{background line} determined by our choice of the relative $\mathit{Spin}$ structure on $L$, the orbit $x$, and the background class $[\alpha]\in H^2(M;\mathbb{Z}_2)$ (cf. \cite{ma}, Definition 3.1). Combining (\ref{eq:M}) and (\ref{eq:R}), it follows that on a boundary stratum of the form ${}_j\overline{\mathcal{M}}(x,y_j)\times{}_{l-j}\overline{\mathcal{R}}(y_j,L,\ring{\beta})$, we have a natural isomorphism
\begin{equation}\label{eq:product}
\lambda^\mathit{top}\left(T{}_{l-j}\overline{\mathcal{R}}(y_j,L,\ring{\beta})\right)\otimes\langle\partial_s\rangle\otimes\lambda^\mathit{top}\left(T{}_j\overline{\mathcal{M}}(x,y_j)\right)\cong\lambda^\mathit{top}(\mathit{TL})\otimes(o_x\otimes\kappa_x^\alpha)^{-1}.
\end{equation}
Comparing the orientation of ${}_j\overline{\mathcal{M}}(x,y_j)\times{}_{l-j}\overline{\mathcal{R}}(y_j,L,\ring{\beta})\subset\partial{}_l\overline{\mathcal{R}}_{k+1}^1(x,L,\ring{\beta})$ with (\ref{eq:product}), we obtain
\begin{equation}
\varepsilon_{4,j}=n+|y_j|.
\end{equation}
For the boundary strata ${}_l^{j,j+1}\overline{\mathcal{R}}_{k+1}^1(x,L,\ring{\beta})$ and ${}_{l-1}\overline{\mathcal{R}}_{k+1}^{S^1}(x,L,\ring{\beta})$, we may orient them so that the natural maps
\begin{equation}
{}_l^{j,j+1}\overline{\mathcal{R}}_{k+1}^1(x,L,\ring{\beta})\rightarrow S^1\times{}_{l-1}\overline{\mathcal{R}}_{k+1}^1(x,L,\ring{\beta})
\end{equation}
and
\begin{equation}
{}_{l-1}\overline{\mathcal{R}}_{k+1}^{S^1}(x,L,\ring{\beta})\rightarrow S^1\times{}_{l-1}\overline{\mathcal{R}}_{k+1}^1(x,L,\ring{\beta})
\end{equation}
induced from the forgetful maps (\ref{eq:fj}) and (\ref{eq:f-S1}) are oriented diffeomorphisms. This provides us with an inductive way to orient the moduli spaces so that
\begin{equation}
\varepsilon_5=\varepsilon_6=0.
\end{equation}

It remains to determine the signs $\varepsilon_2$ and $\varepsilon_3$. For this purpose, we introduce the moduli spaces ${}_l\widetilde{\mathcal{R}}^1(x,L,\ring{\beta})$ of pairs $\left((S,p_1,\cdots,p_l;\ell),u\right)$, where $u:S\rightarrow M$ satisfies (\ref{eq:CR}) and $[u]=\ring{\beta}$, but now we do not quotient by $\mathit{Aut}(S)\cong S^1$. Similarly, define $\widetilde{\mathcal{R}}(L,\beta)$ to be the space of pseudoholomorphic maps $u:(D,\partial D)\rightarrow(M,L)$ with $[u]=\beta$, but without modding out $\mathit{Aut}(D)\cong\mathit{PSL}(2,\mathbb{R})$. Consider the gluing map
\begin{equation}\label{eq:gl}
\mathit{gl}:{}_l\widetilde{\mathcal{R}}^1(x,L,\ring{\beta}_1)\textrm{ }{{}_1\times_{-1}}\textrm{ }\widetilde{\mathcal{R}}(L,\beta_2)\rightarrow{}_l\widetilde{\mathcal{R}}^1(x,L,\ring{\beta}),\textrm{ }\ring{\beta}_1+\beta_2=\ring{\beta},
\end{equation}
where the notation ${}_1\times_{-1}$ means the fiber product on the left-hand side is taken with respect to the evaluation maps at $1\in\partial S$ on the first component, and at $-1\in\partial D$ on the second component. The argument of \cite{fooo1}, Lemma 8.3.10 applies to our case and implies the following:

\begin{lemma}
The map (\ref{eq:gl}) is orientation-preserving.
\end{lemma}

Using this lemma, the signs $\varepsilon_2$ and $\varepsilon_3$ can be computed in the same way as \cite{fooo1}, Section 8.3. We have
\begin{equation}
\varepsilon_2=(k_1-i)(k_2-1)+n+k\textrm{ and }\varepsilon_3=(k_1-i)(k_2-1)+n+1.
\end{equation}

The remaining signs $\varepsilon_7,\cdots,\varepsilon_{13},\varepsilon_{14,0},\cdots,\varepsilon_{14,l},\varepsilon_{15},\varepsilon_{16}$ are determined from the signs $\varepsilon_1,\varepsilon_2,\varepsilon_3,\varepsilon_{4,0},\cdots,\varepsilon_{4,l},\varepsilon_5,\varepsilon_6$ using the formulae
\begin{equation}
\partial([m,m+1]\times X)=\{m+1\}\times X\sqcup(-1)\{m\}\times X\sqcup(-1)[m,m+1]\times\partial X,
\end{equation}
\begin{equation}
[m,m+1]\times(X\times_LY)=([m,m+1]\times X)\times_{[m,m+1]\times L}([m,m+1]\times Y),
\end{equation}
where $X$ and $Y$ are admissible K-spaces.

\Addresses

\newpage

\section{Additional remarks}

The current version of the paper differs slightly from its published version on Selecta Mathematica (N.S.) 32(4), 83 (2026).

The error in the published version was that with the current definition of the complex of degenerate chains $C_\ast^\mathit{nd}$ in Section \ref{section:BV}, the quotient $C_\ast^\mathit{nd}/\mathrm{im}(1-\lambda)$ in equation (3.72) of the published version isn't well-defined, since the cyclic operator $\lambda$ does not preserve the subspace of degenerate chains. The issue was caused by an index error in an earlier arXiv version of this paper, which excludes the possibility $x=\delta_{k,k}(y)$ in the definition of degenerate chains on page 20 of the published version. The issue was then pointed out by an anonymous referee, and was therefore corrected in the published version. However, the author did not notice that this affects the definition of the Connes complex $C_\ast^\lambda$ in (3.72).

The issue is, fortunately, very easy to fix. As we have done in (\ref{eq:Connes}) of this revision, one simply needs to define the Connes' complex as $C_\ast^\lambda:=C_\ast/\mathrm{im}(1-\lambda)$ using the total complex of all chains, instead of just non-degenerate chains. In this way $C_\ast^\lambda$ is no longer a quotient complex of the $S^1$-equivariant complex $C_\ast^{S^1}$ as defined in (\ref{eq:cpx-S1}), but the quasi-isomorphism between $C_\ast$ and its quotient complex of non-degenerate chains $C_\ast^\mathit{nd}$ ensures that there is still a quasi-isomorphism $\phi_\lambda:C_\ast^{S^1}\rightarrow C_\ast^\lambda$ between the $S^1$-equivariant complex and Connes complex, as is confirmed by Lemma \ref{lemma:qis}. The only difference is that this quasi-isomorphism is no longer a quotient map, which does not affect the proofs of the main results. For clarity, the situation is summarized as follows:
\begin{equation}
\begin{tikzcd}
C_\ast(\!(u)\!)/uC_\ast[\![u]\!] \arrow[d,two heads] \arrow[dr,two heads] \\
C_\ast^\mathit{nd}(\!(u)\!)/uC_\ast^\mathit{nd}[\![u]\!] \arrow[r,"\phi_\lambda"'] &C_\ast^\lambda
\end{tikzcd}
\end{equation}
One first constructs the chains using the virtual fundamental chains of the moduli spaces in $C_\ast$, and then projects to $C_\ast^\mathit{nd}$, which is a strict $S^1$-complex, so that one can form the $S^1$-equivariant complex $C_\ast^{S^1}$. We verify these chains satisfy the required properties, and then map them to $C_\ast^\lambda$ using the quasi-isomorphism $\phi_\lambda$, where the chain level string bracket is well-defined, to show that they satisfy Theorem \ref{theorem:chain}, the chain level statement of Theorem \ref{theorem:main}. With this small modification, everything else would then go through.

\end{document}